\documentclass[reqno]{amsart}
\usepackage{hyperref,todonotes,comment,xcolor,xr,fullpage}
\usepackage[color,matrix,arrow]{xy}
\usepackage{graphics,amssymb}
\usepackage{pagecolor,lipsum}
\usepackage{latexsym}
\usepackage{mathrsfs}
\xyoption{curve}

\definecolor{color}{HTML}{AF3235}
\usepackage{hyperref}
\hypersetup{ 
colorlinks,
citecolor=color,
filecolor=color,
linkcolor=color,
urlcolor=color
}

\RequirePackage{tikz-cd}
\RequirePackage{amssymb}
\usetikzlibrary{calc}
\usetikzlibrary{decorations.pathmorphing}

\newtheorem{lemma}{Lemma}[section]
\newtheorem{proposition}[lemma]{Proposition}
\newtheorem{theorem}[lemma]{Theorem}
\newtheorem{corollary}[lemma]{Corollary}

\newtheorem*{theoremA}{Theorem}

\theoremstyle{definition}

\newtheorem{definition}[lemma]{Definition}
\newtheorem{remark}[lemma]{Remark}

\newtheorem{convention}[lemma]{Convention}

\newcommand{\mfk}[1]{\mathfrak{#1}}
\newcommand{\mbb}[1]{\mathbb{#1}}
\newcommand{\mcl}[1]{\mathcal{#1}}
\newcommand{\mrm}[1]{\mathrm{#1}}
\newcommand{\msc}[1]{\mathscr{#1}}
\newcommand{\mbf}[1]{\mathbf{#1}}
\newcommand{\msf}[1]{\mathsf{#1}}
\newcommand{\opn}[1]{\operatorname{#1}}
\newcommand{\ot}{\otimes}
\newcommand{\uBord}{\mathfrak{Bord}}

\DeclareMathOperator{\Hom}{Hom}
\DeclareMathOperator{\End}{End}
\DeclareMathOperator{\Ext}{Ext}
\DeclareMathOperator{\Fun}{Fun}
\DeclareMathOperator{\rep}{rep}
\DeclareMathOperator{\Rep}{Rep}
\DeclareMathOperator{\Spec}{Spec}
\DeclareMathOperator{\Str}{Str}
\DeclareMathOperator{\Ch}{Ch}
\DeclareMathOperator{\QCoh}{QCoh}
\DeclareMathOperator{\Bord}{Bord}
\DeclareMathOperator{\Rib}{Rib}
\DeclareMathOperator{\red}{red}

\renewcommand{\1}{\mathbf{1}}
\renewcommand{\O}{\mathscr{O}}
\renewcommand{\tilde}{\widetilde}
\renewcommand{\hat}{\widehat}

\definecolor{page_color}{HTML}{000000}
\definecolor{text_color}{HTML}{F0EAD6}

\makeindex

\title[Cochain valued TQFT]{Cochain valued TQFTs from nonsemisimple modular tensor categories}
\date{\today}

\author{Agustina Czenky}
\thanks{The first author was partially supported by the Simons Collaboration Grant No. 999367. The second author was supported by NSF
CAREER Grant No. DMS-2239698, and Simons Collaboration Grant No. 999367.}
\address{Department of Mathematics, University of Southern California, Los Angeles, CA 90007}
\email{czenky@usc.edu}

\author{Cris Negron}
\address{Department of Mathematics, University of Southern California, Los Angeles, CA 90007}
\email{cnegron@usc.edu}

\begin{document}

\maketitle

\begin{abstract}
We show that a vector space valued TQFT constructed in work of De Renzi et al.\ \cite{derenzietal23} extends naturally to a topological field theory which takes values in the symmetric monoidal category of linear cochains.  Specifically, we consider a bordism category whose objects are surfaces with markings from the category of cochains $\opn{Ch}(\msc{A})$ over a given modular tensor category (such as the category of small quantum group representations), and whose morphisms are $3$-dimensional bordisms with embedded ribbon graphs traveling between such marked surfaces.  We construct a symmetric monoidal functor from the aforementioned ribbon bordism category to the category of linear cochains.  The values of this theory on surfaces are identified with Hom complexes for $\opn{Ch}(\msc{A})$, and the $3$-manifold invariants are alternating sums of the renormalized Lyubashenko invariant from \cite{derenzietal23}.
\par

We show that our cochain valued TQFT furthermore preserves homotopies, and hence localizes to a theory which takes values in the derived $\infty$-category of dg vector spaces.  The domain for this $\infty$-categorical theory is, up to some approximation, an $\infty$-category of ribbon bordism with labels in the homotopy $\infty$-category $\mcl{K}(\msc{A})$.  We suggest our localized theory as a starting point for the construction of a ``derived TQFT" for the $\infty$-category of derived quantum group representations.
\end{abstract}

\tableofcontents

\section{Introduction}

\subsection{Background}
\label{sect:background}

Consider a finite modular tensor category $\msc{A}$, aka a non-degenerate ribbon tensor category which is additionally finite.  In \cite{derenzietal23} the authors exhibit an (anomalous) $3$-dimensional topological field theory
\begin{equation}\label{eq:113}
Z_{\msc{A}}:\opn{Bord}^{\opn{adm}}_{\msc{A}}=\opn{Bord}_{3,2}(\msc{A})^{\opn{adm}}\to \opn{Vect}
\end{equation}
which, in the case where $\msc{A}$ is semisimple, recovers the usual Reshetikhin-Turaev TQFT.  Though this theory has its quirks, it provides one of the most functional topological theories to date which is constructed from a nonseimisimple modular category, and which submits to a direct analysis.\footnote{There are some other equally explicit nonsemisimple theories in dimensions $3$ and $4$ \cite{costantinoetalII,costantinoetal} (cf.\ \cite{brochieretal21}), but they should all be ``the same" after appropriate translation \cite{balsam,haioun25}.}
\par

The authors of the text \cite{derenzietal23} produce the field theory $Z_{\msc{A}}$ via a universal procedure which begins in dimension three with a normalization of Lyubashenko's invariant, and reproduces Lyubashenko's mapping class group actions in dimension two. In this way $Z_{\msc{A}}$ realizes Lyubashenko's various findings outside of the semisimple setting \cite{lyubashenko95,lyubashenko95II,kerlerlyubashenko01} as projections of a single uniform theory.
\par

In the present text we show that the functor \eqref{eq:113} extends naturally to a TQFT $Z_{\opn{Ch}(\msc{A})}$ which takes values not in vector spaces, but in the symmetric tensor category $\opn{Ch}(\opn{Vect})$ of linear cochains.  This cochain valued theory furthermore respects (certain classes of) homotopies, and hence localizes to provide a type of \emph{homotopical} TQFT, in the higher categorical sense of the term.  This final point is discussed in more detail in Sections \ref{sect:loc_intro} and \ref{sect:zloc} below.
\par

While this paper can certainly be enjoyed as a stand-alone project, our motivations come from the proposed existence of a ``derived" field theory for modular tensor categories.  This derived TQFT should take as input the bounded derived $\infty$-category $\mathcal{D}(\msc{A})$ and should take values in the derived $\infty$-category of linear cochains.  The state spaces in this theory should be infinite-dimensional, and furthermore identified with higher Hochschild cohomologies for the category $\msc{A}$. One also expects that the functor \eqref{eq:113} is recovered as the $0$-th cohomology of its derived counterpart, when interpreted appropriately.
\par

Though we won't provide a detailed accounting of this mythical derived theory, it can be understood physically, in the case where $\msc{A}$ is a category of quantum group representations $(\Rep_q G)_{\opn{small}}$, as a TQFT one obtains via Kapustin-Witten twisting of an associated supersymmetric QFT for the group $G$.  (See \cite{gaiotto19,costellogaiotto19,creutzigetal24} for further elaborations.)  The results of our text provide a vector which points in this derived direction.

\subsection{Results}

We fix a finite modular tensor category $\msc{A}$ over an algebraically closed base field $k$.  As standard examples, one might consider a category of small quantum group representations \cite{gainutdinovlentnerohrmann,negron26}, or of modules over a $C_2$-cofinite vertex operator algebra \cite{mcrae}.
\par

We consider an anomalous bordism category $\opn{Bord}_{\opn{Ch}(\msc{A})}=\Bord_{3,2}(\opn{Ch}(\msc{A}))$ whose objects are oriented surfaces $\Sigma$ with framed marked points, and over each marked point we furthermore specify an object in the category of bounded cochains $\opn{Ch}(\msc{A})$.  Morphisms in $\Bord_{\opn{Ch}(\msc{A})}$ are $3$-dimensional bordisms $M:\Sigma\to \Sigma'$ which come equipped with an embedded, $\msc{A}$-labeled, framed graph $\gamma:\Gamma\to M$ whose intersection with the boundary recovers the given markings on $\Sigma$ and $\Sigma'$.\footnote{There are additional geometric data here which have to do with the anomaly. See Section \ref{sect:bords} for details.}
\par

Inside of the full bordism category $\opn{Bord}_{\opn{Ch}(\msc{A})}$ is a non-full subcategory
\[
\Bord_{\opn{Ch}(\msc{A})}^{\opn{adm}}\ \subseteq\ \Bord_{\opn{Ch}(\msc{A})}
\]
of so-called \emph{admissible} bordisms \cite{derenzietal23}.  This is a symmetric monoidal subcategory which contains all marked surfaces, i.e.\ all objects, all $3$-manifolds with nonempty outgoing boundary, and all closed $3$-manifolds which satisfy a certain algebraic projectivity constraint (see Section \ref{sect:admis}).
\par

\begin{remark}
In terms of unmarked objects, the category $\Bord_{\opn{Ch}(\msc{A})}^{\opn{adm}}$ does contain the usual ``non-compact" bordism category of unmarked surfaces and unmarked $3$-manifolds with nonempty outgoing boundary \cite{lurie08}. The introduction of markings is both natural and necessary from the conformal perspective \cite{fuchsrunkelschweigert02}, for example, and also from a skein theoretic perspective in dimension $4$ \cite{haioun25}.
\end{remark}

As a primary result of our text, we prove the following.

\begin{theoremA}[\ref{thm: Z_Ch}]
Let $\msc{A}$ be a modular tensor category and $\opn{Ch}(\msc{A})$ be the associated braided tensor category of finite length cochains.  There is a symmetric monoidal functor
\begin{equation}\label{eq:150}
Z_{\opn{Ch}(\msc{A})}:\Bord^{\opn{adm}}_{\opn{Ch}(\msc{A})}\to \opn{Ch}(\opn{Vect})
\end{equation}
whose value on a genus $g$ marked surface $\Sigma_{\vec{x}}$ is naturally identified with the Hom complex
\[
Z_{\opn{Ch}(\msc{A})}(\Sigma_{\vec{x}})\cong \Hom^\ast_{\msc{A}}(\1,\opn{Ad}^{\ot g}\otimes(\otimes_{i\in I}x_i^{\pm})).
\]
\end{theoremA}

In the above expression each $x_i^{\pm}$ is $x_i$ when the framing at the $i$-th point in $\Sigma$ reproduces the ambient orientation on $\Sigma$, and $x_i^\ast$ otherwise.  The object $\opn{Ad}$ is the canonical end.  In the case where $\msc{A}=\opn{rep}(A)$ for a Hopf algebra $A$, for example, this object is just the adjoint representation $\opn{Ad}=A$.

The most important aspect of the theory \eqref{eq:150} is, at least from our perspective, the fact that it respects certain notions of homotopy equivalence.  Let us take a moment to speak to this point clearly.
\par

Consider $\Sigma_{\vec{x}}$ a surface $\Sigma$ with $n$ marked points, and marking objects $x_1,\dots,x_n$. Let $\Sigma_{\vec{y}}$ be the same underlying marked surface but with the marking objects $x_i$ replaced by new objects $y_i$.  An algebraic homotopy equivalence $\Sigma_{\vec{x}}\to \Sigma_{\vec{y}}$ is a ribbon bordism whose underlying topology is the product $\Sigma\times I$, along with ``straight lines" traveling from the markings on $\Sigma\times\{0\}$ to the respective markings on $\Sigma\times\{1\}$, and in which these straight lines are labeled by homotopy equivalences between the $x_i$ and $y_i$ in $\opn{Ch}(\msc{A})$.  A generic example of such an equivalence might appear as follows:
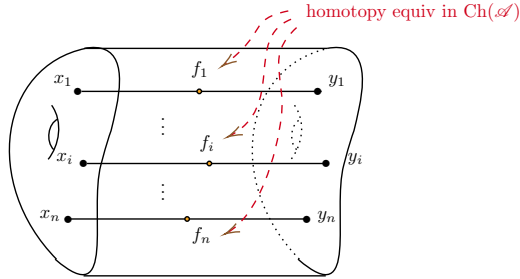
\[
\scalebox{.7}{
\tikzset{every picture/.style={line width=0.75pt}} 
\begin{tikzpicture}[x=0.75pt,y=0.75pt,yscale=-1,xscale=1]
\draw    (73.5,216) .. controls (19.5,161) and (32.5,77) .. (94.5,64) ;
\draw    (93.5,64) .. controls (118.66,59.43) and (122.32,79.26) .. (116.5,107) .. controls (110.68,134.74) and (114.5,114) .. (105.5,148) .. controls (96.5,182) and (117.5,255) .. (72.5,216) ;
\draw    (70,105) .. controls (78.5,117) and (75.5,134) .. (67.5,144) ;
\draw    (71.5,137) .. controls (65.5,131) and (67.5,118) .. (74.5,115) ;
\draw  [dash pattern={on 0.84pt off 2.51pt}]  (248.5,216) .. controls (194.5,161) and (207.5,77) .. (269.5,64) ;
\draw    (268.5,64) .. controls (293.66,59.43) and (297.32,79.26) .. (291.5,107) .. controls (285.68,134.74) and (289.5,114) .. (280.5,148) .. controls (271.5,182) and (292.5,255) .. (247.5,216) ;
\draw  [dash pattern={on 0.84pt off 2.51pt}]  (245,105) .. controls (253.5,117) and (250.5,134) .. (242.5,144) ;
\draw  [dash pattern={on 0.84pt off 2.51pt}]  (246.5,137) .. controls (240.5,131) and (242.5,118) .. (249.5,115) ;
\draw    (94.5,64) -- (268.5,64) ;
\draw    (91.5,95) -- (264.5,95) ;
\draw    (95.5,147) -- (268.5,147) ;
\draw    (79.5,187) -- (252.5,187) ;
\draw    (93,228) -- (267.5,228) ;
\draw  [fill={rgb, 255:red, 245; green, 166; blue, 35 }  ,fill opacity=1 ] (174.75,95) .. controls (174.75,94.1) and (175.48,93.38) .. (176.38,93.38) .. controls (177.27,93.38) and (178,94.1) .. (178,95) .. controls (178,95.9) and (177.27,96.63) .. (176.38,96.63) .. controls (175.48,96.63) and (174.75,95.9) .. (174.75,95) -- cycle ;
\draw  [fill={rgb, 255:red, 245; green, 166; blue, 35 }  ,fill opacity=1 ] (182,147) .. controls (182,146.1) and (182.73,145.38) .. (183.63,145.38) .. controls (184.52,145.38) and (185.25,146.1) .. (185.25,147) .. controls (185.25,147.9) and (184.52,148.63) .. (183.63,148.63) .. controls (182.73,148.63) and (182,147.9) .. (182,147) -- cycle ;
\draw  [fill={rgb, 255:red, 245; green, 166; blue, 35 }  ,fill opacity=1 ] (166,187) .. controls (166,186.1) and (166.73,185.38) .. (167.63,185.38) .. controls (168.52,185.38) and (169.25,186.1) .. (169.25,187) .. controls (169.25,187.9) and (168.52,188.63) .. (167.63,188.63) .. controls (166.73,188.63) and (166,187.9) .. (166,187) -- cycle ;
\draw [color={rgb, 255:red, 208; green, 2; blue, 27 }  ,draw opacity=1 ] [dash pattern={on 4.5pt off 4.5pt}]  (240.5,37) .. controls (210.95,42.91) and (226.03,52.7) .. (194,75.93) ;
\draw [shift={(192.5,77)}, rotate = 324.78] [color={rgb, 255:red, 139; green, 87; blue, 42 }  ,draw opacity=1 ][line width=0.75]    (10.93,-3.29) .. controls (6.95,-1.4) and (3.31,-0.3) .. (0,0) .. controls (3.31,0.3) and (6.95,1.4) .. (10.93,3.29)   ;
\draw [color={rgb, 255:red, 208; green, 2; blue, 27 }  ,draw opacity=1 ] [dash pattern={on 4.5pt off 4.5pt}]  (242.5,43) .. controls (207.85,56.86) and (239.85,106) .. (197.8,127.36) ;
\draw [shift={(196.5,128)}, rotate = 334.49] [color={rgb, 255:red, 139; green, 87; blue, 42 }  ,draw opacity=1 ][line width=0.75]    (10.93,-3.29) .. controls (6.95,-1.4) and (3.31,-0.3) .. (0,0) .. controls (3.31,0.3) and (6.95,1.4) .. (10.93,3.29)   ;
\draw [color={rgb, 255:red, 208; green, 2; blue, 27 }  ,draw opacity=1 ] [dash pattern={on 4.5pt off 4.5pt}]  (246.5,49) .. controls (223.5,70) and (238.84,84.49) .. (235.5,113) .. controls (232.23,140.94) and (219.34,184.61) .. (195.95,197.26) ;
\draw [shift={(194.5,198)}, rotate = 334.49] [color={rgb, 255:red, 139; green, 87; blue, 42 }  ,draw opacity=1 ][line width=0.75]    (10.93,-3.29) .. controls (6.95,-1.4) and (3.31,-0.3) .. (0,0) .. controls (3.31,0.3) and (6.95,1.4) .. (10.93,3.29)   ;

\draw (77,183) node [anchor=north west][inner sep=0.75pt]  [align=left] {$\displaystyle \bullet $};
\draw (88,143) node [anchor=north west][inner sep=0.75pt]  [align=left] {$\displaystyle \bullet $};
\draw (84,91) node [anchor=north west][inner sep=0.75pt] [align=left] {$\displaystyle \bullet $};
\draw (249,183) node [anchor=north west][inner sep=0.75pt]   [align=left] {$\displaystyle \bullet $};
\draw (263,143) node [anchor=north west][inner sep=0.75pt] [align=left] {$\displaystyle \bullet $};
\draw (257,91) node [anchor=north west][inner sep=0.75pt]  [align=left] {$\displaystyle \bullet $};
\draw (70,82) node [anchor=north west][inner sep=0.75pt]   [align=left] {$\displaystyle x_{1}$};
\draw (72,138) node [anchor=north west][inner sep=0.75pt]   [align=left] {$\displaystyle x_{i}$};
\draw (60,179) node [anchor=north west][inner sep=0.75pt]   [align=left] {$\displaystyle x_{n}$};
\draw (147.04,105) node [anchor=north west][inner sep=0.75pt]  [rotate=-0.28] [align=left] {$\displaystyle \vdots $};
\draw (147.03,152) node [anchor=north west][inner sep=0.75pt]  [rotate=-0.16] [align=left] {$\displaystyle \vdots $};
\draw (268,82) node [anchor=north west][inner sep=0.75pt]   [align=left] {$\displaystyle y_{1}$};
\draw (282,138) node [anchor=north west][inner sep=0.75pt]   [align=left] {$\displaystyle y_{i}$};
\draw (260,180) node [anchor=north west][inner sep=0.75pt]   [align=left] {$\displaystyle y_{n}$};
\draw (169,72) node [anchor=north west][inner sep=0.75pt]   [align=left] {$\displaystyle f_{1}$};
\draw (175,125) node [anchor=north west][inner sep=0.75pt]   [align=left] {$\displaystyle f_{i}$};
\draw (169.63,191.63) node [anchor=north west][inner sep=0.75pt]   [align=left] {$\displaystyle f_{n}$};
\draw (252,30) node [anchor=north west][inner sep=0.75pt]  [color={rgb, 255:red, 208; green, 2; blue, 27 }  ,opacity=1 ] [align=left] {homotopy equiv in $\opn{Ch}(\msc{A})$};

\end{tikzpicture}}.
\]
We consider the collection $W_{\opn{GT};\msc{A}}$ of all such algebraic homotopies in $\Bord_{\msc{A}}$, and the usual collection $W_{\opn{Vect}}$ of homotopy equivalences in $\opn{Ch}(\opn{Vect})$.\footnote{The subscript GT stands for geometrically trivial.}

\begin{theoremA}[\ref{thm:zhtop}]
The field theory $Z_{\opn{Ch}(\msc{A})}$ preserves homotopy equivalence,
\[
Z_{\opn{Ch}(\msc{A})}(W_{\opn{GT};\msc{A}})\ \subseteq\ W_{\opn{Vect}}.
\]
\end{theoremA}

We describe some implications of this result below.

\subsection{Localization and proposed derivation}\label{sect:loc_intro}

Below $W$ denotes either $W_{\opn{GT};\msc{A}}$ in the topological setting, or $W_{\opn{Vect}}$ in the linear setting.  By Theorem \ref{thm:zhtop} we understand that the functor $Z_{\opn{Ch}(\msc{A})}$ preserves homotopy equivalence.  Hence we obtain an induced functor on the localizations
\begin{equation}\label{eq:182}
Z_{\opn{Ch}(\msc{A})}[W^{-1}]:\opn{Bord}^{\opn{adm}}_{\opn{Ch}(\msc{A})}[W^{-1}]\to \opn{Ch}(\opn{Vect})[W^{-1}],
\end{equation}
where both the domain and codomain are now symmetric monoidal \emph{$\infty$-categories} \cite{hinich16}.
\par

On the right-hand side of the above functor, the localization $\opn{Ch}(\opn{Vect})[W^{-1}]$ is the $\infty$-category $\mcl{V}ect$ of dg vector spaces, or homotopical vector spaces.  This category can be constructed explicitly as the dg nerve of the dg category of (bounded) linear cochains, or as the stabilization of the $\infty$-category of $k$-linear Kan complexes, and its homotopy category recovers the usual bounded derived category of vector spaces.
\par

On the left hand side, the forgetful functor $\opn{Bord}^{\opn{adm}}_{\opn{Ch}(\msc{A})}\to \Bord_{\ast}$ to the unlabeled $3$-dimensional ribbon bordism category factors through the localization to give a symmetric monoidal functor $\opn{Bord}^{\opn{adm}}_{\opn{Ch}(\msc{A})}[W^{-1}]\to \Bord_{\ast}$, so that the localization realizes \emph{some} symmetric monoidal $\infty$-category of decorated bordisms.  Furthermore, the fiber of this functor over any marked surface is an $\infty$-category which is isomorphic to a cartesian power of the (bounded) homotopy $\infty$-category $\mcl{K}(\msc{A})$.  So, up to some approximation, the localization is an $\infty$-category
\[
\Bord_{\mcl{K}(\msc{A})}^{\opn{adm}}\approx\opn{Bord}^{\opn{adm}}_{\opn{Ch}(\msc{A})}[W^{-1}]
\]
of bordisms with labels in the homotopy $\infty$-category for $\msc{A}$.  Hence, again up to an approximation, the functor \eqref{eq:182} provides a field theory
\begin{equation}\label{eq:200}
Z_{\mcl{K}(\msc{A})}:\Bord_{\mcl{K}(\msc{A})}^{\opn{adm}}\to \mcl{V}ect
\end{equation}
from an $\infty$-category of $\mcl{K}(\msc{A})$-labeled $3$-dimensional bordisms to the $\infty$-category of homotopical vector spaces.
\par

We would claim further that the theory \eqref{eq:200} can be ``derived" to obtain a topological theory
\begin{equation}\label{eq:208}
Z_{Anticipated}:\Bord^{\opn{nc}}_{\mcl{D}(\msc{A})}\to \mcl{V}ect
\end{equation}
from a category of bordisms with labels in the derived $\infty$-category for $\msc{A}$. An important point here is that, in terms of state spaces for example, the values of the derived theory on unmarked surfaces correspond to values of the non-derived theory on surfaces marked by resolutions of the unit.  So there is a material change in the state spaces.
\par

In the case of a quantum group $\msc{A}=(\Rep_q G)_{\opn{small}}$, one can see that this derivation behaves as expected when compared with the Kapustin-Witten twist discussed in Section \ref{sect:background} above, at least at some superficial level.\footnote{In the case $\msc{A}=(\Rep_q G)_{\opn{small}}$, the twisted theory should also have nontrivial constributions from $G^{\opn{Lan}}$-local systems over manifolds in all dimensions, where $G^{\opn{Lan}}$ is the Langlands dual group to $G$.   We expect the appearance of local systems to come directly from a canonical $G^{\opn{Lan}}$-crossed extension of the category $(\Rep_q G)_{\opn{small}}$ which has yet to appear in the literature, as in Turaev's HTQFTs \cite{turaev10} for example.  Also, the ``nc" superscript in \eqref{eq:208} indicates our use of the noncompact bordism category of $3$-manifolds with nonvanishing outgoing boundary.}  So, from Theorems \ref{thm: Z_Ch} and \ref{thm:zhtop}, one can pursue a clear path towards the kinds of homotopical TQFTs which one anticipates from the physical perspective.
\par

An explicit construction of the theory \eqref{eq:208}, following the outline described above, is work in progress \cite{negron}.

\subsection{Structure}

In Section \ref{sect:cats} we cover categorical backgrounds.  Sections \ref{sect:ribbon}--\ref{sect:ev} recall the construction of the ribbon bordism category $\Bord_{3,2}(\msc{A})=\Bord_{\msc{A}}$ and the introduction of skein relations.
\par

Section \ref{sect:separate} introduces a ``separation functor" which produces more-or-less a functor from $\Bord_{\msc{S}\ot\msc{A}}$ to the product $\Bord_{\msc{S}}\times\Bord_{\msc{A}}$ whenever $\msc{S}$ is semisimple.  Hence we ``separate" surfaces and bordisms labeled by the product category $\msc{S}\otimes\msc{A}$ into an $\msc{S}$-factor and an $\msc{A}$-factor.  We use this functor in Section \ref{sect:t_ext} to produce a base changed TQFT
\[
Z(\msc{S}):\Bord_{\msc{S}\ot\msc{A}}\to \msc{S}
\]
from any reasonably well-behaved TQFT $Z:\Bord_{\msc{A}}\to \opn{Vect}$. In Section \ref{sect:state_spaces} we consider the theory $Z_{\msc{A}}$ from De-Renzi, Gainutdinov, Geer, Patureau-Mirand, and Runkel, and calculate the state spaces for the base changed theory $Z_{\msc{A}}(\msc{S})$ via inner-Homs for the action of $\msc{S}$ on $\msc{S}\ot\msc{A}$.
\par

In Section \ref{sect:outline} we outline the construction of the cochain valued theory $Z_{\opn{Ch}(\msc{A})}$ from the base changed theory $Z_{\msc{A}}(\opn{Vect}^{\mathbb{Z}})$ at the category $\opn{Vect}^{\mathbb{Z}}$ of $\mbb{Z}$-graded vector spaces.  We follow up with an explicit construction of $Z_{\opn{Ch}(\msc{A})}$ in Sections \ref{sec: diff def}--\ref{sect:states_parts}. The state spaces and partition functions for the theory $Z_{\opn{Ch}(\msc{A})}$ are described in Section \ref{sect:states_parts} as well. Finally, in Section \ref{sect:htop} we show that the theory $Z_{\opn{Ch}(\msc{A})}$ respects algebraic homotopy equivalence, and we discuss the localization \eqref{eq:182} in relatively explicit terms.

\subsection{Acknowledgements}

Thanks to Ingo Runkel for suggesting \cite{runkelszegedywatts23} as a model for Theorem \ref{thm:Z_AS}, and as a precursor to Theorem \ref{thm:Z_AS}.  Thanks to Jennifer Brown and Nathan Geer for various conversations about low-dimensional TQFTs which aided in our general framing of the topic.  Czenky was partially supported by Simons Collaboration Grant 99367.  Negron was supported by NSF CAREER Grant DMS-2239698 and Simons Collaboration Grant 999367.

\section{Categories and monoidal structures}
\label{sect:cats}

Let $k$ be an algebraically closed field.  All linear structures are $k$-linear.  For example, vector spaces are $k$-vector spaces and linear categories are $k$-linear categories.

\subsection{Linearizations and additive closure}
\label{sect:add}

For a non-linear category $\mbb{B}$ we let $k\mbb{B}$ denote its free linearization, which has objects $\opn{ob}(\mbb{B})$ and morphisms $\Hom_{k\mbb{B}}(x,y)=k\Hom_{\mbb{B}}(x,y)$.  Composition is the unique bilinear map which recovers composition on $\mbb{B}$ when restricting to the basis $\Hom_{\mbb{B}}(x,y)\subseteq \Hom_{k\mbb{B}}(x,y)$.
\par

For a linear, but not necessarily additive category $\msc{B}$ we let $\msc{B}^{\opn{add}}$ denote the additive category whose objects are maps $v:\mbb{Z}_{>0}\to \opn{ob}(\msc{B})\amalg \{0\}$ which take value $0$ at all but finitely many integers. Given such $v$ we take $v_i=v(i)$.  Morphisms in $\msc{B}^{\opn{add}}$ are given by infinite matrices
\[
\Hom_{\msc{B}^{\opn{add}}}(v,w)=\oplus_{i,j>0}\Hom_{\msc{B}}(v_j,w_i)
\]
and composition is defined via matrix multiplication
\[
f'f=\sum_{ij} (\sum_l f'_{il}f_{lj}).
\]
For $v$ and $w$ in $\msc{B}^{\opn{add}}$ we have the biproduct $v\oplus w$ with
\[
(v\oplus w)_i=\left\{\begin{array}{ll} v_{i/2} & \text{if $i$ is even}\\
w_{(i+1)/2} & \text{if $i$ is odd}.
\end{array}
\right.
\]

By an abuse of notation we write $v=\sum_{i=1}^N v_i$ when $v$ is in $\msc{B}^{\opn{add}}$ with all $v_n=0$ at $n>N$.

\begin{remark}
If $\msc{B}$ is already additive then the inclusion $\msc{B}\to \msc{B}^{\opn{add}}$ is an equivalence.  In particular, each formal sum $\sum_{i=1}^Nv_i$ in $\msc{B}^{\opn{add}}$ is isomorphic to the internal sum $\oplus_{i=1}^N v_i$ in $\msc{B}\subseteq \msc{B}^{\opn{add}}$ via the apparent maps.
\end{remark}

\subsection{Ribbon monoidal categories}

In this text we consider both linear and nonlinear monoidal categories.  Almost all monoidal categories are braided and rigid.  For basics on monoidal categories we refer the reader to any standard reference, e.g.\ \cite{bakalovkirillov01,egno15}.
\par

A twist on a braided monoidal categroy $\msc{A}$, with braiding $c$, is a choice of an endomorphism $\theta$ of the identity functor $id_{\msc{A}}$ for which
\[
\theta_{x\ot y} = c_{y,x}c_{x,y}(\theta_x\ot \theta_y).
\]
In the event that $\msc{A}$ is additionally rigid, a ribbon structure on $\msc{A}$ is a choice of a twist $\theta$ for which $(\theta_x)^{\ast} = \theta_{x^{\ast}}$.  We note that any \emph{symmetric} tensor category becomes ribbon under the trivial twist $\theta_x=id_x$.

Any ribbon category is naturally pivotal as well, with pivotal structure $\rho:id_{\msc{A}}\to (-)^{\ast \ast}$ given by the composite $u\theta$, where $u$ is the Drinfeld isomorphism \cite[Sections 8.9, 8.10]{egno15}.  Via the pivotal structure, the left dual $x^\ast$ of any object is naturally a right dual as well
\[
{^\ast x}\overset{\rho_x}\to (({^{\ast}x})^{\ast})^{\ast}=x^{\ast}.
\]
Hence the duality structure on $\msc{A}$ can be expressed via a single anti-equivalence $-^{\ast}:\msc{A}^{\opn{op}}\to \msc{A}$.

\subsection{Tensor categories and modularity}

By a tensor category we mean a rigid, linear, abelian monoidal category which is furthermore locally finite and in which the unit object is simple \cite{egno15}. (A locally finite category is an essentially small, linear, Noetherian and Artinian abelian category in which all Hom spaces are finite-dimensional.) A tensor category $\msc{A}$ is called finite if it has only finitely many simples and enough projectives.  Equivalently, $\msc{A}$ is finite if it is identified, as a linear abelian category, with representations for a finite-dimensional algebra.  We let $\opn{Vect}$ denote the tensor category of finite-dimensional vector spaces.  We always consider $\opn{Vect}$ along with its (unique) symmetric monoidal structure.
\par

For any braided monoidal category $\msc{A}$ we have the M\"uger center $Z_2(\msc{A})$ in $\msc{A}$.  This is the full subcategory of objects $z$ in $\msc{A}$ at which the double braiding is trivial, $c_{-,z}c_{z,-}=id_{z\ot-}$.

\begin{definition}[\cite{shimizu19}]
A finite braided tensor category $\msc{A}$ is called non-degenerate if its M\"uger center is trivial, $\opn{Vect}\cong Z_2(\msc{A})$.  A (finite) modular tensor category is a finite, non-degenerate, ribbon tensor category.
\end{definition}

\subsection{Internal tensor products of categories}
\label{sect:int_prod}

We are interested in two kinds of ``tensor products" for categories. The first is a kind of internal tensor product for arbitrary linear categories.

\begin{definition}
For linear categories $\msc{A}$ and $\msc{B}$, take $\msc{A}\boxtimes\msc{B}$ to be the linear category whose objects are the same as those of $\msc{A}\times \msc{B}$ and whose morphisms are given by tensoring over $k$,
\[
\Hom_{\msc{A}\boxtimes\msc{B}}((a_0,b_0),(a_1,b_1)):=\Hom_{\msc{A}}(a_0,a_1)\otimes_k\Hom_{\msc{B}}(b_0,b_1).
\]
\end{definition}

The category $\msc{A}\boxtimes \msc{B}$ is the universal linear category with a functor $\omega:\msc{A}\times \msc{B}\to \msc{M}$ for which the maps on hom spaces
\[
\Hom_{\msc{A}}(a_0,a_1)\times \Hom_{\msc{B}}(b_0,b_1)\to \Hom_{\msc{M}}(\omega(a_0,b_0),\omega(a_1,b_1))
\]
are bilinear.  When $\msc{A}$ and $\msc{B}$ are linear (braided) monoidal, the product $\msc{A}\boxtimes \msc{B}$ inherits a unique linear (braided) monoidal structure for which the structure map $\msc{A}\times\msc{B}\to \msc{A}\boxtimes\msc{B}$ is a (braided) monoidal functor.  This monoidal structure appears in the expected way
\[
(a,b)\otimes(a',b'):=(a\ot a',b\ot b').
\]

\subsection{Kelley products}

Given locally finite linear abelian categories $\msc{A}$ and $\msc{B}$ we also consider the Kelly \cite{kelly82}, or Deligne tensor product.

\begin{definition}
For locally finite linear abelian categories $\msc{A}$ and $\msc{B}$, the Kelley tensor product $\msc{A}\ot\msc{B}=\msc{A}\ot_{\opn{Vect}}\msc{B}$ is the universal locally finite linear abelian category which receives a functor
\[
\opn{univ}:\msc{A}\times \msc{B}\to \msc{A}\ot \msc{B}
\]
which is right exact in each factor, and which comes equipped with an associative natural isomorphism
\[
\opn{univ}(a\ot_kv,b)\cong \opn{univ}(a,v\ot_kb)
\]
at each finite-dimensional vector space $v$.
\end{definition}

To see that the Kelly product exists, we note that such linear categories can be written as corepresentations $\msc{A}\cong \opn{corep}(A)$, $\msc{B}\cong \opn{corep}(B)$, for coalgebras $A$ and $B$ \cite[Theorem 5.1]{takeuchi77}.  Now for such corepresentation categories one can check that representations of the tensor product of coalgebras realizes the Kelley tensor product,
\[
\msc{A}\otimes\msc{B}=\opn{corep}(A\ot_kB).
\]
In this case the universal functor is just given by tensoring corepresentations $\opn{univ}(a,b)=a\ot b$.  See for example \cite[Lemma 3.2.3]{coulembierflake}, or \cite{douglasspsnyder19} in the case where one of $\msc{A}$ or $\msc{B}$ is finite.

\begin{remark}
Note that indization provides an equivalence between locally finite linear abelian categories and compactly generated presentable linear abelian categories whose subcategories of compact objects are locally finite. We use this equivalence implicitly in order to speak of the Kelly tensor product of locally finite categories (cf.\ \cite{kelly82,chirvasitujfreyd13}).
\end{remark}

Again, when $\msc{A}$ and $\msc{B}$ are (braided) tensor categories the Kelly product $\msc{A}\ot\msc{B}$ admits a unique (braided) tensor structure under which the universal map from $\msc{A}\times\msc{B}$ is (braided) monoidal.

\section{Ribbon graphs in manifolds}
\label{sect:ribbon}

The next two sections are dedicated to a relatively careful presentation of our ribbon bordism category.  When compared with other works on this topic, e.g.\ \cite{derenzietal23,brownhaioun,haioun25}, we are extraordinarily verbose.  So, by all indications, the familiar reader might only skim their contents.

\subsection{Embedding ribbon graphs}
\label{sect:M_rib}

Let $M$ be an oriented $3$-manifold with boundary.  By a (finite, topological) graph we mean a compact space $\Gamma$ with a finite collection of points $\opn{vert}(\Gamma)\subseteq \Gamma$ for which the complement $\Gamma-\opn{vert}(\Gamma)$ is a finite collection of open, oriented intervals with specified smooth structure.  The edges in $\Gamma$ are the intervals $\Gamma-\opn{vert}(\Gamma)$.  When we speak of the edges as a set, we mean the connected components $\opn{edge}(\Gamma)=\pi_0(\Gamma-\opn{vert}(\Gamma))$.
\par

A ribbon graph in $M$ is a continuous embedding $\gamma:\Gamma\to M$ from a directed graph which is smooth away from the vertices, and which comes equipped with a smooth framing along the edges and vertices independently.  (By a framing we mean a trivialization of the pullback of the tangent bundle for $M$.)  Equivalently, we choose a map to the third power of the tangent bundle $\gamma^{fr}:\Gamma\to TM\times_MTM\times_MTM$ with nonvanishing determinant, in which case the projections $\Gamma\to TM$ recover the $x$, $y$, and $z$-vectors for the given framing.  We require that the orientation provided by the framing along $\Gamma$ agrees with the ambient orientation on $M$, and that the $y$-vector for the framing is tangent to $\gamma$ and pointed in the positive direction of travel along $\Gamma$.
\par

We institute a certain smoothness condition at the vertices as well.  Namely, at a given vertex $v$, we require that the $x$ and $y$-vectors along each edge $e$ converge to vectors in the $(x,y)$-plane for the framing at $v$, that the $z$-vectors along each edge converge to the $z$-vector at $v$, and finally that the $y$-vectors along each edge converge to vectors $\vec{y}_e$ with nontivial $y$-component for the framing at $v$.  (This is to say, edges are not allowed to enter along lines which are tangent to the $x$-axis at $v$.)  We require furthermore that the entry vectors $\vec{y}_e$ all point in distinct directions in the $(x,y)$-plane, so that all edges enter at distinct angles.
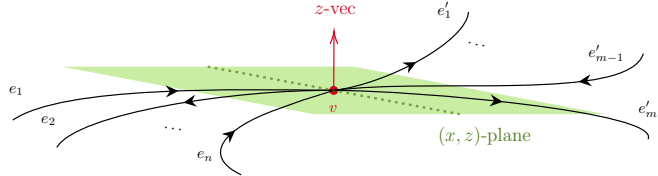
\begin{equation*}	\scalebox{0.6}{
\tikzset{every picture/.style={line width=0.75pt}} 
\begin{tikzpicture}[x=0.75pt,y=0.75pt,yscale=-1,xscale=1]
\draw  [draw opacity=0][fill={rgb, 255:red, 126; green, 211; blue, 33 }  ,fill opacity=0.41 ] (103.68,171) -- (349.05,171) -- (560.87,211) -- (315.5,211) -- cycle ;
\draw [color={rgb, 255:red, 208; green, 2; blue, 27 }  ,draw opacity=1 ]   (332.28,191) -- (332.49,143) ;
\draw [shift={(332.5,141)}, rotate = 90.26] [color={rgb, 255:red, 208; green, 2; blue, 27 }  ,draw opacity=1 ][line width=0.75]    (10.93,-3.29) .. controls (6.95,-1.4) and (3.31,-0.3) .. (0,0) .. controls (3.31,0.3) and (6.95,1.4) .. (10.93,3.29)   ;
\draw  [color={rgb, 255:red, 208; green, 2; blue, 27 }  ,draw opacity=1 ][fill={rgb, 255:red, 208; green, 2; blue, 27 }  ,fill opacity=1 ] (329.03,191) .. controls (329.03,189.21) and (330.48,187.75) .. (332.28,187.75) .. controls (334.07,187.75) and (335.53,189.21) .. (335.53,191) .. controls (335.53,192.79) and (334.07,194.25) .. (332.28,194.25) .. controls (330.48,194.25) and (329.03,192.79) .. (329.03,191) -- cycle ;
\draw    (62.5,216) .. controls (104.5,187) and (251.5,190) .. (332.28,191) ;
\draw [shift={(201.77,191.41)}, rotate = 177.79] [fill={rgb, 255:red, 0; green, 0; blue, 0 }  ][line width=0.08]  [draw opacity=0] (10.72,-5.15) -- (0,0) -- (10.72,5.15) -- (7.12,0) -- cycle    ;
\draw    (99.5,239) .. controls (116.5,210) and (251.5,190) .. (332.28,191) ;
\draw [shift={(205.98,201.31)}, rotate = 350.95] [fill={rgb, 255:red, 0; green, 0; blue, 0 }  ][line width=0.08]  [draw opacity=0] (10.72,-5.15) -- (0,0) -- (10.72,5.15) -- (7.12,0) -- cycle    ;
\draw    (254.5,262) .. controls (220.5,246) and (239.5,223) .. (284.5,207) .. controls (329.5,191) and (285.5,206) .. (332.28,191) ;
\draw [shift={(249.41,224.62)}, rotate = 139.36] [fill={rgb, 255:red, 0; green, 0; blue, 0 }  ][line width=0.08]  [draw opacity=0] (10.72,-5.15) -- (0,0) -- (10.72,5.15) -- (7.12,0) -- cycle    ;
 (10.72,-5.15) -- (0,0) -- (10.72,5.15) -- (7.12,0) -- cycle    ;
\draw    (332.28,191) .. controls (378.37,185.44) and (444.5,149) .. (445.5,125) ;
\draw [shift={(401.7,167.2)}, rotate = 152.33] [fill={rgb, 255:red, 0; green, 0; blue, 0 }  ][line width=0.08]  [draw opacity=0] (10.72,-5.15) -- (0,0) -- (10.72,5.15) -- (7.12,0) -- cycle    ;
\draw    (332.28,191) .. controls (400.5,191) and (607.5,209) .. (596.5,232) ;
\draw [shift={(474.36,200.55)}, rotate = 186.24] [fill={rgb, 255:red, 0; green, 0; blue, 0 }  ][line width=0.08]  [draw opacity=0] (10.72,-5.15) -- (0,0) -- (10.72,5.15) -- (7.12,0) -- cycle    ;
\draw    (332.28,191) .. controls (372.93,185.49) and (435.08,188.01) .. (488.52,186.75) .. controls (541.85,185.49) and (586.51,180.48) .. (592.5,160) ; (10.72,-5.15) -- (0,0) -- (10.72,5.15) -- (7.12,0) -- cycle    ;
\draw [shift={(538.07,183.78)}, rotate = 352.84] [fill={rgb, 255:red, 0; green, 0; blue, 0 }  ][line width=0.08]  [draw opacity=0] (10.72,-5.15) -- (0,0) -- (10.72,5.15) -- (7.12,0) -- cycle    ;
\draw [color={rgb, 255:red, 65; green, 117; blue, 5 }  ,draw opacity=0.65 ][line width=1.5]  [dash pattern={on 1.69pt off 2.76pt}]  (226.36,171) -- (438.19,211) ;

\draw (327,202) node [anchor=north west][inner sep=0.75pt]  [color={rgb, 255:red, 208; green, 2; blue, 27 }  ,opacity=1 ] [align=left] {$\displaystyle v$};
\draw (418.5,220) node [anchor=north west][inner sep=0.75pt]  [color={rgb, 255:red, 74; green, 124; blue, 21 }  ,opacity=1 ] [align=left] [font=\Large] {$(x,z)$-plane};
\draw (188.86,217.53) node [anchor=north west][inner sep=0.75pt]  [rotate=-10.41] [align=left] {$\displaystyle \cdots $};
\draw (314.5,118) node [anchor=north west][inner sep=0.75pt]  [color={rgb, 255:red, 208; green, 2; blue, 27 }  ,opacity=1 ] [align=left] [font=\Large] {$z$-vec};
\draw (444.23,146.14) node [anchor=north west][inner sep=0.75pt]  [rotate=-12.19] [align=left] {$\displaystyle \cdots $};
\draw (58,186) node [anchor=north west][inner sep=0.75pt]   [align=left] {$\displaystyle e_{1}$};
\draw (85,211) node [anchor=north west][inner sep=0.75pt]   [align=left] {$\displaystyle e_{2}$};
\draw (216,241) node [anchor=north west][inner sep=0.75pt]   [align=left] {$\displaystyle e_{n}$};
\draw (418,115) node [anchor=north west][inner sep=0.75pt]   [align=left] {$\displaystyle e'_{1}$};
\draw (545,151) node [anchor=north west][inner sep=0.75pt]   [align=left] {$\displaystyle e'_{m-1}$};
\draw (587,198) node [anchor=north west][inner sep=0.75pt]   [align=left] {$\displaystyle e'_{m}$};

\end{tikzpicture}}
\end{equation*}

\begin{remark}
For loops at $v$, i.e.\ edges $e:(0,1)\to M$ which converge to $v$ at both $0$ and $1$, we really have two tangent vectors associated to $e$.  One is defined by moving towards $0$ and one by moving towards $1$.  Both of these vectors should again lie in the $(x,y)$-plane and point in distinct directions.
\end{remark}

Finally, we require a compatibility at the boundary of $M$ in that the intersection $\Gamma\times_{M}\partial M$ should consist only of $1$-valent vertices in $\Gamma$.  The framings at such vertices should be such that the $(x,z)$-plane is tangent to the boundary itself.  By an abuse of notation we refer to the vertices at the boundary in $M$ as boundary vertices in $\Gamma$, and we refer to all other vertices as internal vertices.

\subsection{Embedded ribbons with labels}

Consider an oriented $3$-manifold $M$ and ribbon graph $\gamma:\Gamma\to M$ as above.  We note first that the implicit framing along $\Gamma$ separates edges attached to a given vertex $v$ into incoming and outgoing edges, in a spatial rather than combinatorial sense, and also provides a linear ordering on these edges.  Let us elaborate.

We have the source and target functions on the edges
\[
s,t:\opn{edge}(\Gamma)\to \opn{vert}(\Gamma)
\]
from which we understand the attached edges at a vertex, with multiplicity, as the set $s^{-1}(v)\amalg t^{-1}(v)$.  For an edge $e_i$ attached to $v$, say of the form $e_i:(1-\epsilon,1]\to M$, the tangent vector as we move towards $1$ either has a positive $y$-component for the framing at $v$, or a negative $y$-component.  In the first case $e_i$ enters from the negative $y$ half-space, and in the latter case it enters from the positive half-space.  Similarly, as we travel towards $0$ in an attached edge of the form $e_j:[0,\epsilon)\to M$ we enter either from the negative or positive $y$ half-space at $v$.

We refer to those edges $e_i$ in the set $s^{-1}(v)\amalg t^{-1}(v)$ which enter from the negative half-space as \emph{spatially incoming} and those which enter from the positive half-space as \emph{spatially outgoing}.  Those edges which are spatially incoming are ordered linearly by moving counterclockwise from the negative $x$ quadrant to the positive $x$ quadrant in the negative $y$ half-space.  Those which are spatially outgoing are linearly ordered by moving clockwise, again from the negative $x$ quadrant to the positive $x$ quadrant in the positive $y$ half-space,
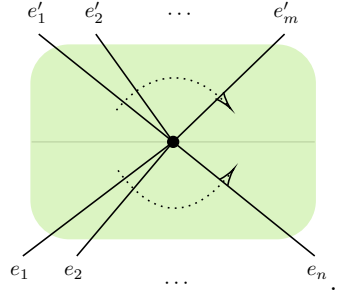
\[
\scalebox{.8}{
\tikzset{every picture/.style={line width=0.75pt}}        
\begin{tikzpicture}[x=0.75pt,y=0.75pt,yscale=-1,xscale=1]
\draw [color={rgb, 255:red, 155; green, 155; blue, 155 }  ,draw opacity=0.5 ]   (76,133) -- (254.5,133) ;
\draw  [color={rgb, 255:red, 0; green, 0; blue, 0 }  ,draw opacity=0 ][fill={rgb, 255:red, 184; green, 233; blue, 134 }  ,fill opacity=.5] (75.25,96.1) .. controls (75.25,82.51) and (86.26,71.5) .. (99.85,71.5) -- (230.65,71.5) .. controls (244.24,71.5) and (255.25,82.51) .. (255.25,96.1) -- (255.25,169.9) .. controls (255.25,183.49) and (244.24,194.5) .. (230.65,194.5) -- (99.85,194.5) .. controls (86.26,194.5) and (75.25,183.49) .. (75.25,169.9) -- cycle ;
\draw  [fill={rgb, 255:red, 0; green, 0; blue, 0 }  ,fill opacity=1 ] (161.88,133) .. controls (161.88,131.14) and (163.39,129.63) .. (165.25,129.63) .. controls (167.11,129.63) and (168.63,131.14) .. (168.63,133) .. controls (168.63,134.86) and (167.11,136.38) .. (165.25,136.38) .. controls (163.39,136.38) and (161.88,134.86) .. (161.88,133) -- cycle ;
\draw    (165.25,133) -- (70.5,205) ;
\draw    (165.25,133) -- (104.5,205) ;
\draw    (165.25,133) -- (254.5,205) ;
\draw    (80.5,65) -- (165.25,133) ;
\draw    (165.25,133) -- (235.5,65) ;
\draw    (116.5,65) -- (165.25,133) ;
\draw  [dash pattern={on 0.84pt off 2.51pt}]  (130.5,151) .. controls (153.16,182.52) and (178.24,182.99) .. (202.4,151.46) ;
\draw [shift={(203.5,150)}, rotate = 126.59] [color={rgb, 255:red, 0; green, 0; blue, 0 }  ][line width=0.75]    (10.93,-3.29) .. controls (6.95,-1.4) and (3.31,-0.3) .. (0,0) .. controls (3.31,0.3) and (6.95,1.4) .. (10.93,3.29)   ;
\draw  [dash pattern={on 0.84pt off 2.51pt}]  (129.5,113) .. controls (151.17,86.4) and (178.66,86.01) .. (200.51,110.85) ;
\draw [shift={(201.5,112)}, rotate = 229.76] [color={rgb, 255:red, 0; green, 0; blue, 0 }  ][line width=0.75]    (10.93,-3.29) .. controls (6.95,-1.4) and (3.31,-0.3) .. (0,0) .. controls (3.31,0.3) and (6.95,1.4) .. (10.93,3.29)   ;

\draw (72,43) node [anchor=north west][inner sep=0.75pt]   [align=left] {$e'_1$};
\draw (108,43) node [anchor=north west][inner sep=0.75pt]   [align=left] {$e'_2$};
\draw (227,43) node [anchor=north west][inner sep=0.75pt]   [align=left] {$e'_m$};
\draw (61,210) node [anchor=north west][inner sep=0.75pt]   [align=left] {$e_1$};
\draw (95,210) node [anchor=north west][inner sep=0.75pt]   [align=left] {$e_2$};
\draw (248,210) node [anchor=north west][inner sep=0.75pt]   [align=left] {$e_n$};
\draw (158,218) node [anchor=north west][inner sep=0.75pt]   [align=left] {$\cdots$};
\draw (160,48) node [anchor=north west][inner sep=0.75pt]   [align=left] {$\cdots$};
\end{tikzpicture}}.
\]

\begin{remark}
Whether an edge is incoming or outgoing in terms of the source and target maps, i.e.\ in terms of the combinatorics of $\Gamma$, has nothing to do with its nature as a spatially incoming or outgoing edge for a given vertex.  As the name suggests, this spacial orientation has to do with the geometry of the given diagram $\gamma:\Gamma\to M$.
\end{remark}

\begin{definition}
Let $\msc{A}$ be a ribbon monoidal category.  An $\msc{A}$-labeling on a ribbon graph in a $3$-manifold $M$ is the following data:
\begin{itemize}
\item The choice of a framed ribbon graph $\gamma:\Gamma\to M$.\vspace{1mm}
\item The choice of an object $x_e$ in $\msc{A}$ over each edge $e$ in $\Gamma$.\vspace{1mm}
\item For each internal vertex $v$, with linearly ordered spatially incoming and outgoing edges $\{e_1,\dots, e_n\}$ and $\{e'_1,\dots, e'_m\}$ respectively, the choice of a map
\[
f_v:x_1^{\pm}\ot\dots\ot x_n^{\pm}\to y_{1}^{\pm}\ot\dots\ot y_{m}^{\pm}.
\]
Here each $x_i^{\pm}$ is $x_{e_i}$ if $e_i$ is combinatorially incoming and $x_{e_i}^{\ast}$ if $e_i$ is combinatorially outgoing.  Similarly, each $y_j^{\pm}$ is $y_{e'_j}$ if $e'_j$ is combinatorially outgoing and $y_{e'_j}^\ast$ if $e'_j$ is combinatorially incoming.
\end{itemize}
\end{definition}

So, the sign in the expressions $x^{\pm}_i$ and $y^{\pm}_j$ just measures the agreement, or disagreement, between the combinatorial and spacial orientations of the given edge at $v$.  Also, in the above expressions each $z_1\ot\cdots \ot z_{t-1}\ot z_t$ explicitly denotes the product with all parentheses arranged to the right $z_1\ot(\cdots (z_{t-1}\ot z_t)\cdots)$.

\begin{convention}
We generally denote a ribbon graph with labeling data
\begin{equation}
T=\{\ \gamma^{fr}:\Gamma\to TM\times_{M}TM\times_{M}TM,\ \msc{A}\text{-data}\ \}
\end{equation}
simply by an arrow $T:\Gamma\to M$.
\end{convention}

\subsection{Homotopies between ribbon graphs}

Consider two ribbon graphs in a fixed $3$-manifold $\gamma:\Gamma\to M$ and $\gamma':\Gamma'\to M$.  A homotopy from $\gamma$ to $\gamma'$ is a choice of an isomorphism $\phi:\Gamma\to \Gamma'$ and a continuous map
\[
h:I\times \Gamma\to M
\]
which is equipped with smooth framings along the vertices $I\times \opn{vert}(\Gamma)$ and edges $I\times (\Gamma-\opn{vert}(\Gamma))$, i.e.\ a trivialization of the pullback of the tangent bundle, for which the following hold:
\begin{enumerate}
\item[(a)] At each time $0\leq t\leq 1$, $h_t:\Gamma\to M$ is a ribbon graph in $M$.
\item[(b)] At times $0$ and $1$ we recover the ribbon graphs $\gamma$ and $\gamma'\phi$ respectively.
\item[(c)] $h$ is constant in some neighborhood of the boundary $\partial \Gamma$.
\end{enumerate}

For $\msc{A}$-labeled ribbon graphs $T$ and $T'$, a homotopy is simply a homotopy on the underlying unlabeled ribbon graphs for which the isomorphism $\phi:\Gamma\to \Gamma'$ preserves the given labels.

\begin{remark}
We note that $h$ needn't be smooth, though its evaluation at any given time $t$ should be smooth, at least away from the vertices.
\end{remark}

\begin{remark}
By considering straight line homotopies, one sees that the equivalence relations provided by homotopy equivalences and isotopy equivalences on ribbon graphs are the same.
\end{remark}

\begin{remark}
Up to homotopy, the framing on an embedded graph $\gamma:\Gamma\to M$ is only determined up to scaling by $\mbb{R}_{>0}$.  Hence the framing can be considered as a map to the third power $\Gamma\to PTM\times_MPTM\times_MPTM$ of the sphere bundle
\[
PTM=(TM-\text{zero section})/\mbb{R}_{>0}
\]
where the three constituent vectors are in generic position.
\end{remark}

\subsection{Comparison with literal ribbons}

One has the more traditional notion of a ``literal" ribbon graph in a $3$-manifold, as presented in Turaev's text \cite{turaev10} for example.  Consider a $3$-manifold $M$, and for the purposes of comparison let us choose a metric on $M$.
\par

Consider an unlabeled ribbon graph $\gamma:\Gamma\to M$ in the sense above.  Via the Riemannian structure, the $x$-vector in the framing along each edge $e$ in $\gamma$ projects onto the normal bundle along $e$ to provide a parallel curve traveling $\epsilon$-distance away from $e$, for some small positive function $\epsilon$.  The straight line homotopy between $e$ and the translated edge now defines a ribbon and the $z$-vector along $e$ orients the ribbon.  We can do this along all edges in $\Gamma$ to replace all edges with oriented ribbons, and by taking $\epsilon$ to $0$ towards the boundary of each edge these ribbons can be taken to be disjoint.
\par

One now replaces each vertex with a sufficiently small coupon to produce from $\gamma$ a corresponding literal ribbon graph, and we note that this ribbon graph is uniquely defined up to isotopy.  (One can only expand or shrink the parameter $\epsilon$, though one should be somewhat careful around the coupons.)  In this way one obtains a map from our ribbon graphs to literal ribbon graphs which is well-defined up to isotopy.  One can check furthermore that homotopy equivalent ribbon graphs, in our sense, map to isotopy equivalent literal ribbon graphs.  One also observes that any literal ribbon graph determines a ribbon graph in our sense, though we leave it to the interested reader to sketch the details.

\section{Marked surfaces and ribbon bordisms}
\label{sect:bords}

We now construct our category $\Bord_{3,2}(\msc{A})=\Bord_{\msc{A}}$ of $\msc{A}$-marked (anomalous) ribbon bordisms in dimension $3$.

\subsection{Aside: Smoothing curves}

The following smoothing result is used throughout the text.

\begin{proposition}[Smoothing]
Let $M$ be a smooth manifold with boundary.  Suppose $\gamma_0:[0,1]\to M$ is a continuous embedding which is smooth away from a finite collection of points $x_1,\dots,x_n$ in the interior.  For any choice of arbitrarily small neighborhoods $N_i\subseteq (0,1)$ around the $x_i$, there is a smooth map $\gamma:[0,1]\to M3$ which agrees with $\gamma_0$ on the complement of the $N_i$, and which is homotopic to $\gamma_0$.  Furthermore, universally across all choices of the $N_i$, the smoothing $\gamma$ is unique up to a sequence of homotopies which are constant away from a compact subset in $(0,1)$.
\end{proposition}

\begin{proof}
This follows by a general theory of smoothing, however one can just produce $\gamma$ by hand.  Around each singular point choose a small ball $B$ in $M$ which contains the given singularity and has connected preimage in $(0,1)$.  Choose now any smooth curve $\gamma_B$ in $B$ which agrees with $\gamma_0$ in a neighborhood of the boundary, and take the straight line homotopy between $\gamma_0$ and $\gamma_B$ in $B$.  One performs this procedure at all singular points, then replaces the relevant segments of $\gamma_0$ with the $\gamma_B$, to produce the smoothed curve $\gamma$.
\end{proof}

\subsection{Marked surfaces and ribbon bordisms}

A marked surface is an oriented surface $\Sigma$ equipped with a collection of distinct framed points $\vec{p}:I\to \Sigma$, where $I$ is a finite linearly ordered set.  Here, again, by a framing we mean a trivialization of the tangent bundle at the given points.  For a ribbon monoidal category $\msc{A}$, an $\msc{A}$-labeled surface is an oriented surface $\Sigma$ along with framed labels by objects in $\msc{A}$,
\[
\vec{x}:I\to \Sigma\times \opn{obj}(\msc{A}).
\]
Note that, unlike in the $3$-dimensional setting, we do not require any compatibility between the framings along $I$ and the orientation on $\Sigma$. We generally let $\Sigma_{\vec{x}}$ denote an $\msc{A}$-marked surface as above.

\begin{remark}
The ordering on the marking set $I$ is formally irrelevant.  However, it is practically useful.
\end{remark}

A bordism between $\msc{A}$-marked surfaces $\Sigma_{\vec{x}}$ and $\Sigma'_{\vec{y}}$ is a $3$-manifold $M$ equipped with an oriented diffeomorphism
\[
\Sigma\amalg \bar{\Sigma}'\overset{\cong}\to \partial M
\]
and an embedded $\msc{A}$-labeled ribbon graph $T:\Gamma\to M$ whose intersection with the boundary recovers the given markings on $\Sigma$ and $\Sigma'$.  (Here $\bar{\Sigma}'$ is just $\Sigma'$ with the opposite orientation.)
\par

Given ribbon bordisms $(M,T):\Sigma_{\vec{x}}\to \Sigma'_{\vec{y}}$ and $(M',T'):\Sigma'_{\vec{y}}\to \Sigma''_{\vec{z}}$ we can glue along the boundary $\Sigma'$ to obtain a new ribbon bordism
\begin{equation}\label{eq:256}
(M\amalg_{\Sigma'}M',T''):\Sigma_{\vec{x}}\to \Sigma''_{\vec{z}}.
\end{equation}
This new bordism has smooth structure obtained by choosing collars along the boundaries in $M$ and $M'$, and the embedded ribbon graph $T''$ is obtained by gluing the two ribbons $T:\Gamma\to M$ and $T':\Gamma'\to M'$ along the boundary $\Sigma'$, then smoothing any singularities which are generated at this boundary.  We also place an identity labeled vertex on any contiguous loop which is generated in this process.  We note that the ribbon bordism $(M\amalg_{\Sigma'}M',T'')$ is well-defined up to a diffeomorphism and a homotopy on the underlying ribbon graph.

\subsection{The marked bordism category}

Fix a ribbon monoidal category $\msc{A}$. We say two ribbon bordisms $(M,T)$ and $(M',T'):\Sigma_{\vec{x}}\to \Sigma'_{\vec{y}}$ between $\msc{A}$-marked surfaces are \emph{equivalent} if there is a diffeomorphism $\psi:M\to M'$ on the underlying oriented manifolds which fixes the boundary and a homotopy between $\msc{A}$-labeled ribbon graphs $h:\psi\circ T\to T'$.  We define the non-anomalous $\msc{A}$-labeled bordism category
\[
\uBord_{\msc{A}}=\uBord_{3,2}(\msc{A})
\]
to be the category of $\msc{A}$-labeled surfaces whose morphisms are equivalence classes $[M,T]:\Sigma_{\vec{x}}\to \Sigma'_{\vec{y}}$ of ribbon bordisms between $\Sigma_{\vec{x}}$ and $\Sigma'_{\vec{y}}$.  Composition is given by gluing along the relevant boundary, as in \eqref{eq:256}.
\par

We consider the category $\uBord_{\msc{A}}$ along with its natural symmetric monoidal structure given by disjoint union.

\subsection{Introducing the anomaly}

Let $\Bord=\Bord_{3,2}$ be the category of unmarked, undecorated, anomalous oriented $3$-bordisms \cite[Section IV.9]{turaev10}.  Objects in this category are oriented surfaces $\Sigma$ with a specified choice of Lagrangian subspace $L_{\Sigma}\subseteq H_1(\Sigma,\mbb{R})$ relative to the intersection form, and morphisms are oriented bordisms $M:\Sigma\to \Sigma'$ equipped with a specified integer $n_M$ called the signature defect.  We consider such bordisms $(M,n_M)$ up to diffeomorphisms which preserve the defect, and this signature defect transforms under composition via the Maslov index
\[
[M',n_{M'}]\circ [M,n_M]=[M'\amalg_{\Sigma'} M,n_M+n_{M'}+\mu(M_\ast(L_{\Sigma}),L_{\Sigma'},(M')^\ast(L_{\Sigma'}))],
\]
where $M_\ast$ and $(M')^\ast$ are as in \cite[Section 4.3]{derenzietal23}.  We formally require $n_{M}=0$ whenever $M$ is the empty manifold as well.

\begin{remark}
The signature defect records the signature of the intersection form for an imagined bounding 4-manifold.  See \cite[Section IV.4.1]{turaev10}.
\end{remark}

We define the anomalous ribbon bordism category associated to a ribbon monoidal category $\msc{A}$ to be the category of ribbon bordisms with independent anomaly data.  So, the $\msc{A}$-decorations and anomaly data do not interact at all.  Formally, if we let $\uBord$ denote the nonanomalous undecorated bordism category, we have the two forgetful functors $\Bord\to \uBord$ and $\uBord_{\msc{A}}\to \uBord$ and define the anomalous ribbon bordism category to be the corresponding fiber product.

\begin{definition}
Given a ribbon monoidal category $\msc{A}$, the anomalous $3$-dimensional ribbon bordism category is the fiber product
\[
\Bord_{\msc{A}}:=\Bord\times_{\uBord}\uBord_{\msc{A}}.
\]
\end{definition}

Though the anomalous setting is the formally correct setting in which to work, for us at least, the anomaly data plays literally no role in our analysis.  So, from now on we make no further mention of the anomaly data and take for granted that every ``surface" comes equipped with a Lagrangian, every ``$3$-manifold" comes equipped with a signature defect, and all compositions of bordisms add signature defects appropriately.

\begin{remark}
Note that any non-anomalous TQFT $Z:\uBord_{\msc{A}}\to \msc{S}$, with symmetric target $\msc{S}$, defines an anomalous TQFT by restricting along the forgetful functor
\[
\Bord_{\msc{A}}\to \uBord_{\msc{A}}\to \msc{S}.
\]
Obviously this anomalous TQFT remains nonanomalous, in the sense that it is constant along variations in the anomaly data.
\end{remark}

From any ribbon tensor functor $F:\msc{A}\to \msc{A}'$ we have an induced symmetric monoidal functor $F_{\ast}:\Bord_{\msc{A}}\to \Bord_{\msc{A}'}$. Of special utility is the functor $\Bord_{\msc{A}}\to \Bord_{\ast}$ induced by the unique map $\msc{A}\to \ast$ to the terminal category.  The category $\Bord_{\ast}$ is, up to superficial issues of notation, the category of anomalous bordisms with unmarked ribbons and the functor $\Bord_{\msc{A}}\to \Bord_{\ast}$ just forgets $\msc{A}$-markings.

\begin{definition}
Given a ribbon bordism $[M,T]$ in $\Bord_{\msc{A}}$ the underlying geometry for $[M,T]$ is the image of this bordism under the forgetful functor to $\Bord_{\ast}$.
\end{definition}

\section{Evaluation functors and skein relations}
\label{sect:ev}

We recall the construction of evaluation functors for ribbon tensor categories. Such functors take ribbon bordisms between $\msc{A}$-marked disks and return values in the marking category $\msc{A}$ itself.  In the special case where $\msc{A}$ is symmetric, such evaluation functors don't distinguish between over-crossings and under-crossings, nor do they recognize ribbon twists.  Hence the corresponding evaluation functor factors through a $2$-dimensional projection of the aforementioned bordism category in this case. We explain the details below.

\subsection{RT evaluation}

Let $\msc{A}$ be a ribbon monoidal category and fix $D\subseteq \mbb{R}^2$ the unit disk.  Notationally, we denote the coordinates on this disk by $x$ and $z$-axes rather than $x$ and $y$-axes.
\par

Consider the category $\opn{Rib}_{\msc{A}}$ whose objects are disks $D\subseteq \mbb{R}^2$ with $\msc{A}$-labeled points lying only on the $x$-axis, and framings provided by either the ambient framing on $\mbb{R}^2$, or the negative $x$ and positive $z$-vectors. Morphisms are equivalence classes of $\msc{A}$-labeled ribbon bordisms $M:D_{\vec{x}}\to D_{\vec{y}}$ whose underlying $3$-manifold (with corners) is a cylinder $\opn{Cyl}_{ab}=D\times [a, b]$.  Here $a<b$, and we consider $M$ up to oriented diffeomorphism on the linear factor $[a,b]$ and homotopy on the embedded ribbon graph.  (So, functionally, we can consider all bordisms as occurring in the standard cylinder $\opn{Cyl}=\opn{Cyl}_{01}$.) We note that the markings on $D_{\vec{x}}$ are naturally ordered via the ordering on the real numbers, and composition is provided by stacking cylinders, i.e.\ gluing along the boundary.
\par

The category $\opn{Rib}_{\msc{A}}$ is rigid monoidal, with product provided by embedding the disjoint union $D\amalg D$ into $D$ as the disks
\[
\{(a,b):(a+{1/2})^2+b^2\leq {1/2}\}\ \ \text{and}\ \ \{(a,b):(a-1/2)^2+b^2\leq 1/2\},
\]
in order.  Rotating embedded disks clockwise endows $\opn{Rib}_{\msc{A}}$ with a braided structure, and a subsequent ribbon structure is given by applying a full clockwise rotation.

\begin{theorem}[\cite{turaev10}]\label{thm:eval_A}
Let $\msc{A}$ be any ribbon monoidal category.  There is a unique ribbon monoidal functor
\[
ev_{\msc{A}}:\opn{Rib}_{\msc{A}}\to \msc{A}
\]
which satisfies the following:
\begin{enumerate}
\item[(a)] The value on a marked disk $D_{\vec{x}}$ is the product $x^{\pm}_1\ot x^{\pm}_2\ot\cdots\ot x^{\pm}_n$, where $x^{\pm}_i$ is $x_i$ if the framing at $i$ reproduces the orientation on $D$ and $x^\ast_i$ otherwise.\vspace{1mm}
\item[(b)] The tensor compatibility $ev_{\msc{A}}(D_{\vec{x}})\ot ev_{\msc{A}}(D_{\vec{y}})\to ev_{\msc{A}}(D_{\vec{x}}\ot D_{\vec{y}})$ is provided by the associator on $\msc{A}$.
\item[(c)] The value on a  bordism of the form
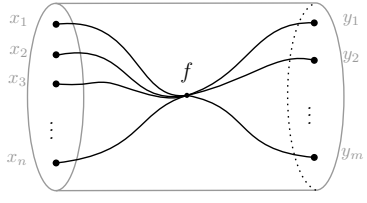
\[
	\scalebox{0.7}{
\tikzset{every picture/.style={line width=0.75pt}} 
\begin{tikzpicture}[x=0.75pt,y=0.75pt,yscale=-1,xscale=1]
\draw  [color={rgb, 255:red, 0; green, 0; blue, 0 }  ,draw opacity=0.4 ] (210.51,153.96) -- (398,153.64) .. controls (409.19,153.62) and (418.31,183.82) .. (418.37,221.1) .. controls (418.44,258.38) and (409.42,288.62) .. (398.24,288.64) -- (210.75,288.96) .. controls (199.57,288.98) and (190.45,258.78) .. (190.38,221.5) .. controls (190.32,184.22) and (199.33,153.98) .. (210.51,153.96) .. controls (221.7,153.94) and (230.82,184.15) .. (230.88,221.43) .. controls (230.95,258.71) and (221.93,288.95) .. (210.75,288.96) ;
\draw  [dash pattern={on 0.84pt off 2.51pt}]  (398.24,288.64) .. controls (370.99,277.67) and (370.99,162.67) .. (398,153.64) ;
\draw  [fill={rgb, 255:red, 0; green, 0; blue, 0 }  ,fill opacity=1 ] (209,169) .. controls (209,167.9) and (209.9,167) .. (211,167) .. controls (212.1,167) and (213,167.9) .. (213,169) .. controls (213,170.1) and (212.1,171) .. (211,171) .. controls (209.9,171) and (209,170.1) .. (209,169) -- cycle ;
\draw  [fill={rgb, 255:red, 0; green, 0; blue, 0 }  ,fill opacity=1 ] (209,191) .. controls (209,189.9) and (209.9,189) .. (211,189) .. controls (212.1,189) and (213,189.9) .. (213,191) .. controls (213,192.1) and (212.1,193) .. (211,193) .. controls (209.9,193) and (209,192.1) .. (209,191) -- cycle ;
\draw  [fill={rgb, 255:red, 0; green, 0; blue, 0 }  ,fill opacity=1 ] (209,269) .. controls (209,267.9) and (209.9,267) .. (211,267) .. controls (212.1,267) and (213,267.9) .. (213,269) .. controls (213,270.1) and (212.1,271) .. (211,271) .. controls (209.9,271) and (209,270.1) .. (209,269) -- cycle ;
\draw  [fill={rgb, 255:red, 0; green, 0; blue, 0 }  ,fill opacity=1 ] (209,212) .. controls (209,210.9) and (209.9,210) .. (211,210) .. controls (212.1,210) and (213,210.9) .. (213,212) .. controls (213,213.1) and (212.1,214) .. (211,214) .. controls (209.9,214) and (209,213.1) .. (209,212) -- cycle ;
\draw  [fill={rgb, 255:red, 0; green, 0; blue, 0 }  ,fill opacity=1 ] (395,168) .. controls (395,166.9) and (395.9,166) .. (397,166) .. controls (398.1,166) and (399,166.9) .. (399,168) .. controls (399,169.1) and (398.1,170) .. (397,170) .. controls (395.9,170) and (395,169.1) .. (395,168) -- cycle ;
\draw  [fill={rgb, 255:red, 0; green, 0; blue, 0 }  ,fill opacity=1 ] (395,195) .. controls (395,193.9) and (395.9,193) .. (397,193) .. controls (398.1,193) and (399,193.9) .. (399,195) .. controls (399,196.1) and (398.1,197) .. (397,197) .. controls (395.9,197) and (395,196.1) .. (395,195) -- cycle ;
\draw  [fill={rgb, 255:red, 0; green, 0; blue, 0 }  ,fill opacity=1 ] (395,265) .. controls (395,263.9) and (395.9,263) .. (397,263) .. controls (398.1,263) and (399,263.9) .. (399,265) .. controls (399,266.1) and (398.1,267) .. (397,267) .. controls (395.9,267) and (395,266.1) .. (395,265) -- cycle ;
\draw    (211,169) .. controls (270.5,162) and (267.75,225.25) .. (305.25,220.25) ;
\draw    (213,191) .. controls (264.5,181) and (259.75,228.25) .. (305.25,220.25) ;
\filldraw (305.25,220.25) circle (1pt);
\draw    (211,212) .. controls (235.03,213.8) and (235.52,209.88) .. (247.5,211) .. controls (259.48,212.12) and (286.75,227.25) .. (305.25,220.25) ;
\draw    (211,269) .. controls (273.5,262) and (267.5,232) .. (305.25,220.25) ;
\draw    (305.25,220.25) .. controls (350.5,206) and (353.5,166) .. (397,168) ;
\draw    (305.25,220.25) .. controls (351.5,210) and (369.5,188) .. (397,195) ;
\draw    (305.25,220.25) .. controls (352.5,229) and (338.5,260) .. (397,265) ;

\draw (204,230) node [anchor=north west][inner sep=0.75pt]   [align=left] {$\displaystyle \vdots $};
\draw (390,219) node [anchor=north west][inner sep=0.75pt]   [align=left] {$\displaystyle \vdots $};
\draw (176,161) node [anchor=north west][inner sep=0.75pt]  [color={rgb, 255:red, 0; green, 0; blue, 0 }  ,opacity=0.4 ] [align=left] {$\displaystyle x_{1}$};
\draw (176,183) node [anchor=north west][inner sep=0.75pt]  [color={rgb, 255:red, 0; green, 0; blue, 0 }  ,opacity=0.4 ] [align=left] {$\displaystyle x_{2}$};
\draw (175,204) node [anchor=north west][inner sep=0.75pt]  [color={rgb, 255:red, 0; green, 0; blue, 0 }  ,opacity=0.4 ] [align=left] {$\displaystyle x_{3}$};
\draw (174,261) node [anchor=north west][inner sep=0.75pt]  [color={rgb, 255:red, 0; green, 0; blue, 0 }  ,opacity=0.4 ] [align=left] {$\displaystyle x_{n}$};
\draw (416,160) node [anchor=north west][inner sep=0.75pt]  [color={rgb, 255:red, 0; green, 0; blue, 0 }  ,opacity=0.4 ] [align=left] {$\displaystyle y_{1}$};
\draw (416,187) node [anchor=north west][inner sep=0.75pt]  [color={rgb, 255:red, 0; green, 0; blue, 0 }  ,opacity=0.4 ] [align=left] {$\displaystyle y_{2}$};
\draw (415,257) node [anchor=north west][inner sep=0.75pt]  [color={rgb, 255:red, 0; green, 0; blue, 0 }  ,opacity=0.4 ] [align=left] {$\displaystyle y_{m}$};
\draw (299,195) node [anchor=north west][inner sep=0.75pt]   [font=\Large] [align=left] {$\displaystyle f$};

\end{tikzpicture}},
\]
is the map $f:x^{\pm}_1\ot\dots\ot x^{\pm}_n\to y^{\pm}_1\ot\dots\ot y^{\pm}_m$.  (Here the $x$-vectors in the framings along edges are constant and the internal vertex inherits its framing from that of the cylinder.)
\item[(d)] The value on ribbon bordisms of the form
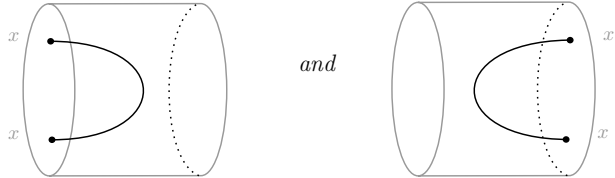
\[
\scalebox{0.75}{
\tikzset{every picture/.style={line width=0.75pt}} 
\begin{tikzpicture}[x=0.75pt,y=0.75pt,yscale=-1,xscale=1]
\draw  [color={rgb, 255:red, 0; green, 0; blue, 0 }  ,draw opacity=0.4 ] (161.86,171.19) -- (264.04,171.03) .. controls (273.63,171.01) and (281.45,196.92) .. (281.52,228.89) .. controls (281.58,260.87) and (273.85,286.8) .. (264.26,286.81) -- (162.09,286.97) .. controls (152.49,286.99) and (144.67,261.08) .. (144.61,229.11) .. controls (144.55,197.13) and (152.27,171.2) .. (161.86,171.19) .. controls (171.46,171.17) and (179.28,197.08) .. (179.34,229.05) .. controls (179.4,261.03) and (171.68,286.96) .. (162.09,286.97) ;
\draw  [dash pattern={on 0.84pt off 2.51pt}]  (262.84,286.81) .. controls (235.13,276.14) and (236.95,180.49) .. (262.62,171.03) ;
\draw  [fill={rgb, 255:red, 0; green, 0; blue, 0 }  ,fill opacity=1 ] (161.29,196.39) .. controls (161.29,195.46) and (162.1,194.71) .. (163.1,194.71) .. controls (164.1,194.71) and (164.92,195.46) .. (164.92,196.39) .. controls (164.92,197.32) and (164.1,198.07) .. (163.1,198.07) .. controls (162.1,198.07) and (161.29,197.32) .. (161.29,196.39) -- cycle ;
\draw  [fill={rgb, 255:red, 0; green, 0; blue, 0 }  ,fill opacity=1 ] (162.19,262.67) .. controls (162.19,261.74) and (163.01,260.99) .. (164.01,260.99) .. controls (165.01,260.99) and (165.82,261.74) .. (165.82,262.67) .. controls (165.82,263.6) and (165.01,264.35) .. (164.01,264.35) .. controls (163.01,264.35) and (162.19,263.6) .. (162.19,262.67) -- cycle ;
\draw    (163.1,196.39) .. controls (246.12,196.39) and (246.12,262.67) .. (164.01,262.67) ;
\draw  [color={rgb, 255:red, 0; green, 0; blue, 0 }  ,draw opacity=0.4 ] (410.02,170.21) -- (511.56,170.03) .. controls (521.24,170.01) and (529.13,196.14) .. (529.19,228.38) .. controls (529.25,260.62) and (521.46,286.78) .. (511.79,286.79) -- (410.24,286.97) .. controls (400.57,286.99) and (392.68,260.86) .. (392.61,228.62) .. controls (392.55,196.38) and (400.34,170.22) .. (410.02,170.21) .. controls (419.69,170.19) and (427.58,196.32) .. (427.64,228.56) .. controls (427.7,260.8) and (419.91,286.95) .. (410.24,286.97) ;
\draw  [dash pattern={on 0.84pt off 2.51pt}]  (510.56,286.81) .. controls (482.92,276.04) and (484.73,179.57) .. (510.34,170.03) ;
\draw  [fill={rgb, 255:red, 0; green, 0; blue, 0 }  ,fill opacity=1 ] (510.63,195.61) .. controls (510.63,194.67) and (511.44,193.92) .. (512.44,193.92) .. controls (513.44,193.92) and (514.25,194.67) .. (514.25,195.61) .. controls (514.25,196.54) and (513.44,197.3) .. (512.44,197.3) .. controls (511.44,197.3) and (510.63,196.54) .. (510.63,195.61) -- cycle ;
\draw  [fill={rgb, 255:red, 0; green, 0; blue, 0 }  ,fill opacity=1 ] (507.91,262.46) .. controls (507.91,261.52) and (508.72,260.77) .. (509.72,260.77) .. controls (510.72,260.77) and (511.53,261.52) .. (511.53,262.46) .. controls (511.53,263.39) and (510.72,264.15) .. (509.72,264.15) .. controls (508.72,264.15) and (507.91,263.39) .. (507.91,262.46) -- cycle ;
\draw    (510.63,195.61) .. controls (426.9,197.3) and (426.9,262.46) .. (511.53,262.46) ;

\draw (133,189) node [anchor=north west][inner sep=0.75pt]  [color={rgb, 255:red, 0; green, 0; blue, 0 }  ,opacity=0.4 ] [align=left] {$\displaystyle x$};
\draw (133,255) node [anchor=north west][inner sep=0.75pt]  [color={rgb, 255:red, 0; green, 0; blue, 0 }  ,opacity=0.4 ] [align=left] {$\displaystyle x$};
\draw (533.12,188.46) node [anchor=north west][inner sep=0.75pt]  [color={rgb, 255:red, 0; green, 0; blue, 0 }  ,opacity=0.4 ] [align=left] {$\displaystyle x$};
\draw (529.4,254.31) node [anchor=north west][inner sep=0.75pt]  [color={rgb, 255:red, 0; green, 0; blue, 0 }  ,opacity=0.4 ] [align=left] {$\displaystyle x$};
\draw (328,204) node [anchor=north west][inner sep=0.75pt]  [font=\Large] [align=left] {and};
\end{tikzpicture}
}
\]
are the evaluation and coevaluation maps $x^{\pm}\ot x^{\mp}\to \1$ and $\1\to x^{\pm}\ot x^{\mp}$, respectively.
\end{enumerate}
\end{theorem}

As before, we arrange parenthesis on the right $z_1\ot\cdots\ot z_{t-1}\ot z_t= z_1\ot (\cdots(z_{t-1}\ot z_t)\cdots)$.  We also use the ribbon structure to identify left and right duals.

\begin{proof}
Supposing existence, uniqueness follows from the fact that the category $\opn{Rib}_{\msc{A}}$ is generated under composition by maps as in (i) and (ii) in conjunction with the ribbon twists and braidings \cite[Theorem 2.5]{turaev10}.  We are left to establish existence.
\par

In the case where $\msc{A}$ is strict one can define $ev_{\msc{A}}$ directly as in \cite{turaev10}.  In the case of general $\msc{A}$ take $\msc{E}=\opn{End}_{\text{mod-}\msc{A}}(\msc{A})$ and consider the strictification $F:\msc{A}\overset{\sim}\to\msc{E}$, $F(x)=x\ot-$. We obtain $ev_{\msc{A}}$ as the unique functor which takes the prescribed values on objects and completes the diagram
\[
\xymatrix{
\opn{Rib}_{\msc{A}}\ar@{..>}[rr]^{ev_{\msc{A}}}\ar[d]_{\opn{Rib}_F} & & \msc{A}\ar[d]^F\\
\opn{Rib}_{\msc{E}}\ar[rr]_{ev_{\msc{E}}} & & \msc{E},
}
\]
with implicit natural isomorphism $ev_{\msc{A}}\opn{Rib}_F\cong F ev_{\msc{A}}$ provided by the associator.
\end{proof}

\begin{definition}
Let $\msc{A}$ be ribbon monoidal.  The evaluation functor for $\msc{A}$ is the ribbon tensor functor $ev_{\msc{A}}:\opn{Rib}_{\msc{A}}\to \msc{A}$ from Theorem \ref{thm:eval_A}.
\end{definition}

\subsection{(Universal) Skein relations}
\label{sect:sk_rel}

Let $\msc{A}$ be a ribbon category.  Two $\msc{A}$-labeled ribbon graphs $T:\Gamma\to M$ and $T':\Gamma'\to M$ are said to be related by a skein relation if there is an oriented embedding $\opn{Cyl}\to M$ from the cylinder which satisfies the following:
\begin{enumerate}
\item[(a)] $T$ and $T'$ intersect $\opn{Cyl}$ transversely along the edges of the respective graphs and along the $x$-axes of the bounding disks for the cylinder.
\item[(b)] The intersections of these graphs with the boundary of $\opn{Cyl}$ agree as collections of $\msc{A}$-labeled framed points on the bounding disks.
\item[(b)] $ev_{\msc{A}}(\opn{Cyl}\times_M\Gamma)=ev_{\msc{A}}(\opn{Cyl}\times_M\Gamma')$.
\item[(c)] For $O_{\Gamma}$ and $O_{\Gamma'}$ the intersections of $\Gamma$ and $\Gamma'$ with the interior of the cylinder $\opn{Cyl}$ in $M$, respectively, there is an isomorphism of directed graphs
\[
\Gamma-O_{\Gamma}\cong \Gamma'-O_{\Gamma'}.
\]
\item[(d)] Under such an isomorphism, $T'|_{\Gamma'-O_{\Gamma'}}$ and $T|_{\Gamma-O_{\Gamma}}$ agree (up to homotopy and smoothing in $M-\opn{Cyl}$).
\end{enumerate}

In short, $T'$ and $T$ are related by a skein relation if $T'$ is obtained from $T$ by replacing some isolated segment of $T$ with a different segment which evaluates to the same morphism in $\msc{A}$ under $ev_{\msc{A}}$.
\par

Though the types of skein relations which are possible rely strongly on the underlying properties of the category $\msc{A}$, there are some universal skein relations which are of interest.

\begin{definition}
Let $D_n$ be the object in $\opn{Rib}_\ast$ provided by the disk with $n$ equidistant points on the axis, all of which are equipped with the framing inherited by the ambient surface.  (Here $\ast$ is the terminal category, with a single object and single morphism.)  The framed braid group $fr\opn{Br}_n$ is the subgroup of $\opn{Aut}_{\Rib_\ast}(D_n)$ generated by the braid group and the twists applied at the various markings.
\end{definition}

We have the normal subgroup $\mbb{Z}^n\subseteq fr\opn{Br}_n$ provided by the twist, and the quotient $fr\opn{Br}_n/\mbb{Z}^n$ recovers the usual braid group.  A framed braid in $\opn{Rib}_{\msc{A}}$ is a map $\xi:D_{\vec{x}}\to D_{\vec{x}}$ whose underlying ribbon diagram has no internal vertices, and whose image along the map $\opn{Rib}_{\msc{A}}\to \opn{Rib}_\ast$ is a framed braid in $\Rib_\ast$.
\par

We have three universal skein relations which are of special interest for us:
\begin{enumerate}
\item[U1](Flipping vertices) Replacing the framing $(x,y,z)$ and labeling morphism $f$, at a vertex $v$ in an $\msc{A}$-labeled ribbon graph $T:\Gamma\to M$, with the framing $(x,-y,z)$ and labeling morphism $f^\ast$,
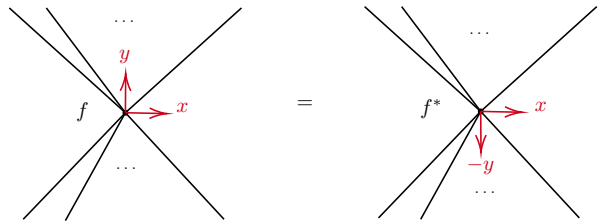
\[
\scalebox{.8}{
\tikzset{every picture/.style={line width=0.75pt}} 
\begin{tikzpicture}[x=0.75pt,y=0.75pt,yscale=-1,xscale=1]
\draw    (60.15,50.84) -- (133.33,116.95) ;
\draw    (133.33,116.95) -- (206.5,50.84) ;
\draw    (133.33,116.95) -- (68.91,184.16) ;
\draw    (133.33,116.95) -- (195.5,184.16) ;
\filldraw (133.33,116.95) circle (1pt);
\draw    (83.5,50.84) -- (133.33,116.95) ;
\draw    (133.33,116.95) -- (95.5,185) ;
\draw [color={rgb, 255:red, 208; green, 2; blue, 27 }  ,draw opacity=1 ]   (133.33,116.95) -- (133.49,94.19) ;
\draw [shift={(133.5,92.19)}, rotate = 90.4] [color={rgb, 255:red, 208; green, 2; blue, 27 }  ,draw opacity=1 ][line width=0.75]    (10.93,-3.29) .. controls (6.95,-1.4) and (3.31,-0.3) .. (0,0) .. controls (3.31,0.3) and (6.95,1.4) .. (10.93,3.29)   ;
\draw    (284.15,50) -- (357.33,116.11) ;
\draw    (357.33,116.11) -- (430.5,50) ;
\draw    (357.33,116.11) -- (292.91,183.31) ;
\draw    (357.33,116.11) -- (419.5,183.31) ;
\filldraw (357.33,116.11) circle (1pt);
\draw    (307.5,50) -- (357.33,116.11) ;
\draw    (357.33,116.11) -- (319.5,184.16) ;
\draw [color={rgb, 255:red, 208; green, 2; blue, 27 }  ,draw opacity=1 ]   (357.33,116.11) -- (357.49,139.97) ;
\draw [shift={(357.5,141.97)}, rotate = 269.61] [color={rgb, 255:red, 208; green, 2; blue, 27 }  ,draw opacity=1 ][line width=0.75]    (10.93,-3.29) .. controls (6.95,-1.4) and (3.31,-0.3) .. (0,0) .. controls (3.31,0.3) and (6.95,1.4) .. (10.93,3.29)   ;
\draw [color={rgb, 255:red, 208; green, 2; blue, 27 }  ,draw opacity=1 ]   (133.33,116.95) -- (157.5,117.46) ;
\draw [shift={(159.5,117.5)}, rotate = 181.21] [color={rgb, 255:red, 208; green, 2; blue, 27 }  ,draw opacity=1 ][line width=0.75]    (10.93,-3.29) .. controls (6.95,-1.4) and (3.31,-0.3) .. (0,0) .. controls (3.31,0.3) and (6.95,1.4) .. (10.93,3.29)   ;
\draw [color={rgb, 255:red, 208; green, 2; blue, 27 }  ,draw opacity=1 ]   (357.33,116.11) -- (381.5,116.61) ;
\draw [shift={(383.5,116.66)}, rotate = 181.21] [color={rgb, 255:red, 208; green, 2; blue, 27 }  ,draw opacity=1 ][line width=0.75]    (10.93,-3.29) .. controls (6.95,-1.4) and (3.31,-0.3) .. (0,0) .. controls (3.31,0.3) and (6.95,1.4) .. (10.93,3.29)   ;

\draw (124.45,56.94) node [anchor=north west][inner sep=0.75pt]  [font=\scriptsize] [align=left] {$\dotsc $};
\draw (125.71,149.93) node [anchor=north west][inner sep=0.75pt] [font=\scriptsize]  [align=left] {$\dotsc $};
\draw (100.72,108.79) node [anchor=north west][inner sep=0.75pt]   [align=left] {$f$};
\draw (240,107.73) node [anchor=north west][inner sep=0.75pt]   [align=left] {=};
\draw (348.45,64.53) node [anchor=north west][inner sep=0.75pt]  [font=\scriptsize] [align=left] {$\dotsc$};
\draw (352,165) node [anchor=north west][inner sep=0.75pt] [font=\scriptsize]  [align=left] {$\dotsc$};
\draw (317.72,106.95) node [anchor=north west][inner sep=0.75pt]   [align=left] {$f^{\ast}$};
\draw (164,109.34) node [anchor=north west][inner sep=0.75pt]  [color={rgb, 255:red, 208; green, 2; blue, 27 }  ,opacity=1 ] [align=left] {$x$};
\draw (128,76) node [anchor=north west][inner sep=0.75pt]  [color={rgb, 255:red, 208; green, 2; blue, 27 }  ,opacity=1 ] [align=left] {$y$};
\draw (390,108.5) node [anchor=north west][inner sep=0.75pt]  [color={rgb, 255:red, 208; green, 2; blue, 27 }  ,opacity=1 ] [align=left] {$x$};
\draw (347,144) node [anchor=north west][inner sep=0.75pt]  [color={rgb, 255:red, 208; green, 2; blue, 27 }  ,opacity=1 ] [align=left] {$-y$};

\end{tikzpicture}}
\]
\item[U2](Internal composition) Composing nearby morphisms in an $\msc{A}$-labeled ribbon graph $T:\Gamma\to M$,
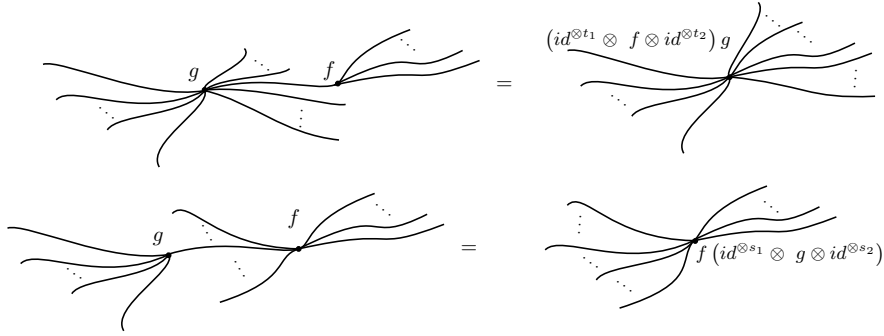
	\begin{equation*}
	\scalebox{0.8}{
		\tikzset{every picture/.style={line width=0.75pt}} 
		\begin{tikzpicture}[x=0.75pt,y=0.75pt,yscale=-1,xscale=1]
\draw    (5.12,313.87) .. controls (16.81,310.22) and (70.61,336.77) .. (106.99,330.44) ;
\draw    (13.43,336.28) .. controls (23.83,325.57) and (64.89,350.41) .. (106.99,330.44) ;
\draw    (45.66,355.28) .. controls (51.9,345.05) and (95.56,346.51) .. (106.99,330.44) ;
\draw    (78.41,378.67) .. controls (69.57,363.56) and (111.67,340.67) .. (106.99,330.44) ;
\draw    (195.88,339.16) .. controls (185.69,341.15) and (138.84,326.69) .. (107.16,330.14) ;
\draw    (188.64,326.96) .. controls (175.22,330.95) and (143.82,319.27) .. (107.16,330.14) ;
\draw    (160.57,316.62) .. controls (155.14,322.19) and (117.12,321.39) .. (107.16,330.14) ;
\draw    (132.05,303.89) .. controls (139.75,312.11) and (103.08,324.57) .. (107.16,330.14) ;
\draw    (106.99,330.44) .. controls (138.18,337.99) and (155.86,357.72) .. (191.72,361.62) ;
\draw    (188.64,326.96) .. controls (197.1,325.06) and (191.76,308.76) .. (236.58,292) ;
\draw    (188.64,326.96) .. controls (247.79,297.43) and (228.05,325.97) .. (269.66,305.13) ;
\draw    (188.64,326.96) .. controls (254.19,312.38) and (231.78,330.04) .. (280.33,311.47) ;
\draw    (335.78,305.21) .. controls (347.48,301.55) and (401.28,328.1) .. (437.66,321.77) ;
\draw    (344.1,327.62) .. controls (354.49,316.9) and (395.56,341.74) .. (437.66,321.77) ;
\draw    (376.33,346.62) .. controls (382.56,336.39) and (426.23,337.85) .. (437.66,321.77) ;
\draw    (409.07,370) .. controls (400.24,354.9) and (442.34,332) .. (437.66,321.77) ;
\draw    (435.97,322.29) .. controls (444.43,320.39) and (439.1,304.09) .. (483.92,287.33) ;
\draw    (435.97,322.29) .. controls (495.12,292.77) and (475.38,321.3) .. (517,300.47) ;
\draw    (435.97,322.29) .. controls (501.52,307.71) and (479.11,325.38) .. (527.67,306.81) ;
\draw    (456.02,274.81) .. controls (463.72,283.03) and (433.59,316.2) .. (437.66,321.77) ;
\draw    (437.66,321.77) .. controls (485.67,326.67) and (486.33,336) .. (529.67,334) ;

\draw (95.24,314.78) node [anchor=north west][inner sep=0.75pt]    {$g$};
\filldraw (107,330) circle (1pt);
\draw (180.08,310.73) node [anchor=north west][inner sep=0.75pt]    {$f$};
\filldraw (191,326) circle (1pt);
\draw (289.33,323) node [anchor=north west][inner sep=0.75pt]    {$=$};
\draw (320,290) node [anchor=north west][inner sep=0.75pt]  [font=\small]  {$\left( id^{\otimes t_{1}} \otimes \ f\otimes id^{\otimes t_{2}}\right) g$};
\filldraw (438,322) circle (1pt);
\draw (40,338) node [anchor=north west][inner sep=0.75pt]  [font=\scriptsize,rotate=-36.72]  {$\dotsc $};
\draw (138,308) node [anchor=north west][inner sep=0.75pt]  [font=\scriptsize,rotate=-36.72]  {$\dotsc $};
\draw (170.5,340) node [anchor=north west][inner sep=0.75pt]  [font=\scriptsize,rotate=-94.06]  {$\dotsc $};
\draw (230,295) node [anchor=north west][inner sep=0.75pt]  [font=\scriptsize,rotate=-36.72]  {$\dotsc $};
\draw (362.79,330) node [anchor=north west][inner sep=0.75pt]  [font=\scriptsize,rotate=-36.72]  {$\dotsc $};
\draw (480,290) node [anchor=north west][inner sep=0.75pt]  [font=\scriptsize,rotate=-36.72]  {$\dotsc $};

\draw (460,280) node [anchor=north west][inner sep=0.75pt]  [font=\scriptsize,rotate=-36.72]  {$\dotsc $};
\draw (520,315) node [anchor=north west][inner sep=0.75pt]  [font=\scriptsize,rotate=-95]  {$\dotsc $};
			
	\end{tikzpicture}}
\end{equation*}
	\begin{equation*}
	\scalebox{0.8}{
		\tikzset{every picture/.style={line width=0.75pt}} 
		\begin{tikzpicture}[x=0.75pt,y=0.75pt,yscale=-1,xscale=1]
\draw    (16.45,449.87) .. controls (28.14,446.22) and (81.94,472.77) .. (118.33,466.44) ;
\draw    (24.77,472.28) .. controls (35.16,461.57) and (76.22,486.41) .. (118.33,466.44) ;
\draw    (56.99,491.28) .. controls (63.23,481.05) and (106.89,482.51) .. (118.33,466.44) ;
\draw    (89.74,514.67) .. controls (80.9,499.56) and (123.01,476.67) .. (118.33,466.44) ;
\draw    (199.97,462.96) .. controls (186.56,466.95) and (155.16,455.27) .. (118.49,466.14) ;
\draw    (199.97,462.96) .. controls (208.43,461.06) and (203.1,444.76) .. (247.92,428) ;
\draw    (199.97,462.96) .. controls (259.12,433.43) and (239.38,461.97) .. (281,441.13) ;
\draw    (199.97,462.96) .. controls (265.52,448.38) and (243.11,466.04) .. (291.67,447.47) ;
\draw    (355.43,463.62) .. controls (365.83,452.9) and (406.89,477.74) .. (448.99,457.77) ;
\draw    (387.66,482.62) .. controls (393.9,472.39) and (437.56,473.85) .. (448.99,457.77) ;
\draw    (447.3,458.29) .. controls (455.77,456.39) and (450.43,440.09) .. (495.25,423.33) ;
\draw    (447.3,458.29) .. controls (506.45,428.77) and (486.71,457.3) .. (528.33,436.47) ;
\draw    (447.3,458.29) .. controls (512.86,443.71) and (490.45,461.38) .. (539,442.81) ;
\draw    (150.33,496.67) .. controls (195.67,482) and (184.33,468) .. (199.97,462.96) ;
\draw    (120.33,440.67) .. controls (130.73,429.95) and (147,462) .. (199.97,462.96) ;
\draw    (369.36,435.48) .. controls (379.75,424.76) and (396.02,456.81) .. (448.99,457.77) ;
\draw    (402.33,500.33) .. controls (453,484.33) and (439,470.33) .. (448.99,457.77) ;

\draw (106.58,450.78) node [anchor=north west][inner sep=0.75pt]    {$g$};
\draw (191.41,436.73) node [anchor=north west][inner sep=0.75pt]    {$f$};
\draw (300.67,459) node [anchor=north west][inner sep=0.75pt]    {$=$};
\draw (448.99,457.77) node [anchor=north west][inner sep=0.75pt]  [font=\small]  {$f\left( id^{\otimes s_{1}} \otimes \ g\otimes id^{\otimes s_{2}}\right)$};
\draw (53.45,472) node [anchor=north west][inner sep=0.75pt]  [font=\scriptsize,rotate=-36.72]  {$\dotsc $};
\draw (248.12,432) node [anchor=north west][inner sep=0.75pt]  [font=\scriptsize,rotate=-36.72]  {$\dotsc $};
\draw (378.79,465) node [anchor=north west][inner sep=0.75pt]  [font=\scriptsize,rotate=-36.72]  {$\dotsc $};
\draw (492.79,430) node [anchor=north west][inner sep=0.75pt]  [font=\scriptsize,rotate=-36.72]  {$\dotsc $};
\draw (160.49,468) node [anchor=north west][inner sep=0.75pt]  [font=\scriptsize,rotate=-57.67,xslant=0.18]  {$\dotsc $};
\draw (140,445) node [anchor=north west][inner sep=0.75pt]  [font=\scriptsize,rotate=-59.8]  {$\dotsc $};
\draw (380,440) node [anchor=north west][inner sep=0.75pt]  [font=\scriptsize,rotate=-95]  {$\dotsc $};
\draw (400,480) node [anchor=north west][inner sep=0.75pt]  [font=\scriptsize,rotate=-36]  {$\dotsc $};
\filldraw (118,467) circle (1pt);
\filldraw (200,463) circle (1pt);
\filldraw (450,458) circle (1pt);
\draw (580,480) node {.};
	\end{tikzpicture}}
\end{equation*}
We also allow for a a type of composition at the boundary, in which an identity labeled vertex can be absorbed into a boundary vertex
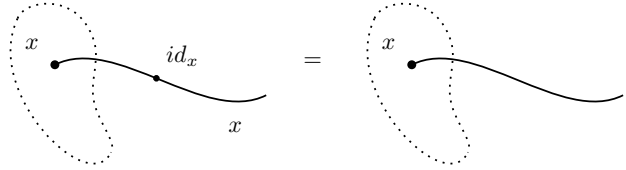
\[
\scalebox{.9}{\tikzset{every picture/.style={line width=0.75pt}} 

\begin{tikzpicture}[x=0.75pt,y=0.75pt,yscale=-1,xscale=1]

\draw  [dash pattern={on 0.84pt off 2.51pt}]  (38.21,21.73) .. controls (41.93,11.91) and (65.65,8.47) .. (76.44,24.85) .. controls (87.22,41.24) and (79.21,50.82) .. (80.5,65) .. controls (81.79,79.18) and (94.56,90.73) .. (90.5,97) ;
\draw  [dash pattern={on 0.84pt off 2.51pt}]  (38.21,21.73) .. controls (18.49,63.28) and (78.1,122.49) .. (90.5,97) ;
\draw    (59.11,47.71) .. controls (93.44,30.64) and (143.21,81.86) .. (177.53,64.79) ;
\draw    (259.11,47.71) .. controls (293.44,30.64) and (343.21,81.86) .. (377.53,64.79) ;
\draw  [dash pattern={on 0.84pt off 2.51pt}]  (238.21,21.73) .. controls (241.93,11.91) and (265.65,8.47) .. (276.44,24.85) .. controls (287.22,41.24) and (279.21,50.82) .. (280.5,65) .. controls (281.79,79.18) and (294.56,90.73) .. (290.5,97) ;
\draw  [dash pattern={on 0.84pt off 2.51pt}]  (238.21,21.73) .. controls (218.49,63.28) and (278.1,122.49) .. (290.5,97) ;

\draw (196.85,42.13) node [anchor=north west][inner sep=0.75pt]   [align=left] {$\displaystyle =$};
\draw (40.87,31.4) node [anchor=north west][inner sep=0.75pt]   [align=left] {$\displaystyle x$};
\draw (154.88,78.02) node [anchor=north west][inner sep=0.75pt]   [align=left] {$\displaystyle x$};
\draw (119.96,35.59) node [anchor=north west][inner sep=0.75pt]   [align=left] {$\displaystyle id_{x}$};
\draw (240.87,31.4) node [anchor=north west][inner sep=0.75pt]   [align=left] {$\displaystyle x$};

\filldraw (59.11,47.71) circle (1.5pt);
\filldraw (116,55.3) circle (1pt);
\filldraw (259.11,47.71) circle (1.5pt);
\draw (400,80) node {.};
\end{tikzpicture}}
\]
\item[U3](Absorbing framed braids) Absorbing framed braids into internal vertices,
	\begin{equation*}
			\scalebox{0.8}{
\tikzset{every picture/.style={line width=0.75pt}} 
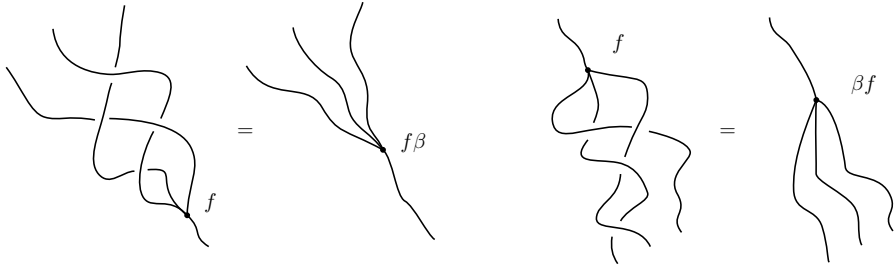
\begin{tikzpicture}[x=0.75pt,y=0.75pt,yscale=-1,xscale=1]

\draw    (36,143.08) .. controls (46.58,157.45) and (51.11,171.82) .. (60.17,174.51) .. controls (69.24,177.2) and (74.25,173.2) .. (91.9,176.31) ;
\draw    (65.46,118.84) .. controls (68.48,143.98) and (92.47,147.28) .. (101.72,147.57) .. controls (110.97,147.87) and (122.11,143.08) .. (133.44,146.67) .. controls (144.77,150.27) and (138.73,162.84) .. (133.44,175.41) ;
\draw    (110.6,104) .. controls (105.31,134.53) and (110.03,122.43) .. (104.74,143.08) ;
\draw    (116.83,208.63) .. controls (109.36,208.63) and (101.3,218.82) .. (94.16,210.43) .. controls (87.03,202.03) and (93.84,186.5) .. (97.19,173.61) .. controls (100.53,160.72) and (100.12,159.65) .. (102.47,152.96) ;
\draw    (100.96,175.41) .. controls (122.52,175.91) and (145.48,180.2) .. (152.27,193.37) .. controls (159.07,206.54) and (150.48,216.71) .. (150.06,236.47) ;
\draw    (150.06,236.47) .. controls (138.73,223.9) and (128.21,233.16) .. (122.76,225.69) .. controls (117.3,218.22) and (120.6,198.76) .. (128.91,183.49) ;
\draw    (125.13,207.73) .. controls (147.8,205.94) and (125.89,216.71) .. (150.06,236.47) ;
\draw    (150.06,236.47) .. controls (160.64,246.35) and (152.33,249.04) .. (163.66,256.22) ;
\filldraw (150.06,236.47) circle (1pt);
\draw    (187.39,142.18) .. controls (201.75,166.43) and (224.92,151) .. (234.98,168.22) .. controls (245.05,185.45) and (248.58,181.69) .. (273.51,195.16) ;
\draw    (216.85,117.94) .. controls (218.11,128.39) and (229.76,143) .. (240.81,149.75) .. controls (251.86,156.5) and (249.37,167.46) .. (253.98,175) .. controls (258.59,182.54) and (265.3,184.48) .. (273.51,195.16) ;
\draw    (260.76,102) .. controls (259.82,107.45) and (252.19,118.13) .. (253.01,130) .. controls (253.83,141.87) and (265.61,147) .. (263.67,166) .. controls (261.73,185) and (268.77,183.31) .. (273.51,195.16) ;
\filldraw (273.51,195.16) circle (1pt);
\draw    (273.51,195.16) .. controls (279.89,201.13) and (281.59,224.28) .. (287.33,227.49) .. controls (293.06,230.7) and (301.62,248.89) .. (306.11,251.73) ;
\draw    (383.92,112.55) .. controls (384.82,124.22) and (392.87,124.22) .. (397.34,131.41) .. controls (401.81,138.59) and (399.13,139.49) .. (402.7,144.88) ;
\draw    (402.7,144.88) .. controls (430.43,152.06) and (439.37,146.67) .. (440.27,160.14) .. controls (441.16,173.61) and (434.9,173.61) .. (425.96,199.65) ;
\draw    (402.7,144.88) .. controls (406.63,159.87) and (378.01,160.59) .. (380.34,177.2) .. controls (382.68,193.82) and (403.6,178.1) .. (430.43,180.8) ;
\draw    (441.16,183.49) .. controls (453.49,186.33) and (465.41,190.65) .. (467.1,196.06) .. controls (468.79,201.47) and (451.9,210.43) .. (459.05,223) .. controls (466.21,235.57) and (453.68,236.47) .. (461.73,247.24) ;
\draw    (402.7,144.88) .. controls (410.75,166.43) and (410.75,166.43) .. (406.28,178.1) ;
\filldraw (402.7,144.88) circle (1pt);
\draw    (402.7,188.88) .. controls (389.89,205.1) and (408.55,200.7) .. (421.49,203.24) .. controls (434.42,205.79) and (438.29,216.75) .. (440.27,222.1) .. controls (442.25,227.45) and (428.64,231.08) .. (423.28,240.06) ;
\draw    (442.06,256.22) .. controls (434.01,239.16) and (407.61,247.91) .. (408.07,238.27) .. controls (408.53,228.62) and (422.38,226.59) .. (423.28,210.43) ;
\draw    (421.49,267) .. controls (417.91,261.61) and (417.01,257.12) .. (417.91,249.94) ;
\draw    (517.19,111.65) .. controls (518.08,123.33) and (526.13,123.33) .. (530.6,130.51) .. controls (535.08,137.69) and (544.02,150.27) .. (546.7,163.73) ;
\draw    (546.7,163.73) .. controls (559.03,166.57) and (564.07,197.92) .. (564.97,206) .. controls (565.86,214.08) and (585.16,209.53) .. (592.32,222.1) .. controls (599.47,234.67) and (586.95,235.57) .. (595,246.35) ;
\draw    (575.32,255.33) .. controls (571.35,246.9) and (576,237.92) .. (569.81,230) .. controls (563.62,222.08) and (546.33,214.88) .. (546.56,210) .. controls (546.79,205.12) and (545.81,179.9) .. (546.7,163.73) ;
\draw    (554.75,266.1) .. controls (553.56,264.3) and (554.22,256.83) .. (550.44,244) .. controls (546.65,231.17) and (532.85,235.52) .. (532.39,220.31) .. controls (531.93,205.1) and (535.97,189.78) .. (546.7,163.73) ;
\filldraw (546.7,163.73) circle (1pt);

\draw (159.29,221.18) node [anchor=north west][inner sep=0.75pt]   [align=left] {$\displaystyle f$};
\draw (180.01,179.93) node [anchor=north west][inner sep=0.75pt]   [align=left] {=};
\draw (283.15,183.47) node [anchor=north west][inner sep=0.75pt]   [align=left] {$\displaystyle f\beta $};
\draw (484.1,179.93) node [anchor=north west][inner sep=0.75pt]   [align=left] {=};
\draw (416.43,120.61) node [anchor=north west][inner sep=0.75pt]   [align=left] {$\displaystyle f$};
\draw (567.11,146.65) node [anchor=north west][inner sep=0.75pt]   [align=left] {$\displaystyle \beta f$};

\draw (615,230) node {.};
\end{tikzpicture}}
\end{equation*}
\end{enumerate}
We note that the relation (U2) allows us to delete any identity labeled vertex which occurs in the middle of an edge, provided doing so does not generate any vertexless loops in the given ribbon graph.
\par

Later we will consider a reduction of the $\msc{A}$-labeled bordism category by the above universal relations
\[
\opn{Bord}_{\msc{A}}/\langle \text{U1, U2, U3}\rangle.
\]
In the linear setting we also introduce a third relation which allows us to transport scalars across vertices as well (see Section \ref{sect:reduced}).

\subsection{RT evaluation for symmetric categories}

We consider evaluation for a symmetric tensor category $\msc{S}$ equipped with the trivial twist.  In this case the evaluation functor
\[
ev_{\msc{S}}:\opn{Rib}_{\msc{S}}\to \msc{S}
\]
does not distinguish between over-crossing and under-crossing, nor does it distinguish between twisted and untwisted ribbons.
\par

Define the symmetric monoidal category $\opn{Str}_{\msc{S}}$ whose objects are $\msc{S}$-labeled disks $D_{\vec{z}}$ and whose morphisms are $\msc{S}$-labeled immersions from a directed graph $T:\Gamma\to D\times [a,b]$, with framings (only) at the vertices and smoothness conditions as in Section \ref{sect:M_rib}.  Two morphisms are considered to be the same if they are related by an oriented diffeomorphism on the linear factor and homotopy on the underlying immersion.
\par

Such an $\msc{S}$-labeled graph $T:\Gamma\to D\times [a,b]$ has ``underlying combinatorics" consisting of the graph $\Gamma$, the specified incoming and outgoing boundary vertices, specifications of spatially incoming and spatially outgoing edges at each vertex, and the $\msc{S}$-labels along edges and vertices.  One sees, via applications of the straight line homotopy, that two morphisms $T$ and $T'$ are equivalent in $\opn{Str}_{\msc{S}}$ if and only if they have the same underlying combinatorics.

Note that we have the forgetful functor $\opn{Rib}_{\msc{S}}\to \opn{Str}_{\msc{S}}$ which is full, but not necessarily faithful.

\begin{lemma}
Let $\msc{S}$ be a symmetric tensor category with trivial twist. The evaluation functor $ev_{\msc{S}}:\Rib_{\msc{S}}\to \msc{S}$ factors through the projection to the category of $\msc{S}$-labeled strings,
\[
\xymatrix{
\Rib_{\msc{S}}\ar[dr]\ar[rr]^{ev_{\msc{S}}} & & \msc{S}\\
	& \opn{Str}_{\msc{S}}\ar@{..>}[ur]_{\exists !} & .
}
\]
\end{lemma}

\begin{proof}
Let $\opn{R}_{\msc{S}}$ be the quotient category obtained from $\opn{Rib}_{\msc{S}}$ by instituting the skein relations which identify over-crossings with under-crossings, and clockwise twists on ribbons with the untwisted ribbons.  One then sees, essentially by applying straight line homotopies again, that two morphisms $T:\Gamma\to D\times [a,b]$ and $T':\Gamma'\to D\times [a,b]$ between labeled disks are identified in $\Hom_{\opn{R}_{\msc{S}}}(D_{\vec{x}},D_{\vec{y}})$ if and only if their underlying combinatorics agree.  Hence the projection $\opn{Rib}_{\msc{S}}\to \opn{Str}_{\msc{S}}$ reduces to an isomorphism $\opn{R}_{\msc{S}}\to \opn{Str}_{\msc{S}}$.  To conclude we note that the evaluation functor is stable under the given skein relations, and hence factors through the projection $\opn{Rib}_{\msc{S}}\to \opn{R}_{\msc{S}}$.
\end{proof}

By an abuse of notation we let
\[
ev_{\msc{S}}:\opn{Str}_{\msc{S}}\to \msc{S}
\]
denote the induced map from $\opn{Str}_{\msc{S}}$, and simply call it the evaluation functor.

\subsection{Tautologial TQFTs for symmetric categories}
\label{sect:s_taut}

For a symmetric monoidal category $\msc{S}$ with trivial twist we have the symmetric monoidal functor
\[
p_{\msc{S}}:\Bord_{\msc{S}}\to \Str_{\msc{S}}
\]
which sends a surface $\Sigma_{\vec{z}}$, with ordered set of labels $\vec{z}:I\to \Sigma\times \opn{obj}(\msc{C})$, to the disk with the same labels $D_{\vec{z}}$ listed in the same order along the $x$-axis.  On morphisms, the functor $p_{\msc{S}}$ sends any bordism $(M,T):\Sigma_{\vec{z}}\to \Sigma'_{\vec{w}}$ to the unique cylindrical bordism $p_{\msc{S}}(M,T):D_{\vec{z}}\to D_{\vec{w}}$ whose underlying combinatorics agrees with that of $T$.  In words, $p_{\msc{S}}$ simply forgets the ambient surfaces on objects, forgets the ambient $3$-manifold on morphisms, and only remembers orderings and labels on edges and vertices in the embedded ribbon graph.

\begin{definition}[\cite{runkelszegedywatts23}]\label{def:ts}
Given any symmetric tensor category $\msc{S}$, the tautological TQFT for $\msc{S}$ is defined as the composite
\[
Ev_{\msc{S}}:=ev_{\msc{S}}\circ p_{\msc{S}}:\Bord_{\msc{S}}\to \Str_{\msc{S}}\to \msc{S}.
\]
\end{definition}

The following fundamental property for the tautological TQFT is employed later in the the text.

\begin{lemma}\label{lem:zs_reason}
Let $\msc{S}$ be a symmetric tensor category, and suppose that two $\msc{S}$-labeled ribbon bordisms
\[
[M,T],\ [M,T']:\Sigma_{\vec{z}}\to \Sigma'_{\vec{w}}
\]
are related by a sequence of skein relations.  Then $Ev_{\msc{S}}(M,T)=Ev_{\msc{S}}(M,T')$.
\end{lemma}

\begin{proof}
It suffices to address the case where $[M,T]$ and $[M,T']$ are related by a single skein relation. In this case we have the implicit cylinder $\xi:\opn{Cyl}\to M$ which isolates the portion of the ribbon graph $T$ to be replaced in the given relation.  Let us note that by shrinking the portion of the graph which lies in the image of the cylinder, then shrinking the image of the cylinder itself, we may assume that $\opn{Cyl}$ has image in an arbitrarily small ball in $M$.
\par

After shrinking if necessary, we can apply a homotopy to move the given cylinder into a collar around the boundary $\Sigma\overset{\sim}\to \partial M$.  Now by cutting $M$, i.e.\ decomposing into a composite, we can reduce to the case where $M$ is the product $\Sigma\times [0,1]$ and the embedded ribbon graph (with isolating cylinder) appears as
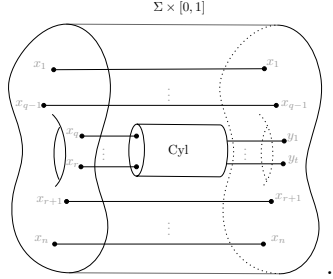
\[
\scalebox{0.5}{
\tikzset{every picture/.style={line width=0.75pt}} 
\begin{tikzpicture}[x=0.75pt,y=0.75pt,yscale=-1,xscale=1]
\draw    (213.5,351) .. controls (186.63,350.7) and (160.37,332.05) .. (156.43,298.02) .. controls (152.5,264) and (166.75,252.43) .. (166.5,227) .. controls (166.25,201.57) and (151.53,184.14) .. (152.51,161.57) .. controls (153.5,139) and (181.84,100.32) .. (210.5,100) ;
\draw    (213.5,351) .. controls (230.72,350.75) and (247.21,336.99) .. (251.86,312.5) .. controls (256.5,288) and (238.43,266.27) .. (239.5,241) .. controls (240.57,215.73) and (259.5,185) .. (254.33,152.34) .. controls (249.16,119.67) and (230.83,100) .. (210.5,100) ;
\draw  [dash pattern={on 0.84pt off 2.51pt}]  (426.5,352) .. controls (399.63,351.7) and (373.37,333.05) .. (369.43,299.02) .. controls (365.5,265) and (379.75,253.43) .. (379.5,228) .. controls (379.25,202.57) and (364.53,185.14) .. (365.51,162.57) .. controls (366.5,140) and (394.84,101.32) .. (423.5,101) ;
\draw    (426.5,352) .. controls (443.72,351.75) and (460.21,337.99) .. (464.86,313.5) .. controls (469.5,289) and (451.43,267.27) .. (452.5,242) .. controls (453.57,216.73) and (472.5,186) .. (467.33,153.34) .. controls (462.16,120.67) and (443.83,101) .. (423.5,101) ;
\draw [color={rgb, 255:red, 0; green, 0; blue, 0 }  ,draw opacity=0.6 ]   (210.5,100) -- (423.5,101) ;
\draw [color={rgb, 255:red, 0; green, 0; blue, 0 }  ,draw opacity=0.6 ]   (213.5,351) -- (426.5,352) ;
\draw    (196.5,191) .. controls (213.5,206) and (214.5,250) .. (199.5,264) ;
\draw    (203.5,257) .. controls (196.5,243) and (196.5,217) .. (204,204) ;
\draw  [dash pattern={on 0.84pt off 2.51pt}]  (405.5,191) .. controls (422.5,206) and (423.5,250) .. (408.5,264) ;
\draw  [dash pattern={on 0.84pt off 2.51pt}]  (412.5,257) .. controls (405.5,243) and (405.5,217) .. (413,204) ;
\draw    (198,146) -- (410.5,145) ;
\filldraw (198,146) circle (1.5pt);
\filldraw (410.5,145) circle (1.5pt);
\draw    (199,322) -- (409.5,322) ;
\filldraw (199,322) circle (1.5pt);
\filldraw (409.5,322) circle (1.5pt);
\draw  (281.9,200.97) -- (365.87,201.72) .. controls (370.21,201.76) and (373.63,213.52) .. (373.5,228) .. controls (373.37,242.48) and (369.75,254.18) .. (365.4,254.14) -- (281.43,253.39) .. controls (277.09,253.36) and (273.67,241.59) .. (273.8,227.11) .. controls (273.93,212.64) and (277.56,200.93) .. (281.9,200.97) .. controls (286.24,201.01) and (289.66,212.78) .. (289.53,227.25) .. controls (289.4,241.73) and (285.77,253.43) .. (281.43,253.39) ;
\draw    (226.5,213) -- (281.5,213) ;
\filldraw (226.5,213) circle (1.5pt);
\filldraw (281.5,213) circle (1.5pt);
\draw    (225.5,244) -- (281.5,244) ;
\filldraw (225.5,244) circle (1.5pt);
\filldraw (281.5,244) circle (1.5pt);
\draw    (373,218) -- (430.5,218) ;
\filldraw (430.5,218) circle (1.5pt);
\draw    (372,241) -- (429.5,241) ;
\filldraw (429.5,241) circle (1.5pt);
\draw    (188,182) -- (422.5,182) ;
\filldraw (188,182) circle (1.5pt);
\filldraw (422.5,182) circle (1.5pt);
\draw    (211,279) -- (417.5,279) ;
\filldraw (211,279) circle (1.5pt);
\filldraw (417.5,279) circle (1.5pt);

\draw (297,75) node [anchor=north west][inner sep=0.75pt]   [align=left] {$\displaystyle \Sigma \times [0,1]$};
\draw (176,136) node [anchor=north west][inner sep=0.75pt]  [color={rgb, 255:red, 0; green, 0; blue, 0 }  ,opacity=0.4 ] [align=left] {$\displaystyle x_{1}$};
\draw (209,204) node [anchor=north west][inner sep=0.75pt]  [color={rgb, 255:red, 0; green, 0; blue, 0 }  ,opacity=0.4 ] [align=left] {$\displaystyle x_{q}$};
\draw (209,235) node [anchor=north west][inner sep=0.75pt]  [color={rgb, 255:red, 0; green, 0; blue, 0 }  ,opacity=0.4 ] [align=left] {$\displaystyle x_{r}$};
\draw (177,312) node [anchor=north west][inner sep=0.75pt]  [color={rgb, 255:red, 0; green, 0; blue, 0 }  ,opacity=0.4 ] [align=left] {$\displaystyle x_{n}$};
\draw (416,312) node [anchor=north west][inner sep=0.75pt]  [color={rgb, 255:red, 0; green, 0; blue, 0 }  ,opacity=0.4 ] [align=left] {$\displaystyle x_{n}$};
\draw (411,134) node [anchor=north west][inner sep=0.75pt]  [color={rgb, 255:red, 0; green, 0; blue, 0 }  ,opacity=0.4 ] [align=left] {$\displaystyle x_{1}$};
\draw (311,156) node [anchor=north west][inner sep=0.75pt]  [color={rgb, 255:red, 0; green, 0; blue, 0 }  ,opacity=0.4 ] [align=left] {$\displaystyle \vdots $};
\draw (388,216) node [anchor=north west][inner sep=0.75pt]  [color={rgb, 255:red, 0; green, 0; blue, 0 }  ,opacity=0.4 ] [align=left] {$\displaystyle \vdots $};
\draw (245,216) node [anchor=north west][inner sep=0.75pt]  [color={rgb, 255:red, 0; green, 0; blue, 0 }  ,opacity=0.4 ] [align=left] {$\displaystyle \vdots $};
\draw (311,291) node [anchor=north west][inner sep=0.75pt]  [color={rgb, 255:red, 0; green, 0; blue, 0 }  ,opacity=0.4 ] [align=left] {$\displaystyle \vdots $};
\draw (158,175) node [anchor=north west][inner sep=0.75pt]  [color={rgb, 255:red, 0; green, 0; blue, 0 }  ,opacity=0.4 ] [align=left] {$\displaystyle x_{q-1}$};
\draw (426,175) node [anchor=north west][inner sep=0.75pt]  [color={rgb, 255:red, 0; green, 0; blue, 0 }  ,opacity=0.4 ] [align=left] {$\displaystyle x_{q-1}$};
\draw (432,209) node [anchor=north west][inner sep=0.75pt]  [color={rgb, 255:red, 0; green, 0; blue, 0 }  ,opacity=0.4 ] [align=left] {$\displaystyle y_{1}$};
\draw (433,231) node [anchor=north west][inner sep=0.75pt]  [color={rgb, 255:red, 0; green, 0; blue, 0 }  ,opacity=0.4 ] [align=left] {$\displaystyle y_{t}$};
\draw (420,269) node [anchor=north west][inner sep=0.75pt]  [color={rgb, 255:red, 0; green, 0; blue, 0 }  ,opacity=0.4 ] [align=left] {$\displaystyle x_{r+1}$};
\draw (179,272) node [anchor=north west][inner sep=0.75pt]  [color={rgb, 255:red, 0; green, 0; blue, 0 }  ,opacity=0.4 ] [align=left] {$\displaystyle x_{r+1}$};
\draw (312,220) node [anchor=north west][inner sep=0.75pt]   [align=left] {Cyl};

\end{tikzpicture}}.
\]

Let $\opn{Cyl}_{T}$ and $\opn{Cyl}_{T'}$ be the cylinders with embedded ribbon graph given by intersecting $\xi$ with $T$ and $T'$ respectively.  By assumption we have
\[
ev_{\msc{S}}(\opn{Cyl}_{T})=ev_{\msc{S}}(\opn{Cyl}_{T'}).
\]
By the definition of the tautological theory, and monoidality of RT evaluation, we have in this case
\[
\begin{array}{rl}
Ev_{\msc{S}}(M,T) & = id_{x_1\ot\cdots \ot x_{t-1}}\ot ev_{\msc{S}}(\opn{Cyl}_{T})\ot id_{x_m\ot\cdots \ot x_n}\vspace{1mm}\\
& = id_{x_1\ot\cdots \ot x_{t-1}}\ot ev_{\msc{S}}(\opn{Cyl}_{T'})\ot id_{x_m\ot\cdots  \ot x_n} = Ev_{\msc{S}}(M,T').
\end{array}
\]
\end{proof}

\section{The separation lemma}
\label{sect:separate}

Let $\msc{S}$ be a semisimple symmetric tensor category with trivial twist and $\msc{A}$ be any ribbon tensor category.  Following methods of Runkel, Szegedy, and Watts \cite{runkelszegedywatts23}, we claim that any TQFT $Z:\Bord_{\msc{A}}\to \opn{Vect}$ admits a ``tautological extension" to the base change $\msc{S}\ot\msc{A}$.  This tautological extension is a TQFT which is now valued in $\msc{S}$,
\[
``Ev_{\msc{S}}\ot Z":\Bord_{\msc{S}\ot\msc{A}}\to \msc{S}.
\]
We deal with the details for this construction in Section \ref{sect:t_ext} below.
\par

For now, we establish a lemma which proposes that one can separate bordisms in $\Bord_{\msc{S}\otimes \msc{A}}$ into an $\msc{S}$-factor and an $\msc{A}$-factor, at least after applying an additive completion.  Specifically, we construct a symmetric monoidal functor
\begin{equation}\label{eq:492}
\Delta:\Bord_{\msc{S}\otimes\msc{A}}\to (k\Bord_{\msc{S}}^{\red}\boxtimes k\Bord_{\msc{A}}^{\red})^{\opn{add}}
\end{equation}
where each $k\Bord_{\square}^{\red}$ is a category of linearly reduced bordisms (see Section \ref{sect:reduced} below).  We then obtain the claimed extension by applying the tautological TQFT $Ev_{\msc{S}}$ in the $\msc{S}$-factor and $Z$ in the $\msc{A}$-factor.

\subsection{Simultaneous bordisms}

We have the rigid (non-linear) braided monoidal category $\msc{S}\times\msc{A}$ and structure map $\msc{S}\times \msc{A}\to \msc{S}\otimes\msc{A}$ which is furthermore braided monoidal.  We have the category of labeled bordisms $\Bord_{\msc{S}\times \msc{A}}$, which one can think of as marked surfaces and ribbon bordisms with pairs of independent labels from $\msc{S}$ and $\msc{A}$.  Indeed, via the projections $\msc{S}\times \msc{A}\to \msc{S}$ and $\msc{S}\times \msc{A}\to \msc{A}$ we have the two maps
\[
\xymatrixrowsep{2mm}
\xymatrix{
&\Bord_{\msc{S}\times \msc{A}}\ar[dl]\ar[dr]\\
\Bord_{\msc{S}} & & \Bord_{\msc{A}}
}
\]
which provide a faithful, symmetric monoidal embedding into the product
\[
\Bord_{\msc{S}\times \msc{A}}\ \subseteq\ \Bord_{\msc{S}}\times \Bord_{\msc{A}}.
\]
The image of this embedding consists of simultaneous pairs of surfaces, and bordisms, whose underlying geometries in the $\msc{S}$ and $\msc{A}$ factors agree.
\par

Via the structure map $\msc{S}\times\msc{A}\to \msc{S}\otimes\msc{A}$ we also obtain a symmetric monoidal functor
\[
\opn{ccat}:\Bord_{\msc{S}\times \msc{A}}\to \Bord_{\msc{S}\otimes\msc{A}},
\]
which we refer to as the concatenation map.  We claim that, after applying the appropriate linearization and relations, the concatenation map becomes a symmetric monoidal equivalence.

\subsection{The reduced concatenation functor}
\label{sect:reduced}

Let $\msc{C}$ be a linear ribbon monoidal category and consider the symmetric monoidal category obtained via the free linearization $k\Bord_{\msc{C}}$.  We produce a new linear symmetric monoidal category $k\Bord_{\msc{C}}^{\red}$ by imposing the universal skein relations (U1)--(U3) from Section \ref{sect:sk_rel} and the following universal linear relation:
\begin{enumerate}
\item[U4] (Linear expansion) Given a $\msc{C}$-labeled ribbon bordism $T:\Gamma\to M$ and specified vertex $v$ which is labeled by a map of the form $f=c_1f_{1}-c_2f_{2}$, and supposing $(M,T_i)$ is obtained from $(M,T)$ by replacing the label $f$ by $f_{i}$ at $v$, we impose an equality
\[
(M,T) = c_1(M,T_1)+c_2(M,T_2).
\]
\end{enumerate}

\begin{definition}\label{def:bord_red}
For a linear ribbon category $\msc{C}$, the category of linearly reduced bordisms
\[
k\Bord_{\msc{C}}^{\red}=(k\Bord_{\msc{C}})^{\opn{red}}=k\Bord_{\msc{C}}/\langle\text{U1--U4}\rangle
\]
is the category whose objects are $\msc{C}$-labeled surfaces and whose morphisms are the quotients
\[
\Hom_{k\Bord_{\msc{C}}^{\red}}(\Sigma_{\vec{x}},\Sigma'_{\vec{y}})=\frac{k\Hom_{\Bord_{\msc{C}}}(\Sigma_{\vec{x}},\Sigma'_{\vec{y}})}{\opn{Span}_k\left\{\ 
\text{relations U1--U4}\ \right\}}.
\]
\end{definition}

Since the relations (U1)--(U4) are all preserved under disjoint union the symmetric monoidal structure on $k\Bord_{\msc{C}}$ induces a unique symmetric monoidal structure on the reduced category under which the canonical projection $k\Bord_{\msc{C}}\to k\Bord_{\msc{C}}^{\red}$ is a symmetric monoidal functor.
\par

Let us consider the possible additional relation:
\begin{enumerate}
\item[U5] (Expansion over coproducts) For an edge $e$ in a $\msc{C}$-labeled ribbon bordism $T:\Gamma\to M$ which is labeled by a sum $x_e=y_e\oplus y'_e$, and $(M,Q)$ and $(M,Q')$ obtained from $(M,T)$ by replacing $x_e$ with $y_e$ and $y'_e$ respectively, we impose an equality
\[
(M,T)=(M,Q)+(M,Q').
\]
\end{enumerate}
To be clear, when we replace $x_e$ with $y_e$ we also replace the label $f_v$ at each vertex attached to $e$ with the appropriate composite $p f_v$ or $f_vi$, where
\[
\begin{array}{c}
p:x^{\pm}_{e_1}\cdots \ot x_e^{\pm}\ot \cdots x_{e_n}^{\pm}\to x^{\pm}_{e_1}\cdots \ot y_e^{\pm}\ot\cdots x_{e_n}^{\pm}\vspace{2mm}\\
\text{and}\ \ i:x^{\pm}_{e_1}\cdots \ot y_e^{\pm}\ot\cdots x_{e_n}^{\pm}\to x^{\pm}_{e_1}\cdots \ot x_e^{\pm}\ot\cdots x_{e_n}^{\pm}
\end{array}
\]
are the projection and inclusion induced by the implicit projection $x_e\to y_e$ and inclusion $y_e\to x_e$.  One performs the same manipulations when replacing $x_e$ with $y'_e$ as well.  We leave the proof of the following to the interested reader.

\begin{lemma}\label{lem:redundant}
The relation {\rm(U5)} holds in $k\Bord_{\msc{C}}^{\red}$.
\end{lemma}
 
For a pair of ribbon tensor categories $\msc{S}$ and $\msc{A}$ we also take
\[
k\Bord_{\msc{S}\times\msc{A}}^{\red}=k\Bord_{\msc{S}\times\msc{A}}/\langle\text{U1--U3, U$'$4}\rangle,
\]
where U$'$4 institutes \emph{bi}linear expansion at the vertices.  We refer to this quotient category as the category of bilinearly reduced bordisms for $\msc{S}\times\msc{A}$.
\par

As a consequence of bilinearity for the structure map $\msc{S}\times \msc{A}\to \msc{S}\otimes \msc{A}$, the linearized concatenation functor
\[
k\Bord_{\msc{S}\times \msc{A}}\to k\Bord_{\msc{S}\otimes\msc{A}}
\]
sends relations of the form (U$'$4) to relations of the form (U4) in $\Bord_{\msc{S}\otimes \msc{A}}$.  Hence we obtain an induced functor from bilinearly reduced bordisms to its linearly reduced counterpart,
\begin{equation}\label{eq:conc}
\opn{ccat}^{\opn{red}}:k\Bord_{\msc{S}\times\msc{A}}^{\red}\to k\Bord_{\msc{S}\otimes\msc{A}}^{\red}.
\end{equation}
One can prove the following by a relatively involved, but direct analysis.  The details are covered in Appendix \ref{sect:proofs}.

\begin{proposition}\label{prop:conc}
The reduced concatenation functor \eqref{eq:conc} is fully faithful.
\end{proposition}

Having accepted Proposition \ref{prop:conc} for the moment, the functor $\opn{ccat}^{\opn{red}}$ \emph{would} be an equivalence if it were essentially surjective.  This is, however, clearly not the case.  Indeed, an object $\Sigma_{\vec{z}}$ in $k\Bord_{\msc{S}\ot\msc{A}}^{\red}$ is in the essential image of $\opn{ccat}^{\opn{red}}$ if and only if all of its labels $z_i$ are in the image of the structure map $\msc{S}\times\msc{A}\to \msc{S}\ot\msc{A}$.  We remedy this situation by imposing an additive closure.

\begin{corollary}\label{cor:conc}
Suppose that $\msc{S}$ is semisimple.  After applying additive closures, the concatenation functor induces a symmetric monoidal equivalence
\[
(\opn{ccat}^{\opn{red}})^{\opn{add}}:(k\Bord_{\msc{S}\times\msc{A}}^{\red})^{\opn{add}}\overset{\sim}\to (k\Bord_{\msc{S}\otimes\msc{A}}^{\red})^{\opn{add}}.
\]
\end{corollary}

We prove Corollary \ref{cor:conc} after establishing a splitting lemma for surfaces in the additively closed, reduced bordism category.

\subsection{Proof of Corollary \ref{cor:conc}}

Let $\msc{C}$ be a linear ribbon monoidal category and consider a marked surface $\Sigma_{\vec{x}}$, with marking objects $(x_k:k\in K)$.  Suppose at a given index $j$ that the object $x_j$ decomposes into a sum $x_{j1}\oplus x_{j2}$ via projections and inclusions
\[
p_{j\lambda}:x_j\to x_{j\lambda}\ \ \text{and}\ \ i_{j\lambda}:x_{j\lambda}\to x_j.
\]
Let $\Sigma_{\vec{x}_{\lambda}}$ be the new marked surface obtained by replacing the marking object $x_j$ with the respective object $x_{j\lambda}$.  By applying the identity on all objects $x_k$ with $k\neq j$, and $p_{j\lambda}$ or $i_{j\lambda}$ at the $j$-th marking object (depending on orientations) we obtain morphisms
\[
\Sigma_{p_{\lambda}}:\Sigma_{\vec{x}}\to \Sigma_{\vec{x}_{\lambda}}\ \ \text{and}\ \ \Sigma_{i_{\lambda}}:\Sigma_{\vec{x}_{\lambda}}\to \Sigma_{\vec{x}}
\]
in the $\msc{C}$-labeled bordism category.  The underlying $3$-manifold for these bordisms is $\Sigma\times [0,1]$, and the given morphisms label straight lines traveling from the markings on $\Sigma\times \{0\}$ to their respective markings on $\Sigma\times\{1\}$.  (See Section \ref{sect:a_sigma}.)

\begin{lemma}\label{lem:sigma_sum}
Consider a marked surface $\Sigma_{\vec{x}}$ in $\Bord_{\msc{C}}^{\opn{red}}$ and a marking object $x_j$ which decomposes as a sum $x_j\cong x_{j1}\oplus x_{j2}$ as above.  In this case the associated projections $\Sigma_{p_{\lambda}}$ determine an isomorphism
\begin{equation}\label{eq:2156}
[\Sigma_{p_{1}}\ \Sigma_{p_{2}}]^t:\Sigma_{\vec{x}}\to \Sigma_{\vec{x}_1}\oplus \Sigma_{\vec{x}_2}
\end{equation}
in the additive closure $(k\Bord_{\msc{C}}^{\opn{red}})^{\opn{add}}$.
\end{lemma}

\begin{proof}
One observes directly that the map
\[
[\Sigma_{i_{1}}\ \Sigma_{i_{2}}]:\Sigma_{\vec{x}_1}\oplus \Sigma_{\vec{x}_2}\to \Sigma_{\vec{x}}
\]
is inverse to the map \eqref{eq:2156}, via internal composition (U1) and the linear relation (U4).
\end{proof}

As a corollary we find that, given a linear ribbon category $\msc{C}$ and a ribbon monoidal functor $F:\msc{C}'\to \msc{C}$ for which $\msc{C}$ is generated by the image of $F$ under direct sums, the corresponding map on additive closures
\[
F_\ast^{\opn{add}}:(k\Bord_{\msc{C}'})^{\opn{add}}\to (k\Bord_{\msc{C}}^{\opn{red}})^{\opn{add}}
\]
is essentially surjective. We now record the proof of Corollary \ref{cor:conc}.

\begin{proof}[Proof of Corollary \ref{cor:conc}]
Since $\msc{S}$ is semisimple the product $\msc{S}\otimes \msc{A}$ is generated by the image of the structure map $\msc{S}\times \msc{A}\to \msc{S}\otimes\msc{A}$ under direct sums (see Lemma \ref{lem:753}).  Thus, by the above discussion, we have that the additive closure of the concatenation functor
\begin{equation}\label{eq:2253}
(\opn{ccat}^{\opn{red}})^{\opn{add}}:(k\Bord_{\msc{S}\times\msc{A}}^{\red})^{\opn{add}}\to (k\Bord_{\msc{S}\otimes\msc{A}}^{\opn{red}})^{\opn{add}}
\end{equation}
is essentially surjective. Since the additive closure of any fully faithful functor remains fully faithful, we also have from Proposition \ref{prop:conc} that the functor \eqref{eq:2253} is fully faithful. Hence it is an equivalence.
\end{proof}

\subsection{Constructing the separation map}

The construction of the separating map $\Delta$ from \eqref{eq:492} is now straightforward.
\par

We consider the map into the product
\[
k\Bord_{\msc{S}\times\msc{A}}\to k\Bord_{\msc{S}}\times k\Bord_{\msc{A}}
\]
provided by the two projections $\msc{S}\times\msc{A}\to \msc{S}$ and $\msc{S}\times \msc{A}\to \msc{A}$, which then composes to provide a map to the internal tensor product of the reduced bordism categories
\[
k\Bord_{\msc{S}\times\msc{A}}\to k\Bord_{\msc{S}}^{\red}\boxtimes k\Bord_{\msc{A}}^{\red}.
\]
One sees directly that this map vanishes on the universal relations (U1)--(U$'$4) so that we obtain an induced symmetric monoidal functor from the bilinear reduction
\[
k\Bord_{\msc{S}\times\msc{A}}^{\red}\to k\Bord_{\msc{S}}^{\red}\boxtimes k\Bord_{\msc{A}}^{\red}.
\]

\begin{lemma}\label{lem:sep}
There is a linear symmetric monoidal functor
\[
\Delta^{\opn{red}}:k\Bord_{\msc{S}\otimes \msc{A}}^{\red}\to (k\Bord_{\msc{S}}^{\red}\boxtimes k\Bord_{\msc{A}}^{\red})^{\opn{add}}
\]
which fits into a ($2$-)commutative diagram
\begin{equation}\label{eq:1411}
\xymatrix{
	& k\Bord_{\msc{S}\times\msc{A}}^{\red}\ar[dr]\ar[dl]_{\opn{ccat}^{\opn{red}}}\\
k\Bord_{\msc{S}\otimes \msc{A}}^{\red}\ar[rr]_{\Delta^{\opn{red}}}& & (k\Bord_{\msc{S}}^{\red}\boxtimes k\Bord_{\msc{A}}^{\red})^{\opn{add}}.
}
\end{equation}
\end{lemma}

\begin{proof}
By Corollary \ref{cor:conc} the additive closure of the reduced concatenation functor is an equivalence.  Hence from any symmetric inverse
\[
\vartheta:(k\Bord_{\msc{S}\otimes\msc{A}}^{\opn{red}})^{\opn{add}}\to (k\Bord_{\msc{S}\times\msc{A}}^{\opn{red}})^{\opn{add}}
\]
we can define $\Delta^{\opn{red}}$ as the composite
\[
\begin{array}{l}
k\Bord_{\msc{S}\otimes \msc{A}}^{\red}\to (k\Bord_{\msc{S}\otimes\msc{A}}^{\opn{red}})^{\opn{add}}\overset{\vartheta}\to (k\Bord_{\msc{S}\times\msc{A}}^{\opn{red}})^{\opn{add}}\to (k\Bord_{\msc{S}}^{\red}\boxtimes k\Bord_{\msc{A}}^{\red})^{\opn{add}}.
\end{array}
\]
\end{proof}

\begin{remark}
One can check that the diagram \eqref{eq:1411} fixes the map $\Delta^{\opn{red}}$ up to a unique natural isomorphism.
\end{remark}

\begin{definition}\label{def:delta}
Given a symmetric tensor category $\msc{S}$ and a ribbon tensor category $\msc{A}$, with $\msc{S}$ semisimple, the separation functor
\[
\Delta:\Bord_{\msc{S}\ot\msc{A}}\to (k\Bord_{\msc{S}}^{\red}\boxtimes k\Bord_{\msc{A}}^{\red})^{\opn{add}}
\]
is the composite of the projection $\Bord_{\msc{S}\ot\msc{A}}\to k\Bord_{\msc{S}\ot\msc{A}}^{\opn{red}}$ with the map $\Delta^{\opn{red}}$ from Lemma \ref{lem:sep}.
\end{definition}

\section{Base change for TQFTs}
\label{sect:t_ext}

We explain how any reasonable field theory $Z:\Bord_{\msc{A}}\to \opn{Vect}$ from the $\msc{A}$-labeled bordism category extends naturally to an $\msc{S}$-valued theory
\[
Z(\msc{S}):\Bord_{\msc{S}\ot\msc{A}}\to \msc{S},
\]
whenever $\msc{S}$ is semisimple and symmetric.  We make a similar claim whenever $\Bord_{\square}$ is replaced with the so-called ``admissible" bordism category.

\begin{remark}
Our consideration of the admissible bordism category is purely practical, as nonsemisimple theories tend to require some constraints on the manifolds under consideration.
\end{remark}

We are especially interested in the case where our symmetric base $\msc{S}$ is the category $\msc{S}=\opn{Vect}^{\mathbb{Z}}$ of $\mathbb{Z}$-graded vector spaces, along with its signed symmetry. In this case the category of cochains $\opn{Ch}(\msc{A})$ is located as a distinguished subcategory in the Kelly product
\[
\opn{Ch}(\msc{A})\ \subseteq\ \msc{A}^{\mathbb{Z}}=\opn{Vect}^{\mathbb{Z}}\otimes \msc{A}.
\]
Ultimately, our cochain valued theory $Z_{\opn{Ch}(\msc{A})}$ discussed in the introduction will be cut out of the base changed theory $Z_{\msc{A}}(\opn{Vect}^{\mathbb{Z}})$ which we construct below.

\subsection{Main conclusion}

\begin{definition}
Let $\msc{A}$ be a ribbon tensor category.  We call a TQFT $Z:\Bord_{\msc{A}}\to \opn{Vect}$ reasonable if there are equalities of linear maps
\[
Z(M,T)=Z(M,T')
\]
whenever the pair of ribbon bordisms $(M,T)$ and $(M,T')$ are related by a sequence of skein relations, and a similar equality of linear maps
\[
Z(M,T)=c_1Z(M,T_1)+c_2Z(M,T_2)
\]
whenever the bordisms $(M,T)$ and $(M,T_i)$ are related by a linear relation as in (U4).
\end{definition}

Given a reasonable field theory $Z$ we let
\[
Z:k\Bord^{\opn{red}}_{\msc{A}}\to \opn{Vect}
\]
denote the induced symmetric monoidal functor from the linearly reduced bordism category, by an abuse of notation.  As the name suggests, all TQFTs which we consider in this text--in particular all Reshetikhin-Turaev style TQFTs--are reasonable in the above sense.  For example, for any symmetric tensor category $\msc{S}$ with trivial twist, the tautological TQFT $Ev_{\msc{S}}:\Bord_{\msc{S}}\to \msc{S}$ is reasonable. Stability under skein relations was covered in Lemma \ref{lem:zs_reason}, and stability under the universal linear relation U4 follows from the fact that ribbon evaluation (from the linearized ribbon category) respects such relations.
\par

The primary claim of the section is as follows.

\begin{theorem}[{cf.\ \cite[Section 4.3]{runkelszegedywatts23}}]\label{thm:Z_AS}
Let $\msc{S}$ be a semisimple symmetric tensor category and $\msc{A}$ be a ribbon tensor category.  For any reasonable TQFT $Z:\Bord_{\msc{A}}\to \opn{Vect}$, there is a canonically defined TQFT
\[
Z(\msc{S}):\Bord_{\msc{S}\otimes \msc{A}}\to \msc{S}
\]
whose restrictions to $\Bord_{\msc{A}}\subseteq \Bord_{\msc{S}\otimes \msc{A}}$ recovers the composition of $Z$ with the unit map $\opn{Vect}\to \msc{S}$.
\end{theorem}

For a given semisimple symmetric tensor category $\msc{S}$ we refer to the theory $Z(\msc{S})$ informally as the \emph{base change} of the theory $Z$ along the unit $\opn{Vect}\to \msc{S}$. We construct the TQFT $Z(\msc{S})$ below, then verify that it has the claimed restriction to $\Bord_{\msc{A}}$.

\subsection{Constructing $Z(\msc{S})$}

Throughout this section we fix a semisimple symmetric tensor category $\msc{S}$ with trivial twist, a ribbon tensor category $\msc{A}$, and a reasonable TQFT $Z$.
\par

Pairing $Z$ with the tautological TQFT for $\msc{S}$, we obtain a linear symmetric monoidal functor
\[
Ev_{\msc{S}}\boxtimes Z:k\Bord_{\msc{S}}^{\red}\boxtimes k\Bord_{\msc{A}}^{\red}\to \msc{S}\boxtimes\opn{Vect}.
\]
We compose with the action map $\msc{S}\boxtimes \opn{Vect}\to \msc{S}$ to obtain a second linear symmetric monoidal functor
\[
k\Bord_{\msc{S}}^{\red}\boxtimes k\Bord_{\msc{A}}^{\red}\to \msc{S}.
\]
As the image here is additive, this functor extends uniquely to a third such functor
\[
(k\Bord_{\msc{S}}^{\red}\boxtimes k\Bord_{\msc{A}}^{\red})^{\opn{add}}\to \msc{S}.
\]
We denote this final map by $Ev_{\msc{S}}\cdot Z$.

\begin{definition}\label{def:Z_AS}
Given a reasonable field theory $Z:\Bord_{\msc{A}}\to \opn{Vect}$, the base changed theory $Z(\msc{S}):\Bord_{\msc{S}\otimes\msc{A}}\to \msc{S}$ is defined as the composite
\[
\begin{array}{l}
Z(\msc{S}):=\vspace{2mm}\\
\Bord_{\msc{S}\otimes \msc{A}}\overset{\Delta}\longrightarrow (k\Bord_{\msc{S}}^{\red}\boxtimes k\Bord_{\msc{A}}^{\red})^{\opn{add}}\overset{Ev_{\msc{S}}\cdot Z}\longrightarrow \msc{S}.
\end{array}
\]
\end{definition}

\subsection{Proof of Theorem \ref{thm:Z_AS}}

We have already produced the TQFT $Z(\msc{S})$.  We need only verify that it recovers $Z$ when restricted to the subcategory of $\msc{A}$-labeled ribbon bordisms in $\Bord_{\msc{S}\otimes \msc{A}}$.

\begin{proof}[Proof of Theorem \ref{thm:Z_AS}]
The inclusion $\msc{A}\to \msc{S}\otimes\msc{A}$ factors through the inclusion $\msc{A}=\{\ast\}\times \msc{A}\to \msc{S}\times \msc{A}$ so that we have the diagram
\[
\xymatrix{
\Bord_{\msc{A}}\ar[rr]\ar[dr]   & & \Bord_{\msc{S}\times \msc{A}}\ar[dl]^{\opn{ccat}}\\
 & \Bord_{\msc{S}\otimes \msc{A}} & .
}
\]
This diagram linearizes then reduces to provide a diagram
\[
\xymatrix{
k\Bord_{\msc{A}}\ar[rr]\ar[dr] & & (k\Bord_{\msc{S}\times \msc{A}}^{\opn{red}})^{\opn{add}}\ar[dl]^{(\opn{ccat}^{\opn{red}})^{\opn{add}}}\\
 & (k\Bord_{\msc{S}\otimes \msc{A}}^{\opn{red}})^{\opn{add}} & .
}
\]
Hence, for any inverse $\vartheta$ to $\opn{ccat}^{\opn{red},\opn{add}}$, the composite
\[
\Bord_{\msc{A}}\to (k\Bord_{\msc{S}\otimes \msc{A}}^{\opn{red}})^{\opn{add}}\overset{\vartheta}\to (k\Bord_{\msc{S}\times \msc{A}}^{\opn{red}})^{\opn{add}}
\]
is isomorphic to the inclusion composed with reduction
\[
\Bord_{\msc{A}}\to k\Bord_{\msc{S}\times \msc{A}}\to (k\Bord_{\msc{S}\times \msc{A}}^{\opn{red}})^{\opn{add}}.
\]
Thus the separation map $\Delta$, which is defined as the composite of $\vartheta|_{\Bord_{\msc{S}\otimes \msc{A}}}$ with the natural map
\[
(k\Bord_{\msc{S}\times \msc{A}}^{\opn{red}})^{\opn{add}}\to (k\Bord_{\msc{S}}^{\red}\boxtimes k\Bord_{\msc{A}}^{\red})^{\opn{add}},
\]
restricts to $\Bord_{\msc{A}}$ to recover the sequence
\[
\Bord_{\msc{A}}\to k\Bord_{\msc{S}\times \msc{A}}\to k\Bord_{\msc{S}}\boxtimes k\Bord_{\msc{A}}\to (k\Bord_{\msc{S}}^{\red}\boxtimes k\Bord_{\msc{A}}^{\red})^{\opn{add}}.
\]
Now applying $Ev_\msc{S}\cdot Z$ recovers the sequence
\begin{equation}\label{eq:1110}
\Bord_{\msc{A}}\to \Bord_{\msc{S}\times \msc{A}}\to \Bord_{\msc{S}}\times \Bord_{\msc{A}}\overset{Ev\times Z}\to \msc{S}\times \opn{Vect}\overset{\opn{act}}\to \msc{S}.
\end{equation}
So, in total, we find that the restriction $Z(\msc{S})|_{\Bord_{\msc{A}}}$ is isomorphic to the above composite \eqref{eq:1110}.
\par

Restricting along the unit map $\ast\to \msc{S}$, for $\msc{S}$ considered as a non-linear monoidal category, produces a (2-)diagram
\[
\xymatrix{
\Bord_{\ast}\ar[r]^{Ev_\ast}\ar[d] & \ast\ar[d]^{\opn{unit}}\\
\Bord_{\msc{S}}\ar[r]_{Ev_\msc{S}} & \msc{S},
}
\]
via naturality of evaluation on string diagrams.  In particular, the TQFT $Ev_{\msc{S}}$ is constant along surfaces and bordisms labeled by the unit object and identity morphisms.  Hence, via the (2-)diagram
\[
\xymatrix{
\Bord_{\msc{A}}\ar[r]\ar[d] & \Bord_{\ast\times \msc{A}}\ar[r] & \Bord_\ast\times \Bord_{\msc{A}}\ar[d]\ar[r]^{Ev_\ast\times Z} & \ast\times \opn{Vect}\ar[r]^{act\cong p_2} & \opn{Vect}\ar[d]^{\opn{unit}}\\
\Bord_{\msc{S}\times \msc{A}}\ar[rr] & & \Bord_{\msc{S}}\times \Bord_\msc{A}\ar[rr]_{Ev_{\msc{S}}\cdot Z}  & & \msc{S}
}
\]
we see, by following around the top, that the sequence \eqref{eq:1110} is isomorphic to the TQFT $Z$ composed with the unit map $\opn{Vect}\to \msc{S}$.  This gives the desired identification $Z(\msc{S})|_{\Bord_{\msc{A}}}\cong \opn{unit}\circ Z$.
\end{proof}

\subsection{Theorem \ref{thm:Z_AS} under admissibility}
\label{sect:admis}

The particular theory we are interested in is not defined on the full category of $\msc{A}$-labeled ribbon bordisms, but on a subcategory of so-called ``admissible" bordisms.  Let us introduce the relevant category of admissible bordisms, then translate Theorem \ref{thm:Z_AS} to the admissible setting.

\begin{lemma}
\begin{enumerate}
\item For arbitrary $s$ in $\msc{S}$ and projective $x$ in $\msc{A}$, the product $s\ot x$ is projective in $\msc{S}\otimes \msc{A}$.
\item An object $z$ in $\msc{S}\otimes \msc{A}$ is projective if and only if $z$ is isomorphic to a sum $\oplus_i s_i\ot x_i$ in which all the $x_i$ are projective in $\msc{A}$.
\end{enumerate}
\end{lemma}

\begin{proof}
For a set $\Lambda$ and a $\Lambda$-indexed collection of additive categories $\msc{C}_\lambda$, take $\oplus_{\lambda\in \Lambda}\msc{C}_\lambda$ the full subcategory in the product $\prod_{\lambda\in \Lambda}\msc{C}_\lambda$ whose objects are all tuples $(x_\lambda:\lambda\in \Lambda)$ in which all but finitely many of the $x_\lambda$ are zero.  As an additive category $\msc{S}\otimes \msc{A}$ is equivalent to the sum $\oplus_{\lambda\in \opn{Irred}(\msc{S})}\msc{A}$ via the map
\[
[s_\lambda\ot-:\lambda\in \opn{Irred}(\msc{S})]:\oplus_{\lambda}\msc{A}\to \msc{S}\otimes \msc{A}.
\]
(This follows by Lemma \ref{lem:708} for example.)  Both statements (1) and (2) follow from the above equivalence.
\end{proof}

\begin{definition}[\cite{derenzietal23}]
A bordism $(M,T)$ in $\Bord_{\msc{S}\otimes\msc{A}}$ is called admissible if, for each component $M_c$ in $M$, $M_c$ contains an edge in $T$ which is labeled by a projective in $\msc{S}\otimes \msc{A}$, or $M_c$ has nonvanishing outgoing boundary.
\par

A bordism $(M',T'_{\msc{S}},T'_{\msc{A}})$ in $\Bord_{\msc{S}\times \msc{A}}$ is called admissible if every component $M'_c$ in $M'$ either contains an edge in $T'_{\msc{A}}$ which is labeled by a projective, or has nonvanishing outgoing boundary.
\end{definition}

\begin{remark}
One can consider the case $\msc{S}=\opn{Vect}$ to recover the original notion of admissible bordisms in $\Bord_{\msc{A}}$, as it appears in the text \cite{derenzietal23}.
\end{remark}

We note that admissible bordisms are stable under composition in their respective ambient categories.

\begin{definition}[\cite{derenzietal23}]
The subcategories
\[
\Bord^{\opn{adm}}_{\msc{S}\otimes\msc{A}}\ \subseteq\ \Bord_{\msc{S}\times \msc{A}}\ \ \text{and}\ \ \Bord^{\opn{adm}}_{\msc{S}\times \msc{A}}\ \subseteq\ \Bord_{\msc{S}\times\msc{A}}
\]
are the non-full symmetric monoidal subcategories which contain all objects and all admissible bordisms.
\end{definition}

We can define categories of reduced admissible bordisms $k\Bord^{\opn{adm};\opn{red}}_{\msc{C}}$, for $\msc{C}=\msc{S}\otimes \msc{A}$ and $\msc{C}=\msc{S}\times \msc{A}$, by introducing relations for internal composition and (bi)linear expansions at vertices, just as in Section \ref{sect:reduced}.

\begin{proposition}
There is a symmetric monoidal functor
\[
\Delta^{\opn{adm}}:\Bord_{\msc{S}\otimes \msc{A}}^{\opn{adm}}\to (k\Bord_{\msc{S}}^{\red}\boxtimes k\Bord^{\opn{adm};\red}_{\msc{A}})^{\opn{add}}
\]
which fits into a diagram
\[
\xymatrix{
\Bord_{\msc{S}\otimes \msc{A}}^{\opn{adm}}\ar[d]\ar[rr]^(.35){\Delta^{\opn{adm}}}& &  (k\Bord_{\msc{S}}^{\red}\boxtimes k\Bord^{\opn{adm};\red}_{\msc{A}})^{\opn{add}}\ar[d]\\
\Bord_{\msc{S}\otimes \msc{A}}\ar[rr]^(.35){\Delta}& &  (k\Bord_{\msc{S}}^{\red}\boxtimes k\Bord^{\red}_{\msc{A}})^{\opn{add}}.
}
\]
\end{proposition}

\begin{proof}
One repeats the proof of Proposition \ref{prop:conc} word for word, replacing $\Bord$ with $\Bord^{\opn{adm}}$, to see that the functor
\[
\opn{ccat}^{\opn{adm};\red}:k\Bord_{\msc{S}\times \msc{A}}^{\opn{adm};\red}\to k\Bord_{\msc{S}\otimes \msc{A}}^{\opn{adm};\red}
\]
is fully faithful.  The analog of Corollary \ref{cor:conc} follows, as the proof of Lemma \ref{lem:sigma_sum} is still valid in the admissible setting.  Hence we can construct $\Delta^{\opn{adm}}$ exactly as in the proof of Lemma \ref{lem:sep}.
\end{proof}

We can again define reasonable theories $Z:\Bord^{\opn{adm}}_{\msc{A}}\to \opn{Vect}$ in the admissible context as those theories which are stable under skein relations and respect the linear relation (U4). Given such a theory, we have the base change to $\Bord^{\opn{adm};\red}_{\msc{S}\otimes\msc{A}}$ at arbitrary semisimple symmetric $\msc{S}$.

\begin{theorem}\label{thm:adm_Z_AS}
Let $Z:\Bord^{\opn{adm}}_{\msc{A}}\to \opn{Vect}$ be a reasonable theory from the admissible bordism category, and $\msc{S}$ be semisimple and symmetric with trivial twist.  There is a symmetric monoidal functor
\[
Z(\msc{S}):\Bord^{\opn{adm}}_{\msc{S}\ot\msc{A}}\to \msc{S}
\]
which recovers $Z$ when restricted to the subcategory $\Bord^{\opn{adm}}_{\msc{A}}$.
\end{theorem}

\begin{proof}
As in the general setting, we can define $Z(\msc{S})$ as the composite of $\Delta^{\opn{adm}}$ with the functor
\[
\opn{act}(Ev_{\msc{S}}\boxtimes Z):(k\Bord_{\msc{S}}^{\red}\boxtimes k\Bord^{\opn{adm};\red}_{\msc{A}})^{\opn{add}}\to (\msc{S}\boxtimes \opn{Vect})^{\opn{add}}\to \msc{S}.
\]
\end{proof}

\subsection{The universal Lyubashenko theory}

In \cite{derenzietal23} De-Renzi, Gainutdinov, Geer, Patureau-Mirand, and Runkel construct the primary target of our study.

\begin{theorem}[{\cite[Theorem 4.12]{derenzietal23}}]\label{thm:dggpr}
Given a finite modular tensor category $\msc{A}$ there exists a symmetric monoidal functor
\[
Z_{\msc{A}}:\Bord^{\opn{adm}}_{\msc{A}}\to \opn{Vect}.
\]
In the case where $\msc{A}$ is semisimple, so that $\Bord^{\opn{adm}}_{\msc{A}}=\Bord_{\msc{A}}$, the TQFT $Z_{\msc{A}}$ recovers the usual Reshetikhin-Turaev theory.
\end{theorem}

\begin{proof}
In \cite{derenzietal23} the authors define a theory $V_{\msc{A}}:\Bord^{\opn{adm}'}_{\msc{A}}\to \opn{Vect}$, where $\Bord_{\msc{A}}^{\opn{adm}'}$ is the subcategory of marked surfaces and ribbon bordisms $(M,T)$ in which each component of $M$ either has nonvanishing \emph{incoming} boundary or a projectively marked edge.  We have the symmetric anti-equivalence
\[
\overline{\hspace{3mm}\!}:\Bord_{\msc{A}}^{\opn{adm}}\to (\opn{Bord}_{\msc{A}}^{\opn{adm}'})^{\opn{op}}
\]
given by reversing orientations and the anti-equivalence $\opn{Vect}\to \opn{Vect}^{\opn{op}}$ provided by linear duality.  We then obtain our theory $Z_{\msc{A}}$ as the composite
\[
\Bord_{\msc{A}}^{\opn{adm}}\overset{-}\to (\opn{Bord}^{\opn{adm}'}_{\msc{A}})^{\opn{op}}\overset{V_{\msc{A}}}\to \opn{Vect}^{\opn{op}}\overset{\ast}\to \opn{Vect}.
\]
\end{proof}

\begin{remark}
We apply a duality in our definition of $Z_{\msc{A}}$, relative to the original construction from \cite{derenzietal23}, in order to facilitate a calculation of the state spaces via morphisms in $\msc{A}$.  See Section \ref{sect:state_spaces} below.
\end{remark}

The theory $Z_{\msc{A}}$ is constructed via a modification of Lyubashenko's invariant for closed $3$-manifolds \cite[Theorem 3.8]{derenzietal23}, and a universal construction which employs this $3$-manifold invariant as a type of form on a free space of ribbon bordisms \cite[Section 4.3]{derenzietal23}.

\begin{definition}
The universal Lyubashenko theory (or DGGPR theory, or admissible Reshetikhin-Turaev theory) associated to a finite modular tensor category $\msc{A}$ is the theory $Z_{\msc{A}}$ of Theorem \ref{thm:dggpr}.
\end{definition}

By \cite[Proposition 4.11]{derenzietal23} the universal Lyubashenko theory respects all skein relations, and it is clear from construction that it respects the linear relation (U4) from Section \ref{sect:reduced} as well.

\begin{lemma}[\cite{derenzietal23}]
Given a finite modular tensor category $\msc{A}$, the TQFT $Z_{\msc{A}}:\Bord_{\msc{A}}^{\opn{adm}}\to \opn{Vect}$ is constant along all skein relations and all linear relations of the form {\rm (U4)}. This is to say, the theory $Z_{\msc{A}}$ is reasonable.
\end{lemma}

As an immediate consequence, we apply Theorem \ref{thm:adm_Z_AS} to obtain a natural extension for the theory $Z_{\msc{A}}$ along any semisimple symmetric tensor category. 

\begin{proposition}
For any finite modular tensor category $\msc{A}$, and any semisimple symmetric tensor category $\msc{S}$, the universal Lyubaschenko theory $Z_{\msc{A}}$ admits an extension to an $\msc{S}$-valued TQFT
\[
Z_{\msc{A}}(\msc{S}):\Bord_{\msc{S}\ot\msc{A}}^{\opn{adm}}\to \msc{S}. 
\]
\end{proposition}

\section{State spaces in the base changed Lyubaschenko theory}
\label{sect:state_spaces}

We show that the state spaces in the base changed Lyubaschenko theory $Z_{\msc{A}}(\msc{S})$ are identified with inner-Homs for the action of $\msc{S}$ on the labeling category $\msc{S}\ot \msc{A}$.  We begin by recalling the analogous result from \cite{derenzietal23}, which identifies the state spaces in $Z_{\msc{A}}$ with linear Homs for $\msc{A}$.  Before that, however, we set some of the technical foundations for our discussion.

\subsection{Non-linear reductions}

For a not-necessarily-linear ribbon category $\msc{C}$ we consider the \emph{non}-linear reduction of the bordism category
\[
\Bord^{\red}_{\msc{C}}:=\Bord_{\msc{C}}/\langle \text{U1, U2, U3}\rangle,
\]
where we recall that (U1) accounts for dualization at vertices, (U2) accounts for internal composition, and (U3) accounts for the local actions of the framed braid group at vertices.  As the construction $\Bord^{\red}_{\square}$ accepts arbitrary ribbon tensor functors, any ribbon tensor functor $F:\msc{C}\to \msc{D}$ induces a symmetric tensor functor on reduced bordisms
\[
F_{\ast}:\Bord^{\red}_{\msc{C}}\to \Bord^{\red}_{\msc{D}}.
\]

\subsection{Categories at the boundary}
\label{sect:a_sigma}

Consider the forgetful functor $\Bord_\msc{C}^{\opn{red}}\to\Bord_\ast^{\opn{red}}$ on non-linear reduced bordisms.  Given a $\ast$-marked surface $\mbf{\Sigma}$, with framed markings $I\to \Sigma$, we have the fiber product
\[
\xymatrix{
\msc{C}_{\mbf{\Sigma}}\ar[r]\ar[d] & \Bord_{\msc{C}}^{\opn{red}}\ar[d]\\
\ast\ar[r]_(.35){\mbf{\Sigma}} & \Bord^{\opn{red}}_{\ast}.
}
\]
At each index $i\in I$ take $\opn{sgn}(i)=+$ if the framing at $i$ recovers the ambient orientation on $\Sigma$ and $\opn{sgn}(i)=-$ otherwise.
\par

For $\msc{C}^+=\msc{C}$ and $\msc{C}^-=\msc{C}^{\opn{op}}$, the above fiber product is isomorphic to the cartesian product
\[
\prod_{i\in I}\msc{C}^{\opn{sgn}(i)}\overset{\sim}\to \msc{C}_{\mbf{\Sigma}},
\]
where the above map sends a tuple of objects $\vec{x}$ in $\prod_{i\in I}\msc{C}^{\opn{sgn}(i)}$ to the associated marked surface $\Sigma_{\vec{x}}$ over $\mbf{\Sigma}$.  Each tuple of maps $\vec{f}=(f_i\in \Hom_{\msc{C}^{\pm}}(x_i,y_i):i\in I)$ in $\prod_i\msc{C}^{\opn{sgn}(i)}$ is sent to the map $\Sigma_{\vec{f}}:\Sigma_{\vec{x}}\to \Sigma_{\vec{y}}$ whose underlying $3$-manifold is the product $\Sigma\times [0,1]$, underlying ribbon graph is the product $I\times [0,1]\to \Sigma\times [0,1]$ of the given markings $I\to \Sigma$ with $id_{[0,1]}$, and whose labeling morphisms are given by the $f_i$.  To illustrate,
	\begin{equation*}
		\scalebox{0.5}{
\tikzset{every picture/.style={line width=0.75pt}} 
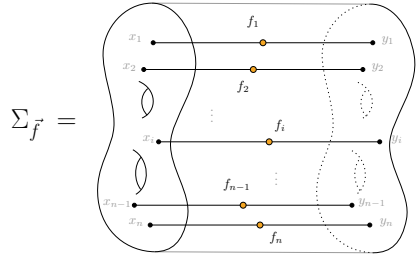
\begin{tikzpicture}[x=0.75pt,y=0.75pt,yscale=-1,xscale=1]

\draw    (213.5,351) .. controls (186.63,350.7) and (160.37,332.05) .. (156.43,298.02) .. controls (152.5,264) and (172.75,246.43) .. (172.5,221) .. controls (172.25,195.57) and (160.51,171.57) .. (161.5,149) .. controls (162.49,126.43) and (181.84,100.32) .. (210.5,100) ;
\draw    (213.5,351) .. controls (230.72,350.75) and (241.86,341.5) .. (246.5,317) .. controls (251.14,292.5) and (228.43,266.27) .. (229.5,241) .. controls (230.57,215.73) and (259.5,185) .. (254.33,152.34) .. controls (249.16,119.67) and (230.83,100) .. (210.5,100) ;
\draw [color={rgb, 255:red, 0; green, 0; blue, 0 }  ,draw opacity=0.4 ]   (210.5,100) -- (431.5,101) ;
\draw [color={rgb, 255:red, 0; green, 0; blue, 0 }  ,draw opacity=0.4 ]   (213.5,351) -- (434.5,352) ;
\draw    (197.5,179) .. controls (215.75,186.6) and (216.83,208.9) .. (200.72,216) ;
\draw    (206.09,213.12) .. controls (198.57,206.03) and (198.57,192.85) .. (206.63,186.26) ;
\draw    (212,140) -- (433.5,140) ;
\draw [shift={(322.75,140)}, rotate = 0] [color={rgb, 255:red, 0; green, 0; blue, 0 }  ][fill={rgb, 255:red, 245; green, 166; blue, 35 }  ,fill opacity=1  ][line width=0.75]   (0, 0) circle [x radius= 3, y radius= 3]   ;
\filldraw (212,140) circle (1.5pt);
\filldraw (433.5,140) circle (1.5pt);
\draw    (209,324) -- (430.5,324) ;
\draw [shift={(319.75,324)}, rotate = 0] [color={rgb, 255:red, 0; green, 0; blue, 0 }  ][fill={rgb, 255:red, 245; green, 166; blue, 35 }  ,fill opacity=1  ][line width=0.75]   (0, 0) circle [x radius= 3, y radius= 3]   ;
\filldraw (209,324) circle (1.5pt);
\filldraw (430.5,324) circle (1.5pt);
\draw    (193,304) -- (412.5,304) ;
\draw [shift={(302.75,304)}, rotate = 0] [color={rgb, 255:red, 0; green, 0; blue, 0 }  ][fill={rgb, 255:red, 245; green, 166; blue, 35 }  ,fill opacity=1  ][line width=0.75]   (0, 0) circle [x radius= 3, y radius= 3]  ;
\filldraw (193,304) circle (1.5pt);
\filldraw (412.5,304) circle (1.5pt);
\draw    (202.5,167) -- (423.5,167) ;
\draw [shift={(313,167)}, rotate = 0] [color={rgb, 255:red, 0; green, 0; blue, 0 }  ][fill={rgb, 255:red, 245; green, 166; blue, 35 }  ,fill opacity=1  ][line width=0.75]   (0, 0) circle [x radius= 3, y radius= 3] ;
\filldraw (202.5,167) circle (1.5pt);
\filldraw (423.5,167) circle (1.5pt);
\draw    (191.5,248) .. controls (208.5,257.25) and (209.5,284.37) .. (194.5,293) ;
\draw    (199.5,288.07) .. controls (192.5,279.44) and (192.5,263.41) .. (200,255.4) ;
\draw  [dash pattern={on 0.84pt off 2.51pt}]  (434.5,352) .. controls (407.63,351.7) and (381.37,333.05) .. (377.43,299.02) .. controls (373.5,265) and (393.75,247.43) .. (393.5,222) .. controls (393.25,196.57) and (381.51,172.57) .. (382.5,150) .. controls (383.49,127.43) and (402.84,101.32) .. (431.5,101) ;
\draw    (434.5,352) .. controls (451.72,351.75) and (462.86,342.5) .. (467.5,318) .. controls (472.14,293.5) and (449.43,267.27) .. (450.5,242) .. controls (451.57,216.73) and (480.5,186) .. (475.33,153.34) .. controls (470.16,120.67) and (451.83,101) .. (431.5,101) ;
\draw  [dash pattern={on 0.84pt off 2.51pt}]  (418.5,180) .. controls (436.75,187.6) and (437.83,209.9) .. (421.72,217) ;
\draw  [dash pattern={on 0.84pt off 2.51pt}]  (427.09,214.12) .. controls (419.57,207.03) and (419.57,193.85) .. (427.63,187.26) ;
\draw  [dash pattern={on 0.84pt off 2.51pt}]  (412.5,249) .. controls (429.5,258.25) and (430.5,285.37) .. (415.5,294) ;
\draw  [dash pattern={on 0.84pt off 2.51pt}]  (420.5,289.07) .. controls (413.5,280.44) and (413.5,264.41) .. (421,256.4) ;
\draw    (217.5,240) -- (440.5,240) ;
\draw [shift={(329,240)}, rotate = 0] [color={rgb, 255:red, 0; green, 0; blue, 0 }  ][fill={rgb, 255:red, 245; green, 166; blue, 35 }  ,fill opacity=1  ][line width=0.75]   (0, 0) circle [x radius= 3, y radius= 3]   ;
\filldraw (217.5,240) circle (1.5pt);
\filldraw (440.5,240) circle (1.5pt);

\draw (67,206) node [anchor=north west][inner sep=0.75pt]  [font=\Huge] [align=left] {$\displaystyle \Sigma_{\vec{f}} \ =$};
\draw (186,130) node [anchor=north west][inner sep=0.75pt]  [color={rgb, 255:red, 0; green, 0; blue, 0 }  ,opacity=0.4 ] [align=left] {$\displaystyle x_{1}$};
\draw (186,316) node [anchor=north west][inner sep=0.75pt]  [color={rgb, 255:red, 0; green, 0; blue, 0 }  ,opacity=0.4 ] [align=left] {$\displaystyle x_{n}$};
\draw (439,316) node [anchor=north west][inner sep=0.75pt]  [color={rgb, 255:red, 0; green, 0; blue, 0 }  ,opacity=0.4 ] [align=left] {$\displaystyle y_{n}$};
\draw (441,132) node [anchor=north west][inner sep=0.75pt]  [color={rgb, 255:red, 0; green, 0; blue, 0 }  ,opacity=0.4 ] [align=left] {$\displaystyle y_{1}$};
\draw (268,198) node [anchor=north west][inner sep=0.75pt]  [color={rgb, 255:red, 0; green, 0; blue, 0 }  ,opacity=0.4 ] [align=left] {$\displaystyle \vdots $};
\draw (333,262) node [anchor=north west][inner sep=0.75pt]  [color={rgb, 255:red, 0; green, 0; blue, 0 }  ,opacity=0.4 ] [align=left] {$\displaystyle \vdots $};
\draw (180,159) node [anchor=north west][inner sep=0.75pt]  [color={rgb, 255:red, 0; green, 0; blue, 0 }  ,opacity=0.4 ] [align=left] {$\displaystyle x_{2}$};
\draw (451,232) node [anchor=north west][inner sep=0.75pt]  [color={rgb, 255:red, 0; green, 0; blue, 0 }  ,opacity=0.4 ] [align=left] {$\displaystyle y_{i}$};
\draw (417.5,296) node [anchor=north west][inner sep=0.75pt]  [color={rgb, 255:red, 0; green, 0; blue, 0 }  ,opacity=0.4 ] [align=left] {$\displaystyle y_{n-1}$};
\draw (160,297) node [anchor=north west][inner sep=0.75pt]  [color={rgb, 255:red, 0; green, 0; blue, 0 }  ,opacity=0.4 ] [align=left] {$\displaystyle x_{n-1}$};
\draw (431,159) node [anchor=north west][inner sep=0.75pt]  [color={rgb, 255:red, 0; green, 0; blue, 0 }  ,opacity=0.4 ] [align=left] {$\displaystyle y_{2}$};
\draw (200,232) node [anchor=north west][inner sep=0.75pt]  [color={rgb, 255:red, 0; green, 0; blue, 0 }  ,opacity=0.4 ] [align=left] {$\displaystyle x_{i}$};
\draw (305,111) node [anchor=north west][inner sep=0.75pt]   [align=left] {$\displaystyle f_{1}$};
\draw (295,177) node [anchor=north west][inner sep=0.75pt]   [align=left] {$\displaystyle f_{2}$};
\draw (333,216) node [anchor=north west][inner sep=0.75pt]   [align=left] {$\displaystyle f_{i}$};
\draw (281,276) node [anchor=north west][inner sep=0.75pt]   [align=left] {$\displaystyle f_{n-1}$};
\draw (327.75,329) node [anchor=north west][inner sep=0.75pt]   [align=left] {$\displaystyle f_{n}$};

\end{tikzpicture}}
\end{equation*}
with each internal vertex inheriting its framing from the ambient ribbon.

\subsection{Ends and coends}

Since $\msc{A}$ is braided the product map $\msc{A}\times\msc{A}\to \msc{A}$ has a natural tensor structure.  In fact, a choice of tensor structure on this map is equivalent to the choice of a braiding on $\msc{A}$.  Since this map is exact in each coordinate, and in particular right exact, it induces a right exact map $m:\msc{A}\otimes\msc{A}\to \msc{A}$.  Furthermore any right exact monoidal functor between rigid tensor categories is also left exact.  Hence $m$ admits both a left and right adjoint.  For $\opn{Ind}$ the right adjoint to $m$ we can furthermore calculate the left adjoint via conjugation with duality $\opn{Ind}({^\ast-})^\ast$.

\begin{definition}
The canonical coend for $\msc{A}$ is the object $C=C_{\msc{A}}:=m\opn{Ind}(\1)$.  The canonical end is the dual $E=E_{\msc{A}}:=m\opn{Ind}(\1)^\ast$.
\end{definition}

For example, if we consider representations of a quasitriangular Hopf algebra $\msc{A}=\opn{rep}(A)$, with $R$-matrix $R$, the object $C$ is the cocycle twist of the dual $C=A^\ast_R$ via the $R$-matrix \cite[Section 3.3]{lyubashenko95}.  Here we consider $A^\ast$ along with its adjoint action.  The end is then simply the algebra $A$ itself, considered along with its adjoint action.  Both $C$ and $E$ are seen directly to be Hopf algebras in $\msc{A}$ in this case, and this Hopf structure extends to the case of general $\msc{A}$. See for example \cite[Section 3.1]{shimizu19}.

\subsection{State spaces in the universal Lyubashenko theory}

Let $\mbf{\Sigma}$ be a genus $g$ marked surface in $\Bord_{\ast}^{\red}$ and consider the boundary category $\msc{A}_{\mbf{\Sigma}}$ as above.  We have the functor
\begin{equation}\label{eq:state}
\msc{A}_{\mbf{\Sigma}}=\prod_{i\in I}\msc{A}^{\opn{sgn}(i)} \underset{\sim}{\overset{\opn{duality}}\longrightarrow} \prod_{i\in I}\msc{A}\overset{\opn{mult}}\to \msc{A}\overset{\Hom_{\msc{A}}(\1,E^{\otimes g}\ot-)}\longrightarrow \opn{Vect}.
\end{equation}
Let us denote this composite by $\Hom_{\msc{A}}(\1,E^{\ot g}\ot-):\msc{A}_{\mbf{\Sigma}}\to \opn{Vect}$, by an abuse of notation.  We also consider the initial portion of this sequence
\begin{equation}\label{eq:1303}
\msc{A}_{\mbf{\Sigma}}=\prod_{i\in I}\msc{A}^{\opn{sgn}(i)} \underset{\sim}{\overset{\opn{duality}}\longrightarrow} \prod_{i\in I}\msc{A}\overset{\opn{mult}}\to \msc{A}
\end{equation}
in isolation, and refer to this map as the \emph{oriented product} functor for $\msc{A}_{\mbf{\Sigma}}$.

\begin{theorem}[\cite{derenzietal23}]\label{thm:states}
Given a genus $g$ marked surface $\mbf{\Sigma}$ in $\Bord_{\ast}^{\opn{red}}$, there is a natural isomorphism of functors $\Hom_{\msc{A}}(\1,E^{\ot g}\ot-)\cong Z_{\msc{A}}|_{\msc{A}_{\mbf{\Sigma}}}$, i.e. a (2-)diagram
\[
\xymatrix{
\msc{A}_{\mbf{\Sigma}}\ar[rr]\ar[drr]|{\hole}_(.35){\Hom_{\msc{A}}(\1,E^{\ot g}\ot-)} & & \Bord_{\msc{A}}^{\opn{adm};\opn{red}}\ar[d]^{Z_{\msc{A}}}\\
	& & \opn{Vect}.
}
\]
\end{theorem}

To say things directly, if we have a genus $g$ marked surface $\Sigma_{\vec{x}}$ with labels $\{x_i:i\in I\}$ then we have an isomorphism
\[
Z_{\msc{A}}(\Sigma_{\vec{x}})\cong \Hom_{\msc{A}}\left(\1,E^{\ot g}\ot x_I\right),
\]
where $x_I=\otimes_{i\in I}x_i^{\pm}$ is the image of the tuple $\{x_i:i\in I\}$ under the oriented product map \eqref{eq:1303}.  We note here that the natural isomorphism of Theorem \ref{thm:states}, though it is natural, is not \emph{canonical}.  It depends on the choice of certain geometric input which ``merges" the vertices in $\mbf{\Sigma}$ to a single vertex.
\par

We provide a proof of Theorem \ref{thm:states} at the conclusion of the subsection, after we clarify some related points.  We begin with the case of a single marking, which is more-or-less covered in the original text.

\begin{proposition}[{\cite[Proposition 4.17]{derenzietal23}}]\label{prop:1322}
Let $\Sigma_x$ be a genus $g$ surface with a single positively oriented marking object $x$.  Then there is a natural isomorphism 
\[
\psi_{\Sigma_x}:Z_{\msc{A}}(\Sigma_x)\overset{\cong}\to \Hom_{\msc{A}}(\1,E^{\ot g}\ot x).
\]
\end{proposition}

\begin{proof}
Note that $\Sigma_x$ is isomorphic to the negatively marked surface $\Sigma_{{^\ast x}}$ in $\Bord^{\opn{adm};\opn{red}}_{\msc{A}}$ via the map $\Sigma_{id_x}$.  We recall also that our state spaces are obtained from those of \cite{derenzietal23} by orientation reversal and dualizing.
\par

For $P_1\to \1$ the projective cover of the unit, \cite[Proposition 4.17]{derenzietal23} gives an isomorphism
\[
\left(\Hom_{\msc{A}}(P_1,E^{\ot g}\ot {^\ast x})/\opn{rad}\right)^\ast\cong Z_{\msc{A}}(\Sigma_x),
\]
where we mod out by the radical of the form
\[
(-,-)_{g,{^\ast x}}:\Hom_{\msc{A}}(C^{\ot g}\ot x^\ast,\1)\ot_k\Hom_{\msc{A}}(P_1,E^{\ot g}\ot x^\ast)\to k
\]
from \cite[Section 4.1]{derenzietal23}.  Naturality in $x$ is clear from the specific construction employed in \cite[Section 4.7]{derenzietal23}.
\par

Via \cite[Lemma 4.1]{derenzietal23} the form $(-,-)_{g,{^\ast x}}$ induces an isomorphism
\[
\Hom_{\msc{A}}(C^{\ot g}\ot {^\ast x},\1)\overset{\sim}\to \left(\Hom_{\msc{A}}(P_1,E^{\ot g}\ot {^\ast x})/\opn{rad}\right)^\ast
\]
which is natural in $x$.  From this natural isomorphism we obtain the claimed identification
\[
\Hom_{\msc{A}}(\1, E^{\ot g}\ot x)\cong \Hom_{\msc{A}}(C^{\ot g}\ot {^\ast x},\1)\cong Z_{\msc{A}}(\Sigma_x).
\]
\end{proof}

The following reduction to the case of a single marking is also implicit in \cite{derenzietal23}.  Compare for example \cite[Theorem 1.1]{derenzietal23} to \cite[Section 4.7]{derenzietal23}.

\begin{lemma}\label{lem:merge}
Let $\Sigma_{\vec{x}}$ be a nonempty, connected, marked surface with ordered labelling set $I$, and take $x_I$ the image of $\Sigma_{\vec{x}}$ under the oriented product map \eqref{eq:1303}.  Let $\Sigma_{x_I}$ be the same surface with a single positive marking of value $x_I$ at an arbitrary point $q:\ast\to \Sigma$.  There is a bordism $[M,T]:\Sigma_{\vec{x}}\to \Sigma_{x_I}$ which evaluates to an isomorphism
\[
Z_{\msc{A}}(M,T):Z_{\msc{A}}(\Sigma_{\vec{x}})\overset{\cong}\to Z_{\msc{A}}(\Sigma_{x_I}).
\]
\end{lemma}

\begin{proof}[Construction]
Consider the ambient manifold $\Sigma\times [0,1]$.  Choose an arbitrary point $q$ in $\Sigma$ and consider the points $p_i:\ast\to \Sigma$ at which the $x_i$ sit.  We can take an embedding from an open disk $\mbb{D}^o\to \Sigma$ whose image contains all of the given points.  By following along non-intersecting paths $\bar{\gamma}_i:[0,1-\epsilon]\to \Sigma$ in the image of $\mbb{D}^o$ we attach the points $(p_i,0)$ in $\Sigma\times [0,1-\epsilon]$ to the point $(q,1-\epsilon)$.  We extend each $\bar{\gamma}_i$ to a path $\gamma_i:[0,1]\to \Sigma$ by letting $\gamma_i$ be of constant value $q$ on $[1-\epsilon,1]$.  These $\gamma_i$ together provide a ribbon graph
\[
\gamma:\Gamma=(I\times [0,1])\amalg_{I\times [1-\epsilon,1]}[1-\epsilon,1]\to \Sigma\times [0,1],
\]
and this defines a bordism $\Sigma_{\opn{merge}}:\mbf{\Sigma}=\Sigma_{\{p_i\}}\to \Sigma_q$ in $\Bord_\ast^{\opn{red}}$ from the surface $\Sigma$ with markings at the $p_i$ to $\Sigma$ with a single marking at $q$.  (The choice of the framing can be taken arbitrarily.)  We label the $p_i$ by our objects $x_i$, label $q$ by $x_I$, and label the unique internal vertex in $\Gamma$ by the identity map to obtain a ribbon bordism $\Sigma_{\opn{merge}}(\vec{x}):\Sigma_{\vec{x}}\to \Sigma_{x_I}$,
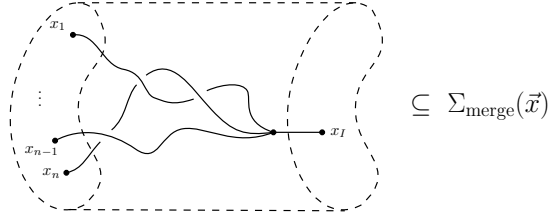
\[
\scalebox{.6}{
\tikzset{every picture/.style={line width=0.75pt}} 

\begin{tikzpicture}[x=0.75pt,y=0.75pt,yscale=-1,xscale=1]

\draw  [dash pattern={on 4.5pt off 4.5pt}]  (225.5,57) .. controls (250.5,102) and (213.37,98.82) .. (211.5,138) .. controls (209.63,177.18) and (251.5,189) .. (216.5,221) ;
\draw  [dash pattern={on 4.5pt off 4.5pt}]  (216.5,221) .. controls (181.52,239.95) and (151.39,191.4) .. (154.5,141) .. controls (157.61,90.6) and (196.19,32.32) .. (225.5,57) ;
\draw  [dash pattern={on 4.5pt off 4.5pt}]  (458.5,57) .. controls (483.5,102) and (446.37,98.82) .. (444.5,138) .. controls (442.63,177.18) and (484.5,189) .. (449.5,221) ;
\draw  [dash pattern={on 4.5pt off 4.5pt}]  (449.5,221) .. controls (414.52,239.95) and (384.39,191.4) .. (387.5,141) .. controls (390.61,90.6) and (429.19,32.32) .. (458.5,57) ;
\draw    (207,78) .. controls (227.84,75.79) and (231.01,103.47) .. (249.5,107) .. controls (267.99,110.53) and (266.25,127.77) .. (275.5,131) .. controls (284.75,134.23) and (290.5,138) .. (308.5,132) ;
\filldraw (207,78) circle (1.5pt);
\draw    (239.5,158) .. controls (252.5,147) and (251.5,130) .. (260.5,120) ;
\draw    (268.5,112) .. controls (290.62,95.41) and (305.45,121.63) .. (323.5,141) .. controls (341.55,160.37) and (351.5,163) .. (375.5,160) ;
\draw    (319.5,125) .. controls (335.09,118.35) and (349.63,118.39) .. (353.5,125) .. controls (357.37,131.61) and (347.5,151) .. (375.5,160) ;
\filldraw (375.5,160) circle (1.5pt);
\draw    (192,167) .. controls (214.5,155) and (239.5,160) .. (259.5,174) .. controls (279.5,188) and (280.38,159.89) .. (299.5,157) .. controls (318.62,154.11) and (342.5,173) .. (375.5,160) ;
\filldraw (192,167) circle (1.5pt);
\draw    (201.5,194) .. controls (215.5,191) and (218.5,180) .. (230,168) ;
\filldraw (201.5,194) circle (1.5pt);
\draw  [dash pattern={on 4.5pt off 4.5pt}]  (211,51) -- (445.5,51) ;
\draw  [dash pattern={on 4.5pt off 4.5pt}]  (204,225) -- (435.5,226) ;
\draw    (375.5,160) -- (416.5,160) ;
\filldraw (416.5,160) circle (1.5pt);

\draw (176,116) node [anchor=north west][inner sep=0.75pt]   [align=left] {$\displaystyle \vdots $};
\draw (186,65) node [anchor=north west][inner sep=0.75pt]   [align=left] {$\displaystyle x_{1}$};
\draw (165,170) node [anchor=north west][inner sep=0.75pt]   [align=left] {$\displaystyle x_{n-1}$};
\draw (180,189) node [anchor=north west][inner sep=0.75pt]   [align=left] {$\displaystyle x_{n}$};
\draw (422,155) node [anchor=north west][inner sep=0.75pt]   [align=left] {$\displaystyle x_I$};
\draw (489,124) node [anchor=north west][inner sep=0.75pt] [font=\huge]  [align=left] {$\displaystyle \subseteq \ \Sigma _{\opn{merge}}(\vec{x})$};

\end{tikzpicture}}.
\]

By following these same paths backwards $\gamma_i^{-1}:[0,1]\to \Sigma$, i.e.\ by applying the orientation reversing automorphisms $1-t:[0,1]\to [0,1]$ in all places, we produce another ribbon bordism $\Sigma_{\opn{merge}}'(\vec{x}):\Sigma_{x_I}\to \Sigma_{\vec{x}}$.  Since $Z_{\msc{A}}$ is constant along all skein relations, one sees that the two maps
\[
Z_{\msc{A}}(\Sigma_{\opn{merge}}(\vec{x}))\ \ \text{and}\ \ Z_{\msc{A}}(\Sigma'_{\opn{merge}}(\vec{x}))
\]
are mutually inverse isomorphisms.  (The point is that the composite
\[
\Sigma'_{\opn{merge}}(\vec{x})\circ\Sigma_{\opn{merge}}(\vec{x}),
\]
for example, is equivalent to the identity via a skein relation which separates the sole ribbon connecting the two internal vertices into $I$ parallel ribbons.)  As both of these bordisms are admissible, and $Z_{\msc{A}}$ is reasonable, $\Sigma_{\opn{merge}}(\vec{x})$ provides the desired isomorphism.
\end{proof}

We now provide the claimed identification of state spaces.

\begin{proof}[Proof of Theorem \ref{thm:states}]
For a fixed $\ast$-labeled bordism $\Sigma_{\opn{merge}}$ as in the proof of Lemma \ref{lem:merge}, and its labeled counterparts, we have diagrams
\[
\xymatrix{
Z_{\msc{A}}(\Sigma_{\vec{x}})\ar[d]_{Z_{\msc{A}}(\Sigma_{\opn{merge}}(\vec{x}))}\ar[rr]^{\Sigma_{\vec{f}}} & & Z_{\msc{A}}(\Sigma_{\vec{y}})\ar[d]^{Z_{\msc{A}}(\Sigma_{\opn{merge}}(\vec{y}))}\\
Z_{\msc{A}}(\Sigma_{x_I})\ar[rr]_{\Sigma_{f_I}} & & Z_{\msc{A}}(\Sigma_{y_I})
}
\]
at each tuple of maps $\vec{f}=(f_i:i\in I)$ in the fiber $\msc{A}_{\mbf{\Sigma}}$.  (Here $f_I$ is the image of $\Sigma_{\vec{f}}$ under the oriented product \eqref{eq:1303}.)  We compose with the isomorphism $\psi$ from Proposition \ref{prop:1322} to obtain isomorphisms for the state spaces which fit into diagrams
\[
\xymatrix{
Z_{\msc{A}}(\Sigma_{\vec{x}})\ar[d]_{\psi'}\ar[rr]^{\Sigma_{\vec{f}}} & & Z_{\msc{A}}(\Sigma_{\vec{y}})\ar[d]^{\psi'}\\
\Hom_{\msc{A}}(\1,E^{\ot g}\ot x_I)\ar[rr]_{f_\ast} & & \Hom_{\msc{A}}(\1,E^{\ot g}\ot y_I).
}
\]
The map $\psi'$ provides the claimed isomorphism.
\end{proof}

\subsection{State spaces in the base change}

Consider the base changed category $\msc{S}\ot\msc{A}$ at semisimple symmetric $\msc{S}$.  We have the inner Homs
\[
\underline{\Hom}:(\msc{S}\ot\msc{A})^{\opn{op}}\times (\msc{S}\ot\msc{A})\to \msc{S}
\]
for the tensor action of $\msc{S}$ on $\msc{S}\ot\msc{A}$, which appears on monomials as
\[
\underline{\Hom}(s\ot x,s'\ot x')=(s^\ast\ot s')\ot_k \Hom_{\msc{A}}(x,x'),
\]
via Lemma \ref{lem:708} for example. Given a genus $g$ marked surface $\mbf{\Sigma}$ in $\Bord_{\ast}^{\red}$ we define the functor
\[
\underline{\Hom}(\1,E^{\ot g}\ot-):(\msc{S}\ot \msc{A})_{\mbf{\Sigma}}\to \msc{S}
\]
by applying dualities, tensoring, then applying inner-Hom just as in \eqref{eq:state}.

\begin{theorem}\label{thm:ind_states}
For any genus $g$ surface $\mbf{\Sigma}$ in $\Bord_{\ast}^{\red}$, there is a (2-)diagram
\[
\xymatrix{
(\msc{S}\ot\msc{A})_{\mbf{\Sigma}}\ar[rr]\ar[drr]|{\hole}_(.35){\underline{\Hom}(\1,E^{\ot g}\ot-)} & & \Bord_{\msc{S}\ot\msc{A}}^{\opn{adm};\opn{red}}\ar[d]^{Z_{\msc{A}}(\msc{S})}\\
 & & \msc{S}.
}
\]
\end{theorem}

More directly, given a genus $g$ surface $\Sigma_{\vec{z}}$ in $\Bord_{\msc{S}\ot\msc{A}}$ with markings from an ordered set $I$, we again have a natural isomorphism $Z_{\msc{A}}(\msc{S})(\Sigma_{\vec{z}})\overset{\sim}\to \underline{\Hom}(\1,E^{\ot g}\ot z_I)$ where $z_I$ is the oriented product for $\Sigma_{\vec{z}}$.

\begin{proof}
Take $\underline{Z}=Z(\msc{S})$ and let $I$ be the labeling set for $\mbf{\Sigma}$.  Throughout the proof we take $\Bord=\Bord^{\opn{adm}}$, to ease notation.
\par

Let $\{s_{\lambda_0}:\lambda_0\in\Lambda\}$ be a complete list of distinct simples in $\msc{S}$.  At each $\lambda_0$ consider the map $F_{\lambda_0}:\msc{A}\to \msc{S}\ot\msc{A}$, $x\mapsto s_{\lambda_0}\ot x$, and its right adjoint $-_{\lambda_0}:\msc{S}\ot\msc{A}\to \msc{A}$.  We recall that these maps assemble into an equivalence
\[
F:\oplus_{\lambda_0\in \Lambda}\msc{A}\to \msc{S}\ot\msc{A}
\]
(Lemma \ref{lem:753}), so that there is a natural isomorphism $z\overset{\sim}\to \oplus_{\lambda_0} s_{\lambda_0}\ot z_{\lambda_0}$ at each $z$ in $\msc{S}\ot \msc{A}$.
\par

For a tuple $\lambda:I\to \Lambda$ and $\Sigma_{\vec{z}}$ in $\msc{S}\ot\msc{A}_{\mbf{\Sigma}}$ we define the marked surface $\Sigma_{\vec{z};\lambda}$ in $\msc{S}\ot\msc{A}_{\mbf{\Sigma}}$ with underlying geometry $\mbf{\Sigma}$ and $i$-th labeling object $s_{\lambda(i)}\ot z_{i\lambda(i)}$, whenever all of the objects $z_{i\lambda(i)}$ are nonzero.  We have the natural isomorphism
\begin{equation}\label{eq:1430}
\eta:\Sigma_{\vec{z}}\overset{\sim}\to \oplus_{\lambda}\Sigma_{\vec{z};\lambda}
\end{equation}
in $(k\Bord_{\msc{S}\ot\msc{A}}^{\opn{red}})^{\opn{add}}$ which is induced by the various projections $\Sigma_{p;\lambda}:\Sigma_{\vec{z}}\to \Sigma_{\vec{z};\lambda}$ (see Lemma \ref{lem:sigma_sum}), where the sum here is over all functions $\lambda$ at which the labeling objects $s_{\lambda(i)}\ot z_{i\lambda(i)}$ are nonzero. (We note that a surface $\Sigma_{\vec{w}}$ in $k\Bord_{\msc{S}\ot\msc{A}}^{\red}$ is a zero object whenever one of the marking objects $w_i$ is zero.)  These maps $\eta$ provide, in particular, a natural isomorphism between the two functors
\[
\xymatrix{
\msc{S}\ot \msc{A}_{\mbf{\Sigma}} \ar@<.5ex>[rr]^(.4){\oplus_{\lambda}\Sigma_{-;\lambda}} \ar@<-.5ex>[rr] & & (k\Bord_{\msc{S}\ot\msc{A}}^{\opn{red}})^{\opn{add}},
}
\]
where the bottom map is the expected composite $\msc{S}\ot\msc{A}_{\mbf{\Sigma}}\to \Bord_{\msc{S}\ot\msc{A}}^{\opn{red}}\to (k\Bord_{\msc{S}\ot\msc{A}}^{\opn{red}})^{\opn{add}}$.

As each $\Sigma_{\vec{z};\lambda}$ is in the image of the concatenation functor, we have the apparent lift of the functor $\oplus_{\lambda}\Sigma_{-;\lambda}$ to the category of $\msc{S}\times \msc{A}$-labeled bordisms and a subsequent (2-)diagram
\[
\xymatrix{
\msc{S}\ot\msc{A}_{\mbf{\Sigma}}\ar[rr]^(.4){\oplus_{\lambda}\Sigma_{-;\lambda}}\ar[drr] & & (k\Bord_{\msc{S}\times \msc{A}}^{\opn{red}})^{\opn{add}}\ar[d]^{(\opn{ccat}^{\opn{red}})^{\opn{add}}}\\
	& & (k\Bord_{\msc{S}\ot\msc{A}}^{\opn{red}})^{\opn{add}}.
}
\]
This diagram then provides a natural isomorphism between the restriction $\underline{Z}|_{\msc{S}\ot\msc{A}_{\mbf{\Sigma}}}$ and the composite
\begin{equation}\label{eq:1446}
\xymatrix{
\msc{S}\ot\msc{A}_{\mbf{\Sigma}}\ar[r]^(.4){\oplus_{\lambda}\Sigma_{-;\lambda}}& (k\Bord_{\msc{S}\times \msc{A}}^{\opn{red}})^{\opn{add}}\ar[r] & (k\Bord_{\msc{S}}^{\opn{red}}\boxtimes k\Bord_{\msc{A}}^{\opn{red}})^{\opn{add}}\ar[r]^(.75){Ev\cdot Z} & \msc{S}.
}
\end{equation}
We can furthermore restrict codomains to the fibers over $\mbf{\Sigma}$ to equate this composite with a composite
\[
\xymatrix{
\msc{S}\ot \msc{A}_{\mbf{\Sigma}} \ar[r]^(.4){\oplus_{\lambda}\Sigma_{-;\lambda}}& (k(\msc{S}\times \msc{A})_{\mbf{\Sigma}})^{\opn{add}}\ar[r] & (k\msc{S}_{\mbf{\Sigma}}\boxtimes k\msc{A}_{\mbf{\Sigma}})^{\opn{add}}\ar[r]^(.7){Ev\cdot Z} & \msc{S}.
}
\]
Finally by Theorem \ref{thm:states} the above composite is isomorphic to the map
\begin{equation}\label{eq:1464}
\xymatrix{
\msc{S}\ot \msc{A}_{\mbf{\Sigma}}\ar[r]^(.4){\oplus_{\lambda}\Sigma_{-;\lambda}}& (k(\msc{S}\times \msc{A})_{\mbf{\Sigma}})^{\opn{add}}\ar[r] & (k\msc{S}_{\mbf{\Sigma}}\boxtimes k\msc{A}_{\mbf{\Sigma}})^{\opn{add}}\ar[rr]^(.65){Ev\cdot \Hom_{\msc{A}}} & & \msc{S},
}
\end{equation}
where the final functor is explicitly the composite
\begin{equation}\label{eq:1470}
\opn{act}\circ (Ev\boxtimes \Hom_{\msc{A}}(\1,E^{\ot g}\ot-)):(k\msc{S}_{\mbf{\Sigma}}\boxtimes k\msc{A}_{\mbf{\Sigma}})^{\opn{add}}\to (\msc{S}\boxtimes \opn{Vect})^{\opn{add}}\to \msc{S}.
\end{equation}

Now, the first two maps in \eqref{eq:1464} compose to a functor
\[
\msc{A}_{\mbf{\Sigma}}\to (k\msc{S}_{\mbf{\Sigma}}\boxtimes k\msc{A}_{\mbf{\Sigma}})^{\opn{add}}
\]
which takes a surface $\Sigma_{\vec{z}}$ to the sum $\oplus_{\lambda}(\Sigma_{\vec{z};\lambda}^{\msc{S}},\ \Sigma_{\vec{z};\lambda}^{\msc{A}})$, where $\Sigma_{\vec{z};\lambda}^{\msc{S}}$ and $\Sigma_{\vec{z};\lambda}^{\msc{A}}$ have underlying geometry $\mbf{\Sigma}$ and $i$-th labels $s_{\lambda(i)}$ and $z_{i\lambda(i)}$, respectively, and which takes each map $\Sigma_{\vec{f}}:\Sigma_{\vec{z}}\to \Sigma_{\vec{w}}$ to the sum $\oplus_{\lambda} (id,\  \Sigma_{\vec{f};\lambda}^{\msc{A}})$, where the $i$-th  ribbon in $\Sigma_{\vec{f};\lambda}^{\msc{A}}$ is labeled by the map $f_{i\lambda(i)}:z_{i\lambda(i)}\to w_{i\lambda(i)}$.  The map \eqref{eq:1470} subsequently sends each pair $(\Sigma^{\msc{S}}_{\vec{z};\lambda},\ \Sigma^{\msc{A}}_{\vec{z};\lambda})$ to $s_{\lambda}\ot \Hom_{\msc{A}}(\1,E^{\ot g}\ot z_{I;\lambda})$ and each map $(id,\  \Sigma_{\vec{f};\lambda})$ to $id_{s_{\lambda}}\ot \Hom_{\msc{A}}(\1, E^{\ot g}\ot f_{I;\lambda})$, where
\begin{equation}\label{eq:1482}
s_{\lambda}=\otimes_{i\in I} s_{\lambda(i)}^{\pm},\ \ z_{I;\lambda}=\otimes_{i\in I}z_{i\lambda(i)}^{\pm},\ \ \text{and}\ \ f_{I;\lambda}=\otimes_{i\in I} f_{i\lambda(i)}^{\pm}.
\end{equation}
(As before, the superscript $\pm$ indicates an application of the dual at negative markings.)
\par

In total we find that the sequence \eqref{eq:1464} is isomorphic to the functor given by
\[
\Sigma_{\vec{z}}\mapsto \oplus_{\lambda:I\to \Lambda}s_{\lambda}\ot\Hom(\1,E^{\ot g}\ot z_{I;\lambda})\ \ \text{and}\ \ \Sigma_{\vec{f}}\mapsto \oplus_{\lambda:I\to \Lambda}s_{\lambda}\ot \Hom(\1,E^{\ot g}\ot f_{I;\lambda}).
\]
But now the oriented product $z_I$ of $\vec{z}$ itself decomposes naturally as the sum $z_I\cong\oplus_{\lambda}s_{\lambda}\ot z_{I;\lambda}$, and under this decomposition $f_I=\oplus_{\lambda} id_{s_{\lambda}}\ot f_{I;\lambda}$.  Hence the above expression for \eqref{eq:1464} can be rewritten as
\[
\Sigma_{\vec{z}}\mapsto \underline{\Hom}(\1,E^{\ot g}\ot z_I)\ \ \text{and}\ \ \Sigma_{\vec{f}}\mapsto \underline{\Hom}(\1,E^{\ot g}\ot f_I).
\]
So we observe the claimed isomorphism
\[
\underline{Z}|_{\msc{A}_{\mbf{\Sigma}}}\cong \eqref{eq:1446}\cong \eqref{eq:1464}\cong \underline{\Hom}(\1,E^{\ot g}\ot-).
\]
\end{proof}

\subsection{Example: Graded objects in $\msc{A}$}

In the case $\msc{S}=\opn{Vect}^{\mbb{Z}}$, the category of $\mbb{Z}$-graded vector spaces, the product $\opn{Vect}^{\mbb{Z}}\ot\msc{A}$ is the tensor category $\msc{A}^{\mbb{Z}}$ of finite length graded objects in $\msc{A}$.  This is, equivalently, the full subcategory of functors $\opn{Fun}(\mbb{Z},\msc{A})$ which take zero value at all but finitely many integers, along with all natural transformations.  The structure map
\[
\opn{Vect}^{\mbb{Z}}\times \msc{A}\to \msc{A}^{\mbb{Z}}
\]
takes a pair consisting of a graded vector space $V=\oplus_{t\in \mbb{Z}}V^t$ and an object $x^0$ in $\msc{A}$ to the graded object $V\ot_k x^0$ which is $V^t\ot_k x^0$ in each degree $t$.
\par

In any case, Theorem \ref{thm:adm_Z_AS} provides a graded extension $\underline{Z}_{\msc{A}}=Z_{\msc{A}}(\opn{Vect}^{\mbb{Z}}):\Bord^{\opn{adm}}_{\msc{A}^{\mbb{Z}}}\to \opn{Vect}^{\mbb{Z}}$ of the universal Lyubaschenko theory whose state spaces are naturally identified with the graded Homs for objects in $\msc{A}^{\mbb{Z}}$,
\[
\underline{Z}_{\msc{A}}(\Sigma_{\vec{z}})\overset{\sim}\to\underline{\Hom}_{\msc{A}}(\1,E^{\ot g}\ot z_I),
\]
where $E$ is taken in degree $0$ and $z_I$ is the oriented product for $\Sigma_{\vec{z}}$.

\section{Outlining a cochain valued theory}
\label{sect:outline}

We explain how, in principle, one produces the cochain valued theory
\[
Z_{\opn{Ch}(\msc{A})}:\Bord_{\opn{Ch}(\msc{A})}^{\opn{adm}}\to \opn{Ch}(\opn{Vect})
\]
from the base changed theory $Z_{\msc{A}}(\opn{Vect}^{\mbb{Z}})$ introduced in Section \ref{sect:t_ext}. The specific details for the construction of $Z_{\opn{Ch}(\msc{A})}$ then appear in Sections \ref{sec: diff def}--\ref{sect:states_parts}.

\subsection{Recollections on cochains}

Let $\msc{A}$ be a locally finite abelian category and $\msc{A}^{\mbb{Z}}$ be the category of locally finite graded objects in $\msc{A}$, considered as a module category over $\opn{Vect}^{\mbb{Z}}$.  We write each graded object as $x=\oplus_{n\in \mbb{Z}} x^n$, where $x^n$ is the homogeneous degree $n$ portion of $x$.
\par

For any integer $r$ we take $k(r)$ the graded vector space which is just $k$ concentrated in degree $r$ and define the shift autofunctor as
\[
[r]:=k(-r)\ot_k-:\msc{A}^{\mbb{Z}}\to \msc{A}^{\mbb{Z}}.
\]
For unknown reasons we write $x[r]:=[r](x)$, and we note that the $m$-th degree component in the shift $(x[r])^m$ is the object $k(-r)\ot x^{m+r}\cong x^{m+r}$.

\begin{definition}
A (finite length) cochain complex over $\msc{A}$ is a graded object $x$ in $\msc{A}^{\mbb{Z}}$ which is equipped with a map $d_x:x\to x[1]$ in $\msc{A}^{\mbb{Z}}$ which satisfies $(d_x[1])d_x=0$.  A morphism between cochains is a graded map $f:x\to y$ which satisfies $d_yf=f[1]d_x$. We let $\opn{Ch}(\msc{A})$ denote the category of (finite length) cochains over $\msc{A}$.
\end{definition}

We can express a cochain $x$ in the usual pictorial manner
\[
\cdots\to x^{n-1}\overset{d_x^{n-1}}\to x^n\overset{d_x^n}\to x^{n+1}\to \cdots,
\]
at which point a morphism $f:x\to y$ in $\opn{Ch}(\msc{A})$ appears as a commuting diagram
\[
\xymatrix{
\cdots\ar[r] & x^{n-1}\ar[r]^{d_x^{n-1}}\ar[d]_{f^{n-1}}& x^n\ar[r]^{d_x^n}\ar[d]_{f^n}&  x^{n+1}\ar[r]\ar[d]_{f^{n+1}} & \cdots\\
\cdots\ar[r] & y^{n-1}\ar[r]_{d_y^{n-1}}& y^n\ar[r]_{d_y^n}&  y^{n+1}\ar[r] & \cdots.
}
\]

\subsection{Tensor structures on cochains}

When $\msc{A}$ is a tensor category the category $\opn{Ch}(\msc{A})$ inherits a tensor structure so that, for cochains $x$ and $y$, $x\ot y$ is the cochain which has degree $n$ component $\oplus_{n_1+n_2=n}x^{n_1}\ot y^{n_2}$ and differential
\[
d_{x\ot y}=d_x\ot 1_y+(\tau_{x k(-1)}\ot 1_y)(1_x\ot d_y):x\ot y\to (x\ot y)[1].
\]
Here we employ the signed symmetry
\[
\tau_{x V}:x\ot_k V\to V\ot_k x,\ \ \tau_{x V}|_{x^m\ot_k V^n}=(-1)^{m\cdot n}\tau^{\heartsuit}_{x^mV^n},
\]
where $\tau^{\heartsuit}$ is the unsigned symmetry for the $\opn{Vect}$-action on $\msc{A}$.  Note that this application of the symmetry $\tau$ is necessary since the maps $d_x\ot 1$ and $1\ot d_y$ take values in the distinct spaces $k(-1)\ot x\ot y$ and $x\ot k(-1)\ot y$, respectively.  So, up to issues of notation, we have
\[
d_{x\ot y}|_{x^n\ot y^m}=``d_x\ot 1+ (-1)^n1\ot d_y"|_{x^n\ot y^m}.
\]
\par

As in the graded setting, we have an action of $\opn{Ch}(\opn{Vect})$ on $\opn{Ch}(\msc{A})$ and define $x[r]$ to be the complex $k(-r)\ot_k x$. This operation extends to an autequivalence
\[
[r]:\opn{Ch}(\msc{A})\to \opn{Ch}(\msc{A}),
\]
where we note that applying $[r]$ alters the differential on $x$ by a negative sign, ``$d_{x[r]}=(-1)^rd_x$".
\par

When $\msc{A}$ is braided, the braiding $c^{\heartsuit}$ on $\msc{A}$ induces a braided structure $c$ on $\opn{Ch}(\msc{A})$ given by the signed opeation $c|_{x^m\ot y^n}=(-1)^{m\cdot n}c^{\heartsuit}|_{x^m\ot y^n}$, and when $\msc{A}$ is ribbon $\opn{Ch}(\msc{A})$ is ribbon with ribbon element $\theta_x=\sum_m\theta_{x^m}^{\heartsuit}$.

\subsection{An outline}\label{sect:outliner}  We omit in this section all admissible labels, as the construction for the admissible and non-admissible cases are analogous.

Let $\msc{A}$ be a finite ribbon tensor category. 
Given a reasonable field theory  $Z:\Bord_{\msc A} \to \opn{Vect}$, via the identification  $\opn{Vect}^{\mathbb Z} \otimes \msc{A}= \msc{A}^{\mathbb Z}$ we can obtain a symmetric monoidal functor
\[\underline{Z}:=Z(\opn{Vect}^{\mathbb Z}) :\Bord_{\msc{A}^{\mathbb Z}}\to \opn{Vect}^{\mathbb Z}\] 
as in Definition \ref{def:Z_AS}. Explicitly, this is
given by the composition 
\[
\Bord_{\msc{A}^{\mathbb{Z}}} \cong \Bord_{\opn{Vect}^{\mathbb Z}\otimes \msc{A}}\xrightarrow{\Delta} (k\Bord_{\opn{Vect}^{\mathbb Z}}^{\red}\boxtimes k\Bord^{\red}_{\msc{A}})^{\opn{add}} \xrightarrow{Ev_{\opn{Vect}^{\mathbb Z}}\cdot Z} \opn{Vect}^{\mathbb Z}.
\]

Recall that the category of cochains $\opn{Ch}(\msc{A})$ is a distinguished subcategory 
\[
\opn{Ch}(\msc{A})\ \subseteq\ \msc{A}^{\mathbb{Z}}=\opn{Vect}^{\mathbb{Z}}\otimes \msc{A},
\]
so objects in $\opn{Ch}(\msc{A})$ are objects in $\msc{A}^{{\mathbb Z}}$ equipped with a differential, and we have a faithful functor $\Bord_{\Ch(\msc{A})}\hookrightarrow \Bord_{\msc{A}^{\mathbb Z}}$.  We can thus consider the restriction
	\begin{equation}\label{eq: Z_Ch}
		\underline{Z}|_{\Bord_{\Ch(\msc{A})}}:  
		\Bord_{\Ch(\msc{A})} \hookrightarrow \Bord_{\msc{A}^{\mathbb{Z}}} \xrightarrow{\underline{Z}}  \opn{Vect}^{\mathbb Z}.
	\end{equation}
We include the commutative diagram below to illustrate the compositions involved in the definition of $\underline{Z}|_{\Bord_{\Ch(\msc{A})}}$.
We will use the commutativity of this diagram frequently in the rest of this paper. 
\[\begin{tikzcd}
	{\Bord_{\opn{Vect}^{\mathbb Z}\otimes \msc{A}}} & {(k\Bord_{\opn{Vect}^{\mathbb Z}}^{\red}\boxtimes k\Bord^{\red}_{\msc{A}})^{\opn{add}}} \\
	{\Bord_{\msc{A}^{\mathbb{Z}}} } &&
	\\
	{\Bord_{\Ch(\msc{A})} } & {\opn{Vect}^{\mathbb Z}}
	\arrow["{\Delta}", from=1-1, to=1-2]
	\arrow["{{\underline{Z}}}"{pos=0.4}, from=1-1, to=3-2]
	\arrow["Ev_{\opn{Vect}^{\mathbb Z}}\cdot Z", from=1-2, to=3-2]
	\arrow["{ \cong}"{marking, allow upside down}, draw=none, from=2-1, to=1-1]
	\arrow[hook, from=3-1, to=2-1]
	\arrow[""', from=3-1, to=3-2]
\end{tikzcd}
\]

In general, we understand that the separation map $\Delta$ is only defined up to a natural isomorphism.  However, in our case we can define the functor
\[
\Delta:\Bord_{\msc{A}^{\mbb{Z}}}\to (k\Bord_{\opn{Vect}^{\mathbb Z}}^{\red}\boxtimes k\Bord^{\red}_{\msc{A}})^{\opn{add}}
\]
explicitly by taking an object $\Sigma_{\vec{x}}$ with marked points $p:I\to \Sigma$ to the formal sum
\[
\oplus_{m:I\to \mathbb{Z}}\Sigma_{k(m)}\ot\Sigma_{\vec{x}^{m}},
\]
where $\Sigma_{k(m)}$ and $\Sigma_{\vec{x}^{m}}$ denote the surface $\Sigma$ with respective marking objects $k(m_i)$ and $x_i^{m_i}$ at the points $p_i$.  Morphisms are defined by lifting along the concatenation functor
\[
\oplus_m\Sigma_{k(m)}\ot\Sigma_{\vec{x}^{m}}\mapsto \oplus_m\Sigma_{k(m)\ot \vec{x}^m}\cong \Sigma_{\vec{x}},
\]
where $\Sigma_{k(m)\ot \vec{x}^m}$ is the expected surface with $i$-th marking object $k(m_i)\ot x^{m_i}_i$.  Hence $\underline{Z}$ is given by the sequence
\[
\Sigma_{\vec{x}}\overset{\Delta}\mapsto 
\oplus_{m:I\to \mathbb{Z}}\Sigma_{k(m)}\ot\Sigma_{\vec{x}^{m}}\overset{Ev\cdot Z}\mapsto \oplus_{m:I\to \mathbb{Z}}\left(\otimes_{i\in I}k(\pm m_i)\right)\ot Z(\Sigma_{\vec{x}^m}).
\]

Our goal in what follows is to show that the symmetric monoical functor $\underline{Z}|_{\Bord_{\Ch(\msc{A})}}$ lifts naturally to the category of finite length cochains $\Ch(\opn{Vect})$. That is, we will prove the following theorem.

\begin{theorem}\label{thm: Zast}
	Let $\msc{A}$ be a ribbon tensor category and $\opn{Ch}(\msc{A})$ be the associated ribbon tensor category of finite length cochains.  Given a reasonable theory $Z:\Bord_{\msc{A}}\to \opn{Vect}$ there is a symmetric monoidal functor
	\begin{equation}\label{eq:150}
		Z^*:\Bord_{\opn{Ch}(\msc{A})}\to \opn{Ch}(\opn{Vect})
	\end{equation}
	which fits into a diagram
	\begin{equation}\label{eq:2457}
	\xymatrix{
	\Bord_{\opn{Ch}(\msc{A})}\ar[r]^{Z^\ast}\ar[d]_{forget} & \opn{Ch}(\opn{Vect})\ar[d]^{forget}\\
	\Bord_{\msc{A}^{\mathbb{Z}}}\ar[r]^{\underline{Z}} & \opn{Vect}^{\mathbb{Z}},
	}
	\end{equation}
and whose restriction to $\Bord_{\msc{A}}\subseteq \Bord_{\Ch(\msc{A})}$ recovers the composition of $Z$ with the unit $\opn{Vect} \to \opn{Vect}^{\mathbb Z}$.
\end{theorem}

We have the expected admissible variant of this result as well.

\begin{theorem}\label{thm: Z_Ch}
Given a ribbon tensor category $\msc{A}$ and reasonable theory $Z:\Bord_{\msc{A}}^{\opn{adm}}\to \opn{Vect}$ there is a symmetric monoidal functor
	\[
		Z^*:\Bord_{\opn{Ch}(\msc{A})}^{\opn{adm}}\to \opn{Ch}(\opn{Vect})
	\]
	which fits into a diagram as in \eqref{eq:2457}.  In the specific case where $\msc{A}$ is modular, and $Z$ is the universal Lyubashenko theory $Z_{\msc{A}}$, we furthermore have a natural identification of the state spaces
\begin{equation}\label{eq:zch_states}
Z^*_{\msc{A}}(\Sigma_{\vec{x}})\cong \Hom^\ast_{\msc{A}}(\1,E^{\ot g}\otimes x_I)
\end{equation}
at any marked genus $g$ surface $\Sigma_{\vec{x}}$.
\end{theorem}

The theory $Z_{\msc{A}}^{\ast}$ which we construct, in the case of universal Lyubashenko, is the theory which we've denoted $Z_{\opn{Ch}(\msc{A})}$ in the introduction.  

Since there is no meaningful difference between the proofs of Theorems \ref{thm: Zast} and (the first part of) \ref{thm: Z_Ch}, we formally present the arguments for Theorem \ref{thm: Zast} and leave it to the interested reader to check that all statements apply word-for-word to the admissible setting as well.  The theory $Z^*$ from Theorem \ref{thm: Zast} is constructed in Section \ref{sec: diff def} below, and the fact that $Z^*$ is actually a TQFT is verified in Sections \ref{sec: diff comm} and \ref{sect:states_parts}.  See in particular  Definition \ref{def:zast}, Proposition \ref{comm of differential}, and Proposition \ref{prop:zast_sym}.  The calculation of state spaces \eqref{eq:zch_states} is covered in Section \ref{sect:zch_states}.

\section{The differential}\label{sec: diff def}

We read diagrams in this section from \emph{top to bottom}, and we use coupons instead of vertices for better legibility of labels in  the string diagrams. 

\subsection{Relations for the category of strings}\label{sect:red_strings}
Consider the category $k\Str_{\opn{Vect}^{\mathbb Z}}$ of strings over $\opn{Vect}^{\mathbb Z}$. 
Throughout any unlabeled coupons in a string diagram such as
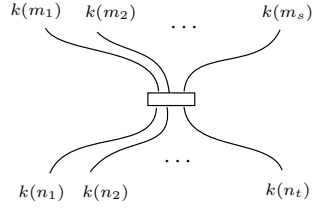
\[
\scalebox{0.9}{\tikzset{every picture/.style={line width=0.5pt}}   
\begin{tikzpicture}[x=0.75pt,y=0.75pt,yscale=-1,xscale=1]
\draw   (340.43,114.13) -- (366.71,114.13) -- (366.71,121.21) -- (340.43,121.21) -- cycle ;
\draw    (283.09,77.82) .. controls (291,93) and (345.21,87.56) .. (346.4,113.24) ;
\draw    (414.5,78.71) .. controls (404,103) and (363.13,87.56) .. (360.74,113.24) ;
\draw    (312,80) .. controls (325,93) and (349.99,87.56) .. (352.38,114.13) ;
\draw    (415.69,157.52) .. controls (409,139) and (363.37,147.64) .. (360.72,122.02) ;
\draw    (285.48,159.29) .. controls (296,134) and (344.28,148.2) .. (345.21,122.47) ;
\draw    (308.17,158.4) .. controls (324,137) and (355.08,148.57) .. (351.18,122.1) ;

\draw (262.11,62.08) node [anchor=north west][inner sep=0.75pt]  [font=\scriptsize] [align=left] {$\displaystyle k( m_{1})$};
\draw (303.93,62.97) node [anchor=north west][inner sep=0.75pt]  [font=\scriptsize] [align=left] {$\displaystyle k( m_{2})$};
\draw (402.98,62.97) node [anchor=north west][inner sep=0.75pt]  [font=\scriptsize] [align=left] {$\displaystyle k( m_{s})$};
\draw (266.6,163.92) node [anchor=north west][inner sep=0.75pt]  [font=\scriptsize] [align=left] {$\displaystyle k( n_{1})$};
\draw (302.44,163.92) node [anchor=north west][inner sep=0.75pt]  [font=\scriptsize] [align=left] {$\displaystyle k( n_{2})$};
\draw (405.08,163.03) node [anchor=north west][inner sep=0.75pt]  [font=\scriptsize] [align=left] {$\displaystyle k( n_{t})$};
\draw (351.61,73.25) node [anchor=north west][inner sep=0.75pt]   [align=left] {$\displaystyle \cdots $};
\draw (348.03,148.52) node [anchor=north west][inner sep=0.75pt]   [align=left] {$\displaystyle \cdots $};

\end{tikzpicture}},
\]
with $\sum_i \pm m_i =\sum_j \pm n_j$, is implicitly labeled by the identification $k(\pm m_1)\ot\dots \ot k(\pm m_s)\overset{\sim}\to k(\pm n_1)\ot\dots\ot k(\pm n_t)$ which sends $1\ot\dots\ot 1$ to $1\ot\dots\ot 1$.  Here the $\pm$ signs account for the directedness of the ribbons, since we can take $k(m)^\ast=k(-m)$.
\par

We introduce a reduction of the category of $\opn{Vect}^{\mathbb{Z}}$-labeled strings by the following skein relations:
\begin{itemize}
	\item[U1-U4] These relations are the same as the ones defined in Sections \ref{sect:sk_rel} and \ref{sect:reduced}.
	\item[V5] (Crossing signs) Uncrossing of two strings labelled by $k(m),$ 
	\begin{equation}\label{rel:V5}
		\begin{tikzpicture}[x=0.75pt,y=0.75pt,yscale=-1,xscale=1]

			\draw    (147.79,399.75) .. controls (147.79,450.75) and (101.73,420.75) .. (101.73,469.75) ;
			\draw    (102.51,399.75) .. controls (102.51,450.75) and (148.58,420.25) .. (148.58,469.25) ;
			\draw    (240.4,399.5) -- (239.8,470) ;
			\draw    (270.4,399.9) -- (269.8,470.4) ;
			
			\draw (91,389.18) node [anchor=north west][inner sep=0.75pt]  [font=\tiny]  {$\ k( m) \ \ \ \ \ \  k( m)$};
			\draw (91,472) node [anchor=north west][inner sep=0.75pt]  [font=\tiny]  {$\ k( m) \ \ \ \ \ \ k( m)$};
			\draw (156.6,423) node [anchor=north west][inner sep=0.75pt]    {$=\ ( -1){^{m^2}} \ \ \ \ \ \ \ \ \ \ \  \ \ \ . $};
			\draw (224.07,389.18) node [anchor=north west][inner sep=0.75pt]  [font=\tiny]  {$\ k( m) \ \ \ \  k( m)$};
			\draw (220.87,472) node [anchor=north west][inner sep=0.75pt]  [font=\tiny]  {$\ k( m) \ \ \ \ \ k( m)$};
		\end{tikzpicture}
	\end{equation}
	\item[V6] (Absorbing evaluation and coevaluation)  Absorbing evaluation and coevaluation maps into a coupon,
	\begin{equation}\label{rel:V6}
		\tikzset{every picture/.style={line width=0.5pt}} 
\begin{tikzpicture}[x=0.75pt,y=0.75pt,yscale=-1,xscale=1]

\draw    (170.46,71.88) .. controls (171.35,93.84) and (160.75,82.86) .. (159.87,100.59) .. controls (158.99,118.33) and (171.43,128) .. (173.19,110.26) ;
\draw    (177.53,71.88) -- (177.53,104.81) ;
\draw    (177.43,110.26) -- (177.53,147.04) ;
\draw   (265.32,105.29) -- (279.87,105.29) -- (279.87,111.21) -- (265.32,111.21) -- cycle ;
\draw    (275.77,70.67) -- (275.77,105.29) ;
\draw    (272.6,111.21) -- (272.24,147.52) ;
\draw    (269.59,70.67) -- (269.57,105.01) ;
\draw    (368.22,145.35) .. controls (369.1,123.39) and (358.51,134.37) .. (357.63,116.64) .. controls (356.74,98.9) and (369.02,88.23) .. (370.79,105.97) ;
\draw    (375.28,145.35) -- (375.28,112.42) ;
\draw    (375.03,105.97) -- (375.28,70.19) ;
\draw    (454.74,147.04) -- (454.74,112.42) ;
\draw    (451.4,106.44) -- (451.21,70.19) ;
\draw    (448.56,147.04) -- (448.56,113.26) ;
\draw   (168.95,104.34) -- (183.49,104.34) -- (183.49,110.26) -- (168.95,110.26) -- cycle ;
\draw   (163.52,98.83) .. controls (161.5,101.05) and (160.27,103.27) .. (159.86,105.5) .. controls (159.46,103.27) and (158.26,101.05) .. (156.25,98.81) ;
\draw   (361.12,112.67) .. controls (359.09,114.89) and (357.87,117.11) .. (357.46,119.33) .. controls (357.06,117.11) and (355.86,114.88) .. (353.85,112.65) ;
\draw   (365.33,106.25) -- (379.88,106.25) -- (379.88,112.17) -- (365.33,112.17) -- cycle ;
\draw   (444.13,106.72) -- (458.68,106.72) -- (458.68,112.65) -- (444.13,112.65) -- cycle ;

\draw (196,101.33) node [anchor=north west][inner sep=0.75pt]   [align=left] {=};
\draw (147,60.04) node [anchor=north west][inner sep=0.75pt]  [font=\tiny] [align=left] {$\displaystyle k( m_{1})$};
\draw (176.24,60.04) node [anchor=north west][inner sep=0.75pt]  [font=\tiny] [align=left] {$\displaystyle k( m_{2})$};
\draw (170.41,149.88) node [anchor=north west][inner sep=0.75pt]  [font=\tiny] [align=left] {$\displaystyle k( n)$};
\draw (264.24,149.51) node [anchor=north west][inner sep=0.75pt]  [font=\tiny] [align=left] {$\displaystyle k( n)$};
\draw (244,60.4) node [anchor=north west][inner sep=0.75pt]  [font=\tiny] [align=left] {$\displaystyle k( m_{1})$};
\draw (272.72,60.4) node [anchor=north west][inner sep=0.75pt]  [font=\tiny] [align=left] {$\displaystyle k( m_{2})$};
\draw (303.96,99.27) node [anchor=north west][inner sep=0.75pt]   [align=left] {and};
\draw (415.65,101.33) node [anchor=north east][inner sep=0.75pt]  [rotate=-180,xscale=-1] [align=left] {=};
\draw (348.16,147.34) node [anchor=north west][inner sep=0.75pt]  [font=\tiny] [align=left] {$\displaystyle k( n_{1})$};
\draw (372.88,147.34) node [anchor=north west][inner sep=0.75pt]  [font=\tiny] [align=left] {$\displaystyle k( n_{2})$};
\draw (367.17,60.35) node [anchor=north west][inner sep=0.75pt]  [font=\tiny] [align=left] {$\displaystyle k( m)$};
\draw (443.09,60.35) node [anchor=north west][inner sep=0.75pt]  [font=\tiny] [align=left] {$\displaystyle k( m)$};
\draw (427.61,148.19) node [anchor=north west][inner sep=0.75pt]  [font=\tiny] [align=left] {$\displaystyle k( n_{1})$};
\draw (452.33,149.03) node [anchor=north west][inner sep=0.75pt]  [font=\tiny] [align=left] {$\displaystyle k( n_{2})$};
\draw (215.06,97.61) node [anchor=north west][inner sep=0.75pt]   [align=left] {$\displaystyle ( -1)^{m_{1}}$};

\end{tikzpicture}.
	\end{equation}
\end{itemize}
We note that the signs in (V5) and (V6) are implied by the symmetric structure $k(m)\otimes k(m)\xrightarrow{\sim} k(m)\otimes k(m)$ in  $\opn{Vect}^{\mathbb Z}$ and the evaluation morphisms $k(m_1)\ot k(m_1)^\ast\to \1$ provided by the choice of trivial ribbon structure.  Note also that internal composition in conjunction with (V6) imply a variant of (V6) where the left-most string is oriented upwards, instead of downwards. We define the category of \emph{linearly reduced strings}  $k\Str^{\opn{red}}_{\opn{Vect}^{\mathbb Z}}$ by 
\[
k\Str_{\opn{Vect}^{\mathbb Z}}^{\red}:=k\Str_{\opn{Vect}^{\mathbb Z}} / \left<
\opn{U1}-\opn{U4}, \opn{V5},\opn{V6} \right>.\]

We have that the (linearized) evaluation functor $ev:k\opn{Str}_{\opn{Vect}^{\mathbb{Z}}}\to \opn{Vect}^{\mathbb{Z}}$ factors through the reduction, that the (linearized) forgetful functor $p:k\opn{Bord}_{\opn{Vect}^{\mathbb{Z}}}\to k\opn{Str}_{\opn{Vect}^{\mathbb{Z}}}$ fits into a diagram
\[
\xymatrix{
k\opn{Bord}_{\opn{Vect}^{\mathbb{Z}}}\ar[rr]^{p}\ar[d] & & k\opn{Str}_{\opn{Vect}^{\mathbb{Z}}}\ar[d]\\
k\opn{Bord}_{\opn{Vect}^{\mathbb{Z}}}^{\red}\ar[rr]_{\exists !} & & k\opn{Str}_{\opn{Vect}^{\mathbb{Z}}}^{\red},
}
\]
and that the reduction of the tautological TQFT also factors through $k\opn{Str}^{\opn{red}}_{\opn{Vect}^{\mathbb{Z}}}$ (see Section \ref{sect:s_taut}).  We therefore obtain a diagram
\[
\xymatrix{
 & k\opn{Str}^{\opn{red}}_{\opn{Vect}^{\mathbb{Z}}}\ar[dr]^{ev_{\opn{Vect}^{\mathbb{Z}}}} \\
k\Bord_{\opn{Vect}^{\mathbb{Z}}}^{\opn{red}}\ar[rr]_{Ev_{\opn{Vect}^{\mathbb{Z}}}}\ar[ur]^{p} & & \opn{Vect}^{\mathbb{Z}}
}
\]
and a factoring of the base changed TQFT $\underline{Z}:\Bord_{\msc{A}^{\mbb{Z}}}\to \opn{Vect}^{\mathbb{Z}}$ as
\[
\begin{array}{l}
\underline{Z}=(ev\cdot Z)(p\boxtimes id)\Delta:\vspace{1mm}\\
\Bord_{\msc{A}^{\mbb{Z}}}\to (k\Bord_{\opn{Vect}^{\mathbb Z}}^{\red}\boxtimes k\Bord^{\red}_{\msc{A}})^{\opn{add}}\to (k\Str_{\opn{Vect}^{\mathbb Z}}^{\red}\boxtimes k\Bord^{\red}_{\msc{A}})^{\opn{add}}\to \opn{Vect}^{\mathbb{Z}},
\end{array}
\]
at any reasonable $Z:\Bord_{\msc{A}}\to \opn{Vect}$.  We therefore open up the possibility of using string diagrams in our analysis of the graded theory $\underline{Z}$.

\subsection{Defining a differential $d_{\Sigma_{\vec{x}}}$ for $\Sigma_{\vec{x}}$.}
\label{sect:diff_def}

Let $x_i=\oplus_{n \in \mathbb Z} x_i^n$ be complexes in $\Ch(\msc{A})$ with differentials  $d_{x_i}:x_i\to x_i[1]$, and let $\vec{x}:=(x_1, \dots, x_t)$. 
In our efforts to prove Theorem \ref{thm: Zast},  for every labelled surface $\Sigma_{\vec{x}}$ in $\Bord_{\Ch{(\msc{A})}}$ we will associate to $\underline{Z}(\Sigma_{\vec{x}})$ a map 
\[d_{\Sigma_{\vec{x}}}: \underline{Z}(\Sigma_{\vec{x}}) \to \underline{Z} (\Sigma_{\vec{x}})[1]\] satisfying $(d_{\Sigma_{\vec{x}}}[1])d_{\Sigma_{\vec{x}}}=0$. Below, for a tuple of integers $m:\{1,\dots,t\}\to \mathbb{Z}$, we let $k(m)$ denote the object in $k\opn{Str}^{\red}_{\opn{Vect}^{\mathbb{Z}}}$ provided by the disk $D$ with labels $k(m_i)$ along the $x$-axis.  We note that this object is identified with the product $k(m_1)\ot \dots\ot k(m_t)$ in $k\opn{Str}^{\red}_{\opn{Vect}^{\mathbb{Z}}}$.

Let us assume for the moment that all of the markings on $\Sigma_{\vec{x}}$ are \emph{positive}, i.e.\ that their framings reproduce the ambient orientation on $\Sigma$.  Following the discussions in Sections \ref{sect:outliner} and \ref{sect:red_strings}, we have that 
\[
\underline{Z}(\Sigma_{\vec{x}}) = \oplus_{m:\{1, \dots, t\}\to \mathbb Z} k(m_1)\ot\dots\ot k(m_t)\otimes Z(\Sigma_{\vec{x}^m})
\]
in $\opn{Vect}^{\mathbb Z}$, where $\Sigma_{\vec{x}^m}$ has marking object $x_i^{m_i}$ at the $i$-th marked point. 
We begin by defining an auxiliary differential
\[
\tilde d_{\Sigma_{\vec{x}}}:\oplus_{m:\{1, \dots, t\}\to \mathbb Z} k(m)\otimes \Sigma_{\vec{x}^m} \to \oplus_{m:\{1, \dots, t\}\to \mathbb Z} 
\left(  \oplus_{i=1}^t k(-1)\otimes k(m+\delta_i)\otimes \Sigma_{\vec{x}^{m+\delta_i}}     \right) 
\]
in $(k\Str_{\opn{Vect}^{\mathbb Z}}^{\red}\boxtimes k\Bord^{\red}_{\msc{A}})^{\opn{add}}$, where $\delta_i:\{1, \dots, t\}\to \mathbb Z$ is $\delta_i(j):=\delta_{i,j}$, in the following way. Each summand 
\[
\tilde d^m_{\Sigma_{\vec{x}}}:k(m)\otimes \Sigma_{\vec{x}^m} \to \oplus_{i=1}^t k(-1)\otimes k(m+\delta_i)\otimes \Sigma_{\vec{x}^{m+\delta_i}}     
\]
is given by the sum
\begin{equation}\label{eq: differential diagram}
	\tikzset{every picture/.style={line width=0.5pt}} 
	\begin{tikzpicture}[x=0.75pt,y=0.75pt,yscale=-1,xscale=1]
		
		\draw   (160.89,1211.48) -- (167.56,1211.48) -- (167.56,1216) -- (160.89,1216) -- cycle ;
		\draw    (162.8,1216.26) .. controls (162.8,1249.21) and (119.98,1218.15) .. (119.98,1249.81) ;
		\draw    (120.17,1202.54) .. controls (120.17,1195.11) and (118.77,1222.49) .. (122.75,1230.69) .. controls (126.72,1238.88) and (136.85,1235.86) .. (135.7,1249.81) ;
		\draw    (164.2,1200.95) -- (164.23,1211.47) ;
		\draw    (165.4,1216.2) -- (165.46,1250.19) ;
		\draw    (327.99,1202.57) -- (327.9,1229.55) -- (327.86,1245.09) ;
		\draw  [draw opacity=1 ][fill={rgb, 255:red, 255; green, 255; blue, 255 }  ,fill opacity=1 ] (269.54,1216.89) -- (295.93,1216.89) -- (295.93,1233.4) -- (269.54,1233.4) -- cycle ;
		\draw    (235.8,1202.8) .. controls (236.26,1190.43) and (328.27,1191) .. (327.99,1202.57) ;
		\draw    (235.8,1202.8) -- (235.7,1229.79) -- (235.67,1245.32) ;
		\draw    (200.65,1200.11) -- (200.6,1250.6) ;
		\draw    (235.67,1245.32) .. controls (235.6,1258.47) and (327.61,1257.9) .. (327.86,1245.09) ;
		\draw 
		(283.03,1207.81) -- (283.05,1217.52) ;
		\draw 
		(283.74,1234.28) -- (283.75,1240.2) ;
		\draw 
		  (253.94,1207.44) -- (253.96,1239.82) ;
		\draw 
		(249.39,1208.19) -- (249.4,1240.58) ;
		\draw 
		(313.88,1207.44) -- (313.89,1239.82) ;
		\draw 
		(309.32,1208.19) -- (309.34,1240.58) ;
		\draw [color={rgb, 255:red, 155; green, 155; blue, 155 }  ,draw opacity=1 ] [dotted]  (238,1247) .. controls (239,1246.5) and (262.5,1246) .. (280,1246) .. controls (297.5,1246) and (317.2,1245.76) .. (327,1247.5) ;
		\draw [color={rgb, 255:red, 155; green, 155; blue, 155 }  ,draw opacity=1 ] [dotted]  (238,1200.75) .. controls (240.5,1203.25) and (258.5,1203) .. (280.5,1203) .. controls (302.5,1203) and (322,1202) .. (324.5,1201.25) ;
		
		\draw (273,1219) node [anchor=north west][inner sep=0.75pt]  [font=\tiny]  {$d_{x_ i}^{m_{i}}$};
		\draw (112.2,1187.8) node [anchor=north west][inner sep=0.75pt]  [font=\tiny]  {$k( m_{1}) \ \ \ \ k( m_{i}) \ \ \ k( m_{t}) \ $};
		\draw (105,1250.8) node [anchor=north west][inner sep=0.75pt]  [font=\tiny]  {$k( -1) \ \ \ \ \ k( m+\delta _{i}) \ $};
		\draw (80,1210) node [anchor=north west][inner sep=0.75pt]  [font=\large]  {$  \sum\limits_{1\leq i\leq t} \ \ \ \ ... \ \ \ \ \ \ ...\ \ \ \ \ \otimes $};
		\draw (330,1195) node [anchor=north west][inner sep=0.75pt]  [color={rgb, 255:red, 0; green, 0; blue, 0 }  ,opacity=1 ]  {$\Sigma _{\vec{x}^{m}}$};
		\draw (330.32,1235.28) node [anchor=north west][inner sep=0.75pt]  [color={rgb, 255:red, 0; green, 0; blue, 0 }  ,opacity=1 ]  {$\Sigma _{\vec{x}^{m+\delta _{i}}}$};
		
	\end{tikzpicture}
\end{equation}
where the diagram on the left is the map 
\[
k(m)=k(m_1)\otimes \dots \otimes k(m_t) \to k(-1)\otimes k(m+\delta_i)
\]
in $k\Str_{\opn{Vect}^{\mathbb Z}}^{\red}$ with a single  coupon denoting the canonical isomorphism $k(m_i)\xrightarrow{\sim} k(-1)\otimes k(m_i),$ followed by a composition of braidings. On the right, the diagram represents the morphism $\Sigma_{d_{x_i}^{m_i}}:\Sigma_{\vec{x}^m}\to \Sigma_{\vec{x}^{m+\delta_i}}$ in $\Bord^{\red}_{\msc{A}}$  as defined in Section \ref{sect:a_sigma}, where $\Sigma_{d_{x_i}^{m_i}}$ is specifically associated to the tuple $(1, \dots, d_{x_i}^{m_i}, \dots, 1)$.

Since $Ev_{\opn{Vect}^{\mathbb Z}}$ is the composition  $k\Bord_{\opn{Vect}^{\mathbb Z}}^{\red} \xrightarrow{p_{\opn{Vect}^{\mathbb Z}}}k\Str_{\opn{Vect}^{\mathbb Z}}^{\red}\xrightarrow{ev_{\opn{Vect}^{\mathbb Z}}}\opn{Vect}^{\mathbb Z}$, then
\[
\underline{Z}(\Sigma_{\vec{x}})=\left( ev_{\opn{Vect}^{\mathbb Z}}\cdot Z \right) \left( \oplus_{m:\{1, \dots, t\}\to \mathbb Z} k(m)\otimes \Sigma_{\vec{x}^m}\right), \text{and}
\]
\[
\underline{Z}(\Sigma_{\vec{x}})[1]=\left( ev_{\opn{Vect}^{\mathbb Z}}\cdot Z\right) \left( \oplus_{m:\{1, \dots, t\}\to \mathbb Z} k(-1)\otimes k(m+\delta_i)\otimes \Sigma_{\vec{x}^{m+\delta_i}}\right).
\]
Hence evaluating the morphism $\tilde{d}_{\Sigma_{\vec{x}}}$ in the additive closure of $k\opn{Str}_{\opn{Vect}^{\mbb{Z}}}^{\opn{red}}\boxtimes\Bord_{\msc{A}}^{\opn{red}}$ produces a map of graded spaces
\begin{equation}\label{eq:476}
d_{\Sigma_{\vec{x}}}=ev_{\opn{Vect}^{\mathbb{Z}}}\cdot Z(\tilde{d}_{\Sigma_{\vec{x}}}):\underline{Z}(\Sigma_{\vec{x}})\to \underline{Z}(\Sigma_{\vec{x}})[1].
\end{equation}

We now consider the case where the marked points in $\Sigma_{\vec{x}}$ have generic framings.  Supposing we have underlying points $p:\{1,\dots,t\}\to \Sigma$, let $\Sigma_{\vec{x}}^+$ denote the same underlying surface with the same points $p$, but with all negative framings replaced by positive framings, by flipping the $x$-vector, and all negatively framed objects $x_i$ replaced by their dual $x_i^\ast$.  We have the apparent isomorphism
\[
\Sigma^{flip}=\Sigma^{flip}_{\vec{x}}:\Sigma_{\vec{x}}\to \Sigma_{\vec{x}}^+
\]
in $\Bord_{\opn{Ch}(\msc{A})}^{\opn{red}}$ provided by the bordism $\Sigma\times [0,1]$ equipped with straight lines traveling between the $i$-th markings, each of which is split by a single internal vertex/coupon which is labeled by the identity $id_{x^\ast}$.  (So, $\Sigma^{flip}$ looks like $\Sigma_{\vec{id}}$ but there is a flipping of ribbon orientations across internal vertices.)

For general $\Sigma_{\vec{x}}$ in $\Bord_{\opn{Ch}(\msc{A})}$ we now define the differential via the differential on $\Sigma_{\vec{x}}^{+}$ and this flip bordism.

\begin{definition}\label{def: differential}
Let $Z:\Bord_{\msc{A}}\to \opn{Vect}$ be a reasonable field theory and consider a marked surface $\Sigma_{\vec{x}}$ in $\Bord_{\opn{Ch}(\msc{A})}$ with corresponding positively marked surface $\Sigma^+_{\vec{x}}$. We define the differential on the graded state space $\underline{Z}(\Sigma_{\vec{x}})$ as the conjugate
\[
d_{\Sigma_{\vec{x}}}:=Z(\Sigma^{flip})^{-1}d_{\Sigma^+_{\vec{x}}}Z(\Sigma^{flip}):\underline{Z}(\Sigma_{\vec{x}})\to \underline{Z}(\Sigma_{\vec{x}})[1],
\]
where $d_{\Sigma^+_{\vec{x}}}$ is as in \eqref{eq:476}.
\end{definition}

After dealing with various annoyances with orientations, one sees that the differential $d_{\Sigma_{\vec{x}}}$ is always defined as in the positively marked case, except that each $\Sigma_{d_{x_i}^{m_i}}$ factor in \eqref{eq: differential diagram} is replaced with $-(-1)^{m_i}\Sigma_{d_{x_i}^{m_i-1}}$ at negative markings.  Or equivalently, each downwardly oriented coupon $d_{x_i}^{m_i}$ in \eqref{eq: differential diagram} is replaced with the downwardly oriented coupon $d_{x_i^\ast}^{-m_i}$.  In the strings factor, we travel down along an upwardly oriented string from $k(m_i)$ to $k(m_i-1)$ at such negative markings.

The following lemma allows us to reduce our analysis, generally, to the case of positively marked surfaces.

\begin{lemma}\label{lem:d_pos}
Let $(M,T):\Sigma_{\vec{x}}\to \Sigma_{\vec{y}}'$ be a bordism in $\Bord_{\opn{Ch}(\msc{A})}$ and 
\[
(N,Q)=\Sigma^{flip}(M,T)(\Sigma^{flip})^{-1}:\Sigma^+_{\vec{x}}\to {\Sigma'}^+_{\vec{y}}
\]
be the corresponding bordism between positively marked surfaces.  We have
\[
d_{\Sigma_{\vec{y}}}\ \underline{Z}(M,T)= \underline{Z}(M,T)[1]\ d_{\Sigma_{\vec{x}}}
\ \ \text{iff}\ \ d_{\Sigma_{\vec{y}}^+}\ \underline{Z}(N,Q)= \underline{Z}(N,Q)[1]\ d_{\Sigma_{\vec{x}}^+}.
\]
\end{lemma}

\begin{proof}
Apparent.
\end{proof}

\subsection{The differential satisfies $d^2=0$}

To prove that $ d_{\Sigma_{\vec{x}}}[1]d_{\Sigma_{\vec{x}}}=0$,
\begin{lemma}
	We have that $(\tilde d_{\Sigma_{\vec{x}}}[1] )\tilde d_{\Sigma_{\vec{x}}} =0$ in $(k\Str_{\opn{Vect}^{\mathbb Z}}^{\red}\boxtimes k\Bord^{\red}_{\msc{A}})^{\opn{add}}$.
\end{lemma}

\begin{proof}
It suffices to consider the case of a positively marked surface, by the definition of $d_{\Sigma_{\vec{x}}}$ itself.  Recall that, at each $m$, $
	\tilde d^m_{\Sigma_{\vec{x}}}:k(m)\otimes \Sigma_{\vec{x}^m} \to \oplus_{i=1}^t k(-1)\otimes k(m+\delta_i)\otimes \Sigma_{\vec{x}^{m+\delta_i}}$ 
is as in \eqref{eq: differential diagram}. Hence, at each summand $i$, we have that 	
\begin{align*}
	\tilde d^{m+\delta_i}_{\Sigma_{\vec{x}}}[1]: k(-1)\otimes k(m+\delta_i)\otimes \Sigma_{\vec{x}^{m+\delta_i}} \to \oplus_{j=1}^t
	  k(-1)\otimes k(-1)\otimes k(m+\delta_i+\delta_j) \otimes \Sigma_{\vec{x}^{m+\delta_i+\delta_j}}
\end{align*}
is given by 
\begin{equation*}
	\scalebox{.9}{%
	\tikzset{every picture/.style={line width=0.5pt}} 
}
\end{equation*}
	We compute below the composition $\tilde d_{\Sigma_{\vec{x}}}[1]\tilde d^m_{\Sigma_{\vec{x}}}$ via graphical calculus:
\begin{equation*}
	\scalebox{.9}{
		\tikzset{every picture/.style={line width=0.5pt}} 
		%
}
\end{equation}
	where in the last equality we are using that $d_{x_i}^{m_i+1} \ d_{x_i}^{m_i}=0$  for all $i$. Thus to prove that $\tilde d_{\Sigma_{\vec{x}}}[1]\tilde d^m_{\Sigma_{\vec{x}}}=0$ for all $m$, it is enough to show that the remaining terms in \eqref{eq:d^2=0} cancel out in pairs. In fact, for fixed $i$ and $j$, we write below the terms corresponding to the pair $(i,j)$ in \eqref{eq:d^2=0}, 
	\begin{equation}\label{eq: (i,j) pair}
		\scalebox{0.9}{
\tikzset{every picture/.style={line width=0.5pt}} 

%
}
\end{equation*}
where in the second equality we are just undoing the cross between the first and second strands using relation (V5), see \eqref{rel:V5}. It follows then that \eqref{eq: (i,j) pair} is equal to zero, and so $\tilde d^m_{\Sigma_{\vec{x}}}[1]\tilde d^m_{\Sigma_{\vec{x}}}=0$ for all $m$, as desired. 
\end{proof}

\begin{corollary}
		The composition $(d_{\Sigma_{\vec{x}}}[1] ) d_{\Sigma_{\vec{x}}} =0$ in $\opn{Vect}^{\mathbb Z}$.
\end{corollary}
\begin{proof}
	This follows from applying $ev_{\opn{Vect}^{\mathbb Z}}\cdot Z$ to $(\tilde d_{\Sigma_{\vec{x}}}[1] ) \tilde d_{\Sigma_{\vec{x}}} =0$.
\end{proof}

\begin{corollary}
	For any $\Sigma_{\vec{x}}$ in $\Bord_{\Ch(\msc{A})}$ the differential $d_{\Sigma_{\vec{x}}}:\underline{Z}(\Sigma_{\vec{x}})\to \underline{Z}(\Sigma_{\vec{x}})[1]$ endows the state space $\underline{Z}(\Sigma_{\vec{x}})$ with the structure of a linear cochain complex.
\end{corollary}

\begin{definition}\label{def:zast}
Consider a ribbon tensor category $\msc{A}$ and a reasonable field theory $Z:\Bord_{\msc{A}}\to \opn{Vect}$.  For any labelled surface $\Sigma_{\vec{x}}$ in $\Bord_{\Ch(\msc{A})}$ we define $Z^{\ast}(\Sigma_{\vec{x}})$ as the complex \[Z^*(\Sigma_{\vec{x}}):=(\underline{Z}(\Sigma_{\vec{x}}), d_{\Sigma_{\vec{x}}}).\]
For any bordism $(M,T)$ in $\Bord_{\Ch(\msc{A})}$ we define $Z^*(M,T):=\underline{Z}(M,T)$.
\end{definition}

At this point we are not claiming that $Z^{\ast}(M,T)$ is a cochain map, though of course we argue this point below.

\section{Bordisms and the differential}\label{sec: diff comm}

In pursuing the proof of Theorem \ref{thm: Zast}, we need the following proposition.

\begin{proposition}\label{comm of differential}
	For each ribbon bordism $(M,T):\Sigma_{\vec{x}}\to \Sigma'_{\vec{y}}$ in $\Bord_{\opn{Ch}(\msc{A})}$, the map $d_{\Sigma_{\vec{x}}}$ satisfies 
	\[
	Z^*(M,T)[1] \ d_{\Sigma_{\vec{x}}} = d_{\Sigma'_{\vec{y}}}\  Z^*(M,T).
	\]
\end{proposition}

This section is dedicated to proving this result; we will do so via a series of lemmas.
\par

For $(M,T):\Sigma_{\vec{x}}\to\Sigma'_{\vec{y}}$ in  $\Bord_{\Ch(\msc{A})},$ we represent its image through $(p\boxtimes 1)\Delta$ in $(k\opn{Str}_{\opn{Vect}^{\mathbb Z}}^{\red}\boxtimes k\Bord^{\red}_{\msc{A}})^{\opn{add}}$
by the picture below,
\begin{equation*}
	\tikzset{every picture/.style={line width=0.5pt}} 
	\begin{tikzpicture}[x=0.75pt,y=0.75pt,yscale=-1,xscale=1]
		
		\draw    (789.2,77.27) .. controls (791.6,101.4) and (786,147) .. (795.55,160.06) ;
		\draw    (678.44,77.58) .. controls (676.4,103.8) and (673.2,151.4) .. (685.03,161.44) ;
		\draw  [color={rgb, 255:red, 128; green, 128; blue, 128 }  ,draw opacity=1 ] (685.21,90.37) -- (783.56,90.37) -- (783.56,147.75) -- (685.21,147.75) -- cycle ;
		\draw  [color={rgb, 255:red, 255; green, 255; blue, 255 }  ,draw opacity=1 ][line width=3] [line join = round][line cap = round] (748.65,150.84) .. controls (748.65,151.99) and (748.65,153.14) .. (748.65,154.3) ;
		\draw  [color={rgb, 255:red, 255; green, 255; blue, 255 }  ,draw opacity=1 ][line width=3] [line join = round][line cap = round] (752.25,168.12) .. controls (752.48,170.01) and (753.12,171.85) .. (753.79,173.65) ;
		\draw    (678.44,77.58) .. controls (678.9,65.21) and (789.48,65.71) .. (789.2,77.27) ;
		\draw    (685.03,161.44) .. controls (692.4,173) and (795.3,172.88) .. (795.55,160.06) ;
		\draw [color={rgb, 255:red, 155; green, 155; blue, 155 }  ,draw opacity=1 ] [dotted={on 0.84pt off 2.51pt}]  (685.03,161.44) .. controls (686.8,158.6) and (727.6,159.4) .. (740.4,159.4) .. controls (753.2,159.4) and (786,159.4) .. (795.8,161.14) ;
		\draw [color={rgb, 255:red, 155; green, 155; blue, 155 }  ,draw opacity=1 ] [dotted={on 0.84pt off 2.51pt}]  (678.44,77.58) .. controls (682.31,83.65) and (717.2,82.6) .. (738,82.6) .. controls (758.8,82.6) and (788.84,82.28) .. (788.96,76.19) ;
		\draw  [color={rgb, 255:red, 128; green, 128; blue, 128 }  ,draw opacity=1 ] (535.21,91.37) -- (633.56,91.37) -- (633.56,148.75) -- (535.21,148.75) -- cycle ;
		\draw    (540,69.5) -- (540,91.25) ;
		\draw    (549.5,69.5) -- (549.5,91.25) ;
		\draw    (621,70) -- (621,91.75) ;
		\draw    (630,70) -- (630,91.75) ;
		\draw    (540,148.5) -- (540,170.25) ;
		\draw    (549.5,148.5) -- (549.5,170.25) ;
		\draw    (621,149) -- (621,170.75) ;
		\draw    (630,149) -- (630,170.75) ;
		
		\draw (685.21,92) node [anchor=north west][inner sep=0.75pt]  [color={rgb, 255:red, 128; green, 128; blue, 128 }  ,opacity=1 ]  {$T_{\msc{A} ,m,n}$};
		\draw (796.17,147.65) node [anchor=north west][inner sep=0.75pt]  [color={rgb, 255:red, 0; green, 0; blue, 0 }  ,opacity=1 ]  {$\Sigma' _{\vec{y}^{n}}$};
		\draw (790,65) node [anchor=north west][inner sep=0.75pt]  [color={rgb, 255:red, 0; green, 0; blue, 0 }  ,opacity=1 ]  {$\Sigma _{\vec{x}^{m}}$};
		\draw (535.21,92) node [anchor=north west][inner sep=0.75pt]  [color={rgb, 255:red, 128; green, 128; blue, 128 }  ,opacity=1 ]  {$T_{\opn{Vect}^{\mathbb Z},m,n}$};
		\draw (456.95,105) node [anchor=north west][inner sep=0.75pt]    {$\sum\limits_{m,n:\{1, \dots, t\} \to \mathbb Z} \ \ \ \ \ \ \ \ \ \ \ \ \ \ \ \ \ \ \ \ \ \ \ \ \ \otimes $};
		\draw (520,59) node [anchor=north west][inner sep=0.75pt]  [font=\tiny]  {$k(m_1) \ \ \ \ \ \ \ \ \ \ \ \ \ \ \ \ \ \ \ \ \ \ \ \  k(m_t) $};
		\draw (520,170) node [anchor=north west][inner sep=0.75pt]  [font=\tiny]  {$k(n_1) \ \ \ \ \ \ \ \ \ \ \ \ \ \ \ \ \ \ \ \ \ \ \ \  k(n_l) $};
		\draw (575,75) node [anchor=north west][inner sep=0.75pt]    {$\cdots $};
		\draw (575,155) node [anchor=north west][inner sep=0.75pt]    {$\cdots $};
	\end{tikzpicture}
\end{equation*}
where each summand is a morphism $k(m)\otimes \Sigma_{\vec{x}^m} \to k(n)\otimes \Sigma'_{\vec{y}^n}$ in $k\opn{Str}_{\opn{Vect}^{\mathbb Z}}^{\red}\boxtimes k\Bord^{\red}_{\msc{A}}$, and $T$ factors into the appropriate formal
sum over all $m,n\in \mathbb{Z}^t$ of $(T_{\opn{Vect}^{\mathbb Z}, m,n}, T_{\msc{A},m,n}).$ To prove the lemmas, we will use graphical calculus in $(k\opn{Str}_{\opn{Vect}^{\mathbb Z}}^{\red}\boxtimes k\Bord^{\red}_{\msc{A}})^{\opn{add}}$, replacing the boxes for $T_{\opn{Vect}^{\mathbb Z}, m,n}$ and $T_{\msc{A},m,n}$ by the corresponding string/ribbon diagram in each case. 

\subsection{Commuting with permutations}
We show first that if $(M,T)$ is a bordism in $\Bord_{\Ch(\msc{A})}$ such that $T$ can be thought of as a ``permutation of ribbons", then $Z^*(M,T)\in \Ch(\opn{Vect})$.

\begin{lemma}\label{lemma 1}
	Let $(M,T):\Sigma_{\vec{x}}\to\Sigma'_{\vec{y}}$ be a connected bordism in  $\opn{Bord}_{\opn{Ch}(\msc{A})}$ such that the following hold:
	\begin{itemize}
	\item There is a set bijection between the labeling sets $\sigma:I\to J$.
	\item $T$ consists (only) of couponless ribbons traveling from the $i$-th marking on $\Sigma_{\vec{x}}$ to the $\sigma(i)$-th marking on $\Sigma'_{\vec{y}}$.
	\end{itemize}
	Then 
	$Z^*(M,T)[1] \ d_{\Sigma_{\vec{x}}} = d_{\Sigma'_{\vec{y}}}\  Z^*(M,T).$
\end{lemma}

	\begin{proof}
By Lemma \ref{lem:d_pos} in conjunction with the internal composition relation (U2) we reduce to the case where $\Sigma_{\vec{x}}$ and $\Sigma'_{\vec{y}}$ are positively marked.  We are also free to take $I=J=\{1,\dots,t\}$, so that $\sigma$ is a permutation in $S_t$.
	
	We start by assuming $\sigma$ is given by a single transposition of consecutive labels, i.e., that  there exists  $1\leq j\leq t-1$ such that $\sigma(i)=i$ for all $i\ne j, j-1$, and $\sigma(j)=j-1$.
	We represent $(p\boxtimes 1)\Delta(M,T)$ by 
	\begin{equation*}\scalebox{0.9}{
			\tikzset{every picture/.style={line width=0.5pt}} 
			\begin{tikzpicture}[x=0.75pt,y=0.75pt,yscale=-1,xscale=1]
				
				\draw    (366.2,346.27) .. controls (368.6,370.4) and (363,416) .. (372.55,429.06) ;
				\draw    (255.44,346.58) .. controls (253.4,372.8) and (250.2,420.4) .. (262.03,430.44) ;
				\draw  [color={rgb, 255:red, 128; green, 128; blue, 128 }  ,draw opacity=1 ] (262.21,359.37) -- (360.56,359.37) -- (360.56,416.75) -- (262.21,416.75) -- cycle ;
				\draw  [color={rgb, 255:red, 255; green, 255; blue, 255 }  ,draw opacity=1 ][line width=3] [line join = round][line cap = round] (325.65,419.84) .. controls (325.65,420.99) and (325.65,422.14) .. (325.65,423.3) ;
				\draw  [color={rgb, 255:red, 255; green, 255; blue, 255 }  ,draw opacity=1 ][line width=3] [line join = round][line cap = round] (329.25,437.12) .. controls (329.48,439.01) and (330.12,440.85) .. (330.79,442.65) ;
				\draw    (65.19,348.1) -- (66.13,421.25) ;
				\draw    (190.12,348.1) -- (191.05,421.75) ;
				\draw    (152.79,349.75) .. controls (152.79,400.75) and (106.73,370.75) .. (106.73,419.75) ;
				\draw    (107.51,349.75) .. controls (107.51,400.75) and (153.58,370.25) .. (153.58,419.25) ;
				\draw    (255.44,346.58) .. controls (255.9,334.21) and (366.48,334.71) .. (366.2,346.27) ;
				\draw    (262.03,430.44) .. controls (269.4,442) and (372.3,441.88) .. (372.55,429.06) ;
				\draw [color={rgb, 255:red, 155; green, 155; blue, 155 }  ,draw opacity=1 ] [dotted={on 0.84pt off 2.51pt}]  (262.03,430.44) .. controls (263.8,427.6) and (304.6,428.4) .. (317.4,428.4) .. controls (330.2,428.4) and (363,428.4) .. (372.8,430.14) ;
				\draw [color={rgb, 255:red, 155; green, 155; blue, 155 }  ,draw opacity=1 ] [dotted={on 0.84pt off 2.51pt}]  (255.44,346.58) .. controls (259.31,352.65) and (294.2,351.6) .. (315,351.6) .. controls (335.8,351.6) and (365.84,351.28) .. (365.96,345.19) ;
				
				\draw (262.21,359.37) node [anchor=north west][inner sep=0.75pt]  [color={rgb, 255:red, 128; green, 128; blue, 128 }  ,opacity=1 ]  {$T_{\msc{A} ,m}$};
				\draw (373.17,416.65) node [anchor=north west][inner sep=0.75pt]  [color={rgb, 255:red, 0; green, 0; blue, 0 }  ,opacity=1 ]  {$\Sigma'_{\vec{y}^{n}}$};
				\draw (368.54,335) node [anchor=north west][inner sep=0.75pt]  [color={rgb, 255:red, 0; green, 0; blue, 0 }  ,opacity=1 ]  {$\Sigma _{\vec{x}^{m}}$};
				\draw (-7,370) node [anchor=north west][inner sep=0.75pt]    {$\sum\limits_{m:\{1, \dots, t\} \to \mathbb Z} \ \ \ \ \cdots \ \ \ \ \ \ \ \ \ \ \ \  \ \ \cdots \ \ \ \  \ \ \ \ \ \otimes \ \ \ \ \ \ \ \ \ \ \ \  \ \ \ \ \ \ \ \ \ \ \ \ \ \  \ \ \ \ \ \ \ \ \ \ , $};
				\draw (56.27,333.98) node [anchor=north west][inner sep=0.75pt]  [font=\tiny]  {$k( m_{1}) \ \ \ \ k( m_{j}) \ \ \  k( m_{j+1}) \ \ k( m_{t}) \ $};
				\draw (50,420.38) node [anchor=north west][inner sep=0.75pt]  [font=\tiny]  {$k( m_{1})  \ \ \ k( m_{j+1})   \ \ \ \ k( m_{j}) \ \ \ \ \ k( m_{t}) \ $};
			\end{tikzpicture}
		}
	\end{equation*}	
	When composing  $\tilde d^m_{\Sigma_{\vec{x}}}$ as given in \eqref{eq: differential diagram} with the sum above, we obtain for fixed $m$ the following sum
	\begin{equation*}
	\scalebox{.9}{%
		\tikzset{every picture/.style={line width=0.5pt}} 
		
		\begin{tikzpicture}[x=0.75pt,y=0.75pt,yscale=-1,xscale=1]
			
			\draw   (417.54,852.4) -- (424.7,852.4) -- (424.7,856.77) -- (417.54,856.77) -- cycle ;
			\draw    (422.39,856.63) -- (423.05,891.28) ;
			\draw    (419.59,857.01) .. controls (419.59,888.84) and (372.67,864.56) .. (372.67,895.13) ;
			\draw    (373.53,846.21) .. controls (373.53,839.04) and (372.03,865.48) .. (376.29,873.39) .. controls (380.55,881.31) and (391.41,878.39) .. (390.19,891.86) ;
			\draw    (421.12,844.28) -- (421.12,852.39) ;
			\draw    (473.31,844.61) -- (473.39,847.56) -- (473.04,889.7) ;
			\draw   (140.89,850.53) -- (147.56,850.53) -- (147.56,855.05) -- (140.89,855.05) -- cycle ;
			\draw    (142.8,855.31) .. controls (142.8,888.26) and (99.98,857.2) .. (99.98,888.86) ;
			\draw    (100.17,841.59) .. controls (100.17,834.16) and (98.77,861.54) .. (102.75,869.74) .. controls (106.72,877.93) and (116.85,874.91) .. (115.7,888.86) ;
			\draw    (144.2,840) -- (144.23,850.52) ;
			\draw    (372.67,895.13) -- (373.07,934.73) ;
			\draw    (445.37,890.78) .. controls (445.37,922.56) and (423.07,903.8) .. (423.07,934.33) ;
			\draw    (423.05,891.28) .. controls (423.05,923.06) and (445.87,904.6) .. (445.87,935.13) ;
			\draw    (473.04,889.7) -- (473.07,934.33) ;
			\draw    (390.19,891.86) -- (390.67,934.73) ;
			\draw    (145.41,854.72) -- (145.46,891.24) ;
			\draw    (191.85,889.37) .. controls (191.85,922.27) and (168.34,901.14) .. (168.34,932.76) ;
			\draw    (168.09,888.99) .. controls (168.09,921.9) and (191.83,900.89) .. (191.83,932.51) ;
			\draw    (99.98,888.86) -- (100.2,930.2) ;
			\draw    (115.7,888.86) -- (115.8,930.4) ;
			\draw    (145.46,891.24) -- (145.22,932.51) ;
			\draw    (167.29,842.32) -- (168.09,888.99) ;
			\draw    (191.05,842.69) -- (191.85,889.37) ;
			\draw    (445.27,843.61) -- (445.35,846.56) -- (445.37,890.78) ;
			\draw  [color={rgb, 255:red, 255; green, 255; blue, 255 }  ,draw opacity=1 ][line width=3] [line join = round][line cap = round] (287.56,835.85) .. controls (287.75,837.14) and (288.25,838.39) .. (288.78,839.62) ;
			\draw    (218.71,886.5) -- (218.73,931.13) ;
			\draw    (218.97,841.41) -- (219.05,844.36) -- (218.71,886.5) ;
			\draw    (335.17,835.1) .. controls (331,849.25) and (340.5,921.25) .. (336.15,932.61) ;
			\draw    (243.89,835.35) .. controls (242.21,857.1) and (242.6,921.18) .. (249.46,932.8) ;
			\draw  [color={rgb, 255:red, 128; green, 128; blue, 128 }  ,draw opacity=1 ] (251.53,879.8) -- (330,879.8) -- (330,927.39) -- (251.53,927.39) -- cycle ;
			\draw  [color={rgb, 255:red, 255; green, 255; blue, 255 }  ,draw opacity=1 ][line width=3] [line join = round][line cap = round] (301.75,896.12) .. controls (301.75,897.07) and (301.75,898.03) .. (301.75,898.99) ;
			\draw    (243.89,835.35) .. controls (244.27,825.1) and (335.39,825.51) .. (335.17,835.1) ;
			\draw [color={rgb, 255:red, 0; green, 0; blue, 0 }  ,draw opacity=1 ]   (243,867.25) .. controls (249.07,876.84) and (333.29,877.38) .. (333.5,866.75) ;
			\draw [color={rgb, 255:red, 155; green, 155; blue, 155 }  ,draw opacity=1 ] [dash pattern={on 0.84pt off 2.51pt}]  (244,869.75) .. controls (245.46,867.39) and (282.95,868.25) .. (293.5,868.25) .. controls (304.05,868.25) and (327.93,867.31) .. (336,868.75) ;
			\draw [color={rgb, 255:red, 155; green, 155; blue, 155 }  ,draw opacity=1 ] [dash pattern={on 0.84pt off 2.51pt}]  (243.89,835.35) .. controls (247.08,840.39) and (273.43,839.12) .. (290.57,839.12) .. controls (307.71,839.12) and (334.87,839.25) .. (334.97,834.21) ;
			\draw    (249.46,932.8) .. controls (256.2,943.18) and (335.91,943.24) .. (336.15,932.61) ;
			\draw [color={rgb, 255:red, 0; green, 0; blue, 0 }  ,draw opacity=1 ]   (262.91,840.12) -- (262.93,866.99) ;
			\draw [color={rgb, 255:red, 0; green, 0; blue, 0 }  ,draw opacity=1 ]   (258.63,840.09) -- (258.64,866.95) ;
			\draw [color={rgb, 255:red, 0; green, 0; blue, 0 }  ,draw opacity=1 ]   (323.27,839.71) -- (323.29,866.58) ;
			\draw [color={rgb, 255:red, 0; green, 0; blue, 0 }  ,draw opacity=1 ]   (319.39,839.54) -- (319.4,866.4) ;
			\draw [color={rgb, 255:red, 155; green, 155; blue, 155 }  ,draw opacity=1 ] [dash pattern={on 0.84pt off 2.51pt}]  (251.66,933.86) .. controls (252.6,933.45) and (274.69,933.03) .. (291.15,933.03) .. controls (307.6,933.03) and (326.13,932.84) .. (335.34,934.28) ;
			\draw  [color={rgb, 255:red, 0; green, 0; blue, 0 }  ,draw opacity=1 ][line width=3] [line join = round][line cap = round] (304.83,849.71) .. controls (305.07,851.39) and (305.72,853.02) .. (306.41,854.63) ;
			\draw  [color={rgb, 255:red, 0; green, 0; blue, 0 }  ,draw opacity=1 ][fill={rgb, 255:red, 255; green, 255; blue, 255 }  ,fill opacity=1 ] (275,846.6) -- (309,846.6) -- (309,863.06) -- (275,863.06) -- cycle ;
			\draw [color={rgb, 255:red, 0; green, 0; blue, 0 }  ,draw opacity=1 ]   (290.57,838.12) -- (290.59,846.74) ;
			\draw [color={rgb, 255:red, 0; green, 0; blue, 0 }  ,draw opacity=1 ]   (290.78,863.85) -- (290.79,869.11) ;
			\draw  [color={rgb, 255:red, 255; green, 255; blue, 255 }  ,draw opacity=1 ][line width=3] [line join = round][line cap = round] (553.56,842.85) .. controls (553.75,844.14) and (554.25,845.39) .. (554.78,846.62) ;
			\draw    (601.17,842.1) .. controls (597,856.25) and (606.5,928.25) .. (602.15,939.61) ;
			\draw    (509.89,842.35) .. controls (508.21,864.1) and (508.6,928.18) .. (515.46,939.8) ;
			\draw  [color={rgb, 255:red, 128; green, 128; blue, 128 }  ,draw opacity=1 ] (517.53,886.8) -- (596,886.8) -- (596,934.39) -- (517.53,934.39) -- cycle ;
			\draw  [color={rgb, 255:red, 255; green, 255; blue, 255 }  ,draw opacity=1 ][line width=3] [line join = round][line cap = round] (567.75,903.12) .. controls (567.75,904.07) and (567.75,905.03) .. (567.75,905.99) ;
			\draw    (509.89,842.35) .. controls (510.27,832.1) and (601.39,832.51) .. (601.17,842.1) ;
			\draw [color={rgb, 255:red, 0; green, 0; blue, 0 }  ,draw opacity=1 ]   (509,874.25) .. controls (515.07,883.84) and (599.29,884.38) .. (599.5,873.75) ;
			\draw [color={rgb, 255:red, 155; green, 155; blue, 155 }  ,draw opacity=1 ] [dash pattern={on 0.84pt off 2.51pt}]  (510,876.75) .. controls (511.46,874.39) and (548.95,875.25) .. (559.5,875.25) .. controls (570.05,875.25) and (593.93,874.31) .. (602,875.75) ;
			\draw [color={rgb, 255:red, 155; green, 155; blue, 155 }  ,draw opacity=1 ] [dash pattern={on 0.84pt off 2.51pt}]  (509.89,842.35) .. controls (513.08,847.39) and (539.43,846.12) .. (556.57,846.12) .. controls (573.71,846.12) and (600.87,846.25) .. (600.97,841.21) ;
			\draw    (515.46,939.8) .. controls (522.2,950.18) and (601.91,950.24) .. (602.15,939.61) ;
			\draw [color={rgb, 255:red, 0; green, 0; blue, 0 }  ,draw opacity=1 ]   (528.91,847.12) -- (528.93,873.99) ;
			\draw [color={rgb, 255:red, 0; green, 0; blue, 0 }  ,draw opacity=1 ]   (524.63,847.09) -- (524.64,873.95) ;
			\draw [color={rgb, 255:red, 0; green, 0; blue, 0 }  ,draw opacity=1 ]   (589.27,846.71) -- (589.29,873.58) ;
			\draw [color={rgb, 255:red, 0; green, 0; blue, 0 }  ,draw opacity=1 ]   (585.39,846.54) -- (585.4,873.4) ;
			\draw [color={rgb, 255:red, 155; green, 155; blue, 155 }  ,draw opacity=1 ] [dash pattern={on 0.84pt off 2.51pt}]  (517.66,940.86) .. controls (518.6,940.45) and (540.69,940.03) .. (557.15,940.03) .. controls (573.6,940.03) and (592.13,939.84) .. (601.34,941.28) ;
			\draw  [color={rgb, 255:red, 0; green, 0; blue, 0 }  ,draw opacity=1 ][line width=3] [line join = round][line cap = round] (570.83,856.71) .. controls (571.07,858.39) and (571.72,860.02) .. (572.41,861.63) ;
			\draw  [color={rgb, 255:red, 0; green, 0; blue, 0 }  ,draw opacity=1 ][fill={rgb, 255:red, 255; green, 255; blue, 255 }  ,fill opacity=1 ] (541,853.6) -- (575,853.6) -- (575,870.06) -- (541,870.06) -- cycle ;
			\draw [color={rgb, 255:red, 0; green, 0; blue, 0 }  ,draw opacity=1 ]   (556.57,845.12) -- (556.59,853.74) ;
			\draw [color={rgb, 255:red, 0; green, 0; blue, 0 }  ,draw opacity=1 ]   (556.78,870.85) -- (556.79,876.11) ;
			
			\draw (255,881.3) node [anchor=north west][inner sep=0.75pt]  [color={rgb, 255:red, 128; green, 128; blue, 128 }  ,opacity=1 ]  {$T_{\msc{A} ,m+\delta_i}$};
			\draw (280,848.23) node [anchor=north west][inner sep=0.75pt]  [font=\tiny]  {$d_{x_i}^{m_{i}}$};
			\draw (520,889) node [anchor=north west][inner sep=0.75pt]  [color={rgb, 255:red, 128; green, 128; blue, 128 }  ,opacity=1 ]  {$T_{\msc{A} ,m+\delta_j}$};
			\draw (544,854) node [anchor=north west][inner sep=0.75pt]  [font=\tiny]  {$d_{x_j}^{m_{j}}$};
			\draw (87,830) node [anchor=north west][inner sep=0.75pt]  [font=\tiny]  {$k( m_{1}) \ \ k( m_{i})  k( m_{j})  k( m_{j+1})   k( m_{t}) \ $};
			\draw (80,935) node [anchor=north west][inner sep=0.75pt]  [font=\tiny]  {$k(-1)  k( m_{1})  k( m_{i}+1) \ \ \ \ \ \ \ \ \ \ k( m_{t}) \ $};
			
			\draw (367.37,830) node [anchor=north west][inner sep=0.75pt]  [font=\tiny]  {$k( m_{1})  \ \ \ k( m_{j})  k( m_{j+1})  \ k( m_{t}) \ $};
			\draw (355,935) node [anchor=north west][inner sep=0.75pt]  [font=\tiny]  {$k(-1) k( m_{1})  \ \ \ \ \ \ k( m_{j}+1)k( m_{t}) \ $};
			
			\draw (76,875) node [anchor=north west][inner sep=0.75pt]  [font=\large]  {$\sum\limits _{i< j} \ \ \ \ \ \  ...\ \ \ ...\ \ \ \ \ \ \ ...\ \  \otimes \ \ \ \ \ \ \ \ \ \ \ \ \ \ \ \ \ \ \ \ \ \ +\ \ \ \ \ \ \ \ ...\ \ \ \ \ \ \ \ \ ...\ \ \ \otimes \ \ \ \ \ \ \ \  \ \ \ \ \ \ \ \ \ \ \ \ \ \ +$};

		\end{tikzpicture}
	}
\end{equation*}
\begin{equation}\label{eq:Lemma A.1,side 1}
	\scalebox{0.9}{
		\tikzset{every picture/.style={line width=0.5pt}} 
		
		\begin{tikzpicture}[x=0.75pt,y=0.75pt,yscale=-1,xscale=1]
			
			\draw   (159.14,995.69) -- (168.7,995.69) -- (168.7,1000.37) -- (159.14,1000.37) -- cycle ;
			\draw    (165.8,1001.07) -- (165.74,1031.95) ;
			\draw    (161.87,1000.63) .. controls (161.87,1034.67) and (97.8,1003.16) .. (97.8,1035.87) ;
			\draw    (98.05,987.2) .. controls (98.05,979.57) and (96.26,1007.71) .. (101.32,1016.13) .. controls (106.37,1024.55) and (119.24,1021.44) .. (117.78,1035.79) ;
			\draw    (163.93,986.67) -- (163.92,995.68) ;
			\draw    (187.4,984.66) -- (187.14,1035.13) ;
			\draw    (97.8,1035.87) -- (98.2,1074.67) ;
			\draw    (165.74,1031.95) .. controls (165.74,1066.25) and (143.07,1042.61) .. (143.07,1075.56) ;
			\draw    (143,1031.56) .. controls (143,1065.86) and (165.55,1042.61) .. (165.55,1075.56) ;
			\draw    (187.14,1035.13) -- (187,1074.67) ;
			\draw    (117.78,1035.79) -- (118.2,1075.87) ;
			\draw    (133,986.27) .. controls (133,978.63) and (131.6,1010.27) .. (137,1016.77) .. controls (142.4,1023.27) and (142.4,1020.87) .. (143,1031.56) ;
			\draw   (462.22,990.96) -- (470.29,990.96) -- (470.29,995.48) -- (462.22,995.48) -- cycle ;
			\draw    (467.4,995.64) -- (468.43,1035.3) ;
			\draw    (464.52,995.73) .. controls (464.52,1028.64) and (365.8,998.78) .. (365.8,1030.4) ;
			\draw    (366.1,982.24) .. controls (366.1,974.82) and (364.4,1002.17) .. (369.21,1010.35) .. controls (374.02,1018.53) and (386.28,1015.51) .. (384.89,1029.45) ;
			\draw    (394.81,983.21) .. controls (394.81,975.79) and (393.11,1003.13) .. (397.92,1011.32) .. controls (402.74,1019.5) and (415.23,1015.5) .. (413.85,1029.44) ;
			\draw    (466.25,982.56) -- (466.25,990.95) ;
			\draw    (365.8,1030.4) -- (366.2,1070.8) ;
			\draw    (468.43,1035.3) -- (469.36,1073.41) ;
			\draw    (437.56,1029.81) .. controls (437.56,1062.68) and (414.1,1041.58) .. (414.1,1073.16) ;
			\draw    (413.85,1029.44) .. controls (413.85,1062.31) and (437.53,1041.32) .. (437.53,1072.9) ;
			\draw    (384.89,1029.45) -- (385.59,1072.12) ;
			\draw    (418.52,983.58) .. controls (418.52,976.16) and (416.82,1003.5) .. (421.63,1011.69) .. controls (426.44,1019.87) and (438.94,1015.87) .. (437.56,1029.81) ;
			\draw    (493.37,979.53) -- (493.45,982.49) -- (493.11,1024.62) ;
			\draw    (493.11,1024.62) -- (493.13,1069.26) ;
			\draw  [color={rgb, 255:red, 255; green, 255; blue, 255 }  ,draw opacity=1 ][line width=3] [line join = round][line cap = round] (249.56,985.37) .. controls (249.75,986.65) and (250.25,987.91) .. (250.78,989.13) ;
			\draw    (297.17,984.62) .. controls (293,998.77) and (302.5,1070.77) .. (298.15,1082.12) ;
			\draw    (205.89,984.87) .. controls (204.21,1006.62) and (204.6,1070.7) .. (211.46,1082.32) ;
			\draw  [color={rgb, 255:red, 128; green, 128; blue, 128 }  ,draw opacity=1 ] (213.53,1029.31) -- (292,1029.31) -- (292,1076.9) -- (213.53,1076.9) -- cycle ;
			\draw  [color={rgb, 255:red, 255; green, 255; blue, 255 }  ,draw opacity=1 ][line width=3] [line join = round][line cap = round] (263.75,1045.63) .. controls (263.75,1046.59) and (263.75,1047.55) .. (263.75,1048.5) ;
			\draw    (205.89,984.87) .. controls (206.27,974.61) and (297.39,975.02) .. (297.17,984.62) ;
			\draw [color={rgb, 255:red, 0; green, 0; blue, 0 }  ,draw opacity=1 ]   (205,1016.77) .. controls (211.07,1026.35) and (295.29,1026.9) .. (295.5,1016.27) ;
			\draw [color={rgb, 255:red, 155; green, 155; blue, 155 }  ,draw opacity=1 ] [dash pattern={on 0.84pt off 2.51pt}]  (206,1019.27) .. controls (207.46,1016.91) and (244.95,1017.77) .. (255.5,1017.77) .. controls (266.05,1017.77) and (289.93,1016.83) .. (298,1018.27) ;
			\draw [color={rgb, 255:red, 155; green, 155; blue, 155 }  ,draw opacity=1 ] [dash pattern={on 0.84pt off 2.51pt}]  (205.89,984.87) .. controls (209.08,989.91) and (235.43,988.64) .. (252.57,988.64) .. controls (269.71,988.64) and (296.87,988.77) .. (296.97,983.72) ;
			\draw    (211.46,1082.32) .. controls (218.2,1092.7) and (297.91,1092.75) .. (298.15,1082.12) ;
			\draw [color={rgb, 255:red, 0; green, 0; blue, 0 }  ,draw opacity=1 ]   (224.91,989.64) -- (224.93,1016.5) ;
			\draw [color={rgb, 255:red, 0; green, 0; blue, 0 }  ,draw opacity=1 ]   (220.63,989.6) -- (220.64,1016.47) ;
			\draw [color={rgb, 255:red, 0; green, 0; blue, 0 }  ,draw opacity=1 ]   (285.27,989.23) -- (285.29,1016.09) ;
			\draw [color={rgb, 255:red, 0; green, 0; blue, 0 }  ,draw opacity=1 ]   (281.39,989.06) -- (281.4,1015.92) ;
			\draw [color={rgb, 255:red, 155; green, 155; blue, 155 }  ,draw opacity=1 ] [dash pattern={on 0.84pt off 2.51pt}]  (213.66,1083.38) .. controls (214.6,1082.96) and (236.69,1082.55) .. (253.15,1082.55) .. controls (269.6,1082.55) and (288.13,1082.35) .. (297.34,1083.79) ;
			\draw  [color={rgb, 255:red, 0; green, 0; blue, 0 }  ,draw opacity=1 ][line width=3] [line join = round][line cap = round] (266.83,999.23) .. controls (267.07,1000.91) and (267.72,1002.54) .. (268.41,1004.14) ;
			\draw  [color={rgb, 255:red, 0; green, 0; blue, 0 }  ,draw opacity=1 ][fill={rgb, 255:red, 255; green, 255; blue, 255 }  ,fill opacity=1 ] (237,996.12) -- (271,996.12) -- (271,1012.58) -- (237,1012.58) -- cycle ;
			\draw [color={rgb, 255:red, 0; green, 0; blue, 0 }  ,draw opacity=1 ]   (252.57,987.64) -- (252.59,996.26) ;
			\draw [color={rgb, 255:red, 0; green, 0; blue, 0 }  ,draw opacity=1 ]   (252.78,1013.36) -- (252.79,1018.62) ;
			\draw  [color={rgb, 255:red, 255; green, 255; blue, 255 }  ,draw opacity=1 ][line width=3] [line join = round][line cap = round] (564.56,979.37) .. controls (564.75,980.65) and (565.25,981.91) .. (565.78,983.13) ;
			\draw    (612.17,978.62) .. controls (608,992.77) and (617.5,1064.77) .. (613.15,1076.12) ;
			\draw    (520.89,978.87) .. controls (519.21,1000.62) and (519.6,1064.7) .. (526.46,1076.32) ;
			\draw  [color={rgb, 255:red, 128; green, 128; blue, 128 }  ,draw opacity=1 ] (528.53,1023.31) -- (607,1023.31) -- (607,1070.9) -- (528.53,1070.9) -- cycle ;
			\draw  [color={rgb, 255:red, 255; green, 255; blue, 255 }  ,draw opacity=1 ][line width=3] [line join = round][line cap = round] (578.75,1039.63) .. controls (578.75,1040.59) and (578.75,1041.55) .. (578.75,1042.5) ;
			\draw    (520.89,978.87) .. controls (521.27,968.61) and (612.39,969.02) .. (612.17,978.62) ;
			\draw [color={rgb, 255:red, 0; green, 0; blue, 0 }  ,draw opacity=1 ]   (520,1010.77) .. controls (526.07,1020.35) and (610.29,1020.9) .. (610.5,1010.27) ;
			\draw [color={rgb, 255:red, 155; green, 155; blue, 155 }  ,draw opacity=1 ] [dash pattern={on 0.84pt off 2.51pt}]  (521,1013.27) .. controls (522.46,1010.91) and (559.95,1011.77) .. (570.5,1011.77) .. controls (581.05,1011.77) and (604.93,1010.83) .. (613,1012.27) ;
			\draw [color={rgb, 255:red, 155; green, 155; blue, 155 }  ,draw opacity=1 ] [dash pattern={on 0.84pt off 2.51pt}]  (520.89,978.87) .. controls (524.08,983.91) and (550.43,982.64) .. (567.57,982.64) .. controls (584.71,982.64) and (611.87,982.77) .. (611.97,977.72) ;
			\draw    (526.46,1076.32) .. controls (533.2,1086.7) and (612.91,1086.75) .. (613.15,1076.12) ;
			\draw [color={rgb, 255:red, 0; green, 0; blue, 0 }  ,draw opacity=1 ]   (539.91,983.64) -- (539.93,1010.5) ;
			\draw [color={rgb, 255:red, 0; green, 0; blue, 0 }  ,draw opacity=1 ]   (535.63,983.6) -- (535.64,1010.47) ;
			\draw [color={rgb, 255:red, 0; green, 0; blue, 0 }  ,draw opacity=1 ]   (600.27,983.23) -- (600.29,1010.09) ;
			\draw [color={rgb, 255:red, 0; green, 0; blue, 0 }  ,draw opacity=1 ]   (596.39,983.06) -- (596.4,1009.92) ;
			\draw [color={rgb, 255:red, 155; green, 155; blue, 155 }  ,draw opacity=1 ] [dash pattern={on 0.84pt off 2.51pt}]  (528.66,1077.38) .. controls (529.6,1076.96) and (551.69,1076.55) .. (568.15,1076.55) .. controls (584.6,1076.55) and (603.13,1076.35) .. (612.34,1077.79) ;
			\draw  [color={rgb, 255:red, 0; green, 0; blue, 0 }  ,draw opacity=1 ][line width=3] [line join = round][line cap = round] (581.83,993.23) .. controls (582.07,994.91) and (582.72,996.54) .. (583.41,998.14) ;
			\draw  [color={rgb, 255:red, 0; green, 0; blue, 0 }  ,draw opacity=1 ][fill={rgb, 255:red, 255; green, 255; blue, 255 }  ,fill opacity=1 ] (552,990.12) -- (586,990.12) -- (586,1006.58) -- (552,1006.58) -- cycle ;
			\draw [color={rgb, 255:red, 0; green, 0; blue, 0 }  ,draw opacity=1 ]   (567.57,981.64) -- (567.59,990.26) ;
			\draw [color={rgb, 255:red, 0; green, 0; blue, 0 }  ,draw opacity=1 ]   (567.78,1007.36) -- (567.79,1012.62) ;
			
			\draw (350,970) node [anchor=north west][inner sep=0.75pt]  [font=\tiny]  {$k( m_{1}) \  k( m_{j}) k( m_{j+1}) \ \ \ k( m_{i})  \ k( m_{t}) \ $};
			\draw (345,1075) node [anchor=north west][inner sep=0.75pt]  [font=\tiny]  {$\ k(-1) \ k( m_{1}) k( m_{j+1}) \  k( m_{i}+1)   k( m_{t}) \ $};
			\draw (85,1075) node [anchor=north west][inner sep=0.75pt]  [font=\tiny]  {$k(-1)\ \  \ \   k( m_{j+1}+1)k( m_{t}) \ $};
			\draw (80,970) node [anchor=north west][inner sep=0.75pt]  [font=\tiny]  {$k( m_{1}) \  k( m_{j}) k( m_{j+1}) \ k( m_{t}) \ $};
			\draw (82,1022) node [anchor=north west][inner sep=0.75pt]  [font=\large]  {$+\ \ \ \ \ \  ...\ \ \ \ \ \ \ ...\  \otimes \ \ \ \ \ \ \ \ \ \ \ \ \ \ \ \ \ \ \ \ \ +\sum\limits _{i >j+1} \ \ \ \ \ \ \  ...\ \ \ \ \ \ \ 
				\ \   ...\ \ \  ...\ \ \ \otimes\ \ \ \ \ \ \ \ \ \ \ \ \ \ \ \ \ \ \ \ \ .$};
			\draw (215,1030) node [anchor=north west][inner sep=0.75pt]  [color={rgb, 255:red, 128; green, 128; blue, 128 }  ,opacity=1 ]  {$T_{\msc{A} ,m+\delta_{j+1}}$};
			\draw (237,996) node [anchor=north west][inner sep=0.75pt]  [font=\tiny]  {$d_{x_{ j+1}}^{m_{j+1}}$};
			\draw (530,1025) node [anchor=north west][inner sep=0.75pt]  [color={rgb, 255:red, 128; green, 128; blue, 128 }  ,opacity=1 ]  {$T_{\msc{A} ,m+\delta_i}$};
			\draw (558,991) node [anchor=north west][inner sep=0.75pt]  [font=\tiny]  {$d_{x_i}^{m_{i}}$};
			
		\end{tikzpicture}
		
	}
\end{equation}

	We look first at the factors in \eqref{eq:Lemma A.1,side 1} ocurring in $\Bord_{\msc{A}}^{\red}$ (right-hand side of each term).  We use  implicitly that $M$ has a collar on its in and out boundaries, which allows us to think of the coupon labelled by $d_{x_i}^{m_i}$ as living in the in-boundary collar of $M$.
	Picture \eqref{eq:Lemma A.1, diff} shows how we can move the coupon labelled by $d_{x_i}^{m_i}$ through $M$: we surround the $i$-th string in $T_{\msc{A},m}$ by a small enough tubular neighborhood, passing the coupon through the string, and we obtain an equivalent bordism where the $d_{x_i}^{m_i}$-labelled coupon has been ``pushed" from the in-boundary of $M$ to the out-boundary of $M$. 
	
	On the other hand, for the factors in \eqref{eq:Lemma A.1,side 1} ocurring in $\Str_{\msc{S}}^{\red}$ (left-hand side of each term), we can use naturality of the braiding on $\Str_{\msc{S}}^{\red}$ to move the coupon from top to bottom. 	
	\begin{equation}\label{eq:Lemma A.1, diff}
		\scalebox{0.85}{
			\tikzset{every picture/.style={line width=0.5pt}} 
			
			\begin{tikzpicture}[x=0.75pt,y=0.75pt,yscale=-1,xscale=1]
				
				\draw 
			 (274.33,185) -- (274.25,197.45) ;
				\draw  
				(266.61,169.41) -- (282.24,169.41) -- (282.24,184) -- (266.61,184) -- cycle ;
				\draw 
				(275.44,144.29) -- (275.1,169) ;
				\draw
				  (237.14,37.09) .. controls (235.69,52.16) and (283.6,93.24) .. (275.44,144.29) ;
				\draw [color={rgb, 255:red, 74; green, 144; blue, 226 }  ,draw opacity=1 ](248.01,37.09) .. controls (247.17,45.75) and (260.83,62.68) .. (273.75,83.16) .. controls (286.66,103.64) and (288.02,124.45) .. (285.98,145.59) ;
				\draw  [color={rgb, 255:red, 74; green, 144; blue, 226 }  ,draw opacity=1 ]  (226.4,37.09) .. controls (224.94,52.16) and (272.05,95.19) .. (264.23,144.94) ;
				\draw  [color={rgb, 255:red, 74; green, 144; blue, 226 }  ,draw opacity=1 ]
			 (264.23,144.94) -- (261.6,198.88) ;
				\draw  [color={rgb, 255:red, 74; green, 144; blue, 226 }  ,draw opacity=1 ] 
				(285.98,145.59) -- (284.42,174.51) -- (283.09,198.94) ;
				\draw  [color={rgb, 255:red, 74; green, 144; blue, 226 }  ,draw opacity=1 ] (261.6,198.88) .. controls (261.61,196.93) and (266.42,195.36) .. (272.36,195.37) .. controls (278.29,195.39) and (283.1,196.98) .. (283.09,198.94) .. controls (283.09,200.89) and (278.27,202.46) .. (272.34,202.45) .. controls (266.4,202.43) and (261.59,200.84) .. (261.6,198.88) -- cycle ;
				\draw  [color={rgb, 255:red, 74; green, 144; blue, 226 }  ,draw opacity=1 ](226.4,37.09) .. controls (226.4,35.13) and (231.21,33.55) .. (237.14,33.55) .. controls (243.08,33.55) and (247.89,35.13) .. (247.89,37.09) .. controls (247.89,39.04) and (243.08,40.63) .. (237.14,40.63) .. controls (231.21,40.63) and (226.4,39.04) .. (226.4,37.09) -- cycle ;
				\draw  [color={rgb, 255:red, 255; green, 255; blue, 255 }  ,draw opacity=1 ][line width=3] [line join = round][line cap = round] (234.49,45.44) .. controls (234.49,46.31) and (234.49,47.18) .. (234.49,48.04) ;
				\draw  [color={rgb, 255:red, 255; green, 255; blue, 255 }  ,draw opacity=1 ][line width=3] [line join = round][line cap = round] (246.04,45.12) .. controls (246.04,45.88) and (246.04,46.63) .. (246.04,47.39) ;
				\draw  [color={rgb, 255:red, 255; green, 255; blue, 255 }  ,draw opacity=1 ][line width=3] [line join = round][line cap = round] (258.96,45.44) .. controls (258.96,46.63) and (258.96,47.83) .. (258.96,49.02) ;
				\draw  [color={rgb, 255:red, 255; green, 255; blue, 255 }  ,draw opacity=1 ][line width=3] [line join = round][line cap = round] (267.8,44.14) .. controls (267.8,45.62) and (269.49,46.62) .. (270.52,47.72) ;
				\draw    (142.43,35.59) .. controls (144.82,81.42) and (139.24,168.04) .. (148.76,192.84) ;
				\draw    (32.1,36.16) .. controls (30.07,85.98) and (26.88,176.4) .. (38.67,195.47) ;
				\draw  [color={rgb, 255:red, 128; green, 128; blue, 128 }  ,draw opacity=1 ] (39.43,78.99) -- (137.4,78.99) -- (137.4,188) -- (39.43,188) -- cycle ;
				\draw  [color={rgb, 255:red, 255; green, 255; blue, 255 }  ,draw opacity=1 ][line width=3] [line join = round][line cap = round] (102.03,175.34) .. controls (102.03,177.52) and (102.03,179.71) .. (102.03,181.9) ;
				\draw  [color={rgb, 255:red, 255; green, 255; blue, 255 }  ,draw opacity=1 ][line width=3] [line join = round][line cap = round] (105.62,208.16) .. controls (105.85,211.75) and (106.49,215.24) .. (107.15,218.67) ;
				\draw    (32.1,36.16) .. controls (32.56,12.67) and (142.7,13.61) .. (142.43,35.59) ;
				\draw    (38.67,195.47) .. controls (46,217.43) and (148.51,217.19) .. (148.76,192.84) ;
				\draw [color={rgb, 255:red, 155; green, 155; blue, 155 }  ,draw opacity=1 ] [dotted={on 0.84pt off 2.51pt}]  (38.67,195.47) .. controls (40.43,190.08) and (81.07,191.6) .. (93.82,191.6) .. controls (106.57,191.6) and (139.24,191.6) .. (149,194.89) ;
				\draw [color={rgb, 255:red, 155; green, 155; blue, 155 }  ,draw opacity=1 ] [dotted={on 0.84pt off 2.51pt}]  (32.1,36.16) .. controls (35.96,47.7) and (70.71,45.7) .. (91.43,45.7) .. controls (112.15,45.7) and (142.07,45.09) .. (142.19,33.54) ;
				\draw 
				  (57.67,71.71) -- (57.11,89.5) ;
				\draw  
				(49.07,56.81) -- (66.07,56.81) -- (66.07,71.31) -- (49.07,71.31) -- cycle ;
				\draw 
				 (57.68,37.68) -- (57.46,56.09) ;
				\draw 
				 (57.11,89.5) .. controls (55.66,104.57) and (105.26,146.56) .. (95.3,195.7) ;
				\draw [color={rgb, 255:red, 74; green, 144; blue, 226 }  ,draw opacity=1 ]   (67.52,89.5) .. controls (71.51,108.23) and (82.64,112.37) .. (94.47,132.21) .. controls (106.3,152.05) and (110.28,174.8) .. (106.05,195.7) ;
				\draw [color={rgb, 255:red, 74; green, 144; blue, 226 }  ,draw opacity=1 ]  (45.91,89.5) .. controls (48.73,117.01) and (94.51,146.56) .. (84.55,195.7) ;
				\draw  [color={rgb, 255:red, 74; green, 144; blue, 226 }  ,draw opacity=1 ]
				(68.07,69.26) -- (67.99,72.14) -- (67.52,89.5) ;
				\draw  [color={rgb, 255:red, 74; green, 144; blue, 226 }  ,draw opacity=1 ] 
				(46.46,69.26) -- (45.91,89.5) ;
				\draw [color={rgb, 255:red, 74; green, 144; blue, 226 }  ,draw opacity=1 ] 
				(47.29,36.24) -- (46.46,69.26) ;
				\draw  [color={rgb, 255:red, 74; green, 144; blue, 226 }  ,draw opacity=1 ]
				(68.45,37.02) -- (67.99,72.14) ;
				\draw  [color={rgb, 255:red, 74; green, 144; blue, 226 }  ,draw opacity=1 ]   (47.29,36.24) .. controls (47.29,34.28) and (52.11,32.7) .. (58.04,32.7) .. controls (63.98,32.7) and (68.79,34.28) .. (68.79,36.24) .. controls (68.79,38.19) and (63.98,39.78) .. (58.04,39.78) .. controls (52.11,39.78) and (47.29,38.19) .. (47.29,36.24) -- cycle ;
				\draw   [color={rgb, 255:red, 74; green, 144; blue, 226 }  ,draw opacity=1 ]  (84.55,195.7) .. controls (84.55,193.75) and (89.36,192.16) .. (95.3,192.16) .. controls (101.23,192.16) and (106.05,193.75) .. (106.05,195.7) .. controls (106.05,197.65) and (101.23,199.24) .. (95.3,199.24) .. controls (89.36,199.24) and (84.55,197.65) .. (84.55,195.7) -- cycle ;
				\draw    (303.97,34.92) .. controls (306.36,80.75) and (300.78,167.37) .. (310.3,192.18) ;
				\draw    (193.64,35.5) .. controls (191.61,85.31) and (188.42,175.73) .. (200.2,194.8) ;
				\draw  [color={rgb, 255:red, 128; green, 128; blue, 128 }  ,draw opacity=1 ] (199.63,48.99) -- (297.6,48.99) -- (297.6,158) -- (199.63,158) -- cycle ;
				\draw    (193.64,35.5) .. controls (194.09,12) and (304.24,12.95) .. (303.97,34.92) ;
				\draw    (200.2,194.8) .. controls (207.54,216.76) and (310.04,216.53) .. (310.3,192.18) ;
				\draw [color={rgb, 255:red, 155; green, 155; blue, 155 }  ,draw opacity=1 ] [dotted]  (200.2,194.8) .. controls (201.97,189.41) and (242.61,190.93) .. (255.36,190.93) .. controls (268.11,190.93) and (300.78,190.93) .. (310.54,194.23) ;
				\draw [color={rgb, 255:red, 155; green, 155; blue, 155 }  ,draw opacity=1 ] [dotted]  (193.64,35.5) .. controls (197.49,47.04) and (232.25,45.04) .. (252.97,45.04) .. controls (273.69,45.04) and (303.61,44.42) .. (303.73,32.87) ;
				
				\draw (160,100) node [anchor=north west][inner sep=0.75pt]    {$=$};
				\draw (270,172) node [anchor=north west][inner sep=0.75pt]  [font=\tiny]  {$d$};
				\draw (228.9,160) node [anchor=north west][inner sep=0.75pt]  [font=\scriptsize,color={rgb, 255:red, 128; green, 128; blue, 128 }  ,opacity=1 ]  {$x_i^{m_{i}}$};
				\draw (250,30) node [anchor=north west][inner sep=0.75pt]  [font=\scriptsize,color={rgb, 255:red, 128; green, 128; blue, 128 }  ,opacity=1 ]  {$x_i^{m_{i}}$};
				\draw (212,193) node [anchor=north west][inner sep=0.75pt]  [font=\scriptsize,color={rgb, 255:red, 128; green, 128; blue, 128 }  ,opacity=1 ]  {$x_i^{m_{i} +1}$};
				\draw (88.59,81.61) node [anchor=north west][inner sep=0.75pt]  [color={rgb, 255:red, 128; green, 128; blue, 128 }  ,opacity=1 ]  {$T_{\msc{A} ,m}$};
				\draw (70,31) node [anchor=north west][inner sep=0.75pt]  [font=\scriptsize,color={rgb, 255:red, 128; green, 128; blue, 128 }  ,opacity=1 ]  {$x_i^{m_{i}}$};
				\draw (68.66,65) node [anchor=north west][inner sep=0.75pt]  [font=\scriptsize,color={rgb, 255:red, 128; green, 128; blue, 128 }  ,opacity=1 ]  {$x_i^{m_{i} +1}$};
				\draw (53,60) node [anchor=north west][inner sep=0.75pt]  [font=\tiny]  {$d$};
				\draw (95.3,196) node [anchor=north west][inner sep=0.75pt]  [font=\scriptsize,color={rgb, 255:red, 128; green, 128; blue, 128 }  ,opacity=1 ]  {$x_i^{m_{i} +1}$};
				\draw (202.13,139.61) node [anchor=north west][inner sep=0.75pt]  [color={rgb, 255:red, 128; green, 128; blue, 128 }  ,opacity=1 ]  {$T_{\msc{A} ,m}$};
			\end{tikzpicture}
		}
	\end{equation}

It follows then that the sum \eqref{eq:Lemma A.1,side 1} is equal to 
\eqref{eq:Lemma A.1,side 2} below,
	\begin{equation*}
	\scalebox{0.9}{
		\tikzset{every picture/.style={line width=0.5pt}} 
		\begin{tikzpicture}[x=0.75pt,y=0.75pt,yscale=-1,xscale=1]
			
			\draw    (423.86,551.07) .. controls (423.86,584.75) and (400.62,564.94) .. (400.62,597.3) ;
			\draw    (401.01,551.07) .. controls (401.01,584.75) and (424.25,564.61) .. (424.25,596.97) ;
			\draw   (420.38,597.84) -- (428.26,597.84) -- (428.26,602.21) -- (420.38,602.21) -- cycle ;
			\draw    (426.33,602.47) -- (426.45,641.3) ;
			\draw    (423.25,602.8) .. controls (423.25,634.68) and (370.75,609.67) .. (370.75,640.3) ;
			\draw    (370.75,552.8) .. controls (370.08,589.22) and (369.14,612.44) .. (373.84,620.37) .. controls (378.54,628.3) and (391.85,628.46) .. (390.5,641.97) ;
			\draw    (400.5,612.43) .. controls (399.67,623.47) and (414.33,627.8) .. (414.67,641.8) ;
			\draw    (400.62,597.3) -- (400.5,612.43) ;
			\draw    (195.07,553.04) .. controls (195.07,588.83) and (169.43,567.78) .. (169.43,602.16) ;
			\draw    (169.86,553.04) .. controls (169.86,588.83) and (195.5,567.43) .. (195.5,601.81) ;
			\draw   (143.11,601.58) -- (150.68,601.58) -- (150.68,606.49) -- (143.11,606.49) -- cycle ;
			\draw    (148,607.25) -- (148.08,643.58) ;
			\draw    (145.27,606.77) .. controls (145.27,642.56) and (109,619) .. (110.5,642.5) ;
			\draw  [color={rgb, 255:red, 255; green, 255; blue, 255 }  ,draw opacity=1 ][line width=3] [line join = round][line cap = round] (204.74,654.32) .. controls (204.74,652.52) and (204.54,650.64) .. (203.44,648.98) ;
			\draw    (110,553) .. controls (109.35,593.88) and (109.29,616.16) .. (113.81,625.06) .. controls (118.32,633.96) and (130,627.25) .. (130,642.75) ;
			\draw    (195.5,601.81) -- (195.46,643.25) ;
			\draw    (145.49,553.1) -- (146.45,596.23) -- (146.64,601.31) ;
			\draw    (169.43,602.16) -- (169.82,643.46) ;
			\draw    (450,551.97) -- (449.67,642.47) ;
			\draw    (216,553) -- (216,642.5) ;
			\draw    (333.97,550.92) .. controls (335.94,570.93) and (327.08,637.59) .. (334.95,648.42) ;
			\draw [color={rgb, 255:red, 0; green, 0; blue, 0 }  ,draw opacity=1 ]   (242.69,551.17) .. controls (241.01,572.92) and (241.4,637) .. (248.26,648.62) ;
			\draw  [color={rgb, 255:red, 128; green, 128; blue, 128 }  ,draw opacity=1 ] (246.67,557.78) -- (327.71,557.78) -- (327.71,605.37) -- (246.67,605.37) -- cycle ;
			\draw  [color={rgb, 255:red, 255; green, 255; blue, 255 }  ,draw opacity=1 ][line width=3] [line join = round][line cap = round] (300.55,611.94) .. controls (300.55,612.89) and (300.55,613.85) .. (300.55,614.8) ;
			\draw    (242.69,551.17) .. controls (243.07,540.91) and (334.19,541.32) .. (333.97,550.92) ;
			\draw [color={rgb, 255:red, 0; green, 0; blue, 0 }  ,draw opacity=1 ]   (242.91,610.57) .. controls (248.98,620.16) and (330.91,620.9) .. (331.11,610.27) ;
			\draw [color={rgb, 255:red, 155; green, 155; blue, 155 }  ,draw opacity=1 ] [dash pattern={on 0.84pt off 2.51pt}]  (242.91,610.57) .. controls (244.37,608.22) and (282.71,608.61) .. (293.26,608.61) .. controls (303.81,608.61) and (323.04,608.83) .. (331.11,610.27) ;
			\draw [color={rgb, 255:red, 155; green, 155; blue, 155 }  ,draw opacity=1 ] [dash pattern={on 0.84pt off 2.51pt}]  (242.69,551.17) .. controls (245.88,556.21) and (274.63,555.34) .. (291.77,555.34) .. controls (308.91,555.34) and (333.67,555.07) .. (333.77,550.02) ;
			\draw  [color={rgb, 255:red, 0; green, 0; blue, 0 }  ,draw opacity=1 ][fill={rgb, 255:red, 255; green, 255; blue, 255 }  ,fill opacity=1 ] (274,627.56) -- (304.93,627.56) -- (304.93,642.23) -- (274,642.23) -- cycle ;
			\draw    (248.26,648.62) .. controls (255,659) and (334.71,659.06) .. (334.95,648.42) ;
			\draw [color={rgb, 255:red, 0; green, 0; blue, 0 }  ,draw opacity=1 ]   (289.82,619.5) -- (289.83,628.12) ;
			\draw [color={rgb, 255:red, 0; green, 0; blue, 0 }  ,draw opacity=1 ]   (289.64,643.01) -- (289.66,648.27) ;
			\draw [color={rgb, 255:red, 0; green, 0; blue, 0 }  ,draw opacity=1 ]   (261.05,619.61) -- (261.06,646.47) ;
			\draw [color={rgb, 255:red, 0; green, 0; blue, 0 }  ,draw opacity=1 ]   (256.76,619.57) -- (256.77,646.43) ;
			\draw [color={rgb, 255:red, 0; green, 0; blue, 0 }  ,draw opacity=1 ]   (321.41,619.2) -- (321.42,646.06) ;
			\draw [color={rgb, 255:red, 0; green, 0; blue, 0 }  ,draw opacity=1 ]   (317.52,619.02) -- (317.53,645.89) ;
			\draw [color={rgb, 255:red, 155; green, 155; blue, 155 }  ,draw opacity=1 ] [dash pattern={on 0.84pt off 2.51pt}]  (250.46,649.68) .. controls (251.4,649.27) and (273.49,648.85) .. (289.95,648.85) .. controls (306.4,648.85) and (324.93,648.65) .. (334.14,650.09) ;
			\draw    (567.76,549.32) .. controls (569.73,569.33) and (560.87,635.99) .. (568.74,646.82) ;
			\draw    (476.48,549.57) .. controls (474.8,571.32) and (475.19,635.4) .. (482.05,647.02) ;
			\draw  [color={rgb, 255:red, 128; green, 128; blue, 128 }  ,draw opacity=1 ] (480.45,556.18) -- (561.5,556.18) -- (561.5,603.77) -- (480.45,603.77) -- cycle ;
			\draw  [color={rgb, 255:red, 255; green, 255; blue, 255 }  ,draw opacity=1 ][line width=3] [line join = round][line cap = round] (534.33,610.34) .. controls (534.33,611.29) and (534.33,612.25) .. (534.33,613.2) ;
			\draw    (476.48,549.57) .. controls (476.86,539.31) and (567.98,539.72) .. (567.76,549.32) ;
			\draw [color={rgb, 255:red, 0; green, 0; blue, 0 }  ,draw opacity=1 ]   (476.7,608.97) .. controls (482.77,618.56) and (564.69,619.3) .. (564.9,608.67) ;
			\draw [color={rgb, 255:red, 155; green, 155; blue, 155 }  ,draw opacity=1 ] [dash pattern={on 0.84pt off 2.51pt}]  (476.7,608.97) .. controls (478.16,606.62) and (516.5,607.01) .. (527.05,607.01) .. controls (537.6,607.01) and (556.83,607.23) .. (564.9,608.67) ;
			\draw [color={rgb, 255:red, 155; green, 155; blue, 155 }  ,draw opacity=1 ] [dash pattern={on 0.84pt off 2.51pt}]  (476.48,549.57) .. controls (479.67,554.61) and (508.42,553.74) .. (525.56,553.74) .. controls (542.7,553.74) and (567.46,553.47) .. (567.56,548.42) ;
			\draw    (482.05,647.02) .. controls (488.79,657.4) and (568.5,657.46) .. (568.74,646.82) ;
			\draw [color={rgb, 255:red, 0; green, 0; blue, 0 }  ,draw opacity=1 ]   (494.84,618.01) -- (494.85,644.87) ;
			\draw [color={rgb, 255:red, 0; green, 0; blue, 0 }  ,draw opacity=1 ]   (490.55,617.97) -- (490.56,644.83) ;
			\draw [color={rgb, 255:red, 0; green, 0; blue, 0 }  ,draw opacity=1 ]   (555.2,617.6) -- (555.21,644.46) ;
			\draw [color={rgb, 255:red, 0; green, 0; blue, 0 }  ,draw opacity=1 ]   (551.31,617.42) -- (551.32,644.29) ;
			\draw [color={rgb, 255:red, 155; green, 155; blue, 155 }  ,draw opacity=1 ] [dash pattern={on 0.84pt off 2.51pt}]  (484.24,648.08) .. controls (485.18,647.67) and (507.28,647.25) .. (523.74,647.25) .. controls (540.19,647.25) and (558.72,647.05) .. (567.93,648.49) ;
			\draw  [color={rgb, 255:red, 0; green, 0; blue, 0 }  ,draw opacity=1 ][fill={rgb, 255:red, 255; green, 255; blue, 255 }  ,fill opacity=1 ] (509,625.56) -- (539.93,625.56) -- (539.93,640.23) -- (509,640.23) -- cycle ;
			\draw [color={rgb, 255:red, 0; green, 0; blue, 0 }  ,draw opacity=1 ]   (524.82,617.5) -- (524.83,626.12) ;
			\draw [color={rgb, 255:red, 0; green, 0; blue, 0 }  ,draw opacity=1 ]   (524.64,641.01) -- (524.66,646.27) ;
			
			\draw (513,625) node [anchor=north west][inner sep=0.75pt]  [font=\tiny]  {$d_{x_j}^{m_{j}}$};
			\draw (480,558) node [anchor=north west][inner sep=0.75pt]  [color={rgb, 255:red, 128; green, 128; blue, 128 }  ,opacity=1 ]  {$T_{\msc{A},m}$};
			
			\draw (278,628) node [anchor=north west][inner sep=0.75pt]  [font=\tiny]  {$d_{x_i}^{m_{i}}$};
			\draw (248,560) node [anchor=north west][inner sep=0.75pt]  [color={rgb, 255:red, 128; green, 128; blue, 128 }  ,opacity=1 ]  {$T_{\msc{A},m}$};
			
			\draw (94.67,539.58) node [anchor=north west][inner sep=0.75pt]  [font=\tiny]  {$k( m_{1}) \  k( m_{i})  k( m_{j})  k( m_{j+1})  k( m_{t}) \ $};
			\draw (85,644) node [anchor=north west][inner sep=0.75pt]  [font=\tiny]  {$k(-1)  k( m_{1})  k( m_{i}+1) \ \ \ \ \  \ k( m_{t}) \ $};
			\draw (82,590) node [anchor=north west][inner sep=0.75pt]  [font=\large]  {$\sum\limits _{i< j} \ \ \ ...\ \ \ \ \ ...\ \ \ \ \ \ \ ...\ \  \otimes \ \ \ \ \ \ \ \ \ \ \ \ \ \ \ \ \ \ \ \ \ +\ \ \ \ \ ...\ \ \ \ \ \ \ \ \  ... \ \  \otimes \ \ \ \ \ \ \ \ \ \ \  \ \ \ \ \  \ \ \ \ \ +$};
			\draw (355,540) node [anchor=north west][inner sep=0.75pt]  [font=\tiny]  {$k( m_{1})  k( m_{j}) k( m_{j+1}) k( m_{t}) \ $};
			\draw (350,642.5) node [anchor=north west][inner sep=0.75pt]  [font=\tiny]  {$k(-1)  k( m_{1}) \ \  k( m_{j}+1)k( m_{t})$};

		\end{tikzpicture}
	}
\end{equation*}
\begin{equation}\label{eq:Lemma A.1,side 2}
	\scalebox{0.9}{
		\tikzset{every picture/.style={line width=0.5pt}} 
		\begin{tikzpicture}[x=0.75pt,y=0.75pt,yscale=-1,xscale=1]
			
			\draw    (169.55,693.5) .. controls (169.55,731.26) and (143,709.05) .. (143,745.32) ;
			\draw    (143.45,693.5) .. controls (143.45,731.26) and (170,708.68) .. (170,744.95) ;
			\draw   (139.54,744.98) -- (147.61,744.98) -- (147.61,749.76) -- (139.54,749.76) -- cycle ;
			\draw    (145,749.5) -- (144.87,783.27) ;
			\draw    (141.84,750.03) .. controls (141.84,784.82) and (102.35,754.85) .. (103.5,783.5) ;
			\draw  [color={rgb, 255:red, 255; green, 255; blue, 255 }  ,draw opacity=1 ][line width=3] [line join = round][line cap = round] (180.56,793.11) .. controls (180.56,791.36) and (180.34,789.54) .. (179.17,787.93) ;
			\draw    (104.38,693.83) .. controls (103.69,733.57) and (102.8,754.49) .. (107.61,763.14) .. controls (112.41,771.8) and (124.88,769.76) .. (123.5,784.5) ;
			\draw    (170,744.95) -- (170.09,783.37) ;
			\draw    (423.8,693.24) .. controls (423.8,728.59) and (397.44,707.8) .. (397.44,741.76) ;
			\draw    (397.89,693.24) .. controls (397.89,728.59) and (424.25,707.45) .. (424.25,741.42) ;
			\draw   (450.27,740.49) -- (457.48,740.49) -- (457.48,745.35) -- (450.27,745.35) -- cycle ;
			\draw    (455.25,745.77) -- (455.48,783.97) ;
			\draw    (452.33,745.62) .. controls (452.33,780.98) and (368.47,753.54) .. (369.5,782.65) ;
			\draw  [color={rgb, 255:red, 255; green, 255; blue, 255 }  ,draw opacity=1 ][line width=3] [line join = round][line cap = round] (466.99,792.32) .. controls (466.99,790.55) and (466.8,788.69) .. (465.76,787.06) ;
			\draw    (370.34,692.77) .. controls (369.72,733.15) and (368.92,754.4) .. (373.22,763.2) .. controls (377.51,771.99) and (388.45,768.74) .. (387.21,783.71) ;
			\draw    (453.5,694.15) -- (453.45,735.21) -- (453.63,740.23) ;
			\draw    (397.44,741.76) -- (397.44,755.4) ;
			\draw    (424.25,741.53) -- (424.25,755.17) ;
			\draw    (189.5,694.5) -- (189.5,784) ;
			\draw    (397.44,755.4) .. controls (396.61,766.44) and (411.67,769.15) .. (412,783.15) ;
			\draw    (424.25,755.17) .. controls (423.41,766.2) and (438.1,769.65) .. (438.43,783.65) ;
			\draw    (476,693.07) -- (475.67,783.57) ;
			\draw    (315.17,691.1) .. controls (317.14,711.11) and (308.28,777.78) .. (316.15,788.61) ;
			\draw    (223.89,691.35) .. controls (222.21,713.1) and (222.6,777.18) .. (229.46,788.8) ;
			\draw  [color={rgb, 255:red, 128; green, 128; blue, 128 }  ,draw opacity=1 ] (227.87,697.96) -- (308.91,697.96) -- (308.91,745.55) -- (227.87,745.55) -- cycle ;
			\draw  [color={rgb, 255:red, 255; green, 255; blue, 255 }  ,draw opacity=1 ][line width=3] [line join = round][line cap = round] (281.75,752.12) .. controls (281.75,753.07) and (281.75,754.03) .. (281.75,754.99) ;
			\draw    (223.89,691.35) .. controls (224.27,681.1) and (315.39,681.51) .. (315.17,691.1) ;
			\draw [color={rgb, 255:red, 0; green, 0; blue, 0 }  ,draw opacity=1 ]   (224.11,750.75) .. controls (230.18,760.34) and (312.11,761.08) .. (312.31,750.45) ;
			\draw [color={rgb, 255:red, 155; green, 155; blue, 155 }  ,draw opacity=1 ] [dash pattern={on 0.84pt off 2.51pt}]  (224.11,750.75) .. controls (225.57,748.4) and (263.91,748.79) .. (274.46,748.79) .. controls (285.01,748.79) and (304.24,749.01) .. (312.31,750.45) ;
			\draw [color={rgb, 255:red, 155; green, 155; blue, 155 }  ,draw opacity=1 ] [dash pattern={on 0.84pt off 2.51pt}]  (223.89,691.35) .. controls (227.08,696.39) and (255.83,695.52) .. (272.97,695.52) .. controls (290.11,695.52) and (314.87,695.25) .. (314.97,690.21) ;
			\draw    (229.46,788.8) .. controls (236.2,799.18) and (315.91,799.24) .. (316.15,788.61) ;
			\draw [color={rgb, 255:red, 0; green, 0; blue, 0 }  ,draw opacity=1 ]   (242.25,759.79) -- (242.26,786.65) ;
			\draw [color={rgb, 255:red, 0; green, 0; blue, 0 }  ,draw opacity=1 ]   (237.96,759.75) -- (237.97,786.62) ;
			\draw [color={rgb, 255:red, 0; green, 0; blue, 0 }  ,draw opacity=1 ]   (302.61,759.38) -- (302.62,786.24) ;
			\draw [color={rgb, 255:red, 0; green, 0; blue, 0 }  ,draw opacity=1 ]   (298.72,759.21) -- (298.73,786.07) ;
			\draw [color={rgb, 255:red, 155; green, 155; blue, 155 }  ,draw opacity=1 ] [dash pattern={on 0.84pt off 2.51pt}]  (231.66,789.86) .. controls (232.6,789.45) and (254.69,789.03) .. (271.15,789.03) .. controls (287.6,789.03) and (306.13,788.84) .. (315.34,790.28) ;
			\draw    (604.16,692.7) .. controls (606.13,712.71) and (597.27,779.38) .. (605.14,790.21) ;
			\draw    (512.88,692.95) .. controls (511.2,714.7) and (511.59,778.78) .. (518.45,790.4) ;
			\draw  [color={rgb, 255:red, 128; green, 128; blue, 128 }  ,draw opacity=1 ] (516.85,699.56) -- (597.9,699.56) -- (597.9,747.15) -- (516.85,747.15) -- cycle ;
			\draw  [color={rgb, 255:red, 255; green, 255; blue, 255 }  ,draw opacity=1 ][line width=3] [line join = round][line cap = round] (570.73,753.72) .. controls (570.73,754.67) and (570.73,755.63) .. (570.73,756.59) ;
			\draw    (512.88,692.95) .. controls (513.26,682.7) and (604.38,683.11) .. (604.16,692.7) ;
			\draw [color={rgb, 255:red, 0; green, 0; blue, 0 }  ,draw opacity=1 ]   (513.1,752.35) .. controls (519.17,761.94) and (601.09,762.68) .. (601.3,752.05) ;
			\draw [color={rgb, 255:red, 155; green, 155; blue, 155 }  ,draw opacity=1 ] [dash pattern={on 0.84pt off 2.51pt}]  (513.1,752.35) .. controls (514.56,750) and (552.9,750.39) .. (563.45,750.39) .. controls (574,750.39) and (593.23,750.61) .. (601.3,752.05) ;
			\draw [color={rgb, 255:red, 155; green, 155; blue, 155 }  ,draw opacity=1 ] [dash pattern={on 0.84pt off 2.51pt}]  (512.88,692.95) .. controls (516.07,697.99) and (544.82,697.12) .. (561.96,697.12) .. controls (579.1,697.12) and (603.86,696.85) .. (603.96,691.81) ;
			\draw    (518.45,790.4) .. controls (525.19,800.78) and (604.9,800.84) .. (605.14,790.21) ;
			\draw [color={rgb, 255:red, 0; green, 0; blue, 0 }  ,draw opacity=1 ]   (531.24,761.39) -- (531.25,788.25) ;
			\draw [color={rgb, 255:red, 0; green, 0; blue, 0 }  ,draw opacity=1 ]   (526.95,761.35) -- (526.96,788.22) ;
			\draw [color={rgb, 255:red, 0; green, 0; blue, 0 }  ,draw opacity=1 ]   (591.6,760.98) -- (591.61,787.84) ;
			\draw [color={rgb, 255:red, 0; green, 0; blue, 0 }  ,draw opacity=1 ]   (587.71,760.81) -- (587.72,787.67) ;
			\draw [color={rgb, 255:red, 155; green, 155; blue, 155 }  ,draw opacity=1 ] [dash pattern={on 0.84pt off 2.51pt}]  (520.64,791.46) .. controls (521.58,791.05) and (543.68,790.63) .. (560.14,790.63) .. controls (576.59,790.63) and (595.12,790.44) .. (604.33,791.88) ;
			\draw  [color={rgb, 255:red, 0; green, 0; blue, 0 }  ,draw opacity=1 ][fill={rgb, 255:red, 255; green, 255; blue, 255 }  ,fill opacity=1 ] (255,768.06) -- (285.93,768.06) -- (285.93,782.73) -- (255,782.73) -- cycle ;
			\draw [color={rgb, 255:red, 0; green, 0; blue, 0 }  ,draw opacity=1 ]   (270.82,760) -- (270.83,768.62) ;
			\draw [color={rgb, 255:red, 0; green, 0; blue, 0 }  ,draw opacity=1 ]   (270.64,783.51) -- (270.66,788.77) ;
			\draw  [color={rgb, 255:red, 0; green, 0; blue, 0 }  ,draw opacity=1 ][fill={rgb, 255:red, 255; green, 255; blue, 255 }  ,fill opacity=1 ] (545,770.06) -- (575.93,770.06) -- (575.93,784.73) -- (545,784.73) -- cycle ;
			\draw [color={rgb, 255:red, 0; green, 0; blue, 0 }  ,draw opacity=1 ]   (560.82,762) -- (560.83,770.62) ;
			\draw [color={rgb, 255:red, 0; green, 0; blue, 0 }  ,draw opacity=1 ]   (560.64,785.51) -- (560.66,790.77) ;
			
			\draw (340,782.5) node [anchor=north west][inner sep=0.75pt]  [font=\tiny]  {$\ k(-1) k( m_{1})  k( m_{j+1})  k( m_{i}+1)  k( m_{t}) \ $};
			\draw (80,785.75) node [anchor=north west][inner sep=0.75pt]  [font=\tiny]  {$k(-1) k( m_{1}) k( m_{j+1}+1) k( m_{t})$};
			\draw (85,730) node [anchor=north west][inner sep=0.75pt]  [font=\large]  {$+\ \ \ \ ...\ \ \ \ \ \ \ \ \  ...\ \ \otimes \ \ \ \ \ \ \ \ \ \ \ \ \ \ \  \ \ \ \ \ \ \ +\sum\limits _{i >j+1} \ \ ...\ \ \ \ \ \ \ \ \  ...\  \ \   ... \ \ \ \otimes \ \ \ \ \ \ \ \ \ \ \ \ \ \ \ \ \ \ \ \ \ \ ,$};	
			\draw (255,767) node [anchor=north west][inner sep=0.75pt]  [font=\tiny]  {$d_{x_{j+1}}^{m_{j+1}}$};
			\draw (228.16,700) node [anchor=north west][inner sep=0.75pt]  [color={rgb, 255:red, 128; green, 128; blue, 128 }  ,opacity=1 ]  {$T_{\msc{A},m}$};
			\draw (550,772) node [anchor=north west][inner sep=0.75pt]  [font=\tiny]  {$d_{x_i}^{m_{i}}$};
			\draw (517,700) node [anchor=north west][inner sep=0.75pt]  [color={rgb, 255:red, 128; green, 128; blue, 128 }  ,opacity=1 ]  {$T_{\msc{A},m}$};
			\draw (350,679) node [anchor=north west][inner sep=0.75pt]  [font=\tiny]  {$k( m_{1}) k( m_{j})  k( m_{j+1})  k( m_{i})   k( m_{t}) \ $};
			
			\draw (90,682.5) node [anchor=north west][inner sep=0.75pt]  [font=\tiny]  {$k( m_{1})  k( m_{j}) k( m_{j+1}) \ k( m_{t}) \ $};
		\end{tikzpicture}
	}
\end{equation}
	for all $m:\{1,\dots, t\}\to \mathbb Z$. Note that the sum over all $m$ of \eqref{eq:Lemma A.1,side 2}  is $(p\boxtimes 1)\Delta(T,M)$ followed by $\tilde d_{\Sigma'_{\vec{y}^{n}}}$.
	Hence, applying the functor  $ev_{\opn{Vect}_{\mathbb Z}}\cdot Z$ to \eqref{eq:Lemma A.1,side 1} and \eqref{eq:Lemma A.1,side 2}, respectively,  we obtain the desired equality $Z^*(M,T)[1] \ d_{\Sigma_{\vec{x}}} = d_{\Sigma'_{\vec{y}}}\  Z^* (M,T)$ for the case when $\sigma$ is a transposition.
	
	 We now prove the general statement by induction on the length of $\sigma$. When the length of $\sigma$ is 1 the statement follows by the preceding argument.  Assume then that the statement is true 
	 for permutations of length less than or equal to $l$, and let $\sigma$ be a permutation of length $l+1$. We write $\sigma$ as a (minimal) composition of adjacent transpositions $\sigma=\sigma_{l+1}\dots \sigma_1$. Let $\sigma_{l+1}=(k,k+1)$ for some $k$ and consider the surface $\Sigma'_{\vec{z}}$ which is just $\Sigma'_{\vec{y}}$ with permuted indices $\vec{z}:\{1,\dots,t\}\to \Sigma'\times \opn{obj}\opn{Ch}(\msc{A})$, $\vec{z}=\vec{y}\circ \sigma_{l+1}$.  We take $(M_l, T_l):\Sigma'_{\vec{y}}\to \Sigma'_{\vec{z}}$ the bordism which, up to reindexing, is just the identity $\Sigma_{\vec{id}}$ . Note that $(M_l, T_l)$ is a permutation bordism with underlying permutation $\sigma_{l+1}$, as is its inverse $(M_l,T_l)^{-1}$.  It follows then that the composition $(M_l,T_l)(M,T)$ has underlying permutation of length is $l$. Hence, by induction, $Z^*((M_l,T_l)(M,T))[1] \ d_{\Sigma_{\vec{x}}} = d_{\Sigma'_{\vec{z}}}\  Z^*( (M_l, T_l)(M,T))$. Now, composing on both sides with $Z^*(M_l,T_l)^{-1},$ we obtain
	 \begin{align*}
	 	Z^*((M_l,T_l)^{-1}(M_l,T_l)(M,T))[1] \ d_{\Sigma_{\vec{x}}}& = Z^*(M_l,T_l)^{-1}[1] \ d_{\Sigma'_{\vec{z}}} \  Z^*( (M_l, T_l)(M,T))\\
	 	&= d_{\Sigma'_{\vec{z}}} \  Z^*((M_l,T_l)^{-1} (M_l, T_l)(M,T))
	 \end{align*} 
where in the second equality we are using the base case for induction, that is, that $Z^*(M_l, T_l)^{-1}$ commutes with the differential. Since $(M_l,T_l)^{-1}(M_l,T_l)(M,T)=(M,T)$, we conclude 
$Z^*(M,T)[1] \ d_{\Sigma_{\vec{x}}} = d_{\Sigma'_{\vec{y}}}\  Z^* (M,T)$, as desired. 
\end{proof}

\subsection{Commuting with a specific type of coupon}
We show now that if $(M,T)$ is a bordism in $\Ch(\msc{A})$ such that $T$ features a unique coupon, with a unique incoming ribbon which is attached to the incoming boundary, then $Z^*(M,T)\in \Ch(\opn{Vect}) $.

\begin{lemma}\label{lemma 2}	Let $(M,T):\Sigma_{\vec{x}}\to\Sigma'_{\vec{y}}$ be a connected bordism in  $\opn{Bord}_{\opn{Ch}(\msc{A})}$ such
	 that:
	 \begin{itemize}
	 	\item The surfaces $\Sigma, \Sigma'$ are nonempty and have respective labels
	 	\[
	 	\vec{x}=(x_1, \dots, x_t)\ \ \text{and}\ \ \vec{y}=(x_1, \dots, x_{j-1}, y_1, \dots, y_l, x_{j+1}, \dots, x_t).
	 	\]
	 	\item  The ribbon graph $T$ has a unique coupon with exactly one incoming ribbon, attached to the incoming boundary, and all outgoing ribbons attached to the outgoing boundary.
	 	\item The coupon is labelled by a function $f:x_j\to y_1\otimes\dots \otimes y_l$ in $\opn{Ch}(\msc{A})$.
	 	\item All other ribbons are couponless and travel from $x_i$ to $x_i$, for $i\ne j$. 
	 \end{itemize}
	Then 
	$Z^* (M,T)[1] \ d_{\Sigma_{\vec{x}}} = d_{\Sigma'_{\vec{y}}}\  Z^* (M,T).$
\end{lemma}

\begin{proof}
	We again assume that $\Sigma_{\vec{x}}$ and $\Sigma_{\vec{y}}$ are positively marked, via Lemma \ref{lem:d_pos}.  Let $(M,T) $ be as in the statement. We represent $(p\boxtimes 1)\Delta(M,T)$  by 
	\begin{equation}\label{fig:lemma 2, M}
		\tikzset{every picture/.style={line width=0.5pt}} 
		\begin{tikzpicture}[x=0.75pt,y=0.75pt,yscale=-1,xscale=1]
			
			\draw    (43.49,40.1) -- (43.75,109.6) ;
			\draw    (168.38,40.5) -- (169.56,114.15) ;
			\draw  [color={rgb, 255:red, 0; green, 0; blue, 0 }  ,draw opacity=1 ] (81.98,68.5) -- (131.64,68.5) -- (131.64,83) -- (81.98,83) -- cycle ;
			\draw    (105.62,39.9) -- (106.24,68.9) ;
			\draw    (64.37,110) .. controls (63.87,101.2) and (68.9,99.6) .. (76.45,96.8) .. controls (83.99,94) and (91.03,92.8) .. (89.52,83.2) ;
			\draw    (138.31,110.4) .. controls (137.8,101.6) and (130.76,98.4) .. (125.73,96.8) .. controls (120.7,95.2) and (111.65,91.6) .. (112.15,83.6) ;
			\draw    (73.93,110.8) .. controls (73.43,102) and (78.46,100.4) .. (86,97.6) .. controls (93.55,94.8) and (100.59,93.6) .. (99.08,84) ;
			\draw    (147.86,110) .. controls (147.36,101.2) and (140.32,98) .. (135.29,96.4) .. controls (130.26,94.8) and (121.21,91.2) .. (121.71,83.2) ;
			\draw    (331.6,34.87) .. controls (334,59) and (328.4,104.6) .. (337.95,117.66) ;
			\draw    (220.84,35.18) .. controls (218.8,61.4) and (215.6,109) .. (227.43,119.04) ;
			\draw  [color={rgb, 255:red, 128; green, 128; blue, 128 }  ,draw opacity=1 ] (227.61,47.97) -- (325.96,47.97) -- (325.96,105.35) -- (227.61,105.35) -- cycle ;
			\draw  [color={rgb, 255:red, 255; green, 255; blue, 255 }  ,draw opacity=1 ][line width=3] [line join = round][line cap = round] (291.05,108.44) .. controls (291.05,109.59) and (291.05,110.74) .. (291.05,111.9) ;
			\draw  [color={rgb, 255:red, 255; green, 255; blue, 255 }  ,draw opacity=1 ][line width=3] [line join = round][line cap = round] (294.65,125.72) .. controls (294.88,127.61) and (295.52,129.45) .. (296.19,131.25) ;
			\draw    (220.84,35.18) .. controls (221.3,22.81) and (331.88,23.31) .. (331.6,34.87) ;
			\draw    (227.43,119.04) .. controls (234.8,130.6) and (337.7,130.48) .. (337.95,117.66) ;
			\draw [color={rgb, 255:red, 155; green, 155; blue, 155 }  ,draw opacity=1 ] [dotted]  (227.43,119.04) .. controls (229.2,116.2) and (270,117) .. (282.8,117) .. controls (295.6,117) and (328.4,117) .. (338.2,118.74) ;
			\draw [color={rgb, 255:red, 155; green, 155; blue, 155 }  ,draw opacity=1 ] [dotted]  (220.84,35.18) .. controls (224.71,41.25) and (259.6,40.2) .. (280.4,40.2) .. controls (301.2,40.2) and (331.24,39.88) .. (331.36,33.79) ;
			
			\draw (242,60) node [anchor=north west][inner sep=0.75pt]  [color={rgb, 255:red, 128; green, 128; blue, 128 }  ,opacity=1 ]  {$T_{\msc{A}, m; n }$};
			\draw (345.17,106.65) node [anchor=north west][inner sep=0.75pt]  [color={rgb, 255:red, 0; green, 0; blue, 0 }  ,opacity=1 ]  {$\Sigma' _{\vec{y}^{m,n}}$};
			\draw (335,25) node [anchor=north west][inner sep=0.75pt]  [color={rgb, 255:red, 0; green, 0; blue, 0 }  ,opacity=1 ]  {$\Sigma _{\vec{x}^{m}}$};
			\draw (30,23.01) node [anchor=north west][inner sep=0.75pt]  [font=\tiny]  {$k( m_{1}) \ \ \ \ \ \ \ \ \ \ \ k( m_{j}) \ \ \ \ \ \ \ \ \ \ \ k( m_{t})$};
			\draw (30,115) node [anchor=north west][inner sep=0.75pt]  [font=\tiny]  {$k( m_{1}) k( n_{1}) \ \ \ \ \ \ \ \ \ \ \ \ \ \ \ \ \ k( n_{l}) \ k( m_{t}) $};
			\draw (97,95) node [anchor=north west][inner sep=0.75pt]    {$\cdots $};
			\draw (-35,64) node [anchor=north west][inner sep=0.75pt]    {$\sum\limits_{\substack{m:\{1, \dots, t\}\to \mathbb Z\\ n_1+\dots+n_l=m_j}} \ \ \ \ \cdots   \ \ \ \ \ \ \ \ \ \ \ \ \ \ \ \cdots \ \ \ \ \ \ \ \otimes $};
			
		\end{tikzpicture}
	\end{equation}
	where  $\vec{y}^{m,n}:=(x_1^{m_1}, \dots , x_{j-1}^{m_{j-1}}, y_1^{n_1},\dots, y_l^{n_l}, x_{j+1}^{m_{j+1}}, \dots, x_t^{m_t})$, and the multivalent coupon featured on the left side of the diagram in $\Str_{\opn{Vect}^{\mathbb Z}}^{\red}$ represents the identification $k(m_j)\xrightarrow{\sim} k(n_1)\otimes \dots \otimes k(n_l)$.
	
	Composing  $\tilde d_{\Sigma_{\vec{x}}}$ as in \eqref{eq: differential diagram} with the above, then projecting at fixed $m$ and $n_1+\dots +n_l=m_j$, we obtain
	\[\scalebox{.9}{
	\tikzset{every picture/.style={line width=0.5pt}} 

\begin{tikzpicture}[x=0.75pt,y=0.75pt,yscale=-1,xscale=1]

\draw    (312.35,49.52) .. controls (309.51,66.93) and (317.38,124.68) .. (315.64,135.84) ;
\draw    (231.91,50.43) .. controls (231.12,66.47) and (228.72,123.97) .. (234.7,136.1) ;
\draw  [draw opacity=0] (231.93,50.39) .. controls (235.13,45.93) and (252,42.55) .. (272.33,42.58) .. controls (291.32,42.6) and (307.29,45.59) .. (311.92,49.63) -- (272.33,51.98) -- cycle ; \draw   (231.93,50.39) .. controls (235.13,45.93) and (252,42.55) .. (272.33,42.58) .. controls (291.32,42.6) and (307.29,45.59) .. (311.92,49.63) ;  
\draw  [draw opacity=0][dash pattern={on 0.84pt off 2.51pt}] (231.96,50.48) .. controls (236.38,54.04) and (252.47,56.67) .. (271.64,56.64) .. controls (290.1,56.61) and (305.69,54.14) .. (310.77,50.77) -- (271.64,48.44) -- cycle ; \draw  [color={rgb, 255:red, 155; green, 155; blue, 155 }  ,draw opacity=1 ][dash pattern={on 0.84pt off 2.51pt}] (231.96,50.48) .. controls (236.38,54.04) and (252.47,56.67) .. (271.64,56.64) .. controls (290.1,56.61) and (305.69,54.14) .. (310.77,50.77) ;  
\draw  [draw opacity=0] (234.73,136.14) .. controls (237.93,140.58) and (254.8,143.95) .. (275.12,143.92) .. controls (295.92,143.89) and (313.09,140.32) .. (315.71,135.72) -- (275.12,134.55) -- cycle ; \draw   (234.73,136.14) .. controls (237.93,140.58) and (254.8,143.95) .. (275.12,143.92) .. controls (295.92,143.89) and (313.09,140.32) .. (315.71,135.72) ;  
\draw  [draw opacity=0][dash pattern={on 0.84pt off 2.51pt}] (234.72,135.28) .. controls (237.92,130.82) and (254.79,127.44) .. (275.12,127.47) .. controls (295.93,127.49) and (313.1,131.08) .. (315.71,135.7) -- (275.12,136.87) -- cycle ; \draw  [color={rgb, 255:red, 155; green, 155; blue, 155 }  ,draw opacity=1 ][dash pattern={on 0.84pt off 2.51pt}] (234.72,135.28) .. controls (237.92,130.82) and (254.79,127.44) .. (275.12,127.47) .. controls (295.93,127.49) and (313.1,131.08) .. (315.71,135.7) ;  
\draw  [draw opacity=0] (231.2,85.99) .. controls (231.08,86.16) and (231.02,86.33) .. (231.02,86.51) .. controls (231.03,89.56) and (249.73,92.01) .. (272.8,91.98) .. controls (291.34,91.95) and (307.05,90.33) .. (312.51,88.11) -- (272.8,86.45) -- cycle ; \draw  [color={rgb, 255:red, 155; green, 155; blue, 155 }  ,draw opacity=1 ] (231.2,85.99) .. controls (231.08,86.16) and (231.02,86.33) .. (231.02,86.51) .. controls (231.03,89.56) and (249.73,92.01) .. (272.8,91.98) .. controls (291.34,91.95) and (307.05,90.33) .. (312.51,88.11) ; 
\draw    (239.36,51.82) -- (239.36,83.56) ;
\draw    (243.36,51.82) -- (243.36,83.56) ;
\draw    (272.33,51.98) -- (272.8,86.45) ;
\draw    (304.92,51.09) -- (304.92,84.28) ;
\draw    (300.92,51.09) -- (300.92,84.28) ;
\draw  [fill={rgb, 255:red, 255; green, 255; blue, 255 }  ,fill opacity=1 ] (256.5,60.56) -- (288.62,60.56) -- (288.62,77.87) -- (256.5,77.87) -- cycle ;
\draw   [color={rgb, 255:red, 128; green, 128; blue, 128 }  ,draw opacity=1 ] (240.68,96.55) -- (308.2,96.55) -- (308.2,131.9) -- (240.68,131.9) -- cycle ;
\draw    (193.48,45.22) -- (193.48,139.63) ;
\draw    (145.62,45.08) -- (145.62,103.53) ;
\draw    (98.42,45.08) -- (98.42,139.5) ;
\draw    (57.12,45.84) .. controls (57.12,60.31) and (57.12,57.63) .. (57.12,75.61) .. controls (57.12,93.6) and (74.82,87.4) .. (74.82,110.35) .. controls (74.82,133.3) and (75,116.67) .. (75,139) ;
\draw    (94.48,61.97) .. controls (93.83,74.37) and (94.48,68.79) .. (93.17,81.82) .. controls (91.86,94.84) and (58.17,94.2) .. (57.58,114.1) .. controls (57,134) and (57,120.24) .. (57,139) ;
\draw  [fill={rgb, 255:red, 255; green, 255; blue, 255 }  ,fill opacity=1 ] (91.86,60.11) -- (101.04,60.11) -- (101.04,65.69) -- (91.86,65.69) -- cycle ;
\draw    (133.82,104.77) .. controls (133.82,126.48) and (113.5,122.13) .. (113.5,138.88) ;
\draw    (141.03,104.77) .. controls (141.03,126.48) and (120.71,122.13) .. (120.71,138.88) ;
\draw    (158.07,104.77) .. controls (158.07,126.48) and (178.4,122.13) .. (178.4,138.88) ;
\draw    (150.86,104.77) .. controls (150.86,126.48) and (171.19,122.13) .. (171.19,138.88) ;
\draw  [fill={rgb, 255:red, 255; green, 255; blue, 255 }  ,fill opacity=1 ] (129.56,97.94) -- (161.68,97.94) -- (161.68,109.11) -- (129.56,109.11) -- cycle ;
\draw    (588.17,48.29) .. controls (585.25,65.82) and (593.35,123.98) .. (591.56,135.21) ;
\draw    (505.31,49.21) .. controls (504.49,65.36) and (502.03,123.26) .. (508.19,135.47) ;
\draw  [draw opacity=0] (505.35,49.13) .. controls (508.71,44.66) and (526.06,41.27) .. (546.95,41.3) .. controls (566.43,41.33) and (582.82,44.31) .. (587.66,48.35) -- (546.94,50.77) -- cycle ; \draw   (505.35,49.13) .. controls (508.71,44.66) and (526.06,41.27) .. (546.95,41.3) .. controls (566.43,41.33) and (582.82,44.31) .. (587.66,48.35) ;  
\draw  [draw opacity=0][dash pattern={on 0.84pt off 2.51pt}] (505.42,49.3) .. controls (510.05,52.87) and (526.57,55.49) .. (546.24,55.46) .. controls (565.15,55.44) and (581.15,52.97) .. (586.47,49.6) -- (546.23,47.21) -- cycle ; \draw  [color={rgb, 255:red, 155; green, 155; blue, 155 }  ,draw opacity=1 ][dash pattern={on 0.84pt off 2.51pt}] (505.42,49.3) .. controls (510.05,52.87) and (526.57,55.49) .. (546.24,55.46) .. controls (565.15,55.44) and (581.15,52.97) .. (586.47,49.6) ;  
\draw  [draw opacity=0] (508.24,135.55) .. controls (511.6,140) and (528.94,143.38) .. (549.82,143.35) .. controls (571.2,143.32) and (588.86,139.73) .. (591.61,135.11) -- (549.82,133.91) -- cycle ; \draw   (508.24,135.55) .. controls (511.6,140) and (528.94,143.38) .. (549.82,143.35) .. controls (571.2,143.32) and (588.86,139.73) .. (591.61,135.11) ;  
\draw  [draw opacity=0][dash pattern={on 0.84pt off 2.51pt}] (508.23,134.61) .. controls (511.59,130.13) and (528.93,126.75) .. (549.82,126.78) .. controls (571.21,126.81) and (588.88,130.41) .. (591.62,135.04) -- (549.82,136.25) -- cycle ; \draw  [color={rgb, 255:red, 155; green, 155; blue, 155 }  ,draw opacity=1 ][dash pattern={on 0.84pt off 2.51pt}] (508.23,134.61) .. controls (511.59,130.13) and (528.93,126.75) .. (549.82,126.78) .. controls (571.21,126.81) and (588.88,130.41) .. (591.62,135.04) ;  
\draw  [draw opacity=0] (505.83,85.02) .. controls (505.71,85.19) and (505.65,85.36) .. (505.65,85.53) .. controls (505.66,88.59) and (524.36,91.04) .. (547.43,91.01) .. controls (565.97,90.98) and (581.68,89.36) .. (587.14,87.14) -- (547.43,85.48) -- cycle ; \draw  [color={rgb, 255:red, 155; green, 155; blue, 155 }  ,draw opacity=1 ] (505.83,85.02) .. controls (505.71,85.19) and (505.65,85.36) .. (505.65,85.53) .. controls (505.66,88.59) and (524.36,91.04) .. (547.43,91.01) .. controls (565.97,90.98) and (581.68,89.36) .. (587.14,87.14) ; 
\draw    (513.66,49.36) -- (513.66,81.32) ;
\draw    (517.66,49.36) -- (517.66,81.32) ;
\draw    (546.94,50.77) -- (547.43,85.48) ;
\draw    (579.84,48.63) -- (579.84,82.05) ;
\draw    (575.84,48.63) -- (575.84,82.05) ;
\draw  [fill={rgb, 255:red, 255; green, 255; blue, 255 }  ,fill opacity=1 ] (530.64,59.41) -- (563.73,59.41) -- (563.73,76.84) -- (530.64,76.84) -- cycle ;
\draw   [color={rgb, 255:red, 128; green, 128; blue, 128 }  ,draw opacity=1 ] (514.34,95.65) -- (583.9,95.65) -- (583.9,131.24) -- (514.34,131.24) -- cycle ;
\draw    (469.77,43.97) -- (469.77,139.03) ;
\draw    (420.47,43.82) -- (420.47,102.67) ;
\draw    (356.99,44.59) .. controls (356.99,64.58) and (356.99,53.33) .. (356.99,72.07) .. controls (356.99,90.81) and (375.22,89.56) .. (375.22,105.8) .. controls (375.22,122.03) and (375.22,115.79) .. (375.22,138.27) ;
\draw    (417.09,60.83) .. controls (417.09,78.94) and (410.34,83.94) .. (391.43,85.19) .. controls (372.52,86.44) and (358.26,91.69) .. (357.62,112.56) .. controls (356.99,133.43) and (357.62,121.02) .. (356.99,138.27) ;
\draw  [fill={rgb, 255:red, 255; green, 255; blue, 255 }  ,fill opacity=1 ] (414.39,57.08) -- (423.84,57.08) -- (423.84,62.08) -- (414.39,62.08) -- cycle ;
\draw    (408.31,103.92) .. controls (408.31,125.78) and (387.38,121.41) .. (387.38,138.27) ;
\draw    (415.74,103.92) .. controls (415.74,125.78) and (394.81,121.41) .. (394.81,138.27) ;
\draw    (433.3,103.92) .. controls (433.3,125.78) and (454.23,121.41) .. (454.23,138.27) ;
\draw    (425.87,103.92) .. controls (425.87,125.78) and (446.81,121.41) .. (446.81,138.27) ;
\draw  [fill={rgb, 255:red, 255; green, 255; blue, 255 }  ,fill opacity=1 ] (403.92,97.05) -- (437.01,97.05) -- (437.01,108.29) -- (403.92,108.29) -- cycle ;

\draw (258,62) node [anchor=north west][inner sep=0.75pt]  [font=\tiny] [align=left] {$\displaystyle d_{x_{i}}^{m_{i} -1}$};
\draw (251.94,103.7) node [anchor=north west][inner sep=0.75pt]  [color={rgb, 255:red, 155; green, 155; blue, 155 }  ,opacity=1 ] [align=left] {$\displaystyle T_{\mathscr{A} ,m;n}$};
\draw (77,95) node [anchor=north west][inner sep=0.75pt]   [align=left] {$\displaystyle \cdots $};
\draw (108,88) node [anchor=north west][inner sep=0.75pt]   [align=left] {$\displaystyle \cdots $};
\draw (165,88) node [anchor=north west][inner sep=0.75pt]   [align=left] {$\displaystyle \cdots $};
\draw (137,125) node [anchor=north west][inner sep=0.75pt]   [align=left] {$\displaystyle \cdots $};
\draw (20.52,69.6) node [anchor=north west][inner sep=0.75pt]  [font=\normalsize] [align=left] {$\displaystyle \sum _{i< j}$};
\draw (41,34.07) node [anchor=north west][inner sep=0.75pt]  [font=\tiny] [align=left] {$\displaystyle k( m_{1}) \ \ k( m_{i} -1) \ \ \ \ k( m_{j}) \ \ \ \ \ \ k( m_{t})$};
\draw (41,141.41) node [anchor=north west][inner sep=0.75pt]  [font=\tiny] [align=left] {$\displaystyle k( -1) \ \ \ \ \ k( m_{i}) \ \ \ \ \ \ \ \ k( n) \ \ \ \ \ \ \ k( m_{t})$};
\draw (201.18,86.46) node [anchor=north west][inner sep=0.75pt]   [align=left] {$\displaystyle \otimes $};
\draw (532,59) node [anchor=north west][inner sep=0.75pt]  [font=\tiny] [align=left] {$\displaystyle d_{x_{j}}^{m_{j} -1}$};
\draw (525.62,102.9) node [anchor=north west][inner sep=0.75pt]  [color={rgb, 255:red, 155; green, 155; blue, 155 }  ,opacity=1 ] [align=left] {$\displaystyle T_{\mathscr{A} ,m;n}$};
\draw (380,88) node [anchor=north west][inner sep=0.75pt]   [align=left] {$\displaystyle \cdots $};
\draw (440,88) node [anchor=north west][inner sep=0.75pt]   [align=left] {$\displaystyle \cdots $};
\draw (412,125) node [anchor=north west][inner sep=0.75pt]   [align=left] {$\displaystyle \cdots $};
\draw (329.71,88.05) node [anchor=north west][inner sep=0.75pt]   [align=left] {$\displaystyle +$};
\draw (342,33.64) node [anchor=north west][inner sep=0.75pt]  [font=\tiny] [align=left] {$\displaystyle k( m_{1}) \ \ \ \ \ \ \ \ k( m_{j} -1) \ \ \ \ \ k( m_{t})$};
\draw (342,140.46) node [anchor=north west][inner sep=0.75pt]  [font=\tiny] [align=left] {$\displaystyle k( -1) \ \ \ \ \ \ \ \ \ \ \ \ k( n) \ \ \ \ \ \ \ \ \ k( m_{t})$};
\draw (477.32,84.31) node [anchor=north west][inner sep=0.75pt]   [align=left] {$\displaystyle \otimes $};

\draw (610,88) node {$+$};
\end{tikzpicture}}
	\]
	\begin{equation}\label{eq:lemma 2, side 1}
	\scalebox{.9}{\tikzset{every picture/.style={line width=0.5pt}} 

\begin{tikzpicture}[x=0.75pt,y=0.75pt,yscale=-1,xscale=1]

\draw    (539.15,95.78) .. controls (536.21,113.28) and (544.36,171.33) .. (542.56,182.54) ;
\draw    (455.89,96.7) .. controls (455.06,112.82) and (452.59,170.61) .. (458.78,182.8) ;
\draw  [draw opacity=0] (455.94,96.62) .. controls (459.34,92.16) and (476.76,88.79) .. (497.73,88.81) .. controls (517.28,88.84) and (533.73,91.8) .. (538.62,95.82) -- (497.72,98.26) -- cycle ; \draw   (455.94,96.62) .. controls (459.34,92.16) and (476.76,88.79) .. (497.73,88.81) .. controls (517.28,88.84) and (533.73,91.8) .. (538.62,95.82) ;  
\draw  [draw opacity=0][dash pattern={on 0.84pt off 2.51pt}] (456.01,96.8) .. controls (460.69,100.36) and (477.27,102.97) .. (497.02,102.95) .. controls (516,102.92) and (532.06,100.47) .. (537.42,97.11) -- (497.01,94.7) -- cycle ; \draw  [color={rgb, 255:red, 155; green, 155; blue, 155 }  ,draw opacity=1 ][dash pattern={on 0.84pt off 2.51pt}] (456.01,96.8) .. controls (460.69,100.36) and (477.27,102.97) .. (497.02,102.95) .. controls (516,102.92) and (532.06,100.47) .. (537.42,97.11) ;  
\draw  [draw opacity=0] (458.83,182.88) .. controls (462.23,187.32) and (479.65,190.69) .. (500.62,190.66) .. controls (522.09,190.64) and (539.83,187.06) .. (542.61,182.45) -- (500.61,181.24) -- cycle ; \draw   (458.83,182.88) .. controls (462.23,187.32) and (479.65,190.69) .. (500.62,190.66) .. controls (522.09,190.64) and (539.83,187.06) .. (542.61,182.45) ;  
\draw  [draw opacity=0][dash pattern={on 0.84pt off 2.51pt}] (458.83,181.93) .. controls (462.23,177.48) and (479.64,174.1) .. (500.62,174.13) .. controls (522.1,174.15) and (539.84,177.73) .. (542.61,182.35) -- (500.61,183.57) -- cycle ; \draw  [color={rgb, 255:red, 155; green, 155; blue, 155 }  ,draw opacity=1 ][dash pattern={on 0.84pt off 2.51pt}] (458.83,181.93) .. controls (462.23,177.48) and (479.64,174.1) .. (500.62,174.13) .. controls (522.1,174.15) and (539.84,177.73) .. (542.61,182.35) ;  
\draw  [draw opacity=0] (456.61,132.44) .. controls (456.5,132.61) and (456.44,132.78) .. (456.44,132.95) .. controls (456.44,136.01) and (475.14,138.46) .. (498.22,138.43) .. controls (516.75,138.4) and (532.46,136.78) .. (537.93,134.56) -- (498.21,132.9) -- cycle ; \draw  [color={rgb, 255:red, 155; green, 155; blue, 155 }  ,draw opacity=1 ] (456.61,132.44) .. controls (456.5,132.61) and (456.44,132.78) .. (456.44,132.95) .. controls (456.44,136.01) and (475.14,138.46) .. (498.22,138.43) .. controls (516.75,138.4) and (532.46,136.78) .. (537.93,134.56) ;
\draw    (463.71,98.09) -- (463.71,130) ;
\draw    (467.71,98.09) -- (467.71,130) ;
\draw    (497.72,98.26) -- (498.21,132.9) ;
\draw    (530.71,97.37) -- (530.71,130.72) ;
\draw    (526.71,97.37) -- (526.71,130.72) ;
\draw  [fill={rgb, 255:red, 255; green, 255; blue, 255 }  ,fill opacity=1 ] (481.34,106.88) -- (514.59,106.88) -- (514.59,124.28) -- (481.34,124.28) -- cycle ;
\draw   [color={rgb, 255:red, 128; green, 128; blue, 128 }  ,draw opacity=1 ] (464.96,143.05) -- (534.86,143.05) -- (534.86,178.58) -- (464.96,178.58) -- cycle ;
\draw    (421.53,91.47) -- (421.53,186.35) ;
\draw    (388.95,91.47) -- (388.95,186.35) ;
\draw    (344,91) -- (343.48,150.06) ;
\draw    (279.69,92.09) .. controls (279.69,112.04) and (279.69,100.82) .. (279.69,119.52) .. controls (279.69,138.22) and (298.02,136.97) .. (298.02,153.18) .. controls (298.02,169.39) and (298.02,163.15) .. (298.02,185.59) ;
\draw    (386.92,110.79) .. controls (386.24,136.35) and (334.66,129.49) .. (315.66,130.74) .. controls (296.66,131.98) and (281.05,138.84) .. (281.05,158.16) .. controls (281.05,177.49) and (281,168.35) .. (281,186) ;
\draw  [fill={rgb, 255:red, 255; green, 255; blue, 255 }  ,fill opacity=1 ] (383.52,108.3) -- (392.35,108.3) -- (392.35,113.28) -- (383.52,113.28) -- cycle ;
\draw    (331.27,151.31) .. controls (331.27,173.13) and (310.23,168.76) .. (310.23,185.59) ;
\draw    (338.73,151.31) .. controls (338.73,173.13) and (317.7,168.76) .. (317.7,185.59) ;
\draw    (356.38,151.31) .. controls (356.38,173.13) and (377.42,168.76) .. (377.42,185.59) ;
\draw    (348.91,151.31) .. controls (348.91,173.13) and (369.95,168.76) .. (369.95,185.59) ;
\draw  [fill={rgb, 255:red, 255; green, 255; blue, 255 }  ,fill opacity=1 ] (326.86,144.45) -- (360.11,144.45) -- (360.11,155.67) -- (326.86,155.67) -- cycle ;

\draw (484,109) node [anchor=north west][inner sep=0.75pt]  [font=\tiny] [align=left] {$\displaystyle d_{x_{i}}^{m_{i} -1}$};
\draw (475.33,150.58) node [anchor=north west][inner sep=0.75pt]  [color={rgb, 255:red, 155; green, 155; blue, 155 }  ,opacity=1 ] [align=left] {$\displaystyle T_{\mathscr{A} ,m;n}$};
\draw (427.81,133) node [anchor=north west][inner sep=0.75pt]   [align=left] {$\displaystyle \otimes $};
\draw (310,133) node [anchor=north west][inner sep=0.75pt]   [align=left] {$\displaystyle \cdots $};
\draw (358,133) node [anchor=north west][inner sep=0.75pt]   [align=left] {$\displaystyle \cdots $};
\draw (334,172) node [anchor=north west][inner sep=0.75pt]   [align=left] {$\displaystyle \cdots $};
\draw (270.03,80.15) node [anchor=north west][inner sep=0.75pt]  [font=\tiny] [align=left] {$\displaystyle k( m_{1}) \ \ \ \ \ \ \ \ k( m_{j}) \ \ \ \ \  k( m_{i} -1)  k( m_{t})$};
\draw (265,189) node [anchor=north west][inner sep=0.75pt]  [font=\tiny] [align=left] {$\displaystyle k(-1) \ \ \ \ \ \ \ \ \ \ \ \ k( n) \ \ \ \ \ \ \ \ k( m_{i}) \ \ k( m_{t})$};
\draw (396,133) node [anchor=north west][inner sep=0.75pt]   [align=left] {$\displaystyle \cdots $};
\draw (239.21,117) node [anchor=north west][inner sep=0.75pt]  [font=\normalsize] [align=left] {$\displaystyle \sum _{i >j}$};
\draw (216.71,131.05) node [anchor=north west][inner sep=0.75pt]   [align=left] {$\displaystyle +$};

\end{tikzpicture}}
	\end{equation}
	We look first at the second term in \eqref{eq:lemma 2, side 1}. 
	For each $m_j$ let $f^{m_j}:=f|_{x_j^{m_j}} : x_j^{m_j} \to \oplus_{n_1+\dots+n_l=m_j} y_1^{n_1}\otimes\dots\otimes y_l^{n_l}$,  and for  $1\leq k \leq l$
	let $f_{n}^{m_j}$ denote the projection of $f^{m_j}$ to the summand $y_1^{n_1}\otimes \dots \otimes y_l^{n_l}$.
	Then we can picture the factor in  $\Bord_{\msc{A}}^{\red}$ appearing in the $k$-th summand as 
	\begin{equation*}
		\scalebox{0.8}{
			\tikzset{every picture/.style={line width=0.5pt}} 
			
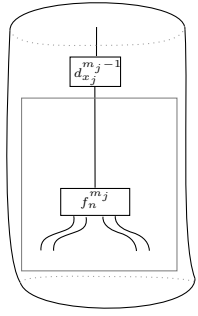
			\begin{tikzpicture}[x=0.75pt,y=0.75pt,yscale=-1,xscale=1]
				
				\draw 
				(75.62,990.1) -- (75.77,1053.76) ;
				\draw  
				 (60,971.33) -- (91.83,971.33) -- (91.83,989.91) -- (60,989.91) -- cycle ;
				\draw 
				   (76.7,952.4) -- (76.66,970.71) ;
				\draw  
				 (54.1,1054.26) -- (97.6,1054.26) -- (97.6,1071.26) -- (54.1,1071.26) -- cycle ;
				\draw 
				  (41.65,1093.34) .. controls (41.24,1086.4) and (45.55,1084.7) .. (51.69,1082.49) .. controls (57.82,1080.28) and (63.55,1079.33) .. (62.33,1071.76) ;
				\draw 
				   (101.86,1093.73) .. controls (101.45,1086.79) and (95.88,1083.75) .. (91.79,1082.49) .. controls (87.7,1081.23) and (80.33,1078.39) .. (80.74,1072.08) ;
				\draw
				   (49.64,1093.54) .. controls (49.23,1086.59) and (53.32,1085.33) .. (59.46,1083.12) .. controls (65.6,1080.91) and (71.33,1079.97) .. (70.1,1072.39) ;
				\draw 
				   (110,1093.54) .. controls (109.59,1086.59) and (103.66,1083.44) .. (99.56,1082.17) .. controls (95.47,1080.91) and (88.11,1078.07) .. (88.52,1071.76) ;
				\draw    (132.43,953.84) .. controls (134.82,999.67) and (129.24,1086.29) .. (138.76,1111.09) ;
				\draw    (22.1,954.41) .. controls (20.07,1004.23) and (16.88,1094.65) .. (28.67,1113.72) ;
				\draw  [color={rgb, 255:red, 128; green, 128; blue, 128 }  ,draw opacity=1 ] (29.43,997.24) -- (127.4,997.24) -- (127.4,1106.25) -- (29.43,1106.25) -- cycle ;
				\draw  [color={rgb, 255:red, 255; green, 255; blue, 255 }  ,draw opacity=1 ][line width=3] [line join = round][line cap = round] (92.03,1093.59) .. controls (92.03,1095.77) and (92.03,1097.96) .. (92.03,1100.15) ;
				\draw  [color={rgb, 255:red, 255; green, 255; blue, 255 }  ,draw opacity=1 ][line width=3] [line join = round][line cap = round] (95.62,1126.41) .. controls (95.85,1130) and (96.49,1133.49) .. (97.15,1136.92) ;
				\draw    (22.1,954.41) .. controls (22.56,930.92) and (132.7,931.86) .. (132.43,953.84) ;
				\draw    (28.67,1113.72) .. controls (36,1135.68) and (138.51,1135.44) .. (138.76,1111.09) ;
				\draw [color={rgb, 255:red, 155; green, 155; blue, 155 }  ,draw opacity=1 ] [dotted]  (28.67,1113.72) .. controls (30.43,1108.33) and (71.07,1109.85) .. (83.82,1109.85) .. controls (96.57,1109.85) and (129.24,1109.85) .. (139,1113.14) ;
				\draw [color={rgb, 255:red, 155; green, 155; blue, 155 }  ,draw opacity=1 ] [dotted]  (22.1,954.41) .. controls (25.96,965.95) and (60.71,963.95) .. (81.43,963.95) .. controls (102.15,963.95) and (132.07,963.34) .. (132.19,951.79) ;
				
				\draw (61,972) node [anchor=north west][inner sep=0.75pt]  [font=\tiny]  {$d_{x_j}^{m_{j}-1}$};
				\draw (65,1054.68) node [anchor=north west][inner sep=0.75pt]  [font=\tiny]  {$f_{n}^{m_j}$};
			\end{tikzpicture}
		}
	\end{equation*}
	Same as in the previous lemma, we move the coupon labelled by the differential in the picture above from the in-boundary collar to the out-boundary collar of $M$, see \eqref{eq:multivalent node}. We surround the $j$-th ribbon in $M$ by a small tubular neighborhood; sliding down the coupon gives us the second equality in \eqref{eq:multivalent node}. Since $f$ is a cochain map, it satisfies
	\[f^{m_j}_{n}d_{x_j}^{m_j-1}= \sum\limits_{k-1}^l (-1)^{n_1+\dots+n_{k-1}}
	(id\ot d_{y_k}^{n_k-1}\ot id) f^{m_j-1}_{n-\delta_k},\]
	and so by linearity of $k\Bord_{\msc{A}}^{\red}$ we get the third equality in \eqref{eq:multivalent node}.
	\begin{equation}\label{eq:multivalent node}
		\scalebox{0.9}{
			\tikzset{every picture/.style={line width=0.5pt}} 

		}
	\end{equation}
	On the other hand, on the first and last terms in \eqref{eq:lemma 2, side 1} we can move down the differential labelled coupon in $k\Bord_{\msc{A}}^{\red}$ same as in \eqref{eq:Lemma A.1, diff}. Then from the above discussion and string relations on $k\Str_{\msc{A}}^{\red}$ we obtain that Equation \eqref{eq:lemma 2, side 1}  is equal to
	\[
	\scalebox{.9}{\tikzset{every picture/.style={line width=0.5pt}} 

%
}
	\end{equation}
	
	If we compose \eqref{fig:lemma 2, M} with $\tilde d_{\Sigma'_{\vec{y}}}$ instead, and project at fixed $m$ and $n$ with $n_1+\dots +n_l=m_j$, we obtain the equation below
	\[
	\scalebox{.9}{\tikzset{every picture/.style={line width=0.5pt}} 

%
}
	\end{equation}
	Since the braiding isomorphism $k(m_r)\otimes k(-1)\xrightarrow{\sim} k(-1)\otimes k(m_r)$ in $\opn{Vect}^{\mathbb Z}$ contributes a factor of $(-1)^{m_r}$ at each crossing, we get that
	\begin{equation*}
		\scalebox{0.9}{\tikzset{every picture/.style={line width=0.5pt}} 

%
}
	\end{equation*}
	in $k\Str_{\opn{Vect}^{\mathbb Z}}^{\red}$. It follows that Equations \eqref{eq:lemma 2, side 1}=\eqref{eq:lemma 2, also side 1} and \eqref{eq:lemma 2, side 2} are equal. Thus the statement follows from applying the functor $ev_{\opn{Vect}^{\mathbb Z}}\cdot Z$ to the sum over all $m:\{1, \dots, t\} \to \mathbb Z$ and all $n_1+\dots +n_l=m_j$ of  \eqref{eq:lemma 2, side 1} and \eqref{eq:lemma 2, side 2}.
\end{proof}

\subsection{Commuting with cups and caps}
In this section we consider bordisms $(M,T)$ in $\Ch(\msc{A})$ such that $T$ has ``cups" and or ``caps", i.e., couponless ribbons attached on both ends to either the incoming or the outgoing boundary, and show that $Z^*(M,T)\in \Ch(\opn{Vect}) $.

\begin{lemma}\label{lemma 3}
Let $(M,T):\Sigma_{\vec{x}}\to\Sigma'_{\vec{y}}$ be a connected bordism in  $\opn{Bord}_{\opn{Ch}(\msc{A})}$ such
that $T$ consists (only) of couponless ribbons either traveling from $\Sigma$ to itself, or from $\Sigma$ to $\Sigma'$.  (So, $T$ consists of a collection of cups and a permutation.)  Then 
$Z^* (M,T)[1] \ d_{\Sigma_{\vec{x}}} = d_{\Sigma'_{\vec{y}}}\  Z^* (M,T).$ 
\end{lemma}

\begin{proof}
	By composing with a permutation if necessary, we may assume that $\Sigma_{\vec{x}}$ has marking set $I=\{1,\dots,t\}$, that $\Sigma_{\vec{y}}$ has marking set $\{1,\dots,t\}\setminus\{r_1,\dots,r_{2l}\}$, and that those ribbons which travel from $\Sigma$ to $\Sigma'$ travel from the $i$-th marking to the $i$-th marking.  Furthermore, by a variant of Lemma \ref{lem:d_pos}, we may assume that all of the markings in the subset $\{1,\dots,t\}\setminus\{r_1,\dots,r_{2l}\}\to \Sigma$ are positively framed and that all ribbons connecting points in the set $\{r_1,\dots, r_{2l}\}$ travel from a higher index to a lower index.
	\par
	
	 Now, we prove the statement for $T$ with a single cup; the proof for multiple cups instead is analogous but more cumbersome.  We represent $(p\boxtimes 1)\Delta(M,T)$ by 
	\begin{equation}
		\scalebox{0.9}{
			\tikzset{every picture/.style={line width=0.5pt}} 
			
			\begin{tikzpicture}[x=0.75pt,y=0.75pt,yscale=-1,xscale=1]
				
				\draw    (40.82,40.1) -- (41.09,109.6) ;
				\draw    (169.71,40.5) -- (170.89,114.15) ;
				\draw  [draw opacity=0] (140.39,41.04) .. controls (140.25,71.18) and (125.91,95.51) .. (107.92,95.75) .. controls (89.76,95.99) and (74.69,71.6) .. (74,41.1) -- (107.18,39.76) -- cycle ; \draw   (140.39,41.04) .. controls (140.25,71.18) and (125.91,95.51) .. (107.92,95.75) .. controls (89.76,95.99) and (74.69,71.6) .. (74,41.1) ;  
				\draw    (327.2,29.27) .. controls (329.6,53.4) and (324,99) .. (333.55,112.06) ;
				\draw    (216.44,29.58) .. controls (214.4,55.8) and (211.2,103.4) .. (223.03,113.44) ;
				\draw  [color={rgb, 255:red, 128; green, 128; blue, 128 }  ,draw opacity=1 ] (223.21,42.37) -- (321.56,42.37) -- (321.56,99.75) -- (223.21,99.75) -- cycle ;
				\draw  [color={rgb, 255:red, 255; green, 255; blue, 255 }  ,draw opacity=1 ][line width=3] [line join = round][line cap = round] (286.65,102.84) .. controls (286.65,103.99) and (286.65,105.14) .. (286.65,106.3) ;
				\draw  [color={rgb, 255:red, 255; green, 255; blue, 255 }  ,draw opacity=1 ][line width=3] [line join = round][line cap = round] (290.25,120.12) .. controls (290.48,122.01) and (291.12,123.85) .. (291.79,125.65) ;
				\draw    (216.44,29.58) .. controls (216.9,17.21) and (327.48,17.71) .. (327.2,29.27) ;
				\draw    (223.03,113.44) .. controls (230.4,125) and (333.3,124.88) .. (333.55,112.06) ;
				\draw [color={rgb, 255:red, 155; green, 155; blue, 155 }  ,draw opacity=1 ] [dotted]  (223.03,113.44) .. controls (224.8,110.6) and (265.6,111.4) .. (278.4,111.4) .. controls (291.2,111.4) and (324,111.4) .. (333.8,113.14) ;
				\draw [color={rgb, 255:red, 155; green, 155; blue, 155 }  ,draw opacity=1 ] [dotted]  (216.44,29.58) .. controls (220.31,35.65) and (255.2,34.6) .. (276,34.6) .. controls (296.8,34.6) and (326.84,34.28) .. (326.96,28.19) ;
				
				\draw (28,28) node [anchor=north west][inner sep=0.75pt]  [font=\tiny]  {$k( m_{1}) \ \ k( m_{j})   \  \ \ \ \ \ \ \ \ \ k(m_k) \ \ k( m_{t})$};
				\draw (30.83,112) node [anchor=north west][inner sep=0.75pt]  [font=\tiny]  {$k( m_{1}) \ \ \ \ \ \ \ \ \ \ \ \ \  \ \ \ \ \ \ \ \ \ \ \ \ \ \ \ \ k( m_{t})$};
				\draw (-30,60) node [anchor=north west][inner sep=0.75pt]    {$\sum\limits_{m:\{1, \dots, t\}\to \mathbb{Z}} \ \ \ \cdots \ \ \ \ \ \ \ \cdots \ \ \ \ \ \ \cdots \ \ \ \ \ \otimes $};
				\draw (225,45) node [anchor=north west][inner sep=0.75pt]  [color={rgb, 255:red, 128; green, 128; blue, 128 }  ,opacity=1 ]  {$T_{\msc{A} ,m}$};
				\draw (340,106.65) node [anchor=north west][inner sep=0.75pt]  [color={rgb, 255:red, 0; green, 0; blue, 0 }  ,opacity=1 ]  {$\Sigma' _{\vec{y}^{m}}$};
				\draw (335,25) node [anchor=north west][inner sep=0.75pt]  [color={rgb, 255:red, 0; green, 0; blue, 0 }  ,opacity=1 ]  {$\Sigma _{\vec{x}^{m}}$};
				\draw (340,40) node {,};
			\end{tikzpicture}
		}
	\end{equation}
where the sum is over all multi-indices with $m_j=m_k$.  Composing  $\tilde d^m_{\Sigma_{\vec{x}}}$  \eqref{eq: differential diagram} with the above for fixed $m$ we obtain
	\begin{equation*}
	\scalebox{0.9}{
		\tikzset{every picture/.style={line width=0.5pt}} 
		\begin{tikzpicture}[x=0.75pt,y=0.75pt,yscale=-1,xscale=1]
			
			\draw   (104.92,201.22) -- (111.58,201.22) -- (111.58,205.74) -- (104.92,205.74) -- cycle ;
			\draw    (106.82,206) .. controls (106.82,238.94) and (39.34,207.89) .. (39.34,239.54) ;
			\draw    (39.53,192.27) .. controls (39.53,184.85) and (38.13,212.23) .. (42.1,220.42) .. controls (46.08,228.62) and (56.2,225.59) .. (55.06,239.54) ;
			\draw    (108.22,190.02) -- (108.25,200.54) ;
			\draw    (109.44,205.4) -- (109.48,241.93) ;
			\draw    (39.34,239.54) -- (39.4,280) ;
			\draw    (55.06,239.54) -- (55.4,280.4) ;
			\draw    (109.48,241.93) -- (109.27,280.4) ;
			\draw  [color={rgb, 255:red, 255; green, 255; blue, 255 }  ,draw opacity=1 ][line width=3] [line join = round][line cap = round] (276.32,287.78) .. controls (276.5,289.06) and (277.01,290.32) .. (277.54,291.54) ;
			\draw    (158.73,190) -- (159.53,279.6) ;
			\draw   (341.64,202.72) -- (348.31,202.72) -- (348.31,207.24) -- (341.64,207.24) -- cycle ;
			\draw    (343.55,207.5) .. controls (343.55,240.44) and (312.73,208.72) .. (312.73,240.38) ;
			\draw    (310.92,193.11) .. controls (310.92,185.68) and (309.52,213.06) .. (313.5,221.25) .. controls (317.47,229.45) and (327.6,226.42) .. (326.45,240.38) ;
			\draw    (344.95,191.52) -- (344.98,202.04) ;
			\draw    (346.4,207.53) -- (346.5,217.55) ;
			\draw    (312.73,240.38) -- (312.79,280.83) ;
			\draw    (326.45,240.38) -- (326.79,281.23) ;
			\draw    (414.79,191.83) -- (415.59,281.43) ;
			\draw  [draw opacity=0] (139.06,219.97) .. controls (138.91,250.1) and (124.58,274.44) .. (106.58,274.68) .. controls (88.43,274.91) and (73.35,250.53) .. (72.67,220.03) -- (105.85,218.69) -- cycle ; \draw   (139.06,219.97) .. controls (138.91,250.1) and (124.58,274.44) .. (106.58,274.68) .. controls (88.43,274.91) and (73.35,250.53) .. (72.67,220.03) ;  
			\draw    (72.67,189.33) -- (72.67,220.03) ;
			\draw    (139.05,189.27) -- (139.06,219.97) ;
			\draw  [draw opacity=0] (394.39,214.73) .. controls (394.4,214.99) and (394.4,215.24) .. (394.41,215.5) .. controls (394.81,246.7) and (384.42,272.14) .. (371.18,272.31) .. controls (358.04,272.49) and (347.04,247.68) .. (346.49,216.78) -- (370.44,215.81) -- cycle ; \draw   (394.39,214.73) .. controls (394.4,214.99) and (394.4,215.24) .. (394.41,215.5) .. controls (394.81,246.7) and (384.42,272.14) .. (371.18,272.31) .. controls (358.04,272.49) and (347.04,247.68) .. (346.49,216.78) ;  
			\draw    (394.5,194.05) -- (394.39,214.73) ;
			\draw  [color={rgb, 255:red, 255; green, 255; blue, 255 }  ,draw opacity=1 ][line width=3] [line join = round][line cap = round] (224.23,185.91) .. controls (224.41,187.19) and (224.92,188.45) .. (225.45,189.67) ;
			\draw    (271.83,185.16) .. controls (267.67,199.31) and (277.17,271.31) .. (272.82,282.66) ;
			\draw    (180.56,185.41) .. controls (178.88,207.16) and (179.27,271.24) .. (186.13,282.86) ;
			\draw  [color={rgb, 255:red, 128; green, 128; blue, 128 }  ,draw opacity=1 ] (188.2,229.85) -- (266.67,229.85) -- (266.67,277.44) -- (188.2,277.44) -- cycle ;
			\draw  [color={rgb, 255:red, 255; green, 255; blue, 255 }  ,draw opacity=1 ][line width=3] [line join = round][line cap = round] (238.41,246.17) .. controls (238.41,247.13) and (238.41,248.09) .. (238.41,249.04) ;
			\draw    (180.56,185.41) .. controls (180.93,175.15) and (272.06,175.56) .. (271.83,185.16) ;
			\draw [color={rgb, 255:red, 0; green, 0; blue, 0 }  ,draw opacity=1 ]   (179.67,217.31) .. controls (185.74,226.89) and (269.96,227.44) .. (270.17,216.81) ;
			\draw [color={rgb, 255:red, 155; green, 155; blue, 155 }  ,draw opacity=1 ] [dash pattern={on 0.84pt off 2.51pt}]  (180.67,219.81) .. controls (182.12,217.45) and (219.62,218.31) .. (230.17,218.31) .. controls (240.71,218.31) and (264.59,217.36) .. (272.67,218.81) ;
			\draw [color={rgb, 255:red, 155; green, 155; blue, 155 }  ,draw opacity=1 ] [dash pattern={on 0.84pt off 2.51pt}]  (180.56,185.41) .. controls (183.75,190.45) and (210.1,189.18) .. (227.24,189.18) .. controls (244.38,189.18) and (271.53,189.31) .. (271.63,184.26) ;
			\draw    (186.13,282.86) .. controls (192.87,293.24) and (272.58,293.29) .. (272.82,282.66) ;
			\draw [color={rgb, 255:red, 0; green, 0; blue, 0 }  ,draw opacity=1 ]   (199.58,190.18) -- (199.59,217.04) ;
			\draw [color={rgb, 255:red, 0; green, 0; blue, 0 }  ,draw opacity=1 ]   (195.29,190.14) -- (195.31,217) ;
			\draw [color={rgb, 255:red, 0; green, 0; blue, 0 }  ,draw opacity=1 ]   (259.94,189.77) -- (259.95,216.63) ;
			\draw [color={rgb, 255:red, 0; green, 0; blue, 0 }  ,draw opacity=1 ]   (256.05,189.6) -- (256.07,216.46) ;
			\draw [color={rgb, 255:red, 155; green, 155; blue, 155 }  ,draw opacity=1 ] [dash pattern={on 0.84pt off 2.51pt}]  (188.32,283.92) .. controls (189.26,283.5) and (211.36,283.09) .. (227.82,283.09) .. controls (244.27,283.09) and (262.8,282.89) .. (272.01,284.33) ;
			\draw  [color={rgb, 255:red, 0; green, 0; blue, 0 }  ,draw opacity=1 ][line width=3] [line join = round][line cap = round] (241.49,199.77) .. controls (241.74,201.45) and (242.39,203.08) .. (243.08,204.68) ;
			\draw  [color={rgb, 255:red, 0; green, 0; blue, 0 }  ,draw opacity=1 ][fill={rgb, 255:red, 255; green, 255; blue, 255 }  ,fill opacity=1 ] (211.67,196.66) -- (245.67,196.66) -- (245.67,213.12) -- (211.67,213.12) -- cycle ;
			\draw [color={rgb, 255:red, 0; green, 0; blue, 0 }  ,draw opacity=1 ]   (227.24,188.18) -- (227.25,196.8) ;
			\draw [color={rgb, 255:red, 0; green, 0; blue, 0 }  ,draw opacity=1 ]   (227.44,213.9) -- (227.46,219.16) ;
			\draw  [color={rgb, 255:red, 255; green, 255; blue, 255 }  ,draw opacity=1 ][line width=3] [line join = round][line cap = round] (485.23,188.91) .. controls (485.41,190.19) and (485.92,191.45) .. (486.45,192.67) ;
			\draw    (532.83,188.16) .. controls (528.67,202.31) and (538.17,274.31) .. (533.82,285.66) ;
			\draw    (441.56,188.41) .. controls (439.88,210.16) and (440.27,274.24) .. (447.13,285.86) ;
			\draw  [color={rgb, 255:red, 128; green, 128; blue, 128 }  ,draw opacity=1 ] (449.2,232.85) -- (527.67,232.85) -- (527.67,280.44) -- (449.2,280.44) -- cycle ;
			\draw  [color={rgb, 255:red, 255; green, 255; blue, 255 }  ,draw opacity=1 ][line width=3] [line join = round][line cap = round] (499.41,249.17) .. controls (499.41,250.13) and (499.41,251.09) .. (499.41,252.04) ;
			\draw    (441.56,188.41) .. controls (441.93,178.15) and (533.06,178.56) .. (532.83,188.16) ;
			\draw [color={rgb, 255:red, 0; green, 0; blue, 0 }  ,draw opacity=1 ]   (440.67,220.31) .. controls (446.74,229.89) and (530.96,230.44) .. (531.17,219.81) ;
			\draw [color={rgb, 255:red, 155; green, 155; blue, 155 }  ,draw opacity=1 ] [dash pattern={on 0.84pt off 2.51pt}]  (441.67,222.81) .. controls (443.12,220.45) and (480.62,221.31) .. (491.17,221.31) .. controls (501.71,221.31) and (525.59,220.36) .. (533.67,221.81) ;
			\draw [color={rgb, 255:red, 155; green, 155; blue, 155 }  ,draw opacity=1 ] [dash pattern={on 0.84pt off 2.51pt}]  (441.56,188.41) .. controls (444.75,193.45) and (471.1,192.18) .. (488.24,192.18) .. controls (505.38,192.18) and (532.53,192.31) .. (532.63,187.26) ;
			\draw    (447.13,285.86) .. controls (453.87,296.24) and (533.58,296.29) .. (533.82,285.66) ;
			\draw [color={rgb, 255:red, 0; green, 0; blue, 0 }  ,draw opacity=1 ]   (460.58,193.18) -- (460.59,220.04) ;
			\draw [color={rgb, 255:red, 0; green, 0; blue, 0 }  ,draw opacity=1 ]   (456.29,193.14) -- (456.31,220) ;
			\draw [color={rgb, 255:red, 0; green, 0; blue, 0 }  ,draw opacity=1 ]   (520.94,192.77) -- (520.95,219.63) ;
			\draw [color={rgb, 255:red, 0; green, 0; blue, 0 }  ,draw opacity=1 ]   (517.05,192.6) -- (517.07,219.46) ;
			\draw [color={rgb, 255:red, 155; green, 155; blue, 155 }  ,draw opacity=1 ] [dash pattern={on 0.84pt off 2.51pt}]  (449.32,286.92) .. controls (450.26,286.5) and (472.36,286.09) .. (488.82,286.09) .. controls (505.27,286.09) and (523.8,285.89) .. (533.01,287.33) ;
			\draw  [color={rgb, 255:red, 0; green, 0; blue, 0 }  ,draw opacity=1 ][line width=3] [line join = round][line cap = round] (502.49,202.77) .. controls (502.74,204.45) and (503.39,206.08) .. (504.08,207.68) ;
			\draw  [color={rgb, 255:red, 0; green, 0; blue, 0 }  ,draw opacity=1 ][fill={rgb, 255:red, 255; green, 255; blue, 255 }  ,fill opacity=1 ] (472.67,196) -- (506.67,196) -- (506.67,216.12) -- (472.67,216.12) -- cycle ;
			\draw [color={rgb, 255:red, 0; green, 0; blue, 0 }  ,draw opacity=1 ]   (488.24,189) -- (488.25,196) ;
			\draw [color={rgb, 255:red, 0; green, 0; blue, 0 }  ,draw opacity=1 ]   (488.44,216.9) -- (488.46,222.16) ;
			
			\draw (30,177.93) node [anchor=north west][inner sep=0.75pt]  [font=\tiny]  {$k( m_{1}) \ \ k( m_{j}) \ \ \  k( m_{i}) \ \ \ \ \ \  k( m_{t}) \ $};
			\draw (20,281.43) node [anchor=north west][inner sep=0.75pt]  [font=\tiny]  {$k(-1) \ k( m_{1})  \ \ \ \ \ k( m_{i} +1) \ \ \ \ k( m_{t}) \ $};
			\draw (305.02,178.77) node [anchor=north west][inner sep=0.75pt]  [font=\tiny]  {$k( m_{1}) \ \ k(m_j)    \ \ \ \ \ \ \ \ \ \  \ k( m_{t})$};
			\draw (295,282.27) node [anchor=north west][inner sep=0.75pt]  [font=\tiny]  {$k(-1) \ k( m_{1}) \ \ \ \  \ \ \ \ \ \ \ \ \ \ \ k( m_{t}) \ $};
			\draw (10,225) node [anchor=north west][inner sep=0.75pt]  [font=\normalsize]  {$\sum\limits_{i\ne j, k} \ \ \ \ \cdots  \ \ \cdots \ \ \ \ \cdots \   \cdots\  \otimes \ \ \ \ \ \ \ \ \ \ \ \ \ \ \ \ \ \ \ \ \ \ \ \ \  +\ \ \ \ \ \ \cdots \ \ \ \   \cdots \ \  \ \cdots \ \ \otimes \ \ \ $};
			
			\draw (189,233) node [anchor=north west][inner sep=0.75pt]  [color={rgb, 255:red, 128; green, 128; blue, 128 }  ,opacity=1 ]  {$T_{\msc{A} ,m+\delta_i}$};
			\draw (215,197.29) node [anchor=north west][inner sep=0.75pt]  [font=\tiny]  {$d_{x_i}^{m_{i}}$};
			\draw (450,237) node [anchor=north west][inner sep=0.75pt]  [color={rgb, 255:red, 128; green, 128; blue, 128 }  ,opacity=1 ]  {$T_{\msc{A} ,m-\delta_j}$};
			\draw (477,195.5) node [anchor=north west][inner sep=0.75pt]  [font=\tiny]  {$d_{x_j^{\ast}}^{-m_{j}}$};
			
		\end{tikzpicture}
	}
\end{equation*}
\begin{equation}\label{eq:lemma 3, side 1}
	\scalebox{0.9}{
		\tikzset{every picture/.style={line width=0.5pt}} 
		
		\begin{tikzpicture}[x=0.75pt,y=0.75pt,yscale=-1,xscale=1]
			
			\draw   (261.04,339.38) -- (267.71,339.38) -- (267.71,343.9) -- (261.04,343.9) -- cycle ;
			\draw    (262.95,344.16) .. controls (262.95,377.11) and (184.13,347.39) .. (184.13,379.04) ;
			\draw    (182.32,331.77) .. controls (182.32,324.35) and (180.92,351.73) .. (184.9,359.92) .. controls (188.87,368.12) and (199,365.09) .. (197.85,379.04) ;
			\draw    (264.35,328.18) -- (264.38,338.7) ;
			\draw    (265.8,344.2) -- (265.9,354.21) ;
			\draw    (184.13,379.04) -- (184.19,419.5) ;
			\draw    (197.85,379.04) -- (198.19,419.9) ;
			\draw    (286.19,330.5) -- (286.99,420.1) ;
			\draw  [draw opacity=0] (265.79,353.4) .. controls (265.8,353.66) and (265.8,353.91) .. (265.81,354.17) .. controls (266.21,385.37) and (255.82,410.81) .. (242.58,410.98) .. controls (229.44,411.15) and (218.44,386.35) .. (217.89,355.45) -- (241.84,354.48) -- cycle ; \draw   (265.79,353.4) .. controls (265.8,353.66) and (265.8,353.91) .. (265.81,354.17) .. controls (266.21,385.37) and (255.82,410.81) .. (242.58,410.98) .. controls (229.44,411.15) and (218.44,386.35) .. (217.89,355.45) ;  
			\draw    (218.33,329.33) -- (217.89,355.45) ;
			\draw  [color={rgb, 255:red, 255; green, 255; blue, 255 }  ,draw opacity=1 ][line width=3] [line join = round][line cap = round] (358.56,326.24) .. controls (358.75,327.53) and (359.25,328.78) .. (359.78,330.01) ;
			\draw [color={rgb, 255:red, 0; green, 0; blue, 0 }  ,draw opacity=1 ]   (406.17,325.49) .. controls (402,339.64) and (411.5,411.64) .. (407.15,423) ;
			\draw [color={rgb, 255:red, 0; green, 0; blue, 0 }  ,draw opacity=1 ]   (314.89,325.74) .. controls (313.21,347.49) and (313.6,411.57) .. (320.46,423.19) ;
			\draw  [color={rgb, 255:red, 128; green, 128; blue, 128 }  ,draw opacity=1 ] (322.53,370.18) -- (401,370.18) -- (401,417.78) -- (322.53,417.78) -- cycle ;
			\draw  [color={rgb, 255:red, 255; green, 255; blue, 255 }  ,draw opacity=1 ][line width=3] [line join = round][line cap = round] (372.75,386.51) .. controls (372.75,387.46) and (372.75,388.42) .. (372.75,389.37) ;
			\draw    (314.89,325.74) .. controls (315.27,315.48) and (406.39,315.9) .. (406.17,325.49) ;
			\draw [color={rgb, 255:red, 0; green, 0; blue, 0 }  ,draw opacity=1 ]   (314,357.64) .. controls (320.07,367.23) and (404.29,367.77) .. (404.5,357.14) ;
			\draw [color={rgb, 255:red, 155; green, 155; blue, 155 }  ,draw opacity=1 ] [dash pattern={on 0.84pt off 2.51pt}]  (315,360.14) .. controls (316.46,357.78) and (353.95,358.64) .. (364.5,358.64) .. controls (375.05,358.64) and (398.93,357.7) .. (407,359.14) ;
			\draw [color={rgb, 255:red, 155; green, 155; blue, 155 }  ,draw opacity=1 ] [dash pattern={on 0.84pt off 2.51pt}]  (314.89,325.74) .. controls (318.08,330.78) and (344.43,329.51) .. (361.57,329.51) .. controls (378.71,329.51) and (405.87,329.64) .. (405.97,324.6) ;
			\draw    (320.46,423.19) .. controls (327.2,433.57) and (406.91,433.63) .. (407.15,423) ;
			\draw [color={rgb, 255:red, 0; green, 0; blue, 0 }  ,draw opacity=1 ]   (323.25,330.51) -- (323.26,357.37) ;
			\draw [color={rgb, 255:red, 0; green, 0; blue, 0 }  ,draw opacity=1 ]   (318.96,330.48) -- (318.97,357.34) ;
			\draw [color={rgb, 255:red, 0; green, 0; blue, 0 }  ,draw opacity=1 ]   (400.27,330.1) -- (400.29,356.97) ;
			\draw [color={rgb, 255:red, 0; green, 0; blue, 0 }  ,draw opacity=1 ]   (394.27,330.1) -- (394.29,356.97) ;
			\draw [color={rgb, 255:red, 155; green, 155; blue, 155 }  ,draw opacity=1 ] [dash pattern={on 0.84pt off 2.51pt}]  (322.66,424.25) .. controls (323.6,423.84) and (345.69,423.42) .. (362.15,423.42) .. controls (378.6,423.42) and (397.13,423.23) .. (406.34,424.67) ;
			\draw  [color={rgb, 255:red, 0; green, 0; blue, 0 }  ,draw opacity=1 ][line width=3] [line join = round][line cap = round] (370.83,340.1) .. controls (371.07,341.78) and (371.72,343.41) .. (372.41,345.01) ;
			\draw  [color={rgb, 255:red, 0; green, 0; blue, 0 }  ,draw opacity=1 ][fill={rgb, 255:red, 255; green, 255; blue, 255 }  ,fill opacity=1 ] (345.33,335.08) -- (377.33,335.08) -- (377.33,354) -- (345.33,354) -- cycle ;
			\draw [color={rgb, 255:red, 0; green, 0; blue, 0 }  ,draw opacity=1 ]   (361.57,329.51) -- (361.42,334.71) ;
			\draw [color={rgb, 255:red, 0; green, 0; blue, 0 }  ,draw opacity=1 ]   (360.98,353.43) -- (360.83,358.64) ;
			
			\draw (324,374) node [anchor=north west][inner sep=0.75pt]  [color={rgb, 255:red, 128; green, 128; blue, 128 }  ,opacity=1 ]  {$T_{\msc{A} ,m+\delta_k}$};
			\draw (350,337) node [anchor=north west][inner sep=0.75pt]  [font=\tiny]  {$d_{x_k}^{m_{k}}$};
			\draw (170,370) node [anchor=north west][inner sep=0.75pt]    {$+ \ \ \ \ \cdots \ \ \ \  \cdots \ \ \ \cdots \ \ \ \otimes$};
			\draw (173,317.43) node [anchor=north west][inner sep=0.75pt]  [font=\tiny]  {$k( m_{1}) \ \ \ \ \ \ \ \ \ \ \ \ \ k(m_k) \ \  k( m_{t})$};
			\draw (167,420.93) node [anchor=north west][inner sep=0.75pt]  [font=\tiny]  {$k(-1) \ k( m_{1}) \ \ \ \ \ \ \ \ \ \ \ \ \ \ \ \ \ k( m_{t}) \ $};
			
			\draw (420,385) node {,}; 
		\end{tikzpicture}
	}
\end{equation}
	where now $m_j=m_k+1$ necessarily.  We look at the last term in \eqref{eq:lemma 3, side 1}. Using the relations in $k\Bord_{\msc{A}}^{\red}$, 
	we slide the coupon labeled by $d$ through $M$ as follows (pictured in \eqref{eq: lemma 3/1}).
	At first the ribbon is labelled by $d_{x_k}^{m_k}$.  We slide this coupon to the left and reverse the orientation, while dualyzing the label, see relation (U1) in Section \ref{sect:sk_rel}. Then, using that $d_{w^*}|_{(w^l)^*} = - (-1)^l (d_w|_{w^l})^*$ for any cochain $w$ and index $l$, we obtain the last equality in \eqref{eq: lemma 3/1}. 
	\begin{equation}\label{eq: lemma 3/1}
		\scalebox{0.9}{
	\tikzset{every picture/.style={line width=0.5pt}} 
	

	
		}
	\end{equation}
	On the other hand, we can use string relations in  $k\Str_{\opn{Vect}^{\mathbb Z}}^{\red}$ to obtain
	\begin{equation}\label{eq: lemma 3/2}
		\scalebox{0.9}{
			\tikzset{every picture/.style={line width=0.5pt}} 
			
			%
			
		}
	\end{equation}
	where in the third equality we are using relation (V5), see \eqref{rel:V5}.
	
	Replacing \eqref{eq: lemma 3/1} and \eqref{eq: lemma 3/2} in \eqref{eq:lemma 3, side 1}, we get
	\begin{equation*}
	\scalebox{0.9}{
		\tikzset{every picture/.style={line width=0.5pt}} 
		
		%
		
	}
\end{equation*}
	Since the second term is equal to zero in $k\Str_{\opn{Vect}^{\mathbb Z}}^{\red}$,  the equation above is just 
	\begin{equation*}
		\scalebox{0.9}{
			\tikzset{every picture/.style={line width=0.5pt}} 
			
			%
		}
	\end{equation*}
	Lastly, via string relations in $k\Str_{\opn{Vect}^{\mathbb Z}}^{\red}$ and using \eqref{eq:Lemma A.1, diff} in $k\Bord_{\msc{A}}^{\red}$, this is the same as
	\begin{equation}\label{eq:lemma 3 side 2}
		\scalebox{0.9}{
			\tikzset{every picture/.style={line width=0.5pt}} 
			
			%
		}
	\end{equation}
	The statement then follows from applying the functor $ev_{\opn{Vect}^{\mathbb Z}}\cdot Z$ to  the sum over all $m:\{1,\dots, t\} \to \mathbb Z$ of the equality $\eqref{eq:lemma 3, side 1}=\eqref{eq:lemma 3 side 2}$. 
\end{proof}

\begin{lemma}\label{lemma 4}
	Let $(M,T):\Sigma_{\vec{x}}\to\Sigma'_{\vec{y}}$ be a connected bordism in  $\opn{Bord}_{\opn{Ch}(\msc{A})}$ such
that $T$ consists (only) of couponless ribbons either traveling from $\Sigma$ to $\Sigma'$, or from $\Sigma'$ to itself.  Then 
$Z^* (M,T)[1] \ d_{\Sigma_{\vec{x}}} = d_{\Sigma'_{\vec{y}}}\  Z^* (M,T).$ 
\end{lemma}

\begin{proof}
	The proof is analogous to that of Lemma \ref{lemma 3}.
\end{proof}

\begin{corollary}\label{cor:cups,caps and crossings}
		Let $(M,T):\Sigma_{\vec{x}}\to\Sigma'_{\vec{y}}$ be a connected bordism in  $\opn{Bord}_{\opn{Ch}(\msc{A})}$ such that all ribbons in $T$ are couponless. That is, each ribbon travels from an incoming to another incoming label, an incoming to an outgoing label, or an outgoing to another outgoing label.   
	Then 
	$Z^* (M,T)[1] \ d_{\Sigma_{\vec{x}}} = d_{\Sigma'_{\vec{y}}}\  Z^* (M,T).$ 
\end{corollary}

\begin{proof}
	 If $M$ has vanishing boundary then $Z^\ast(M,T)$ is an endomorphism of $k$, considered as a complex concentrated in degree $0$, and is trivially a cochain map.  If $M$ has only one nonvanishing boundary, either incoming or outgoing, then $(M,T)$ is necessarily of the form covered in Lemma \ref{lemma 3} or \ref{lemma 4}.  So again $Z^\ast(M,T)$ is a cochain map.  We consider finally the case where both $\Sigma$ and $\Sigma'$ are nonempty.
	
	The ribbon diagram $T$ is given by a composition of cups, caps and permutations of strings. 
We will prove the statement by induction on the number of cups and caps in $T$. If there are none or at most one cup or cap, then $(M,T)$ is as in Lemma \ref{lemma 1}, Lemma \ref{lemma 3} or \ref{lemma 4}, respectively, and the statement follows. Suppose then the statement is true for ribbon diagrams with a total of at most $l$ cups and caps, and assume $T$ has $l+1$ cups and caps.  If $T$ contains only caps, or only cups, then we are again in the situation of Lemma \ref{lemma 1} or Lemma \ref{lemma 2}.  So we consider the case where $T$ contains both cups and caps.

We look at one of the cups in $(M,T)$ and pull it to the opposite boundary so that the cup is now happening inside one of the collars of $M$; we may do this inside a small enough tubular neighborhood not interfering with the topology of $M$, see Picture \eqref{eq:cup/cap}. We can then ``cut" the cobordism into two, so that $(M,T)=(M_c,T_c)(M',T')$, where $(M_c,T_c)$ has underlying manifold $\Sigma'\times [0,1]$ and ribbon diagram $T_c$ featuring a unique cup, and $(M',T')$ is obtained from $(M,T)$ by removing the cup/cap via this procedure.  Hence, by induction, $Z^*(M',T')[1] \ d_{\Sigma_{\vec{x}}} = d_{\Sigma'_{\vec{y}}}\  Z^*( M', T')$. Now, composing on both sides with $Z^*(M_c,T_c),$ we obtain
\begin{align*}
	Z^*((M_c,T_c)(M',T'))[1] \ d_{\Sigma_{\vec{x}}} &= Z^*(M_c,T_c)[1] \ d_{\Sigma'_{\vec{y}}} \  Z^*(M', T')\\
	&= d_{\Sigma'_{\vec{y}}} \  Z^*((M_c,T_c)(M',T'))
\end{align*} 
where in the second equality we are using that $Z^*(M_c, T_c)$ commutes with the differential by Lemmas \ref{lemma 3} and \ref{lemma 4}. 	Thus since $(M_c,T_c)(M',T)=(M,T)$, we conclude 
$Z^*(M,T)[1] \ d_{\Sigma_{\vec{x}}} = d_{\Sigma'_{\vec{y}}}\  Z^* (M,T)$, as desired. 

	\begin{equation}\label{eq:cup/cap}
		\scalebox{0.8}{
\tikzset{every picture/.style={line width=0.5pt}} 

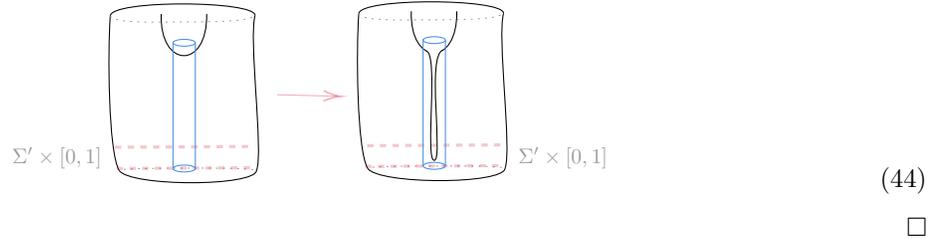
\begin{tikzpicture}[x=0.75pt,y=0.75pt,yscale=-1,xscale=1]
	
	\draw  [color={rgb, 255:red, 255; green, 255; blue, 255 }  ,draw opacity=1 ][line width=3] [line join = round][line cap = round] (208.82,2440.76) .. controls (209,2442.05) and (209.51,2443.3) .. (210.04,2444.53) ;
	\draw  [color={rgb, 255:red, 255; green, 255; blue, 255 }  ,draw opacity=1 ][line width=3] [line join = round][line cap = round] (156.73,2338.89) .. controls (156.91,2340.18) and (157.42,2341.43) .. (157.95,2342.66) ;
	\draw    (204.33,2338.14) .. controls (200.17,2352.29) and (209.67,2424.29) .. (205.32,2435.65) ;
	\draw    (113.06,2338.39) .. controls (111.38,2360.14) and (111.77,2424.22) .. (118.63,2435.84) ;
	\draw    (113.06,2338.39) .. controls (113.43,2328.13) and (204.56,2328.55) .. (204.33,2338.14) ;
	\draw [color={rgb, 255:red, 155; green, 155; blue, 155 }  ,draw opacity=1 ] [dash pattern={on 0.84pt off 2.51pt}]  (113.06,2338.39) .. controls (116.25,2343.43) and (142.6,2342.16) .. (159.74,2342.16) .. controls (176.88,2342.16) and (204.03,2342.29) .. (204.13,2337.25) ;
	\draw    (118.63,2435.84) .. controls (125.37,2446.22) and (205.08,2446.28) .. (205.32,2435.65) ;
	\draw [color={rgb, 255:red, 155; green, 155; blue, 155 }  ,draw opacity=1 ] [dash pattern={on 0.84pt off 2.51pt}]  (121.63,2435.23) .. controls (122.57,2434.82) and (144.67,2434.4) .. (161.12,2434.4) .. controls (177.58,2434.4) and (196.11,2434.21) .. (205.32,2435.65) ;
	\draw    (145.17,2337.17) .. controls (144,2371.75) and (174.5,2372.75) .. (173.83,2337.17) ;
	\draw [color={rgb, 255:red, 208; green, 2; blue, 27 }  ,draw opacity=0.32 ]   (218.17,2388.11) -- (256.42,2389.06) ;
	\draw [shift={(258.42,2389.11)}, rotate = 181.42] [color={rgb, 255:red, 208; green, 2; blue, 27 }  ,draw opacity=0.32 ][line width=0.75]    (10.93,-3.29) .. controls (6.95,-1.4) and (3.31,-0.3) .. (0,0) .. controls (3.31,0.3) and (6.95,1.4) .. (10.93,3.29)   ;
	\draw  [color={rgb, 255:red, 255; green, 255; blue, 255 }  ,draw opacity=1 ][line width=3] [line join = round][line cap = round] (365.49,2439.43) .. controls (365.67,2440.71) and (366.18,2441.96) .. (366.71,2443.19) ;
	\draw  [color={rgb, 255:red, 255; green, 255; blue, 255 }  ,draw opacity=1 ][line width=3] [line join = round][line cap = round] (313.39,2337.56) .. controls (313.58,2338.84) and (314.08,2340.1) .. (314.61,2341.32) ;
	\draw    (361,2336.81) .. controls (356.83,2350.95) and (366.33,2422.95) .. (361.99,2434.31) ;
	\draw    (269.72,2337.06) .. controls (268.04,2358.81) and (268.43,2422.89) .. (275.29,2434.51) ;
	\draw    (269.72,2337.06) .. controls (270.1,2326.8) and (361.23,2327.21) .. (361,2336.81) ;
	\draw [color={rgb, 255:red, 155; green, 155; blue, 155 }  ,draw opacity=1 ] [dash pattern={on 0.84pt off 2.51pt}]  (269.72,2337.06) .. controls (272.91,2342.1) and (299.26,2340.83) .. (316.4,2340.83) .. controls (333.55,2340.83) and (360.7,2340.96) .. (360.8,2335.91) ;
	\draw    (275.29,2434.51) .. controls (282.03,2444.89) and (361.75,2444.94) .. (361.99,2434.31) ;
	\draw [color={rgb, 255:red, 155; green, 155; blue, 155 }  ,draw opacity=1 ] [dash pattern={on 0.84pt off 2.51pt}]  (278.3,2433.9) .. controls (279.24,2433.48) and (301.33,2433.07) .. (317.79,2433.07) .. controls (334.25,2433.07) and (352.77,2432.87) .. (361.99,2434.31) ;
	\draw [dashed, color={rgb, 255:red, 208; green, 2; blue, 27 }  ,draw opacity=0.22 ][line width=1.5]    (275,2420.25) -- (363,2419.75) ;
	\draw [dashed, color={rgb, 255:red, 208; green, 2; blue, 27 }  ,draw opacity=0.22 ][line width=1.5]    (273.79,2433.32) -- (361.79,2432.82) ;
	\draw [dashed, color={rgb, 255:red, 208; green, 2; blue, 27 }  ,draw opacity=0.22 ][line width=1.5]    (116,2421.25) -- (204,2420.75) ;
	\draw [dashed, color={rgb, 255:red, 208; green, 2; blue, 27 }  ,draw opacity=0.22 ][line width=1.5]    (117.12,2434.65) -- (205.12,2434.15) ;
	\draw [color={rgb, 255:red, 74; green, 144; blue, 226 }  ,draw opacity=1 ][fill={rgb, 255:red, 74; green, 144; blue, 226 }  ,fill opacity=1 ]   (152.6,2354.66) -- (152.6,2434.66) ;
	\draw  [color={rgb, 255:red, 74; green, 144; blue, 226 }  ,draw opacity=1 ] (152.6,2355.32) .. controls (152.6,2354.13) and (155.75,2353.17) .. (159.64,2353.17) .. controls (163.52,2353.17) and (166.67,2354.13) .. (166.67,2355.32) .. controls (166.67,2356.51) and (163.52,2357.48) .. (159.64,2357.48) .. controls (155.75,2357.48) and (152.6,2356.51) .. (152.6,2355.32) -- cycle ;
	\draw [color={rgb, 255:red, 74; green, 144; blue, 226 }  ,draw opacity=1 ][fill={rgb, 255:red, 74; green, 144; blue, 226 }  ,fill opacity=1 ]   (166.67,2355.32) -- (166.67,2434.66) ;
	\draw  [color={rgb, 255:red, 74; green, 144; blue, 226 }  ,draw opacity=1 ] (152.6,2434.66) .. controls (152.6,2433.47) and (155.75,2432.5) .. (159.64,2432.5) .. controls (163.52,2432.5) and (166.67,2433.47) .. (166.67,2434.66) .. controls (166.67,2435.86) and (163.52,2436.82) .. (159.64,2436.82) .. controls (155.75,2436.82) and (152.6,2435.86) .. (152.6,2434.66) -- cycle ;
	\draw    (302.83,2335.5) .. controls (301.67,2370.08) and (332.17,2371.08) .. (331.5,2335.5) ;
	\draw [color={rgb, 255:red, 74; green, 144; blue, 226 }  ,draw opacity=1 ][fill={rgb, 255:red, 74; green, 144; blue, 226 }  ,fill opacity=1 ]   (310.27,2352.99) -- (310.27,2433) ;
	\draw  [color={rgb, 255:red, 74; green, 144; blue, 226 }  ,draw opacity=1 ] (310.27,2353.66) .. controls (310.27,2352.47) and (313.42,2351.5) .. (317.3,2351.5) .. controls (321.19,2351.5) and (324.33,2352.47) .. (324.33,2353.66) .. controls (324.33,2354.85) and (321.19,2355.81) .. (317.3,2355.81) .. controls (313.42,2355.81) and (310.27,2354.85) .. (310.27,2353.66) -- cycle ;
	\draw [color={rgb, 255:red, 74; green, 144; blue, 226 }  ,draw opacity=1 ][fill={rgb, 255:red, 74; green, 144; blue, 226 }  ,fill opacity=1 ]   (324.33,2353.66) -- (324.33,2433) ;
	\draw  [color={rgb, 255:red, 74; green, 144; blue, 226 }  ,draw opacity=1 ] (310.27,2433) .. controls (310.27,2431.8) and (313.42,2430.83) .. (317.3,2430.83) .. controls (321.19,2430.83) and (324.33,2431.8) .. (324.33,2433) .. controls (324.33,2434.19) and (321.19,2435.16) .. (317.3,2435.16) .. controls (313.42,2435.16) and (310.27,2434.19) .. (310.27,2433) -- cycle ;
	\draw  [color={rgb, 255:red, 255; green, 255; blue, 255 }  ,draw opacity=1 ][line width=6] [line join = round][line cap = round] (316.25,2362) .. controls (315.86,2361.36) and (317.8,2361.74) .. (318.5,2362) ;
	\draw    (312.25,2360.5) .. controls (320.25,2362.25) and (311.25,2429.25) .. (317.25,2429.5) .. controls (323.25,2429.75) and (312.25,2363) .. (322.75,2360) ;
	
	\draw (50,2420) node [anchor=north west][inner sep=0.75pt]  [font=\normalsize,color={rgb, 255:red, 155; green, 155; blue, 155 }  ,opacity=1 ]  {$\Sigma '\times [ 0,1]$};
	\draw (369,2420) node [anchor=north west][inner sep=0.75pt]  [font=\normalsize,color={rgb, 255:red, 155; green, 155; blue, 155 }  ,opacity=1 ]  {$\Sigma '\times [ 0,1]$};

\end{tikzpicture}
		}
	\end{equation}
\end{proof}

\subsection{Commuting with any coupon}
Here we show that for bordisms $(M,T)$ in $\Ch(\msc{A})$ such that $T$ has a unique coupon satisfy $Z^*(M,T)\in \Ch(\opn{Vect}) $.

\begin{corollary}\label{corollary from lemma 2}Let $(M,T):\Sigma_{\vec{x}}\to\Sigma'_{\vec{y}}$ be a connected bordism in  $\opn{Bord}_{\opn{Ch}(\msc{A})}$ such
	that $M$ has nonvanishing boundary, and that the ribbon diagram
	$T$ has a unique coupon. Then 
	$Z^* (M,T)[1] \ d_{\Sigma_{\vec{x}}} = d_{\Sigma'_{\vec{y}}}\  Z^* (M,T).$
\end{corollary}

\begin{proof}
By dragging the unique coupon in $T$ into a collar around the boundary, then cutting,
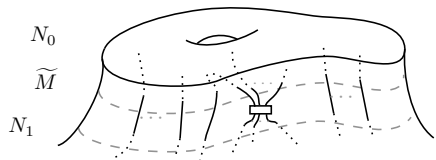
\[
\scalebox{.8}{\tikzset{every picture/.style={line width=0.75pt}} 
\begin{tikzpicture}[x=0.75pt,y=0.75pt,yscale=-1,xscale=1]
\draw    (274.54,109.58) .. controls (261.53,134.96) and (328.93,137.08) .. (346.67,129.67) .. controls (364.41,122.27) and (403.04,112.85) .. (419.99,114.87) .. controls (436.94,116.89) and (464.44,126.3) .. (463.74,109.58) ;
\draw    (274.54,109.58) .. controls (278.58,100.2) and (304.25,88.94) .. (336.03,85.09) .. controls (367.81,81.25) and (398.7,88.72) .. (411.71,89.63) .. controls (424.72,90.53) and (463.07,90.76) .. (463.74,109.58) ;
\draw    (325.81,104.44) .. controls (337.16,111) and (358.9,105.47) .. (371.12,97.89) ;
\draw    (332.89,106.45) .. controls (339.25,101.08) and (349.49,99.71) .. (358.16,103.81) ;
\draw    (245,170.76) .. controls (254.91,169.09) and (276.93,134.05) .. (273.07,116.13) ;
\draw    (463.74,109.58) .. controls (462.36,125.89) and (467,154.75) .. (482,166.08) ;
\draw [color={rgb, 255:red, 74; green, 74; blue, 74 }  ,draw opacity=0.6 ] [dash pattern={on 4.5pt off 4.5pt}]  (270.71,135.92) .. controls (271.8,152.51) and (328.23,153.71) .. (345.96,146.31) .. controls (363.7,138.9) and (402.33,129.48) .. (419.28,131.5) .. controls (436.23,133.52) and (463.73,142.93) .. (463.04,126.22) ;
\draw [color={rgb, 255:red, 74; green, 74; blue, 74 }  ,draw opacity=0.6 ] [dash pattern={on 4.5pt off 4.5pt}]  (258.1,159.82) .. controls (259.68,173.49) and (320.9,178.33) .. (339.65,172.14) .. controls (358.39,165.95) and (397.14,153.7) .. (414.09,155.73) .. controls (431.04,157.75) and (457.43,167.34) .. (467.67,152.99) ;
\draw   (366.14,145.85) -- (379.58,145.85) -- (379.58,150.84) -- (366.14,150.84) -- cycle ;
\draw    (357.84,134.05) .. controls (362.8,136.67) and (369.4,141.47) .. (369.13,146.4) ;
\draw    (362.8,157.65) .. controls (367.2,153.71) and (367.2,157.21) .. (369.28,150.84) ;
\draw    (376.01,145.85) .. controls (376.56,141.91) and (382.61,141.47) .. (386.47,130.99) ;
\draw    (382.61,156.33) .. controls (379.86,155.02) and (376.56,157.21) .. (376.56,151.09) ;
\draw    (323.72,142.79) -- (322.06,166.39) ;
\draw    (298.94,140.16) -- (292.89,163.33) ;
\draw    (413.99,125.74) -- (418.94,150.65) ;
\draw    (435.46,129.24) -- (440.41,154.15) ;
\draw    (367.2,134.92) .. controls (372.71,138.85) and (371.05,141.04) .. (372.16,145.41) ;
\draw    (368.85,158.52) .. controls (371.61,153.71) and (370.5,157.65) .. (372.16,150.65) ;
\draw  [dash pattern={on 0.84pt off 2.51pt}]  (322.06,166.39) -- (318.76,179.06) ;
\draw  [dash pattern={on 0.84pt off 2.51pt}]  (323.49,127.04) -- (323.72,142.79) ;
\draw  [dash pattern={on 0.84pt off 2.51pt}]  (418.94,150.65) -- (423.55,167.34) ;
\draw  [dash pattern={on 0.84pt off 2.51pt}]  (440.41,154.15) -- (444.27,166.82) ;
\draw  [dash pattern={on 0.84pt off 2.51pt}]  (412.52,109.27) -- (413.99,125.74) ;
\draw  [dash pattern={on 0.84pt off 2.51pt}]  (435.37,114.06) -- (435.46,129.24) ;
\draw  [dash pattern={on 0.84pt off 2.51pt}]  (362.8,157.65) -- (355.01,175.54) ;
\draw  [dash pattern={on 0.84pt off 2.51pt}]  (368.85,158.52) -- (367.61,170.07) ;
\draw  [dash pattern={on 0.84pt off 2.51pt}]  (382.61,156.33) -- (396.76,170.75) ;
\draw    (338.03,166.82) .. controls (339.13,160.27) and (345.18,149.78) .. (345.18,137.54) ;
\draw  [dash pattern={on 0.84pt off 2.51pt}]  (345.18,137.54) .. controls (346.83,131.42) and (334.72,128.8) .. (339.68,126.18) .. controls (344.63,123.56) and (352.34,130.55) .. (357.84,134.05) ;
\draw  [dash pattern={on 0.84pt off 2.51pt}]  (338.03,166.82) -- (334.52,181) ;
\draw  [dash pattern={on 0.84pt off 2.51pt}]  (389.22,122.68) -- (386.47,130.99) ;
\draw  [dash pattern={on 0.84pt off 2.51pt}]  (355.09,111.76) .. controls (353.99,120.93) and (362.25,132.3) .. (367.2,134.92) ;
\draw  [dash pattern={on 0.84pt off 2.51pt}]  (281.33,178.19) .. controls (287.38,175.13) and (291.24,170.32) .. (292.89,163.33) ;
\draw  [dash pattern={on 0.84pt off 2.51pt}]  (295.92,112.69) .. controls (297.49,120.89) and (300.52,132.65) .. (298.94,140.16) ;

\draw (294.46,149.35) node [anchor=north west][inner sep=0.75pt]  [color={rgb, 255:red, 155; green, 155; blue, 155 }  ,opacity=1 ,rotate=-358.94] [align=left] {$\displaystyle \cdots $};
\draw (414.35,135.75) node [anchor=north west][inner sep=0.75pt]  [color={rgb, 255:red, 155; green, 155; blue, 155 }  ,opacity=1 ] [align=left] {$\displaystyle \cdots $};
\draw (366.33,128.33) node [anchor=north west][inner sep=0.75pt]  [font=\scriptsize,color={rgb, 255:red, 155; green, 155; blue, 155 }  ,opacity=1 ] [align=left] {$\displaystyle \cdots $};
\draw (364.57,152.16) node [anchor=north west][inner sep=0.75pt]  [font=\scriptsize,color={rgb, 255:red, 155; green, 155; blue, 155 }  ,opacity=1 ] [align=left] {$\displaystyle \cdots $};
\draw (226.15,94.53) node [anchor=north west][inner sep=0.75pt]   [align=left] {$\displaystyle N_{0}$};
\draw (212.48,151.98) node [anchor=north west][inner sep=0.75pt]   [align=left] {$\displaystyle N_{1}$};
\draw (228.51,120.26) node [anchor=north west][inner sep=0.75pt]   [align=left] {$\displaystyle \tilde{M}$};

\end{tikzpicture}
}
\]
we are free to assume that $M=\Sigma\times [0,1]$ for nonempty $\Sigma$, that the unique coupon in $T$ has incoming ribbons connected to the incoming boundary and outgoing ribbons connected to the outgoing boundary, and that all other ribbons are couponless.  (Here we refer to Corollary \ref{cor:cups,caps and crossings} to deal with the excess bordisms $N_0$ and $N_1$ generated in this process.)  We now argue under these assumptions on $(M,T)$.

Suppose first that the coupon in $T$ has multiple incoming strings. As pictured in \eqref{eq: mult strings}, we can first slide the coupon upwards (first equality) through a small enough tubular neighborhood around an incoming string. Then we can use skein relations, see Section \ref{sect:sk_rel} to replace the label so the coupon has a unique incoming string instead (second equality). We can thus ``cut" the bordism $(M,T)=(\tilde M, \tilde T)(M_1, T_1)$, where $(M_1, T_1)$ is as in Lemma \ref{lemma 2} and $(\tilde M, \tilde T)$ is couponless. It follows from Lemma \ref{lemma 2} and Corollary \ref{cor:cups,caps and crossings} that $Z^*(M_1, T_1)$ and $Z^*(\tilde M,\tilde T)$ are in $\Ch(\opn{Vect})$, respectively, and hence so is $Z^*(M,T)$.
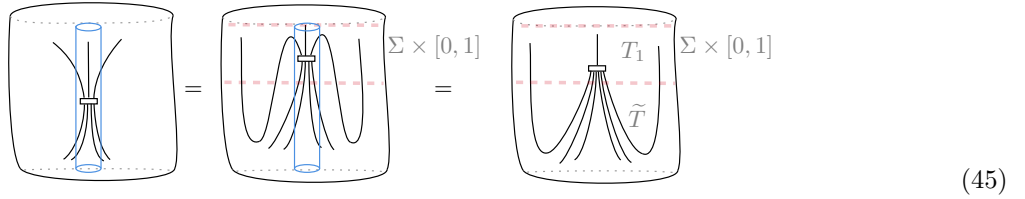
	\begin{equation}\label{eq: mult strings}
		\scalebox{0.9}{
	\tikzset{every picture/.style={line width=0.5pt}} 
	
	\begin{tikzpicture}[x=0.75pt,y=0.75pt,yscale=-1,xscale=1]
		
		\draw  [color={rgb, 255:red, 255; green, 255; blue, 255 }  ,draw opacity=1 ][line width=3] [line join = round][line cap = round] (57.46,1642.62) .. controls (57.65,1643.9) and (58.15,1645.16) .. (58.68,1646.38) ;
		\draw    (105.07,1641.87) .. controls (100.9,1654.26) and (110.4,1717.32) .. (106.05,1727.26) ;
		\draw    (13.79,1642.09) .. controls (12.11,1661.14) and (12.5,1717.26) .. (19.36,1727.43) ;
		\draw    (13.79,1642.12) .. controls (14.17,1631.86) and (105.29,1632.27) .. (105.07,1641.87) ;
		\draw [color={rgb, 255:red, 155; green, 155; blue, 155 }  ,draw opacity=1 ] [dash pattern={on 0.84pt off 2.51pt}]  (13.79,1642.12) .. controls (16.98,1647.16) and (43.33,1645.89) .. (60.47,1645.89) .. controls (77.61,1645.89) and (104.77,1646.02) .. (104.87,1640.97) ;
		\draw    (19.36,1727.43) .. controls (26.1,1736.53) and (105.81,1736.57) .. (106.05,1727.26) ;
		\draw [color={rgb, 255:red, 155; green, 155; blue, 155 }  ,draw opacity=1 ] [dash pattern={on 0.84pt off 2.51pt}]  (21.56,1728.36) .. controls (22.5,1728) and (44.59,1727.64) .. (61.05,1727.64) .. controls (77.5,1727.64) and (96.03,1727.47) .. (105.24,1728.73) ;
		\draw [color={rgb, 255:red, 0; green, 0; blue, 0 }  ,draw opacity=1 ]   (76.17,1654.96) .. controls (67.92,1663.96) and (60.92,1676.46) .. (60.66,1688.42) ;
		\draw  [color={rgb, 255:red, 0; green, 0; blue, 0 }  ,draw opacity=1 ][fill={rgb, 255:red, 255; green, 255; blue, 255 }  ,fill opacity=1 ] (53.06,1688.3) -- (62.56,1688.3) -- (62.56,1691.28) -- (53.06,1691.28) -- cycle ;
		\draw [color={rgb, 255:red, 0; green, 0; blue, 0 }  ,draw opacity=1 ]   (37.72,1656.79) .. controls (45.97,1663.29) and (53.16,1672.52) .. (55.22,1687.79) ;
		\draw [color={rgb, 255:red, 0; green, 0; blue, 0 }  ,draw opacity=1 ][fill={rgb, 255:red, 255; green, 255; blue, 255 }  ,fill opacity=1 ]   (57.79,1687.66) -- (57.66,1671.23) -- (57.53,1656.42) ;
		\draw [color={rgb, 255:red, 0; green, 0; blue, 0 }  ,draw opacity=1 ]   (43.67,1722.32) .. controls (48.57,1719.27) and (55.25,1706.03) .. (55.38,1691.61) ;
		\draw [color={rgb, 255:red, 0; green, 0; blue, 0 }  ,draw opacity=1 ]   (47.94,1722.65) .. controls (53.97,1719.6) and (57.5,1705.28) .. (57.11,1691.41) ;
		\draw [color={rgb, 255:red, 0; green, 0; blue, 0 }  ,draw opacity=1 ]   (69.75,1722.53) .. controls (64.1,1717.29) and (61.25,1703.53) .. (60.75,1691.3) ;
		\draw [color={rgb, 255:red, 0; green, 0; blue, 0 }  ,draw opacity=1 ]   (66.25,1723.28) .. controls (60.22,1720.22) and (59.25,1705.78) .. (59.14,1691.12) ;
		\draw  [color={rgb, 255:red, 255; green, 255; blue, 255 }  ,draw opacity=1 ][line width=3] [line join = round][line cap = round] (176.46,1643.39) .. controls (176.65,1644.61) and (177.15,1645.79) .. (177.68,1646.95) ;
		\draw    (224.07,1642.68) .. controls (219.9,1655.04) and (229.4,1717.97) .. (225.05,1727.89) ;
		\draw    (132.79,1642.9) .. controls (131.11,1661.91) and (131.5,1717.91) .. (138.36,1728.06) ;
		\draw    (132.79,1642.92) .. controls (133.17,1633.21) and (224.29,1633.6) .. (224.07,1642.68) ;
		\draw [color={rgb, 255:red, 155; green, 155; blue, 155 }  ,draw opacity=1 ] [dash pattern={on 0.84pt off 2.51pt}]  (132.79,1642.92) .. controls (135.98,1647.69) and (162.33,1646.48) .. (179.47,1646.48) .. controls (196.61,1646.48) and (223.77,1646.61) .. (223.87,1641.83) ;
		\draw    (138.36,1728.06) .. controls (145.1,1737.65) and (224.81,1737.7) .. (225.05,1727.89) ;
		\draw [color={rgb, 255:red, 155; green, 155; blue, 155 }  ,draw opacity=1 ] [dash pattern={on 0.84pt off 2.51pt}]  (140.56,1729.04) .. controls (141.5,1728.66) and (163.59,1728.28) .. (180.05,1728.28) .. controls (196.5,1728.28) and (215.03,1728.1) .. (224.24,1729.43) ;
		\draw  [color={rgb, 255:red, 255; green, 255; blue, 255 }  ,draw opacity=1 ][line width=3] [line join = round][line cap = round] (339.36,1644.41) .. controls (339.55,1645.63) and (340.05,1646.81) .. (340.58,1647.98) ;
		\draw    (386.97,1643.69) .. controls (382.8,1655.99) and (392.3,1718.57) .. (387.95,1728.44) ;
		\draw    (295.69,1643.93) .. controls (294.01,1663.86) and (294.4,1717.98) .. (301.26,1728.62) ;
		\draw    (295.69,1643.93) .. controls (296.07,1634.2) and (387.19,1634.59) .. (386.97,1643.69) ;
		\draw [color={rgb, 255:red, 155; green, 155; blue, 155 }  ,draw opacity=1 ] [dash pattern={on 0.84pt off 2.51pt}]  (295.69,1643.93) .. controls (298.88,1648.71) and (325.23,1647.51) .. (342.37,1647.51) .. controls (359.51,1647.51) and (386.67,1647.63) .. (386.77,1642.85) ;
		\draw    (301.26,1728.62) .. controls (308,1738.13) and (387.71,1738.18) .. (387.95,1728.44) ;
		\draw [color={rgb, 255:red, 155; green, 155; blue, 155 }  ,draw opacity=1 ] [dash pattern={on 0.84pt off 2.51pt}]  (303.46,1729.59) .. controls (304.4,1729.21) and (326.49,1728.83) .. (342.95,1728.83) .. controls (359.4,1728.83) and (377.93,1728.65) .. (387.14,1729.97) ;
		\draw [color={rgb, 255:red, 0; green, 0; blue, 0 }  ,draw opacity=1 ][fill={rgb, 255:red, 255; green, 255; blue, 255 }  ,fill opacity=1 ]   (204.5,1712.75) .. controls (196.5,1710.5) and (195.4,1621.5) .. (182.2,1663.9) ;
		\draw  [color={rgb, 255:red, 0; green, 0; blue, 0 }  ,draw opacity=1 ] (175,1664.4) -- (184.5,1664.4) -- (184.5,1667.38) -- (175,1667.38) -- cycle ;
		\draw [color={rgb, 255:red, 0; green, 0; blue, 0 }  ,draw opacity=1 ]   (153.5,1712.25) .. controls (161.75,1708.75) and (165,1625.1) .. (177,1663.9) ;
		\draw [color={rgb, 255:red, 0; green, 0; blue, 0 }  ,draw opacity=1 ][fill={rgb, 255:red, 255; green, 255; blue, 255 }  ,fill opacity=1 ]   (179.26,1664.36) -- (179.25,1647) ;
		\draw [color={rgb, 255:red, 0; green, 0; blue, 0 }  ,draw opacity=1 ]   (158.5,1716.25) .. controls (163.4,1713.2) and (177.25,1682.13) .. (177.38,1667.71) ;
		\draw [color={rgb, 255:red, 0; green, 0; blue, 0 }  ,draw opacity=1 ]   (167,1716.75) .. controls (173.03,1713.7) and (179.5,1681.38) .. (179.11,1667.51) ;
		\draw [color={rgb, 255:red, 0; green, 0; blue, 0 }  ,draw opacity=1 ]   (195.25,1717) .. controls (189.6,1711.77) and (183.25,1679.63) .. (182.75,1667.4) ;
		\draw [color={rgb, 255:red, 0; green, 0; blue, 0 }  ,draw opacity=1 ]   (188.25,1715.75) .. controls (182.22,1712.7) and (181.25,1681.88) .. (181.14,1667.22) ;
		\draw [color={rgb, 255:red, 0; green, 0; blue, 0 }  ,draw opacity=1 ][fill={rgb, 255:red, 255; green, 255; blue, 255 }  ,fill opacity=1 ]   (368.07,1719.17) .. controls (360.07,1716.92) and (349.9,1689.13) .. (346.15,1672.9) ;
		\draw  [color={rgb, 255:red, 0; green, 0; blue, 0 }  ,draw opacity=1 ] (337.9,1669.9) -- (347.4,1669.9) -- (347.4,1672.88) -- (337.9,1672.88) -- cycle ;
		\draw [color={rgb, 255:red, 0; green, 0; blue, 0 }  ,draw opacity=1 ]   (315.73,1718.5) .. controls (323.98,1715) and (339.4,1680.13) .. (339.15,1672.63) ;
		\draw [color={rgb, 255:red, 0; green, 0; blue, 0 }  ,draw opacity=1 ][fill={rgb, 255:red, 255; green, 255; blue, 255 }  ,fill opacity=1 ]   (342.82,1669.53) -- (342.82,1652.17) ;
		\draw [color={rgb, 255:red, 0; green, 0; blue, 0 }  ,draw opacity=1 ]   (321.65,1724.63) .. controls (326.55,1721.57) and (340.65,1687.88) .. (340.78,1673.46) ;
		\draw [color={rgb, 255:red, 0; green, 0; blue, 0 }  ,draw opacity=1 ]   (330.04,1722.63) .. controls (334.76,1720.25) and (339.74,1699.97) .. (341.5,1684.61) .. controls (341.99,1680.32) and (342.24,1676.41) .. (342.15,1673.39) ;
		\draw [color={rgb, 255:red, 0; green, 0; blue, 0 }  ,draw opacity=1 ]   (355.9,1722.75) .. controls (350.25,1717.52) and (343.9,1685.38) .. (343.4,1673.15) ;
		\draw [color={rgb, 255:red, 0; green, 0; blue, 0 }  ,draw opacity=1 ]   (361.9,1724.13) .. controls (355.87,1721.07) and (345.9,1689.13) .. (344.79,1673.22) ;
		\draw [dashed, color={rgb, 255:red, 208; green, 2; blue, 27 }  ,draw opacity=0.22 ][line width=1.5]    (136.08,1646.83) -- (222.42,1647.17) ;
		\draw [dashed, color={rgb, 255:red, 208; green, 2; blue, 27 }  ,draw opacity=0.22 ][line width=1.5]    (133.83,1679.17) -- (223,1679.75) ;
		\draw    (143.17,1653.92) .. controls (144.17,1668.42) and (140.5,1717.58) .. (153.5,1712.25) ;
		\draw    (209.17,1654.75) .. controls (210.17,1669.25) and (215.67,1718) .. (204.5,1712.75) ;
		\draw    (305.07,1657.17) .. controls (306.07,1671.67) and (302.73,1723.83) .. (315.73,1718.5) ;
		\draw    (377.07,1658.83) .. controls (378.07,1673.33) and (379.23,1724.42) .. (368.07,1719.17) ;
		\draw [dashed, color={rgb, 255:red, 208; green, 2; blue, 27 }  ,draw opacity=0.22 ][line width=1.5]    (300.2,1647.51) -- (384.54,1647.51) ;
		\draw [color={rgb, 255:red, 74; green, 144; blue, 226 }  ,draw opacity=1 ][fill={rgb, 255:red, 74; green, 144; blue, 226 }  ,fill opacity=1 ]   (50.6,1647.5) -- (50.6,1727.5) ;
		\draw [color={rgb, 255:red, 74; green, 144; blue, 226 }  ,draw opacity=1 ][fill={rgb, 255:red, 74; green, 144; blue, 226 }  ,fill opacity=1 ]   (64.67,1648.16) -- (64.67,1727.5) ;
		\draw  [color={rgb, 255:red, 74; green, 144; blue, 226 }  ,draw opacity=1 ] (50.6,1727.5) .. controls (50.6,1726.31) and (53.75,1725.34) .. (57.64,1725.34) .. controls (61.52,1725.34) and (64.67,1726.31) .. (64.67,1727.5) .. controls (64.67,1728.69) and (61.52,1729.66) .. (57.64,1729.66) .. controls (53.75,1729.66) and (50.6,1728.69) .. (50.6,1727.5) -- cycle ;
		\draw  [color={rgb, 255:red, 74; green, 144; blue, 226 }  ,draw opacity=1 ] (50.6,1648.16) .. controls (50.6,1646.97) and (53.75,1646) .. (57.64,1646) .. controls (61.52,1646) and (64.67,1646.97) .. (64.67,1648.16) .. controls (64.67,1649.35) and (61.52,1650.32) .. (57.64,1650.32) .. controls (53.75,1650.32) and (50.6,1649.35) .. (50.6,1648.16) -- cycle ;
		\draw [color={rgb, 255:red, 74; green, 144; blue, 226 }  ,draw opacity=1 ][fill={rgb, 255:red, 74; green, 144; blue, 226 }  ,fill opacity=1 ]   (173.1,1647.5) -- (173.1,1727.5) ;
		\draw  [color={rgb, 255:red, 74; green, 144; blue, 226 }  ,draw opacity=1 ] (173.1,1648.16) .. controls (173.1,1646.97) and (176.25,1646) .. (180.14,1646) .. controls (184.02,1646) and (187.17,1646.97) .. (187.17,1648.16) .. controls (187.17,1649.35) and (184.02,1650.32) .. (180.14,1650.32) .. controls (176.25,1650.32) and (173.1,1649.35) .. (173.1,1648.16) -- cycle ;
		\draw [color={rgb, 255:red, 74; green, 144; blue, 226 }  ,draw opacity=1 ][fill={rgb, 255:red, 74; green, 144; blue, 226 }  ,fill opacity=1 ]   (187.17,1648.16) -- (187.17,1727.5) ;
		\draw  [color={rgb, 255:red, 74; green, 144; blue, 226 }  ,draw opacity=1 ] (173.1,1727.5) .. controls (173.1,1726.31) and (176.25,1725.34) .. (180.14,1725.34) .. controls (184.02,1725.34) and (187.17,1726.31) .. (187.17,1727.5) .. controls (187.17,1728.69) and (184.02,1729.66) .. (180.14,1729.66) .. controls (176.25,1729.66) and (173.1,1728.69) .. (173.1,1727.5) -- cycle ;
		\draw [dashed, color={rgb, 255:red, 208; green, 2; blue, 27 }  ,draw opacity=0.22 ][line width=1.5]    (297.23,1679.17) -- (386.4,1679.75) ;
		
		\draw (110,1680) node [anchor=north west][inner sep=0.75pt]    {$=$};
		\draw (250,1680) node [anchor=north west][inner sep=0.75pt]    {$=$};
		\draw (354.83,1655) node [anchor=north west][inner sep=0.75pt]  [color={rgb, 255:red, 128; green, 128; blue, 128 }  ,opacity=1 ]  {$T_{1}$};
		\draw (358.9,1690) node [anchor=north west][inner sep=0.75pt]  [color={rgb, 255:red, 128; green, 128; blue, 128 }  ,opacity=1 ]  {$\tilde{T}$};
		\draw (223.9,1651.6) node [anchor=north west][inner sep=0.75pt]  [color={rgb, 255:red, 128; green, 128; blue, 128 }  ,opacity=1 ]  {$\Sigma \times [ 0,1] \ \ \ \ \ \ \ \ \ \ \ \ \ \ \ \ \ \ \ \ \ \ \ \ \ \Sigma \times [ 0,1]$};

	\end{tikzpicture}
		}
	\end{equation}
	
	Suppose now that the coupon has no incoming strings.  We can start by sliding one of the ribbons  as in the first equality of Picture \eqref{eq: no strings},   and then use skein relations to relabel the coupon so that this ribbon is its unique incoming string as pictured in the second equality.  
	We can slide the resulting cap up through a small tubular neighborhood then ``cut" the bordism into  $(M,T)=(\tilde M, \tilde T)(M_1, T_1)$,  where $(M_1, T_1)$ is couponless, and $(\tilde M, \tilde T)$ is as in Lemma \ref{lemma 2}. It follows from Corollary \ref{cor:cups,caps and crossings} and Lemma \ref{lemma 2} that $Z^*(M_1, T_1)$ and $Z^*(\tilde M,\tilde T)$ are in $\Ch(\opn{Vect})$, respectively, and hence so is $Z^*(M,T)$.
	\begin{equation}\label{eq: no strings}
		\scalebox{0.9}{
\tikzset{every picture/.style={line width=0.5pt}} 

\begin{tikzpicture}[x=0.75pt,y=0.75pt,yscale=-1,xscale=1]
	
	\draw  [color={rgb, 255:red, 255; green, 255; blue, 255 }  ,draw opacity=1 ][line width=3] [line join = round][line cap = round] (58.96,1770.12) .. controls (59.15,1771.4) and (59.65,1772.66) .. (60.18,1773.88) ;
	\draw    (106.57,1769.37) .. controls (102.4,1781.76) and (111.9,1844.82) .. (107.55,1854.76) ;
	\draw    (15.29,1769.59) .. controls (13.61,1788.64) and (14,1844.76) .. (20.86,1854.93) ;
	\draw    (15.29,1769.62) .. controls (15.67,1759.36) and (106.79,1759.77) .. (106.57,1769.37) ;
	\draw [color={rgb, 255:red, 155; green, 155; blue, 155 }  ,draw opacity=1 ] [dash pattern={on 0.84pt off 2.51pt}]  (15.29,1769.62) .. controls (18.48,1774.66) and (44.83,1773.39) .. (61.97,1773.39) .. controls (79.11,1773.39) and (106.27,1773.52) .. (106.37,1768.47) ;
	\draw    (20.86,1854.93) .. controls (27.6,1864.03) and (107.31,1864.07) .. (107.55,1854.76) ;
	\draw [color={rgb, 255:red, 155; green, 155; blue, 155 }  ,draw opacity=1 ] [dash pattern={on 0.84pt off 2.51pt}]  (23.06,1855.86) .. controls (24,1855.5) and (46.09,1855.14) .. (62.55,1855.14) .. controls (79,1855.14) and (97.53,1854.97) .. (106.74,1856.23) ;
	\draw  [color={rgb, 255:red, 0; green, 0; blue, 0 }  ,draw opacity=1 ][fill={rgb, 255:red, 255; green, 255; blue, 255 }  ,fill opacity=1 ] (55.31,1796.8) -- (64.81,1796.8) -- (64.81,1799.78) -- (55.31,1799.78) -- cycle ;
	\draw [color={rgb, 255:red, 0; green, 0; blue, 0 }  ,draw opacity=1 ]   (35.25,1841.88) .. controls (42.5,1822.13) and (56.25,1815.88) .. (56.63,1800.11) ;
	\draw [color={rgb, 255:red, 0; green, 0; blue, 0 }  ,draw opacity=1 ]   (48.5,1843.63) .. controls (47.25,1826.63) and (61.75,1820.63) .. (58.36,1799.91) ;
	\draw [color={rgb, 255:red, 0; green, 0; blue, 0 }  ,draw opacity=1 ]   (83.14,1843.95) .. controls (77.49,1838.72) and (73.14,1811.45) .. (62.89,1799.87) ;
	\draw [color={rgb, 255:red, 0; green, 0; blue, 0 }  ,draw opacity=1 ]   (79.5,1848.13) .. controls (73.47,1845.07) and (70.75,1817.88) .. (60.5,1800.05) ;
	\draw  [color={rgb, 255:red, 255; green, 255; blue, 255 }  ,draw opacity=1 ][line width=3] [line join = round][line cap = round] (177.96,1770.89) .. controls (178.15,1772.11) and (178.65,1773.29) .. (179.18,1774.45) ;
	\draw    (225.57,1770.18) .. controls (221.4,1782.54) and (230.9,1845.47) .. (226.55,1855.39) ;
	\draw    (134.29,1770.4) .. controls (132.61,1789.41) and (133,1845.41) .. (139.86,1855.56) ;
	\draw    (134.29,1770.42) .. controls (134.67,1760.71) and (225.79,1761.1) .. (225.57,1770.18) ;
	\draw [color={rgb, 255:red, 155; green, 155; blue, 155 }  ,draw opacity=1 ] [dash pattern={on 0.84pt off 2.51pt}]  (134.29,1770.42) .. controls (137.48,1775.19) and (163.83,1773.98) .. (180.97,1773.98) .. controls (198.11,1773.98) and (225.27,1774.11) .. (225.37,1769.33) ;
	\draw    (139.86,1855.56) .. controls (146.6,1865.15) and (226.31,1865.2) .. (226.55,1855.39) ;
	\draw [color={rgb, 255:red, 155; green, 155; blue, 155 }  ,draw opacity=1 ] [dash pattern={on 0.84pt off 2.51pt}]  (142.06,1856.54) .. controls (143,1856.16) and (165.09,1855.78) .. (181.55,1855.78) .. controls (198,1855.78) and (216.53,1855.6) .. (225.74,1856.93) ;
	\draw  [color={rgb, 255:red, 255; green, 255; blue, 255 }  ,draw opacity=1 ][line width=3] [line join = round][line cap = round] (294.46,1770.91) .. controls (294.65,1772.13) and (295.15,1773.31) .. (295.68,1774.48) ;
	\draw    (342.07,1770.19) .. controls (337.9,1782.49) and (347.4,1845.07) .. (343.05,1854.94) ;
	\draw    (250.79,1770.43) .. controls (249.11,1790.36) and (249.5,1844.48) .. (256.36,1855.12) ;
	\draw    (250.79,1770.43) .. controls (251.17,1760.7) and (342.29,1761.09) .. (342.07,1770.19) ;
	\draw [color={rgb, 255:red, 155; green, 155; blue, 155 }  ,draw opacity=1 ] [dash pattern={on 0.84pt off 2.51pt}]  (250.79,1770.43) .. controls (253.98,1775.21) and (280.33,1774.01) .. (297.47,1774.01) .. controls (314.61,1774.01) and (341.77,1774.13) .. (341.87,1769.35) ;
	\draw    (256.36,1855.12) .. controls (263.1,1864.63) and (342.81,1864.68) .. (343.05,1854.94) ;
	\draw [color={rgb, 255:red, 155; green, 155; blue, 155 }  ,draw opacity=1 ] [dash pattern={on 0.84pt off 2.51pt}]  (258.56,1856.09) .. controls (259.5,1855.71) and (281.59,1855.33) .. (298.05,1855.33) .. controls (314.5,1855.33) and (333.03,1855.15) .. (342.24,1856.47) ;
	\draw [color={rgb, 255:red, 0; green, 0; blue, 0 }  ,draw opacity=1 ]   (151.25,1838.38) .. controls (158.5,1822.38) and (155.5,1779.38) .. (164,1791.88) ;
	\draw  [color={rgb, 255:red, 0; green, 0; blue, 0 }  ,draw opacity=1 ][fill={rgb, 255:red, 255; green, 255; blue, 255 }  ,fill opacity=1 ] (178.31,1798.8) -- (187.81,1798.8) -- (187.81,1801.78) -- (178.31,1801.78) -- cycle ;
	\draw [color={rgb, 255:red, 0; green, 0; blue, 0 }  ,draw opacity=1 ]   (164,1791.88) .. controls (167.25,1795.63) and (171.75,1813.63) .. (179.63,1801.86) ;
	\draw [color={rgb, 255:red, 0; green, 0; blue, 0 }  ,draw opacity=1 ]   (171.5,1845.63) .. controls (170.25,1828.63) and (184.75,1822.63) .. (181.36,1801.91) ;
	\draw [color={rgb, 255:red, 0; green, 0; blue, 0 }  ,draw opacity=1 ]   (206.14,1845.95) .. controls (200.49,1840.72) and (196.14,1813.45) .. (185.89,1801.87) ;
	\draw [color={rgb, 255:red, 0; green, 0; blue, 0 }  ,draw opacity=1 ]   (202.5,1850.13) .. controls (196.47,1847.07) and (193.75,1819.88) .. (183.5,1802.05) ;
	\draw [color={rgb, 255:red, 0; green, 0; blue, 0 }  ,draw opacity=1 ]   (269,1839.13) .. controls (291.52,1819.96) and (279.83,1772.4) .. (291.46,1787.03) .. controls (293.26,1789.29) and (295.62,1793.05) .. (298.75,1798.63) ;
	\draw  [color={rgb, 255:red, 0; green, 0; blue, 0 }  ,draw opacity=1 ][fill={rgb, 255:red, 255; green, 255; blue, 255 }  ,fill opacity=1 ] (294.56,1798.55) -- (304.06,1798.55) -- (304.06,1801.53) -- (294.56,1801.53) -- cycle ;
	\draw [color={rgb, 255:red, 0; green, 0; blue, 0 }  ,draw opacity=1 ]   (287.75,1845.38) .. controls (286.5,1828.38) and (301,1822.38) .. (297.61,1801.66) ;
	\draw [color={rgb, 255:red, 0; green, 0; blue, 0 }  ,draw opacity=1 ]   (322.39,1845.7) .. controls (316.74,1840.47) and (312.39,1813.2) .. (302.14,1801.62) ;
	\draw [color={rgb, 255:red, 0; green, 0; blue, 0 }  ,draw opacity=1 ]   (318.75,1849.88) .. controls (312.72,1846.82) and (310,1819.63) .. (299.75,1801.8) ;
	\draw [color={rgb, 255:red, 74; green, 144; blue, 226 }  ,draw opacity=1 ][fill={rgb, 255:red, 74; green, 144; blue, 226 }  ,fill opacity=1 ]   (275.77,1774.52) -- (303.16,1816.31) ;
	\draw [color={rgb, 255:red, 74; green, 144; blue, 226 }  ,draw opacity=1 ][fill={rgb, 255:red, 74; green, 144; blue, 226 }  ,fill opacity=1 ]   (284.57,1770.75) -- (311.74,1812.19) ;
	\draw  [color={rgb, 255:red, 74; green, 144; blue, 226 }  ,draw opacity=1 ] (303.16,1816.31) .. controls (302.75,1815.69) and (304.34,1814.26) .. (306.71,1813.12) .. controls (309.08,1811.99) and (311.33,1811.57) .. (311.74,1812.19) .. controls (312.14,1812.81) and (310.55,1814.24) .. (308.18,1815.38) .. controls (305.82,1816.52) and (303.56,1816.93) .. (303.16,1816.31) -- cycle ;
	\draw  [color={rgb, 255:red, 74; green, 144; blue, 226 }  ,draw opacity=1 ] (275.99,1774.87) .. controls (275.59,1774.25) and (277.18,1772.82) .. (279.54,1771.68) .. controls (281.91,1770.54) and (284.16,1770.12) .. (284.57,1770.75) .. controls (284.98,1771.37) and (283.39,1772.79) .. (281.02,1773.93) .. controls (278.65,1775.07) and (276.4,1775.49) .. (275.99,1774.87) -- cycle ;
	\draw  [color={rgb, 255:red, 255; green, 255; blue, 255 }  ,draw opacity=1 ][line width=3] [line join = round][line cap = round] (450.56,1769.31) .. controls (450.75,1770.53) and (451.25,1771.71) .. (451.78,1772.88) ;
	\draw    (498.17,1768.59) .. controls (494,1780.89) and (503.5,1843.47) .. (499.15,1853.34) ;
	\draw    (406.89,1768.83) .. controls (405.21,1788.76) and (405.6,1842.88) .. (412.46,1853.52) ;
	\draw    (406.89,1768.83) .. controls (407.27,1759.1) and (498.39,1759.49) .. (498.17,1768.59) ;
	\draw [color={rgb, 255:red, 155; green, 155; blue, 155 }  ,draw opacity=1 ] [dash pattern={on 0.84pt off 2.51pt}]  (406.89,1768.83) .. controls (410.08,1773.61) and (436.43,1772.41) .. (453.57,1772.41) .. controls (470.71,1772.41) and (497.87,1772.53) .. (497.97,1767.75) ;
	\draw    (412.46,1853.52) .. controls (419.2,1863.03) and (498.91,1863.08) .. (499.15,1853.34) ;
	\draw [color={rgb, 255:red, 155; green, 155; blue, 155 }  ,draw opacity=1 ] [dash pattern={on 0.84pt off 2.51pt}]  (414.66,1854.49) .. controls (415.6,1854.11) and (437.69,1853.73) .. (454.15,1853.73) .. controls (470.6,1853.73) and (489.13,1853.55) .. (498.34,1854.87) ;
	\draw [color={rgb, 255:red, 0; green, 0; blue, 0 }  ,draw opacity=1 ]   (425.1,1837.53) .. controls (447.62,1818.36) and (424.2,1753) .. (444.2,1781.4) .. controls (464.2,1809.8) and (451.72,1791.45) .. (454.85,1797.03) ;
	\draw  [color={rgb, 255:red, 0; green, 0; blue, 0 }  ,draw opacity=1 ][fill={rgb, 255:red, 255; green, 255; blue, 255 }  ,fill opacity=1 ] (450.66,1796.95) -- (460.16,1796.95) -- (460.16,1799.93) -- (450.66,1799.93) -- cycle ;
	\draw [color={rgb, 255:red, 0; green, 0; blue, 0 }  ,draw opacity=1 ]   (443.85,1843.78) .. controls (442.6,1826.78) and (457.1,1820.78) .. (453.71,1800.06) ;
	\draw [color={rgb, 255:red, 0; green, 0; blue, 0 }  ,draw opacity=1 ]   (478.49,1844.1) .. controls (472.84,1838.87) and (468.49,1811.6) .. (458.24,1800.02) ;
	\draw [color={rgb, 255:red, 0; green, 0; blue, 0 }  ,draw opacity=1 ]   (474.85,1848.28) .. controls (468.82,1845.22) and (466.1,1818.03) .. (455.85,1800.2) ;
	\draw [dashed, color={rgb, 255:red, 208; green, 2; blue, 27 }  ,draw opacity=0.22 ][line width=1.5]    (408.5,1771.7) -- (492.83,1771.7) ;
	\draw [dashed, color={rgb, 255:red, 208; green, 2; blue, 27 }  ,draw opacity=0.22 ][line width=1.5]    (408.9,1782.9) -- (452.7,1782.9) -- (493.23,1782.9) ;
	\draw [dashed, color={rgb, 255:red, 208; green, 2; blue, 27 }  ,draw opacity=0.22 ][line width=1.5]    (253.3,1783.7) -- (297.1,1783.7) -- (337.63,1783.7) ;
	\draw [dashed, color={rgb, 255:red, 208; green, 2; blue, 27 }  ,draw opacity=0.22 ][line width=1.5]    (253.67,1774.01) -- (297.47,1774.01) -- (338,1774.01) ;
	\draw  [color={rgb, 255:red, 255; green, 255; blue, 255 }  ,draw opacity=1 ][line width=3] [line join = round][line cap = round] (450.47,1769.71) .. controls (450.66,1770.93) and (451.16,1772.11) .. (451.69,1773.28) ;
	\draw [color={rgb, 255:red, 74; green, 144; blue, 226 }  ,draw opacity=1 ][fill={rgb, 255:red, 74; green, 144; blue, 226 }  ,fill opacity=1 ]   (431.77,1773.32) -- (459.17,1815.11) ;
	\draw [color={rgb, 255:red, 74; green, 144; blue, 226 }  ,draw opacity=1 ][fill={rgb, 255:red, 74; green, 144; blue, 226 }  ,fill opacity=1 ]   (440.58,1769.55) -- (467.75,1810.99) ;
	\draw  [color={rgb, 255:red, 74; green, 144; blue, 226 }  ,draw opacity=1 ] (459.17,1815.11) .. controls (458.76,1814.49) and (460.35,1813.06) .. (462.72,1811.92) .. controls (465.09,1810.79) and (467.34,1810.37) .. (467.75,1810.99) .. controls (468.15,1811.61) and (466.56,1813.04) .. (464.19,1814.18) .. controls (461.82,1815.32) and (459.57,1815.73) .. (459.17,1815.11) -- cycle ;
	\draw  [color={rgb, 255:red, 74; green, 144; blue, 226 }  ,draw opacity=1 ] (432,1773.67) .. controls (431.59,1773.05) and (433.18,1771.62) .. (435.55,1770.48) .. controls (437.92,1769.34) and (440.17,1768.92) .. (440.58,1769.55) .. controls (440.99,1770.17) and (439.4,1771.59) .. (437.03,1772.73) .. controls (434.66,1773.87) and (432.41,1774.29) .. (432,1773.67) -- cycle ;
	
	\draw (113,1805) node [anchor=north west][inner sep=0.75pt]    {$=$};
	\draw (230,1805) node [anchor=north west][inner sep=0.75pt]    {$=$};
	\draw (370,1805) node [anchor=north west][inner sep=0.75pt]    {$=$};
	\draw (345,1770) node [anchor=north west][inner sep=0.75pt]  [color={rgb, 255:red, 128; green, 128; blue, 128 }  ,opacity=1 ]  {$\Sigma \times [ 0,1]$};
	\draw (500,1770) node [anchor=north west][inner sep=0.75pt]  [color={rgb, 255:red, 128; green, 128; blue, 128 }  ,opacity=1 ]  {$\Sigma \times [ 0,1]$};
	\draw (478,1807.17) node [anchor=north west][inner sep=0.75pt]  [color={rgb, 255:red, 128; green, 128; blue, 128 }  ,opacity=1 ]  {$\tilde{T}$};
	\draw (478,1770) node [anchor=north west][inner sep=0.75pt]  [color={rgb, 255:red, 128; green, 128; blue, 128 }  ,opacity=1 ]  {$T_{1}$};

\end{tikzpicture}		}
	\end{equation}
\end{proof}

\subsection{Proof of Proposition \ref{comm of differential}}

\begin{proof}[Proof of Proposition \ref{comm of differential}]
	By applying the same argument in each connected component, it is enough to prove the result for $M$ connected. (See alternatively Proposition \ref{prop:zast_sym} below.)
	
	Assume first that $M$ has non-empty boundary. We prove that $Z^*(M,T)$ is in $\Ch(\opn{Vect})$ by induction on the number of coupons in $T$. If $T$ has one coupon, this follows from Corollary \ref{corollary from lemma 2}. Suppose then the statement is true whenever the ribbon graph has at most $l$ coupons, and let $(M,T)$ such that $T$ has exactly $l+1$ coupons. We can choose one of the coupons in $T$ and slide it (thorugh a small enough tubular neighborhood) into the boundary collar of $M$, see Picture \eqref{eq:coupon}. Then we can ``cut" the bordism into $(M,T)=(M_2,T_2)(M_1,T_1)$, where  $(M_1, T_1)$ has underlying manifold $\Sigma \times [0,1]$ and a unique coupon, and $(M_2, T_2)$ is obtained from $(M,T)$ by removing a coupon via this procedure. By the inductive step, $Z^*(M_2,T_2)$ is in $\Ch(\opn{Vect})$, and so is $Z^*(M_1,T_1)$ by Corollary \ref{corollary from lemma 2}. Hence the statement also holds true for $(M,T)$, as desired. 
	\begin{equation}\label{eq:coupon}
		\scalebox{0.9}{
		\tikzset{every picture/.style={line width=0.5pt}} 
		
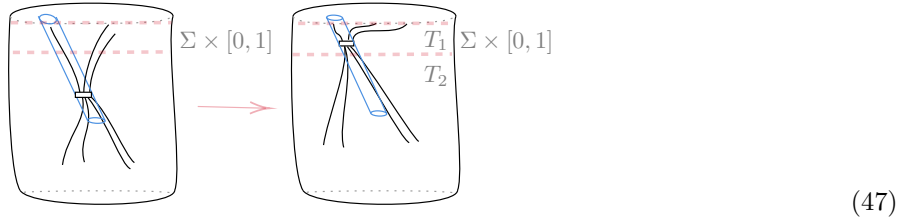
		\begin{tikzpicture}[x=0.75pt,y=0.75pt,yscale=-1,xscale=1]
			
			\draw  [color={rgb, 255:red, 255; green, 255; blue, 255 }  ,draw opacity=1 ][line width=3] [line join = round][line cap = round] (215.15,1986.02) .. controls (215.34,1987.31) and (215.84,1988.56) .. (216.37,1989.79) ;
			\draw  [color={rgb, 255:red, 255; green, 255; blue, 255 }  ,draw opacity=1 ][line width=3] [line join = round][line cap = round] (163.06,1884.15) .. controls (163.25,1885.44) and (163.75,1886.69) .. (164.28,1887.92) ;
			\draw    (210.67,1883.4) .. controls (206.5,1897.55) and (216,1969.55) .. (211.65,1980.91) ;
			\draw    (119.39,1883.65) .. controls (117.71,1905.4) and (118.1,1969.48) .. (124.96,1981.1) ;
			\draw    (119.39,1883.65) .. controls (119.77,1873.4) and (210.89,1873.81) .. (210.67,1883.4) ;
			\draw [color={rgb, 255:red, 155; green, 155; blue, 155 }  ,draw opacity=1 ] [dash pattern={on 0.84pt off 2.51pt}]  (119.39,1883.65) .. controls (122.58,1888.69) and (148.93,1887.42) .. (166.07,1887.42) .. controls (183.21,1887.42) and (210.37,1887.55) .. (210.47,1882.51) ;
			\draw    (124.96,1981.1) .. controls (131.7,1991.48) and (211.41,1991.54) .. (211.65,1980.91) ;
			\draw [color={rgb, 255:red, 155; green, 155; blue, 155 }  ,draw opacity=1 ] [dash pattern={on 0.84pt off 2.51pt}]  (127.16,1982.16) .. controls (128.1,1981.75) and (150.19,1981.33) .. (166.65,1981.33) .. controls (183.1,1981.33) and (201.63,1981.14) .. (210.84,1982.58) ;
			\draw  [color={rgb, 255:red, 255; green, 255; blue, 255 }  ,draw opacity=1 ][line width=3] [line join = round][line cap = round] (373.65,1986.02) .. controls (373.84,1987.31) and (374.34,1988.56) .. (374.87,1989.79) ;
			\draw  [color={rgb, 255:red, 255; green, 255; blue, 255 }  ,draw opacity=1 ][line width=3] [line join = round][line cap = round] (321.56,1884.15) .. controls (321.75,1885.44) and (322.25,1886.69) .. (322.78,1887.92) ;
			\draw    (369.17,1883.4) .. controls (365,1897.55) and (374.5,1969.55) .. (370.15,1980.91) ;
			\draw    (277.89,1883.65) .. controls (276.21,1905.4) and (276.6,1969.48) .. (283.46,1981.1) ;
			\draw    (277.89,1883.65) .. controls (278.27,1873.4) and (369.39,1873.81) .. (369.17,1883.4) ;
			\draw [color={rgb, 255:red, 155; green, 155; blue, 155 }  ,draw opacity=1 ] [dash pattern={on 0.84pt off 2.51pt}]  (277.89,1883.65) .. controls (281.08,1888.69) and (307.43,1887.42) .. (324.57,1887.42) .. controls (341.71,1887.42) and (368.87,1887.55) .. (368.97,1882.51) ;
			\draw    (283.46,1981.1) .. controls (290.2,1991.48) and (369.91,1991.54) .. (370.15,1980.91) ;
			\draw [color={rgb, 255:red, 155; green, 155; blue, 155 }  ,draw opacity=1 ] [dash pattern={on 0.84pt off 2.51pt}]  (285.66,1982.16) .. controls (286.6,1981.75) and (308.69,1981.33) .. (325.15,1981.33) .. controls (341.6,1981.33) and (360.13,1981.14) .. (369.34,1982.58) ;
			\draw [color={rgb, 255:red, 208; green, 2; blue, 27 }  ,draw opacity=0.32 ]   (224.5,1934.03) -- (262.75,1934.99) ;
			\draw [shift={(264.75,1935.03)}, rotate = 181.42] [color={rgb, 255:red, 208; green, 2; blue, 27 }  ,draw opacity=0.32 ][line width=0.75]    (10.93,-3.29) .. controls (6.95,-1.4) and (3.31,-0.3) .. (0,0) .. controls (3.31,0.3) and (6.95,1.4) .. (10.93,3.29)   ;
			\draw [color={rgb, 255:red, 0; green, 0; blue, 0 }  ,draw opacity=1 ]   (148.71,1966.81) .. controls (151.09,1950.41) and (162.4,1942.22) .. (159.94,1929.13) ;
			\draw [color={rgb, 255:red, 0; green, 0; blue, 0 }  ,draw opacity=1 ]   (161.02,1965.26) .. controls (157.03,1951.14) and (167.66,1946.16) .. (161.3,1928.96) ;
			\draw [color={rgb, 255:red, 0; green, 0; blue, 0 }  ,draw opacity=1 ]   (189,1965.53) .. controls (183.52,1961.19) and (175.24,1938.54) .. (164.95,1928.93) ;
			\draw [color={rgb, 255:red, 0; green, 0; blue, 0 }  ,draw opacity=1 ]   (186.8,1969) .. controls (181.4,1966.47) and (174.44,1943.88) .. (163.05,1929.08) ;
			\draw [color={rgb, 255:red, 0; green, 0; blue, 0 }  ,draw opacity=1 ]   (142.07,1889.45) .. controls (144.03,1896.52) and (156.5,1911.49) .. (158.45,1926.33) ;
			\draw [color={rgb, 255:red, 0; green, 0; blue, 0 }  ,draw opacity=1 ]   (178.28,1893.25) .. controls (170.47,1902) and (164.26,1914.23) .. (164.01,1925.84) ;
			\draw [color={rgb, 255:red, 74; green, 144; blue, 226 }  ,draw opacity=1 ][fill={rgb, 255:red, 74; green, 144; blue, 226 }  ,fill opacity=1 ]   (135.18,1883.41) -- (162.81,1941.9) ;
			\draw [color={rgb, 255:red, 74; green, 144; blue, 226 }  ,draw opacity=1 ][fill={rgb, 255:red, 74; green, 144; blue, 226 }  ,fill opacity=1 ]   (145.16,1883.9) -- (172.56,1941.9) ;
			\draw  [color={rgb, 255:red, 74; green, 144; blue, 226 }  ,draw opacity=1 ] (162.81,1941.9) .. controls (162.4,1941.03) and (164.25,1940.32) .. (166.94,1940.32) .. controls (169.63,1940.32) and (172.15,1941.03) .. (172.56,1941.9) .. controls (172.97,1942.77) and (171.12,1943.47) .. (168.43,1943.47) .. controls (165.74,1943.47) and (163.22,1942.77) .. (162.81,1941.9) -- cycle ;
			\draw  [color={rgb, 255:red, 74; green, 144; blue, 226 }  ,draw opacity=1 ] (135.93,1885.04) .. controls (135.26,1883.58) and (136.77,1882.4) .. (139.31,1882.4) .. controls (141.85,1882.4) and (144.45,1883.58) .. (145.12,1885.04) .. controls (145.79,1886.49) and (144.28,1887.67) .. (141.74,1887.67) .. controls (139.2,1887.67) and (136.6,1886.49) .. (135.93,1885.04) -- cycle ;
			\draw [color={rgb, 255:red, 0; green, 0; blue, 0 }  ,draw opacity=1 ]   (174.62,1888.15) .. controls (166.8,1896.9) and (160.91,1914.27) .. (160.66,1925.89) ;
			\draw  [color={rgb, 255:red, 0; green, 0; blue, 0 }  ,draw opacity=1 ][fill={rgb, 255:red, 255; green, 255; blue, 255 }  ,fill opacity=1 ] (155.95,1926.04) -- (164.95,1926.04) -- (164.95,1928.93) -- (155.95,1928.93) -- cycle ;
			\draw [color={rgb, 255:red, 0; green, 0; blue, 0 }  ,draw opacity=1 ]   (295,1955.67) .. controls (297.15,1940.49) and (308.99,1912.48) .. (306.77,1900.36) ;
			\draw [color={rgb, 255:red, 0; green, 0; blue, 0 }  ,draw opacity=1 ]   (306.33,1957) .. controls (302.73,1943.93) and (312.53,1916.29) .. (306.77,1900.36) ;
			\draw [color={rgb, 255:red, 0; green, 0; blue, 0 }  ,draw opacity=1 ]   (348.33,1953.67) .. controls (343.39,1949.65) and (319.37,1909.24) .. (310.07,1900.34) ;
			\draw [color={rgb, 255:red, 0; green, 0; blue, 0 }  ,draw opacity=1 ]   (343,1954.33) .. controls (338.12,1951.99) and (318.64,1914.17) .. (308.36,1900.47) ;
			\draw [color={rgb, 255:red, 0; green, 0; blue, 0 }  ,draw opacity=1 ]   (300.6,1887.8) .. controls (301.8,1896.2) and (303.8,1891) .. (305.54,1897.26) ;
			\draw [color={rgb, 255:red, 0; green, 0; blue, 0 }  ,draw opacity=1 ]   (341.72,1887.72) .. controls (334.66,1895.81) and (311.62,1886.65) .. (311.4,1897.4) ;
			\draw [color={rgb, 255:red, 0; green, 0; blue, 0 }  ,draw opacity=1 ]   (324.57,1887.42) .. controls (317.4,1891.4) and (309.22,1885.85) .. (309,1896.6) ;
			\draw  [color={rgb, 255:red, 0; green, 0; blue, 0 }  ,draw opacity=1 ][fill={rgb, 255:red, 255; green, 255; blue, 255 }  ,fill opacity=1 ] (303.83,1897.4) -- (311.96,1897.4) -- (311.96,1900.07) -- (303.83,1900.07) -- cycle ;
			\draw [color={rgb, 255:red, 74; green, 144; blue, 226 }  ,draw opacity=1 ][fill={rgb, 255:red, 74; green, 144; blue, 226 }  ,fill opacity=1 ]   (296.59,1883.81) -- (321.55,1937.94) ;
			\draw [color={rgb, 255:red, 74; green, 144; blue, 226 }  ,draw opacity=1 ][fill={rgb, 255:red, 74; green, 144; blue, 226 }  ,fill opacity=1 ]   (305.6,1884.26) -- (330.36,1937.94) ;
			\draw  [color={rgb, 255:red, 74; green, 144; blue, 226 }  ,draw opacity=1 ] (321.55,1937.94) .. controls (321.18,1937.14) and (322.85,1936.48) .. (325.28,1936.48) .. controls (327.71,1936.48) and (329.99,1937.14) .. (330.36,1937.94) .. controls (330.73,1938.75) and (329.06,1939.4) .. (326.63,1939.4) .. controls (324.19,1939.4) and (321.92,1938.75) .. (321.55,1937.94) -- cycle ;
			\draw  [color={rgb, 255:red, 74; green, 144; blue, 226 }  ,draw opacity=1 ] (296.8,1884.26) .. controls (296.43,1883.46) and (298.1,1882.8) .. (300.53,1882.8) .. controls (302.96,1882.8) and (305.23,1883.46) .. (305.6,1884.26) .. controls (305.98,1885.07) and (304.31,1885.72) .. (301.87,1885.72) .. controls (299.44,1885.72) and (297.17,1885.07) .. (296.8,1884.26) -- cycle ;
			\draw [dashed, color={rgb, 255:red, 208; green, 2; blue, 27 }  ,draw opacity=0.22 ][line width=1.5]    (120.9,1903.7) -- (164.7,1903.7) -- (209,1903.4) ;
			\draw [dashed, color={rgb, 255:red, 208; green, 2; blue, 27 }  ,draw opacity=0.22 ][line width=1.5]    (119.87,1887.02) -- (163.67,1887.02) -- (209.8,1887) ;
			\draw [dashed, color={rgb, 255:red, 208; green, 2; blue, 27 }  ,draw opacity=0.22 ][line width=1.5]    (280.77,1887.42) -- (324.57,1887.42) -- (367.4,1887.4) ;
			\draw [dashed, color={rgb, 255:red, 208; green, 2; blue, 27 }  ,draw opacity=0.22 ][line width=1.5]    (277.7,1904.9) -- (321.5,1904.9) -- (367,1904.6) ;
			
			\draw (213,1890) node [anchor=north west][inner sep=0.75pt]  [color={rgb, 255:red, 128; green, 128; blue, 128 }  ,opacity=1 ]  {$\Sigma \times [ 0,1]$};
			\draw (370,1890) node [anchor=north west][inner sep=0.75pt]  [color={rgb, 255:red, 128; green, 128; blue, 128 }  ,opacity=1 ]  {$\Sigma \times [ 0,1]$};

\draw (350,1890) node [anchor=north west][inner sep=0.75pt]  [color={rgb, 255:red, 128; green, 128; blue, 128 }  ,opacity=1 ]  {$T_1$};			
\draw (350,1910) node [anchor=north west][inner sep=0.75pt]  [color={rgb, 255:red, 128; green, 128; blue, 128 }  ,opacity=1 ]  {$T_2$};
			
		\end{tikzpicture}
		}
	\end{equation}
	
It remains to look at the case where $(M,T)$ has empty boundary.  As discussed earlier, $Z^\ast(M,T)$ is an endomorphisms of the complex $k$ in this case.  Since $k$ is concentrated in degree $0$, and has vanishing differential, any such linear endomorphism is a cochain map.
\end{proof}

\subsection{$Z^\ast$ is a cochain valued functor}

As an immediate corollary to Proposition \ref{comm of differential} we find that the cochain complexes $Z^\ast(\Sigma_{\vec{x}})$ and homogeneous maps $Z^\ast(M,T)$ associated to surfaces and bordisms in $\Bord_{\opn{Ch}(\msc{A})}$ assemble into a well-defined cochain valued functor.

\begin{corollary}\label{cor:zast_fun}
Given a ribbon tensor category $\msc{A}$ and a reasonable theory $Z:\Bord_{\msc{A}}\to \opn{Vect}$, the constructions from Definition \ref{def:zast} provide a well-defined functor
\[
Z^\ast:\Bord_{\opn{Ch}(\msc{A})}\to \opn{Ch}(\opn{Vect}).
\]
\end{corollary}

\section{Symmetric monoidality, state spaces, and partition functions}
\label{sect:states_parts}

We have already seen that the theory $Z^\ast$ is well-defined as a functor.  We verify symmetric monoidality below, thus completing the proof of Theorem \ref{thm: Zast}.  We then provide an explicit description of the state spaces and partition functions in our cochain variant of the universal Lyubashenko theory.

\subsection{Symmetric monoidality}

We now have that any reasonable theory $Z:\Bord_{\msc{A}}\to \opn{Vect}$ defines a cochain valued functor $Z^\ast:\Bord_{\opn{Ch}(\msc{A})}\to \opn{Ch}(\msc{A})$ which fits into a diagram over the graded TQFT
\[
\xymatrix{
\Bord_{\opn{Ch}(\msc{A})}\ar[r]^{Z^\ast}\ar[d]_{forget} & \opn{Ch}(\msc{A})\ar[d]^{forget}\\
\Bord_{\msc{A}^{\mathbb{Z}}}\ar[r]_{\underline{Z}} &\opn{Vect}^{\mathbb{Z}}.
}
\]
Since the forgetful functors are faithful, we find that the functor $Z^\ast$ inherits a symmetric monoidal structure from that of $\underline{Z}$ provided the monoidal structure maps for $\underline{Z}$ are in fact maps of cochains, whenever one evaluates at disjoint unions of $\opn{Ch}(\msc{A})$-marked surfaces.

\begin{proposition}\label{prop:zast_sym}
For any reasonable theory $Z$, and surface $\Sigma_{\vec{x}}$ with a decomposition $\Sigma^0_{\vec{v}}\amalg\Sigma^1_{\vec{w}}\overset{\sim}\to \Sigma_{\vec{x}}$ in $\Bord_{\opn{Ch}(\msc{A})}$, the homogenous isomorphism
\[
F:Z^{\ast}(\Sigma^0_{\vec{v}})\ot Z^{\ast}(\Sigma^1_{\vec{w}})\overset{\sim}\to Z^{\ast}(\Sigma_{\vec{x}})
\]
defined via the symmetric monoidal structure on $\underline{Z}$ is an isomorphism of cochains.
\end{proposition}

To clarify what's going on here, we have a marked surface $\Sigma_{\vec{x}}$ with a decomposition $\Sigma^0\amalg\Sigma^1\overset{\sim}\to\Sigma$.  Pulling back the markings $I\to \Sigma$ along the maps from the $\Sigma^{\varepsilon}$ provides, by assumption, an ordered splitting the set $I^0\amalg I^1=I$ and respective markings $I^{\varepsilon}\to \Sigma^{\varepsilon}$.  The marking objects $\vec{v}$ and $\vec{w}$ for $\Sigma^0$ and $\Sigma^1$ are similarly obtained via pullback, or restriction if one prefers.

\begin{proof}
We may take $I=\{1,\dots, n\}$ and, for the given splitting $I=I^0\amalg I^1$ as above, take $t=\opn{max}(I^0)$.  Hence $I^0=\{1,\dots,t\}$ and $I^1=\{t+1,\dots,n\}$.  Suppose for the moment that $\Sigma_{\vec{x}}$ is positively marked as well.  For an ordered collection of integers $m:A\to \mathbb{Z}$, and $r$ in $A$, we let $D_l=D_l(m)$ denote the morphisms
\[
\scalebox{.9}{
\tikzset{every picture/.style={line width=0.5pt}} 
\begin{tikzpicture}[x=0.75pt,y=0.75pt,yscale=-1,xscale=1]
\draw    (302.92,26.83) -- (302.92,97.8) ;
\draw   (242.6,56.82) -- (253.81,56.82) -- (253.81,60.09) -- (242.6,60.09) -- cycle ;
\draw    (248.47,27.13) -- (248.47,56.52) ;
\draw    (251.14,59.79) -- (251.14,98.1) ;
\draw    (189.22,27.13) -- (189.22,72.56) ;
\draw    (189.22,72.56) .. controls (189.76,90.67) and (206.84,79.98) .. (206.3,98.1) ;
\draw    (186.55,98.39) .. controls (187.09,71.67) and (244.74,93.94) .. (245.27,59.79) ;

\draw (207,49.01) node [anchor=north west][inner sep=0.75pt] [align=left] {$\displaystyle \dotsc $};
\draw (271.24,49.01) node [anchor=north west][inner sep=0.75pt] [align=left] {$\displaystyle \dotsc $};
\draw (178.52,16) node [anchor=north west][inner sep=0.75pt]  [font=\tiny] [align=left] {$\displaystyle k(m_{a_{1}})$};
\draw (240.01,16) node [anchor=north west][inner sep=0.75pt]  [font=\tiny] [align=left] {$\displaystyle k(m_{l})$};
\draw (292.22,16) node [anchor=north west][inner sep=0.75pt]  [font=\tiny] [align=left] {$\displaystyle k(m_{a_{r}})$};
\draw (235.91,100.) node [anchor=north west][inner sep=0.75pt]  [font=\tiny] [align=left] {$\displaystyle k(m_{l} +1)$};
\draw (168,100.) node [anchor=north west][inner sep=0.75pt]  [font=\tiny] [align=left] {$\displaystyle k(-1)$};
\draw (195.07,100.) node [anchor=north west][inner sep=0.75pt]  [font=\tiny] [align=left] {$\displaystyle k(m_{a_{1}})$};
\draw (291.69,100) node [anchor=north west][inner sep=0.75pt]  [font=\tiny] [align=left] {$\displaystyle k(m_{a_{r}})$};

\draw (125,50) node [anchor=north west][inner sep=0.75pt]  [font=\normalsize] [align=left] {$D_l=$};

\end{tikzpicture}}
\]
in $k\opn{Str}^{\red}_{\opn{Vect}^{\mathbb{Z}}}$.  Then we have identifications
\[
Z^\ast(\Sigma_{\vec{x}})=\oplus_{m:I\to \mathbb{Z}} k(m)\ot Z(\Sigma_{\vec{x}^m})
\]
and
\[
\begin{array}{rl}
Z^\ast(\Sigma_{\vec{v}}^0)\ot Z^{\ast}(\Sigma_{\vec{w}}^1) & = \left(\oplus_{m^0:I^0\to \mathbb{Z}} k(m^0)\ot Z(\Sigma_{\vec{v}^{m^0}}^0)\right)\ot \left(\oplus_{m^1:I^1\to\mathbb{Z}}k(m^1)\ot Z(\Sigma_{\vec{w}^{m^1}}^1)\right)\vspace{2mm}\\
& \cong \oplus_{m:I\to \mathbb{Z}}k(m)\ot Z(\Sigma^0_{\vec{v}^{m|I^0}})\ot Z(\Sigma^1_{\vec{w}^{m|I^1}}),
\end{array}
\]
where $k(m)$ etc.\ are the apparent products in $\opn{Vect}^{\mathbb{Z}}$.  Under these identifications the proposed structure map $F:Z^\ast\ot Z^\ast\to Z^\ast$ is simply induced by that of $Z$,
\begin{equation}\label{eq:7444}
F=\oplus_m id_{k(m)}\ot F_Z,
\end{equation}
and the differentials appear as
\[
d_{Z^\ast(\Sigma_{\vec{x}})}=\sum_m\sum_i ev(D_i)\ot Z(\Sigma_{d_{x_i}^{m_i}})
\]
and
\[
 d_{Z^\ast(\Sigma^0_{\vec{v}})\ot Z^\ast(\Sigma^1_{\vec{w}})}=\sum_m\left(\sum_{i\leq t} ev(D_i)\ot Z(\Sigma^0_{d_{x_i}^{m_i}})\ot Z(id)+\sum_{j>t}ev(D_j)\ot Z(id)\ot Z(\Sigma^0_{d_{x_j}^{m_j}})\right).
\]
\par

Via monoidality of $Z$ we have
\[
F_Z\left(Z(\Sigma^0_{d_{x_i}^{m_i}})\ot Z(id)\right)=Z(\Sigma_{d_{x_i}^{m_i}})F_Z\ \ \text{and}\ \ F_Z\left(Z(id)\ot Z(\Sigma^1_{d_{x_j}^{m_j}})\right)=Z(\Sigma_{d_{x_j}^{m_j}})F_Z.
\]
From these equalities, the above expressions for the differentials, and the formula \eqref{eq:7444} we obtain the desired compatibility with the differential
\[
F[1]\ d_{Z^\ast(\Sigma^0_{\vec{v}})\ot Z^\ast(\Sigma^1_{\vec{w}})}=d_{Z^\ast(\Sigma_{\vec{x}})}\ F.
\]
\par

For general $\Sigma_{\vec{x}}$ we note that the flip map from Section \ref{sect:diff_def} splits over disjoint unions
\[
\Sigma^{flip}_{\vec{x}}=(\Sigma^0_{\vec{v}})^{flip}\amalg (\Sigma^1_{\vec{w}})^{flip}.
\]
We therefore have a diagram of homogeneous morphisms
\[
\xymatrix{
Z^\ast(\Sigma^0_{\vec{v}})\ot Z^\ast(\Sigma^1_{\vec{w}})\ar[rr]^{F}\ar[d]_{Z^{\ast}(flip)\ot Z^{\ast}(flip)} & & Z^\ast(\Sigma_{\vec{x}})\ar[d]^{Z^\ast(flip)}\\
Z^\ast(\Sigma^{0+}_{\vec{v}})\ot Z^\ast(\Sigma^{1+}_{\vec{w}})\ar[rr]^{F} & & Z^\ast(\Sigma^+_{\vec{x}})
}
\]
in which all maps, save for the top one, are known to be cochain isomorphisms.  It follows that the top map is a cochain isomorphism as well.
\end{proof}

\subsection{Proof of Theorem \ref{thm: Zast}}
\label{sect:zast_proof}

\begin{proof}[Proof of Theorem \ref{thm: Zast}]
The theory $Z^\ast$ is well-defined as a functor by Corollary \ref{cor:zast_fun}, fits into the proposed diagram over $\underline{Z}$ by construction, and inherits its symmetric monoidal structure from that of $\underline{Z}$ by Proposition \ref{prop:zast_sym}.  The fact that the restriction $Z^\ast|_{\Bord_{\msc{A}}}$ recovers the original theory follows from the corresponding property for $\underline{Z}$ (see Theorem \ref{thm:Z_AS}).  
\end{proof}

\subsection{State spaces for $Z^*_{\msc{A}}$}\label{sect:zch_states}

As stated earlier, the entire analysis above applies to reasonable theories from the admissible category of $\msc{A}$-labeled bordisms as well.  We therefore have the cochain valued symmetric monoidal functor
\[
Z_{\msc{A}}^{\ast}:\Bord_{\opn{Ch}(\msc{A})}^{\opn{adm}}\to \opn{Ch}(\opn{Vect})
\]
associated to the universal Lyubashenko theory $Z_{\msc{A}}:\Bord^{\opn{adm}}_{\msc{A}}\to \opn{Vect}$, at any given modular tensor category $\msc{A}$.

Let $\Sigma_{\vec{x}}$ be a genus $g$ surface in $\Bord_{\Ch(\msc{A})}^{\opn{adm}}$ with markings from an
ordered set $I$. As stated in Theorem \ref{thm:ind_states}, there is a natural isomorphism \[\underline{Z}_{\msc{A}}(\Sigma_{\vec{x}})\overset{\sim}\to \underline{\Hom}(\1,E^{\ot g}\ot x_I).\]
Recall that $Z^*_{\msc{A}}(\Sigma_{\vec{x}})$ is the graded vector space $\underline{Z}_{\msc{A}}(\Sigma_{\vec{x}})$ together with the differential $d_{\Sigma_{\vec{x}}}$ as in Definition \ref{def: differential}. We show that the natural isomorphism above preserves the differential and thus extends to $Z^*_{\msc{A}}$.

Suppose that $\Sigma_{\vec{x}}$ is positively marked and consider the bordism $\Sigma_{\vec{d_{x_i}}}:\Sigma_{\vec{x}}\to \Sigma_{\vec{x}[\delta_i]}$ as defined in Section \ref{sect:a_sigma}, where $\vec{d_{x_i}}:= (1,\dots,d_{x_i},\dots,1)$ and $[\delta_i]$ denotes the shift $[1]$ applied at the $i$-th marking object.  By naturality, we have the following commutative diagram 
\[\begin{tikzcd}
	{\underline{Z}_{\msc{A}}(\Sigma_{\vec{x}})} && {\underline{\Hom}(\1,E^{\otimes g}\otimes x_I)} \\
	{\underline{Z}_{\msc{A}}( \Sigma_{\vec{x}[\delta_i]}} )&& {\underline{\Hom}(\1,E^{\otimes g}\otimes x_I[\delta_i])}.
	\arrow["\sim", from=1-1, to=1-3]
	\arrow["{\Sigma_{\vec{d_{x_i}}}}"', from=1-1, to=2-1]
	\arrow["(1\ot d_{x_i}\ot 1)_\ast", from=1-3, to=2-3]
	\arrow["\sim"', from=2-1, to=2-3]
\end{tikzcd}\]
We compose now degree by degree $\Sigma_{d_{x_i}^{m_i}}:\underline{Z}_{\msc{A}}(\Sigma_{\vec{x}^{m}}) \to
	\underline{Z}_{\msc{A}}(\Sigma_{\vec{x}[\delta_i]^{m}})$ with the appropriate sign map  $\underline{Z}_{\msc{A}}(\Sigma_{\vec{x}[\delta_i]^m}) \to \underline{Z}_{\msc{A}}(\Sigma_{\vec{x}^{m+\delta_i}})[1]$, namely, composition of the apparent linear identification with multiplication by $(-1)^{\sum_{j<i}m_j}$. Since the signs agree with those coming from the identification ${\underline{\Hom}(\1,E^{\otimes g}\otimes x_I[\delta_i]^m)}\to {\underline{\Hom}(\1,E^{\otimes g}\otimes x^{m+\delta_i}_I)}[1]$ for all $i$, we obtain the commutative diagram
\[\begin{tikzcd}
	{\underline{Z}_{\msc{A}}(\Sigma_{\vec{x}})} && {\underline{\Hom}(\1,E^{\otimes g}\otimes x_I)} \\
	{\underline{Z}_{\msc{A}}(\Sigma_{\vec{x}[\delta_i]})} && {\underline{\Hom}(\1,E^{\otimes g}\otimes x_I[\delta_i])} \\
	{\underline{Z}_{\msc{A}}(\Sigma_{\vec{x}})}[1] && {\underline{\Hom}(\1,E^{\otimes g}\otimes x_I)[1]}.
	\arrow["\sim", from=1-1, to=1-3]
	\arrow["\Sigma_{\vec{d_{x_i}}}"', from=1-1, to=2-1]
	\arrow["{(1\ot d_{x_i}\ot 1)_\ast}", from=1-3, to=2-3]
	\arrow["\sim", from=2-1, to=2-3]
	\arrow["{\text{sign}}"', from=2-1, to=3-1]
	\arrow["{\text{sign}}", from=2-3, to=3-3]
	\arrow["\sim", from=3-1, to=3-3]
\end{tikzcd}\]
Doing the sum over all $i\in I$ of the vertical compositions in the left (resp.\ right) results in the differential $d_{\Sigma_{\vec{x}}}$ (resp.\ $d_{\Hom^{\ast}_{\msc{A}}}$). It follows that we have a natural isomorphism of cochains
\begin{equation}\label{eq:1382}
	Z^*_{\msc{A}}(\Sigma_{\vec{x}})\overset{\sim}\to \Hom^\ast_{\msc{A}}(\1,E^{\ot g}\otimes(x_I)),
\end{equation}
as desired. A slightly more subtle analysis verifies the existence of such an isomorphism \eqref{eq:1382} in the case of a surface with general markings.

\subsection{Partition functions for $Z^\ast_{\msc{A}}$}

For $(M,T)$ a closed $3$-manifold in $\Bord_{\opn{Ch}(\msc{A})}$ we claim that the invariant $Z^\ast(M,T)$ is a type of alternating sum of $3$-manifold invariants $Z^{\ast}(M,T_\mu)$ for associated $\msc{A}$-labeled diagrams $T_\mu$.  This observation is a consequence of the following lemma.

\begin{lemma}
Consider any endomorphism $L:\1\to \1$ of the unit (unlabeled disk) in $\opn{Str}_{\opn{Vect}^{\mathbb{Z}}}$ in which all coupons are labeled by the tautological identification $k(m_1)\ot\dots k(m_t)\overset{\sim}\to k(n_1)\ot\dots k(n_r)$.  We have $ev(L)=\pm 1$.
\end{lemma}

\begin{proof}
We have the non-linear symmetric subcategory $\msc{V}\subseteq \opn{Vect}^{\mbb{Z}}$ whose objects are products $k(m_1)\ot\dots k(m_t)$ and whose morphisms are either empty or $\pm 1$ scalings of the tautological identification.  Note that any $L$ as in the statement is in the image of the functor $\opn{Str}_{\msc{V}}\to \opn{Str}_{\opn{Vect}^{\mathbb{Z}}}$.  The result now follows by commutativity of the diagram
\[
\xymatrix{
\opn{Str}_{\msc{V}}\ar[r]^{ev}\ar[d] & \msc{V}\ar[d]\\
\opn{Str}_{\opn{Vect}^{\mathbb{Z}}}\ar[r]_{ev} & \opn{Vect}^{\mathbb{Z}}.
}
\]
\end{proof}

Practically speaking, for $\msc{V}$ as in the above proof, the image of each bordism in $\Bord_{\opn{Ch}(\msc{A})}$ under our diagonal map $(p\boxtimes 1)\Delta(M,T)$ can be written as a sum of objects $(L_\mu,M,T_\mu)$ in the image of the functor
\[
\opn{Str}_{\msc{V}}\times \Bord_{\msc{A}}\to k\opn{Str}_{\opn{Vect}^{\mathbb{Z}}}^{red}\boxtimes k\Bord_{\msc{A}}^{\red}.
\]
This gives $Z^\ast(M,T)$ as the sum
\[
Z^\ast(M,T)=\sum_{\mu}ev_{\msc{V}}(L_\mu)Z(M,T_\mu)=\sum_{\mu}\pm Z(M,T_\mu).
\]

For concreteness, let us consider the special case of a knot in a closed $3$-manifold.  Take $S$ the graph with a single vertex $v$ and a single edge traveling from $v$ to itself, and consider such a knot $T:S\to M$ labeled by a complex $x$ in $\opn{Ch}(\msc{A})$ and $id_x$.  We have that the $k(m)$-labeled loop $L_m$ in $\opn{Str}_{\opn{Vect}^{\mathbb{Z}}}$ evaluates to the super-dimension $ev(L_m)=(-1)^m$.  Hence for $T_m:S\to M$ the corresponding $x^m$-labeling of the given knot we obtain
\[
Z^\ast(M,T)=\sum_m (-1)^mZ(M,T_m).
\]
We apply this general formula in the case of the universal Lyubashenko theory.

\begin{proposition}
Let $\msc{A}$ be a modular tensor category and $T:S\to M$ be a knot in a closed $3$-manifold, which we label by a bounded complex of projectives $x$ in $\opn{Ch}(\msc{A})$, as above.  Let $T_m:S\to M$ be the corresponding $x^m$-labeling of this knot.  (Note that the both $(M,T)$ and $(M,T_m)$ are admissible bordisms.)  Then we have
\[
Z^\ast_{\msc{A}}(M,T)=\sum_{m\in \mathbb{Z}} (-1)^mZ_{\msc{A}}(M,T_m),
\]
where $Z(M,T_m)$ is specifically the renormalized Lyubashenko invariant from \cite[Theorem 3.8]{derenzietal23}.
\end{proposition}

One can provide a similar formula for the invariant associated to an admissible link in a $3$-manifold.

\section{Homotopy equivalence and localization}
\label{sect:htop}

We prove that the cochain valued theory $Z^\ast:\Bord_{\opn{Ch}(\msc{A})}\to \opn{Ch}(\opn{Vect})$ preserves natural classes of equivalences.  If follows, as we explain below, that one can \emph{localize} $Z^\ast$ to obtain a TQFT which takes values in the $\infty$-category of homotopical vector spaces.  This localized theory takes the form of a symmetric monoidal functor from an $\infty$-category of, more-or-less, ribbon bordisms with labels in the homotopy $\infty$-category for $\msc{A}$.

\subsection{Homotopy equivalences on cochains}
Let $\msc{A}$ be a locally finite abelian category and $\Ch(\msc{A})$ the category of chain complexes in $\msc{A}$. Given two maps $f,g:x\to y$ in $\Ch(\msc{A})$, a \emph{cochain homotopy} $H:f\Rightarrow g$ is a map $H:x[1]\to y$ in $\msc{A}^{\mathbb Z}$ satisfying
\begin{equation}\label{eq:htop_formula}
	f-g = H d - (dH)[-1],
\end{equation}
where we are using implicitly the isomorphism 
\[
\Hom(x[1][-1],y[1][-1])\cong \Hom(x,y)
\]
provided by applications of the counit and unit for $k(-1)$ in order to identify $(dH)[-1]$ as a map from $x$ to $y$.

Note that the counit map $\1\to k(1)\ot k(-1)=k(-1)^\ast\ot k(-1)$ involves a negative sign, since we employ the pivotal structure to identify the left dual with a right dual.  Hence the specification of a map $H$ satisfying \eqref{eq:htop_formula} is the same information as a collection of morphisms 
\[
(H^n : x^n \to y^{n-1} : n\in \mathbb N )
\]
in $\msc{A}$ such that
\[
f^n - g^n = H^{n+1} d_x^n + d_y^{n-1} H^n
\]
for all $n\in \mathbb N$. In this case, we say $f$ and $g$ are \emph{chain homotopic}.  

\begin{definition}
We say a map $\mu:x\to y$ in $\opn{Ch}(\msc{A})$ is a homotopy equivalence if there exists another map $\eta:y \to x$ in $\Ch(\msc{A})$ such that $\eta \circ \mu$ and $\mu \circ \eta$ are homotopic to the identity of $x$ and $y$, respectively. 
\end{definition}

\subsection{Homotopy equivalences on bordisms}
We now define a class of equivalences in $\Bord_{\opn{Ch}(\msc{A})}$, and show that the symmetric monoidal functor $Z^*:\Bord_{\opn{Ch}(\msc{A})}\to \opn{Ch}(\opn{Vect})$ sends these to homotopy equivalences in $\Ch(\opn{Vect})$.

\begin{definition} 
	Let $W_{\opn{GT};\msc{A}}$  be the collection of all bordisms $(M,T)$ in  $\Bord_{\opn{Ch}(\msc{A})}$ of the form $\Sigma_{\vec{\mu}}:\Sigma_{\vec{x}}\to \Sigma_{\vec{y}}$ as described in Section \ref{sect:a_sigma}, where $\vec{\mu}=(\mu_i \in \Hom_{\msc{A}^{\pm}}(x_i,y_i) : i \in I)$ and $\mu_i:x_i\to y_i$ is a homotopy equivalence for all $i\in I$.

	We call $W_{\opn{GT};\msc{A}}$ the set of \emph{algebraic homotopy equivalences} in $\Bord_{\opn{Ch}(\mathcal{A})}$.
\end{definition}

We denote by  $W_{\opn{Vect}}$  the usual collection  of homotopy equivalences in $\opn{Ch(Vect)}$.

\begin{theorem}\label{thm:zhtop}
	The field theory $Z^*: \Bord_{\Ch(\msc{A}) }\to \Ch(\opn{Vect})$ from Theorem \ref{thm: Zast} preserves homotopy equivalences,
	\[
	Z^*(W_{\opn{GT};\msc{A}})\ \subseteq\ W_{\opn{Vect}}.
	\]
\end{theorem}

\begin{proof}
	We prove the statement by graphical calculus in $(k\Str_{\msc{S}}^{\red}\boxtimes k\Bord_{\msc{C}}^{\red})^{\opn{add}}$.	
Consider a bordism $\Sigma_{\vec{\mu}}:\Sigma_{\vec{x}}\to \Sigma_{\vec{y}}$ in $W_{\opn{GT};\msc{A}}$, where   
we assume without loss of generality that $\mu_i=\text{id}_{x_i}$ for all $i\ne j$, and that $\mu_j:x_j\to y_j$ is a homotopy equivalence. In fact, it is enough to consider this case as any bordism in $W_{\opn{GT};\msc{A}}$ can be obtained as a composition of bordisms of this form. 
We represent $(p\boxtimes 1)\Delta(\Sigma_{\vec{\mu}})$ in $(k\Str_{\msc{S}}^{\red}\boxtimes k\Bord_{\msc{C}}^{\red})^{\opn{add}}$ by
	\begin{equation}\label{def:htpy in bord}
		\tikzset{every picture/.style={line width=0.5pt}} 
		\begin{tikzpicture}[x=0.75pt,y=0.75pt,yscale=-1,xscale=1]
			
			\draw    (372.39,35.97) -- (372.3,62.95) -- (372.26,78.49) ;
			\draw  
			[fill={rgb, 255:red, 255; green, 255; blue, 255 }  ,fill opacity=1 ] (313.94,50.29) -- (340.33,50.29) -- (340.33,66.8) -- (313.94,66.8) -- cycle ;
			\draw    (280.2,36.2) .. controls (280.66,23.83) and (372.67,24.4) .. (372.39,35.97) ;
			\draw    (280.2,36.2) -- (280.1,63.19) -- (280.07,78.72) ;
			\draw    (235.85,32.31) -- (235.8,82.8) ;
			\draw    (280.07,78.72) .. controls (280,91.87) and (372.01,91.3) .. (372.26,78.49) ;
			\draw 
			(327.43,41.21) -- (327.45,50.92) ;
			\draw 
			(328.14,67.68) -- (328.15,73.6) ;
			\draw 
			(298.34,40.84) -- (298.36,73.22) ;
			\draw 
			(293.79,41.59) -- (293.8,73.98) ;
			\draw 
			(358.28,40.84) -- (358.29,73.22) ;
			\draw 
			(353.72,41.59) -- (353.74,73.98) ;
			\draw [color={rgb, 255:red, 155; green, 155; blue, 155 }  ,draw opacity=1 ] [dotted]  (282.4,80.4) .. controls (283.4,79.9) and (306.9,79.4) .. (324.4,79.4) .. controls (341.9,79.4) and (361.6,79.16) .. (371.4,80.9) ;
			\draw [color={rgb, 255:red, 155; green, 155; blue, 155 }  ,draw opacity=1 ] [dotted]  (282.4,34.15) .. controls (284.9,36.65) and (302.9,36.4) .. (324.9,36.4) .. controls (346.9,36.4) and (366.4,35.4) .. (368.9,34.65) ;
			\draw    (176.65,31.91) -- (176.6,82.4) ;
			
			\draw (319.18,51) node [anchor=north west][inner sep=0.75pt]  [font=\tiny]  {$\mu_{j}^{m_{j}}$};
			\draw (165,20) node [anchor=north west][inner sep=0.75pt]  [font=\tiny]  {$k( m_{1}) \ \ \ \ \ \ \ \ \ \  k( m_{t}) \ $};
			\draw (88,46) node [anchor=north west][inner sep=0.75pt]  [font=\large]  {$\sum\limits_{m:\{1,\ \dotsc ,t\}\rightarrow \mathbb{Z} \ } \ \ \ \ \ \ ...\ \ \ \ \ \ \ \ \otimes \ \ \ \ \ \ \ \ \ \ \ \ \ \ \ \ \ \ \ \ \ \ \ \ \ \ \  .$};
			\draw (374,30) node [anchor=north west][inner sep=0.75pt]  [color={rgb, 255:red, 0; green, 0; blue, 0 }  ,opacity=1 ]  {$\Sigma _{\vec{x}^{m}}$};
			\draw (374.72,80) node [anchor=north west][inner sep=0.75pt]  [color={rgb, 255:red, 0; green, 0; blue, 0 }  ,opacity=1 ]  {$\Sigma _{\vec{y}^{m}}$};
			\draw (165,83.2) node [anchor=north west][inner sep=0.75pt]  [font=\tiny]  {$k( m_{1}) \ \ \ \ \ \ \ \ \ \  k( m_{t}) \ $};
			
		\end{tikzpicture}
	\end{equation}
We also assume, via Lemma \ref{lem:d_pos}, that $\Sigma_{\vec{x}}$ and $\Sigma_{\vec{y}}$ are positively marked.	
	
	Since $\mu_j:x_j\to y_j$ is a homotopy equivalence, there exists a map $\eta_j:y_j\to x_j$ in $\opn{Ch}(\msc A)$ and a graded morphism $H_j:x_j[1]\to x_j$ such that 
	\[ \eta_j \mu_j - \opn{id}_{x_j} = H_id_{x_j} - (d_{x_j} H_j)[-1].
	\]
Let $\Sigma_{\vec{\eta}}:\Sigma_{\vec{y}}\to \Sigma_{\vec{x}}$ be defined by the tuple $(1,\dots,\eta_j,\dots, 1)$.  We have the corresponding map $(p\boxtimes 1)\Delta(\Sigma_{\vec{\eta}})$ in $(k\Str_{\opn{Vect}^{\mathbb Z}}^{\red}\boxtimes k\Bord_{\msc{C}}^{\red})^{\opn{add}}$ given by  
	\begin{equation*}
		\tikzset{every picture/.style={line width=0.5pt}} 

	\end{equation*}
	for all $m:\{1, \dots, t\}\to \mathbb Z$. We claim that the equation
	\begin{equation}\label{eq: htpy}
	(ev_{\opn{Vect}^{\mathbb Z}}\cdot Z)(H)d_{\Sigma_{\vec{x}}} - d_{\Sigma_{\vec{x}}}[-1] (ev_{\opn{Vect}^{\mathbb Z}}\cdot Z)(H)[-1] = Z^*(\Sigma_{\eta})Z^*(\Sigma_{\mu}) - \opn{id}_{Z^{\ast}(\Sigma_{\vec{x}})}
	\end{equation}
holds in $\opn{Ch(Vect)}$. After absorbing the shifts and implicit applications of coevaluation and evaluation into the functor $ev_{\opn{Vect}^{\mathbb{Z}}}\cdot Z$ we obtain an equivalent expression
	\[
	\begin{array}{l}
	d_{\Sigma_{\vec{x}}}{[-1]} (ev_{\opn{Vect}^{\mathbb{Z}}}\cdot Z)(H){[-1]}\vspace{1mm}\\
	=(ev\cdot Z)\big((\opn{ev}_{k(-1)}\ot id)(k(1)\ot \tilde{d}_{\Sigma_{\vec{x}}})\big)\circ(ev\cdot Z)\big((k(1)\ot H)( \opn{coev}_{k(-1)}\ot id)\big).
	\end{array}
	\]
	The morphism $(k(1)\ot H)(\opn{coev}_{k(-1)}\ot id)$ is the sum over all $m:\{1,\dots, t\}\to \mathbb Z$ of 
	\begin{equation*}
		\scalebox{0.9}{
			\tikzset{every picture/.style={line width=0.5pt}} 
		%
		}
	\end{equation*}
	and $(\opn{ev}_{k(-1)}\ot id)(k(1)\ot \tilde{d}_{\Sigma_{\vec{x}}})$ is the sum over all $m:\{1,\dots, t\}\to \mathbb Z$ of 
	\begin{equation*}
		\scalebox{0.9}{
			\tikzset{every picture/.style={line width=0.5pt}} 
			
			%
		}
	\end{equation*}
	We note that we include arrows on the left-most strings of the preceding two diagrams to clarify the difference in orientation, since by convention our strings typically travel downward.
	
	On one hand, we obtain that $H \tilde d_{\Sigma}$ equals the sum over all $m$ of
	\begin{equation*}
	\scalebox{0.9}{
		\tikzset{every picture/.style={line width=0.5pt}} 
		
		%
	}
\end{equation}

On the other hand, the composition
\[
(\opn{ev}_{k(-1)}\ot id)(k(1)\ot \tilde{d}_{\Sigma_{\vec{x}}})(k(1)\ot H)( \opn{coev}_{k(-1)}\ot id)
\]
is the sum over all $m$ in $\mathbb Z$ of
	\begin{equation*}
	\scalebox{0.9}{
		\tikzset{every picture/.style={line width=0.5pt}} 
		
		%
	}
\end{equation}	
Sliding the coupons on all terms but the first, we obtain that \eqref{eq:htpy 2} is equivalent to 
	\begin{equation*}
	\scalebox{0.9}{
		\tikzset{every picture/.style={line width=0.5pt}} 
		
		%
	}
\end{equation}	
where we note that the strings on all but the first term are now oriented downwards, as is our usual convention.  After applying a homotopy on strings, we observe that the final two terms in \eqref{eq:htpy 1} and \eqref{eq:htpy 2.5} agree so that the difference
\[H\tilde d_{\Sigma_{\vec{x}}} - (\opn{ev}_{k(-1)}\ot id)(k(1)\ot \tilde{d}_{\Sigma_{\vec{x}}})(k(1)\ot H)( \opn{coev}_{k(-1)}\ot id)\]
	is the sum over all $m$ in $\mathbb Z$ of 
\[
	\scalebox{.9}{\tikzset{every picture/.style={line width=0.5pt}} 

%
}
\]
\[
=id_{k(m)}\ot \Sigma_{\vec{\eta}^m}\Sigma_{\vec{\mu}^m}-id_{k(m)}\ot id_{\Sigma_{\vec{x}^m}},
\]
	where in the second equality we are using relation (V6). 
	The desired identity \eqref{eq: htpy} now follows by applying the functor $ev_{\opn{Vect}^{\mathbb Z}}\cdot Z$ to the equation above.
	
	One similarly constructs from a degree $-1$ map $\check{H}_j$ which satisfies $\mu_j \eta_j - \opn{id}_{y_j} = \check{H}_j d_{y_j}-(d_{y_j} \check{H}_j)[-1]$ a corresponding map $\check{H}$ which satisfies
	\[
	(ev_{\opn{Vect}^{\mathbb Z}}\cdot Z)(\check{H})d_{\Sigma_{\vec{y}}} - d_{\Sigma_{\vec{y}}}[-1] (ev_{\opn{Vect}^{\mathbb Z}}\cdot Z)(\check{H})[-1] = Z^*(\Sigma_{\vec{\mu}})Z^*(\Sigma_{\vec{\eta}}) - \opn{id}_{Z^{\ast}(\Sigma_{\vec{y}})}.
	\]
	This equation, along with \eqref{eq: htpy}, verifies that $Z^{\ast}(\Sigma_{\vec{\mu}})$ is in fact a homotopy equivalence.
\end{proof}

\subsection{Localization}\label{sect:zloc}

From the $\infty$-categorical perspective, our symmetric monoidal categories produce cocartesian fibrations
\begin{equation}\label{eq:bord_fib}
\Bord_{\opn{Ch}(\msc{A})}^{\red\ot}\to \opn{Comm}\ \ \text{and}\ \ \opn{Ch}(\opn{Vect})^{\ot}\to \opn{Comm}
\end{equation}
to the commutative operad, which is explicitly the discrete category of finite pointed sets $\opn{Comm}=\opn{Fin}_{\ast}$, and our functor $Z^{\ast}:\Bord_{\opn{Ch}(\msc{A})}^{\red}\to \opn{Ch}(\opn{Vect})$ lifts to a map of such fibrations \cite[Construction 2.0.0.1]{lurieha}.  Indeed, this is what it means for such a map to be symmetric monoidal.
\par

The classes of homotopy equivalences furthermore produce corresponding classes of morphisms (edges) $W^{\ot}_{\opn{GT}:\msc{A}}$ and $W^{\ot}_{\opn{Vect}}$ in $\Bord_{\opn{Ch}(\msc{A})}^{\red\ot}$ and $\opn{Ch}(\opn{Vect})^{\ot}$ which exist only in the fibers over objects $I$ in $\opn{Comm}$.  Said another way, all of these morphisms map to identity morphism in $\opn{Comm}$ under the structure maps.  (See the discussion preceding \cite[Definition A.4]{nikolausscholze18}, for example.)

By \cite[Proposition 2.1.4]{hinich16} and Theorem \ref{thm:zhtop} we can localize the fibrations \eqref{eq:bord_fib} relative to their respective classes of homotopy equivalences to obtain new symmetric monoidal $\infty$-categories and a map of symmetric monoidal $\infty$-categories
\[
\xymatrix{
\Bord_{\opn{Ch}(\msc{A})}^{\red\ot}[(W^{\ot})^{-1}]\ar[rr]^{\mcl{Z}^{\ot}}\ar[dr] & & \opn{Ch}(\opn{Vect})^{\ot}[(W^{\ot})^{-1}]\ar[dl]\\
 & \opn{Comm} & ,
}
\]
where $W^{\ot}$ are the appropriate classes of homotopy equivalences and $\mcl{Z}^{\ot}$ is obtained by localizing the map $(Z^{\ast})^{\ot}$.

Now, the fibration $\opn{Ch}(\opn{Vect})^{\ot}[(W^{\ot})^{-1}]\to \opn{Comm}$ is the symmetric monoidal $\infty$-category $\mcl{V}ect^{\ot}\to \opn{Comm}$ of dg, or homotopical vector spaces \cite[Proposition 1.3.4.5]{lurieha}.  As for the category $\Bord_{\opn{Ch}(\msc{A})}^{\red\ot}[(W^{\ot})^{-1}]$, the forgetful functor $\Bord_{\opn{Ch}(\msc{A})}^{\red\ot}\to \Bord_{\ast}^{\red\ot}$ localizes to produce a symmetric monoidal functor from the localization to the discrete symmetric monoidal category $\Bord_{\ast}^{\red\ot}$, so that $\Bord_{\opn{Ch}(\msc{A})}^{\red\ot}[(W^{\ot})^{-1}]$ can be seen as a symmetric monoidal $\infty$-category of decorated ribbon bordisms.
\par

In addition to this global perspective, we find that the fibers of the underlying functor to the category of unlabeled ribbon bordisms $\opn{Ch}(\msc{A})_{\mbf{\Sigma}}\to \Bord^{\red}_{\opn{Ch}(\msc{A})}$ (see Section \ref{sect:a_sigma}) localize to produce a functor
\[
\mcl{K}(\msc{A})_{\mbf{\Sigma}}:=\prod_{i\in I}\mcl{K}(\msc{A})^{\opn{sgn}(i)}=\opn{Ch}(\msc{A})_{\mbf{\Sigma}}[\opn{HtopEq}^{-1}]\to  \Bord_{\opn{Ch}(\msc{A})}^{\red\ot}[(W^{\ot})^{-1}],
\]
where $\mcl{K}(\msc{A})$ is the bounded homomotopy $\infty$-category for $\msc{A}$ \cite[Proposition 1.3.4.5]{lurieha}, $I$ is the marking set for $\mbf{\Sigma}$, and the sign exponents are as in Section \ref{sect:a_sigma}.  Hence we obtain a diagram
\[
\xymatrix{
\mcl{K}(\msc{A})_{\mbf{\Sigma}}\ar[rr]\ar[d] & & \Bord_{\opn{Ch}(\msc{A})}^{\red}[W^{-1}]\ar[d]\\
\ast\ar[rr]_{\mbf{\Sigma}} & & \Bord^{\red}_\ast
}
\]
from which we view objects in the localized bordism category as surfaces marked by objects in the homotopy $\infty$-category $\mcl{K}(\msc{A})$.  So we can think of the localization 
$\Bord_{\opn{Ch}(\msc{A})}^{\red\ot}[(W^{\ot})^{-1}]$ vaguely (and only vaguely!) as a symmetric monoidal $\infty$-category of ribbon bordisms with labels from the ribbon monoidal $\infty$-category $\mcl{K}(\msc{A})$.

\begin{remark}
In considering the possible markings on multivalent vertices one realizes how complicated, and maybe even preposterous, such a notion of ``ribbon diagrams with labels in a monoidal $\infty$-category" must be. Hence our warning about the vagueness of this language.
\end{remark}

As remarked in the introduction, we propose that the localized theory
\begin{equation}\label{eq:loc_thry}
\mcl{Z}:\Bord_{\opn{Ch}(\msc{A})}^{\red}[W^{-1}]\to \mcl{V}ect
\end{equation}
can be used to develop a kind of derived TQFT from the original theory $Z$.  This line of inquiry is developed further in work in progress \cite{negron}.

\subsection{Continued remarks}
Of course the above discussion is as valid in the admissible setting as it is in the general non-admissible setting.  In assessing the actual ``$\infty$-ness" of the field theory \eqref{eq:loc_thry} we can consider the state spaces for the localized Lyubashenko TQFT, for example.

If we consider the twice marked sphere $\mbf{S}$, with positively marked north pole and negatively marked south pole, Theorem \ref{thm: Z_Ch} calculates the state spaces in the cochain valued theory
\[
Z_{\msc{A}}^{\ast}|_{\opn{Ch}(\msc{A})_{\mbf{S}}}:\opn{Ch}(\msc{A})^{\opn{op}}\times \opn{Ch}(\msc{A})\to \opn{Ch}(\opn{Vect})
\]
as cochain homs $Z_{\msc{A}}^{\ast}|_{\opn{Ch}(\msc{A})_{\mbf{S}}}\cong \Hom^{\ast}_{\msc{A}}$.  It follows that in the localization $\mcl{Z}_{\msc{A}}$ the sate spaces
\[
\mcl{Z}_{\msc{A}}|_{\mcl{K}(\msc{A})_{\mbf{S}}}:\mcl{K}(\msc{A})^{\opn{op}}\times \mcl{K}(\msc{A})\to \mcl{V}ect
\]
are necessarily obtained by localizing the cochain hom functor.  We would propose that these are the linear mapping spaces
\[
\mcl{Z}_{\msc{A}}|_{\mcl{K}(\msc{A})_{\mbf{S}}}\cong \opn{Maps}_{\mcl{K}(\msc{A})}
\]
(cf.\ \cite[Lemma I.3.4]{nikolausscholze18}).  By a linear version of Yoneda, this one functor knows everything about the $\infty$-category $\mcl{K}(\msc{A})$.

\appendix

\section{Proof of Proposition \ref{prop:conc}}
\label{sect:proofs}

We first establish some basic facts about the ribbon bordism category then provide the proof of Proposition \ref{prop:conc}.

\subsection{Exchange of the labeling category}
\label{sect:exc}

Let $F,G:\msc{C}\to \msc{D}$ be ribbon monoidal functors and $\alpha:F\to G$ be a natural isomorphism.  We have the associated functors on the opposite categories $F^{\opn{op}}$ and $G^{\opn{op}}$, and via duality $\alpha$ produces a unique natural isomorphism $F^{\opn{op}}\overset{\sim}\to G^{\opn{op}}$ which fits into a diagram
\begin{equation}\label{eq:2270}
\xymatrix{
F(x^{\ast})\ar[rr]^{\alpha_{x^{\ast}}}\ar[d]_{\cong} & & G(x^{\ast})\ar[d]^{\cong}\\
F(x)^{\ast}\ar@{-->}[rr]_{\exists !}^{\sim} & & G(x)^{\ast}
}
\end{equation}
at each $x$ in $\msc{C}$. Via monoidality of $\alpha$, this unique natural isomorphism completes the expected diagrams with respect to evaluation and coevaluation for $F(x)$ and $G(x)$, as in \eqref{eq:2282} below.

\begin{lemma}\label{lem:nat_dual}
Let $\msc{C}$ be a rigid monoidal category and $\alpha:x\to y$ be an isomorphism in $\msc{C}$.  Then, for arbitrarily chosen duals $x^\ast$ and $y^\ast$, $(\alpha^{-1})^{\ast}:x^{\ast}\to y^{\ast}$ is the unique morphism which fits into diagrams
\begin{equation}\label{eq:2282}
\xymatrix{
x^{\ast}\ot x\ar[rr]^{(\alpha^{-1})^{\ast}\ot\alpha}\ar[dr] & & y^{\ast}\ot y\ar[dl] & \1 & \\
	& \1 & x\ot x^{\ast}\ar[rr]_{\alpha\ot (\alpha^{-1})^{\ast}}\ar[ur] & & y^{\ast}\ot y\ar[ul]
}
\end{equation}
over the evaluation and coevaluation maps for $x$ and $y$. 
\end{lemma}

\begin{proof}
By uniqueness of duals \cite[Proposition 2.10.5]{egno15} there is a unique morphism $\check{\alpha}:y^{\ast}\to x^{\ast}$ which completes the given diagrams.  One verifies via a diagram chase that map $(\alpha^{-1})^{\ast}$, i.e.\ the composite
\[
x^{\ast}\to x^{\ast}\ot y\ot y^{\ast}\overset{1\ot \alpha^{-1}\ot 1}\to x^{\ast}\ot x\ot y^{\ast}\to y^{\ast},
\]
completes the above diagrams, thus giving $\check{\alpha}=(\alpha^{-1})^{\ast}$.
\end{proof}

We apply Lemma \ref{lem:nat_dual} to the case of a natural isomorphism between monoidal functors, discussed above, to see that the natural isomorphism
\[
(\alpha_x^{-1})^{\ast}:F(x)^{\ast}\to G(x)^{\ast}
\]
is the unique isomorphism which completes the diagram \eqref{eq:2270}.  Hence, after identifying $F(x^{\ast})$ and $G(x^{\ast})$ as duals of the objects $F(x)$ and $G(x)$ via the evaluation and coevaluation maps induced by those on $x$, we have $\alpha_{x^{\ast}}=(\alpha_x^{-1})^{\ast}$.

Now, having clarified this issue, we note that $\alpha:F\to G$ and $\alpha^{-1}:F^{\opn{op}}\to G^{\opn{op}}$ define natural isomorphisms
\begin{equation}\label{eq:2266}
\Sigma_{\alpha}=\Sigma_{\alpha_{\vec{x}}}:\Sigma_{F(\vec{x})}\to \Sigma_{G(\vec{x})}
\end{equation}
in the fiber $\msc{D}_{\mbf{\Sigma}}$ over any $\ast$-marked surface.  (See Section \ref{sect:a_sigma}.)  We claim that the $\Sigma_{\alpha}$ together provide a natural isomorphism between the induced functors $F_{\ast}$ and $G_{\ast}$ on non-linearly reduced bordisms
\[
\xymatrix{
\Bord^{\opn{red}}_{\msc{C}}\ar@<-.5ex>[rr]_{G_\ast}\ar@<+.5ex>[rr]^{F_\ast}\ar[dr] & & \Bord^{\opn{red}}_{\msc{D}}\ar[dl]\\
 & \Bord_{\ast}^{\opn{red}} & .
}
\]

\begin{lemma}\label{lem:622}
Let $F,G:\msc{C}\to \msc{D}$ be ribbon monoidal functors and $\alpha:F\to G$ be a natural isomorphism between such functors.  Then the isomorphisms $\Sigma_{\alpha}$ from \eqref{eq:2266} provide a natural isomorphism between the induced symmetric monoidal functors
\[
F_{\ast},G_{\ast}:\Bord^{\opn{red}}_{\msc{C}}\to \Bord^{\opn{red}}_{\msc{D}}.
\]
\end{lemma}

\begin{proof}
Clearly each $\Sigma_{\alpha_{\vec{x}}}$ is invertible in $\Bord^{\opn{red}}_{\msc{D}}$, with inverse $\Sigma_{\alpha^{-1}_{\vec{x}}}$. So we need only verify naturality.  That is, given a map $[M,T]:\Sigma_{\vec{x}}\to \Sigma_{\vec{y}}$ in $\Bord^{\opn{red}}_{\msc{D}}$, we need to very the equality
\[
\Sigma_{\alpha_{\vec{y}}}\circ [M,FT]=[M,GT]\circ \Sigma_{\alpha_{\vec{x}}}.
\]
Equivalently, taking the object subscripts for granted, we need to verify the equality
\begin{equation}\label{eq:639}
\Sigma_{\alpha}\circ [M,FT]\circ \Sigma_{\alpha^{-1}}=[M,GT].
\end{equation}
In $\opn{Bord}_{\msc{D}}^{\opn{red}}$ we have an equality $[M,FT]=[M,T']$ where $T'$ is obtained from $FT$ by replacing each label $F(\xi)$ on an internal vertex with the conjugation of $G(\xi)$ by $\alpha$,
\[
\scalebox{.7}{
\tikzset{every picture/.style={line width=0.75pt}} 
\begin{tikzpicture}[x=0.75pt,y=0.75pt,yscale=-1,xscale=1]
\draw  [fill={rgb, 255:red, 0; green, 0; blue, 0 }  ,fill opacity=1 ] (116,134.13) .. controls (114.14,134.13) and (112.62,132.62) .. (112.62,130.76) .. controls (112.62,128.89) and (114.13,127.38) .. (116,127.38) .. controls (117.86,127.38) and (119.37,128.89) .. (119.37,130.75) .. controls (119.37,132.62) and (117.86,134.13) .. (116,134.13) -- cycle ;
\draw    (116,130.75) -- (215.5,219) ;
\draw [shift={(169.49,178.19)}, rotate = 221.57] [fill={rgb, 255:red, 0; green, 0; blue, 0 }  ][line width=0.08]  [draw opacity=0] (10.72,-5.15) -- (0,0) -- (10.72,5.15) -- (7.12,0) -- cycle    ;
\draw    (116,130.75) -- (213.5,182) ;
\draw [shift={(169.17,158.7)}, rotate = 207.73] [fill={rgb, 255:red, 0; green, 0; blue, 0 }  ][line width=0.08]  [draw opacity=0] (10.72,-5.15) -- (0,0) -- (10.72,5.15) -- (7.12,0) -- cycle    ;
\draw    (116,130.75) -- (210.5,57) ;
\draw [shift={(158.12,97.88)}, rotate = 322.03] [fill={rgb, 255:red, 0; green, 0; blue, 0 }  ][line width=0.08]  [draw opacity=0] (10.72,-5.15) -- (0,0) -- (10.72,5.15) -- (7.12,0) -- cycle    ;
\draw    (14.42,215.56) -- (116,130.75) ;
\draw [shift={(69.05,169.95)}, rotate = 140.14] [fill={rgb, 255:red, 0; green, 0; blue, 0 }  ][line width=0.08]  [draw opacity=0] (10.72,-5.15) -- (0,0) -- (10.72,5.15) -- (7.12,0) -- cycle    ;
\draw    (116,130.75) -- (14.96,60.56) ;
\draw [shift={(61.37,92.8)}, rotate = 34.79] [fill={rgb, 255:red, 0; green, 0; blue, 0 }  ][line width=0.08]  [draw opacity=0] (10.72,-5.15) -- (0,0) -- (10.72,5.15) -- (7.12,0) -- cycle    ;
\draw    (15.02,179.56) -- (116,130.75) ;
\draw [shift={(59.66,157.98)}, rotate = 334.2] [fill={rgb, 255:red, 0; green, 0; blue, 0 }  ][line width=0.08]  [draw opacity=0] (10.72,-5.15) -- (0,0) -- (10.72,5.15) -- (7.12,0) -- cycle    ;
\draw  [dash pattern={on 0.84pt off 2.51pt}]  (116,40.13) -- (116,228.13) ;
\draw  [fill={rgb, 255:red, 0; green, 0; blue, 0 }  ,fill opacity=1 ] (393,135.13) .. controls (391.14,135.13) and (389.62,133.62) .. (389.62,131.76) .. controls (389.62,129.89) and (391.13,128.38) .. (393,128.38) .. controls (394.86,128.38) and (396.37,129.89) .. (396.37,131.75) .. controls (396.37,133.62) and (394.86,135.13) .. (393,135.13) -- cycle ;
\draw    (393,131.75) -- (492.5,220) ;
\draw [shift={(446.49,179.19)}, rotate = 221.57] [fill={rgb, 255:red, 0; green, 0; blue, 0 }  ][line width=0.08]  [draw opacity=0] (10.72,-5.15) -- (0,0) -- (10.72,5.15) -- (7.12,0) -- cycle    ;
\draw    (393,131.75) -- (490.5,183) ;
\draw [shift={(446.17,159.7)}, rotate = 207.73] [fill={rgb, 255:red, 0; green, 0; blue, 0 }  ][line width=0.08]  [draw opacity=0] (10.72,-5.15) -- (0,0) -- (10.72,5.15) -- (7.12,0) -- cycle    ;
\draw    (393,131.75) -- (487.5,58) ;
\draw [shift={(435.12,98.88)}, rotate = 322.03] [fill={rgb, 255:red, 0; green, 0; blue, 0 }  ][line width=0.08]  [draw opacity=0] (10.72,-5.15) -- (0,0) -- (10.72,5.15) -- (7.12,0) -- cycle    ;
\draw    (291.42,216.56) -- (393,131.75) ;
\draw [shift={(346.05,170.95)}, rotate = 140.14] [fill={rgb, 255:red, 0; green, 0; blue, 0 }  ][line width=0.08]  [draw opacity=0] (10.72,-5.15) -- (0,0) -- (10.72,5.15) -- (7.12,0) -- cycle    ;
\draw    (393,131.75) -- (291.96,61.56) ;
\draw [shift={(338.37,93.8)}, rotate = 34.79] [fill={rgb, 255:red, 0; green, 0; blue, 0 }  ][line width=0.08]  [draw opacity=0] (10.72,-5.15) -- (0,0) -- (10.72,5.15) -- (7.12,0) -- cycle    ;
\draw    (292.02,180.56) -- (393,131.75) ;
\draw [shift={(336.66,158.98)}, rotate = 334.2] [fill={rgb, 255:red, 0; green, 0; blue, 0 }  ][line width=0.08]  [draw opacity=0] (10.72,-5.15) -- (0,0) -- (10.72,5.15) -- (7.12,0) -- cycle    ;
\draw  [dash pattern={on 0.84pt off 2.51pt}]  (393,41.13) -- (393,229.13) ;
\draw  [fill={rgb, 255:red, 0; green, 0; blue, 0 }  ,fill opacity=1 ] (320.5,83.75) .. controls (320.5,81.96) and (321.96,80.5) .. (323.75,80.5) .. controls (325.54,80.5) and (327,81.96) .. (327,83.75) .. controls (327,85.54) and (325.54,87) .. (323.75,87) .. controls (321.96,87) and (320.5,85.54) .. (320.5,83.75) -- cycle ;
\draw  [fill={rgb, 255:red, 0; green, 0; blue, 0 }  ,fill opacity=1 ] (459.5,78.75) .. controls (459.5,76.96) and (460.96,75.5) .. (462.75,75.5) .. controls (464.54,75.5) and (466,76.96) .. (466,78.75) .. controls (466,80.54) and (464.54,82) .. (462.75,82) .. controls (460.96,82) and (459.5,80.54) .. (459.5,78.75) -- cycle ;
\draw  [fill={rgb, 255:red, 0; green, 0; blue, 0 }  ,fill opacity=1 ] (316.5,193.75) .. controls (316.5,191.96) and (317.96,190.5) .. (319.75,190.5) .. controls (321.54,190.5) and (323,191.96) .. (323,193.75) .. controls (323,195.54) and (321.54,197) .. (319.75,197) .. controls (317.96,197) and (316.5,195.54) .. (316.5,193.75) -- cycle ;
\draw  [fill={rgb, 255:red, 0; green, 0; blue, 0 }  ,fill opacity=1 ] (314.5,167.75) .. controls (314.5,165.96) and (315.96,164.5) .. (317.75,164.5) .. controls (319.54,164.5) and (321,165.96) .. (321,167.75) .. controls (321,169.54) and (319.54,171) .. (317.75,171) .. controls (315.96,171) and (314.5,169.54) .. (314.5,167.75) -- cycle ;
\draw  [fill={rgb, 255:red, 0; green, 0; blue, 0 }  ,fill opacity=1 ] (457.5,191.75) .. controls (457.5,189.96) and (458.96,188.5) .. (460.75,188.5) .. controls (462.54,188.5) and (464,189.96) .. (464,191.75) .. controls (464,193.54) and (462.54,195) .. (460.75,195) .. controls (458.96,195) and (457.5,193.54) .. (457.5,191.75) -- cycle ;
\draw  [fill={rgb, 255:red, 0; green, 0; blue, 0 }  ,fill opacity=1 ] (465.5,171.75) .. controls (465.5,169.96) and (466.96,168.5) .. (468.75,168.5) .. controls (470.54,168.5) and (472,169.96) .. (472,171.75) .. controls (472,173.54) and (470.54,175) .. (468.75,175) .. controls (466.96,175) and (465.5,173.54) .. (465.5,171.75) -- cycle ;

\draw (99,88) node [anchor=north west][inner sep=0.75pt]   [align=left] {$F(\xi)$};
\draw (376,89) node [anchor=north west][inner sep=0.75pt]   [align=left] {$G(\xi)$};
\draw (247,120) node [anchor=north west][inner sep=0.75pt]   [align=left] {=};
\draw (320,60) node [anchor=north west][inner sep=0.75pt]   [align=left] {$\alpha^{-1}_{x_1}$};
\draw (322,195) node [anchor=north west][inner sep=0.75pt]   [align=left] {$\alpha_{x_n}$};
\draw (295,144) node [anchor=north west][inner sep=0.75pt]   [align=left] {$\alpha^{-1}_{x_{n-1}}$};
\draw (435,195) node [anchor=north west][inner sep=0.75pt]   [align=left] {$\alpha^{-1}_{y_m}$};
\draw (440,60) node [anchor=north west][inner sep=0.75pt]   [align=left] {$\alpha_{y_1}$};
\draw (458,145) node [anchor=north west][inner sep=0.75pt]   [align=left] {$\alpha^{-1}_{y_{m-1}}$};
\draw (458,110) node [anchor=north west][inner sep=0.75pt]   [align=left] {$\vdots$};
\draw (323,110) node [anchor=north west][inner sep=0.75pt]   [align=left] {$\vdots$};
\draw (170,110) node [anchor=north west][inner sep=0.75pt]   [align=left] {$\vdots$};
\draw (50,110) node [anchor=north west][inner sep=0.75pt]   [align=left] {$\vdots$};

\end{tikzpicture}}\ .
\]
Here each vertex we introduce inherits its framing from the ambient framing on the corresponding edge.
\par

We now have
\[
\Sigma_{\alpha}\circ [M,FT]\circ \Sigma_{\alpha^{-1}}=\Sigma_{\alpha}\circ [M,T']\circ \Sigma_{\alpha^{-1}}.
\]
Let $[M,T'']$ be the composite
\[
[M,T'']=\Sigma_{\alpha}\circ [M,T']\circ \Sigma_{\alpha^{-1}}.
\]
The bordism $[M,T'']$ is obtained directly from $[M,FT]$ by replacing all $F$'s with $G$'s then inserting a segment of the form
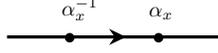
\[
\scalebox{.8}{\tikzset{every picture/.style={line width=0.75pt}} 
\begin{tikzpicture}[x=0.75pt,y=0.75pt,yscale=-1,xscale=1]
\draw [line width=1.5]    (0,132) -- (135.5,132) ;
\draw [shift={(75,132)}, rotate = 180] [fill={rgb, 255:red, 0; green, 0; blue, 0 }  ][line width=0.08]  [draw opacity=0] (10.72,-5.15) -- (0,0) -- (10.72,5.15) -- (7.12,0) -- cycle    ;
\draw  [fill={rgb, 255:red, 0; green, 0; blue, 0 }  ,fill opacity=1 ] (37,132.5) .. controls (37,131.12) and (38.12,130) .. (39.5,130) .. controls (40.88,130) and (42,131.12) .. (42,132.5) .. controls (42,133.88) and (40.88,135) .. (39.5,135) .. controls (38.12,135) and (37,133.88) .. (37,132.5) -- cycle ;
\draw  [fill={rgb, 255:red, 0; green, 0; blue, 0 }  ,fill opacity=1 ] (93,132.5) .. controls (93,131.12) and (94.12,130) .. (95.5,130) .. controls (96.88,130) and (98,131.12) .. (98,132.5) .. controls (98,133.88) and (96.88,135) .. (95.5,135) .. controls (94.12,135) and (93,133.88) .. (93,132.5) -- cycle ;

\draw (33,106) node [anchor=north west][inner sep=0.75pt]   [align=left] {$\alpha_x^{-1}$};
\draw (88,112) node [anchor=north west][inner sep=0.75pt]   [align=left] {$\alpha_x$};

\end{tikzpicture}}
\]
in the middle of each edge.  In the reduced bordism category, we can close each of these segments via internal composition to obtain an equality $[M,T'']=[M,GT]$, which gives the desired equality \eqref{eq:639}.
\end{proof}

\begin{corollary}\label{cor:bord_repl}
If $F:\msc{C}\to \msc{D}$ is an equivalence of ribbon tensor categories, then the induced map
\[
F_{\ast}:\Bord_{\msc{C}}^{\opn{red}}\to \Bord_{\msc{D}}^{\opn{red}}
\]
is an equivalence of symmetric monoidal categories.
\end{corollary}

\begin{proof}
For a choice of inverse $F^{-1}:\msc{D}\to \msc{C}$, Lemma \ref{lem:622} tells us that the two composites $F_{\ast}\circ F^{-1}_{\ast}$ and $F^{-1}_{\ast}\circ F_{\ast}$ are naturally isomorphic to the identity.
\end{proof}

\begin{remark}
Supposing $\msc{C}$ and $\msc{D}$ are linear, the above results also hold when the non-linear reduction $\Bord_{\square}^{\opn{red}}$ is replaced with the linear reduction $k\Bord_{\square}^{\opn{red}}$, and the proofs in the linear setting are exactly the same.  Indeed, the linear version of Corollary \ref{cor:bord_repl} is implied already by its non-linear counterpart.
\end{remark}

\subsection{A particular construction of $\msc{S}\otimes \msc{A}$}

We explain how, in a semisimple setting, one can construct the Kelly product $\msc{S}\ot\msc{A}$ via an additive completion of the internal product $\msc{S}\boxtimes\msc{A}$.  We begin with an explicit description of morphisms in the Kelly product.

\begin{lemma}\label{lem:708}
Let $\msc{S}$ and $\msc{A}$ be locally finite linear abelian categories, and suppose that $\msc{S}$ is semisimple.  The following hold:
\begin{enumerate}
\item[(i)] The category $\msc{S}\otimes \msc{A}$ is spanned, under finite sums, by the monomials $s\ot a$ with $s$ simple in $\msc{S}$.
\item[(ii)] For each pair of objects $s$ and $s'$ in $\msc{S}$, $a$ and $a'$ in $\msc{A}$, the map
\[
\Hom_{\msc{S}}(s,s')\otimes_k\Hom_{\msc{A}}(a,a')\to \Hom_{\msc{S}\otimes \msc{A}}(s\otimes a,s'\otimes a'),\ \ f\otimes_k f'\mapsto f\otimes f',
\]
is an isomorphism.
\end{enumerate}
\end{lemma}

\begin{proof}
(i) Under our hypotheses we have $\msc{S}\cong\opn{corep}(S)$ and $\msc{A}\cong \opn{corep}(\mfk{A})$ for a coalgebras $S$ and $\mfk{A}$ \cite[Theorem 5.1]{takeuchi77}.  Hence $\msc{S}\otimes \msc{A}$ can be realized as the category of $S\otimes \mfk{A}$-comodules \cite[Lemma 3.2.3]{coulembierflake}.  For this proof we prefer to use modules however, and so identify the product $\msc{S}\ot\msc{A}$ with the category of  finite-dimensional (discrete topological) $A=\mfk{A}^\ast$-modules with a commuting coaction of $S$, i.e. simultaneous modules and comodules on which $A$ acts by $S$-comodule endomorphsims.
\par

Consider any set of representatives $\{s_\lambda:\lambda\in \Lambda\}$ for the simples in $\msc{S}$.  For any object $m$ in $\msc{S}\otimes \msc{A}$ and simple $s_\lambda$ in $\msc{S}$ define $m_\lambda$ as the image of the map
\[
\opn{eval}:\Hom_{\msc{S}}(s_\lambda,m)\otimes_k s_\lambda\to m.
\]
We note that each $\Hom_{\msc{S}}(s_\lambda,m)$ is an $A$-module and that the above map is a morphism of $A$-modules in $\msc{S}$.  Furthermore, the above map is an injective map of $S$-comodules, via semisimplicity, and hence an isomorphism onto its image $m_\lambda\subseteq m$.  In particular, we have
\[
\Hom_{\msc{S}}(s_\lambda,m)\otimes_k s_\lambda\overset{\sim}\to m_\lambda,
\]
and we subsequently obtain a decomposition
\[
\oplus_\lambda \Hom_{\msc{S}}(s_\lambda,m)\otimes_k s_\lambda\overset{\sim}\to \oplus_{\lambda} m_\lambda =m.
\]
So we see that $\msc{S}\otimes \msc{A}$ is spanned by the monomidal objects $s\ot a$ with $s$ simple.
\par

(ii) Via the proof of (i), it suffices to prove the result in the case where $s$ and $s'$ are simple.  When $s$ and $s'$ are non-isomorphic both hom spaces are $0$, and there is nothing to check.  In the case where $s$ is isomorphic to $s'$ we can assume $s=s'$, at which point we have the isomorphism of $A$-bimodules
\[
\Hom_k(a,a')\cong \Hom_\msc{S}(s,s)\otimes_k \Hom_k(a,a')\to \Hom_\msc{S}(s\ot a,s\ot a').
\]
Taking central elements for the $A$-bimodule structures now provides the claimed isomorphism
\[
\Hom_{\msc{S}}(s,s)\otimes_k \Hom_{\msc{A}}(a,a')\overset{\sim}\to \Hom_{\msc{S}\otimes\msc{A}}(s\ot a,s\ot a').
\]
\end{proof}

At this point one should recall our construction of the additive completion from Section \ref{sect:add}.

\begin{lemma}\label{lem:753}
For $\msc{S}$ and $\msc{A}$ as in Lemma \ref{lem:708}, the structure map $\msc{S}\times\msc{A}\to \msc{S}\otimes \msc{A}$ induces an equivalence
\[
(\msc{S}\boxtimes\msc{A})^{\opn{add}}\overset{\sim}\to \msc{S}\otimes \msc{A}.
\]
Furthermore, when $\msc{S}$ and $\msc{A}$ are (braided) monoidal this equivalence is an equivalence of (braided) monoidal categories.
\end{lemma}

\begin{proof}
By Lemma \ref{lem:708} (a) the given functor is essentially surjective and by (b) it is fully faithful.  In the event that both $\msc{S}$ and $\msc{A}$ are (braided) monoidal the given functor admits a unique (braided) monoidal structure so that we have a diagram of (braided) monoidal functors
\[
\xymatrix{
	& \msc{S}\times \msc{A}\ar[dr]\ar[dl]\\
\msc{S}\boxtimes \msc{A}\subseteq (\msc{S}\boxtimes\msc{A})^{\opn{add}}\ar[rr] & & \msc{S}\otimes\msc{A}.
}
\]
\end{proof}

\subsection{Fullness of reduced concatenation}

By a \emph{monomial surface} in $\Bord_{\msc{S}\ot\msc{A}}$ we mean a surface which is in the image of the concatenation functor $\Bord_{\msc{S}\times\msc{A}}\to \Bord_{\msc{S}\ot\msc{A}}$, and by a \emph{monomial bordism} between such surfaces we mean a ribbon bordism which is in the image of concatenation.  So, a monomial surface $\Sigma_{\vec{z}}$ is one in which all the marking objects are of the form $z_i=s_i\ot x_i$, and a monoidal bordism $(M,T)$ is one in which all edges in the diagram $T$ are labeled by objects of the form $z_e=s_e\ot x_e$.
\par

We note that, for monomial surfaces $\Sigma_{\vec{z}}$ and $\Sigma_{\vec{w}}$, the space
\[
\Hom_{k\Bord_{\msc{S}\ot\msc{A}}^{\red}}(\Sigma_{\vec{z}},\Sigma_{\vec{w}})
\]
is spanned by monomial surfaces.  This follows by the redundant relation (U5) from Section \ref{sect:reduced} in conjunction with Lemma \ref{lem:708} (i).  Hence we observe the following.

\begin{lemma}\label{lem:837}
For $\msc{S}$ a semisimple symmetric tensor category, and $\msc{A}$ any ribbon tensor category, the reduced concatenation functor $\opn{ccat}^{\opn{red}}:k\Bord_{\msc{S}\times \msc{A}}^{\opn{red}}\to k\Bord_{\msc{S}\otimes\msc{A}}^{\red}$ is full.
\end{lemma}

\subsection{Proof of Proposition \ref{prop:conc}}
\label{sect:conc_proof}

We now prove that the reduced concatenation functor \eqref{eq:conc} is fully faithful.

\begin{proof}[Proof of Proposition \ref{prop:conc}]
We are free to assume $\msc{S}$ and $\msc{A}$ are small, and to take the explicit expression $\msc{S}\ot\msc{A}=(\msc{S}\boxtimes\msc{A})^{\opn{add}}$, by Corollary \ref{cor:bord_repl} and Lemma \ref{lem:753}.  Recall that objects in $(\msc{S}\boxtimes\msc{A})^{\opn{add}}$ are formal sums $\mbb{Z}_{\geq 0}\to \msc{S}\boxtimes\msc{A}$ which take only finitely many nonzero values, and recall also that $\msc{S}\boxtimes \msc{A}$ is equal to $\msc{S}\times \msc{A}$ at the level of objects.
\par

Take $k\Bord'_{\msc{S}\ot \msc{A}}$ the quotient of $k\Bord_{\msc{S}\ot \msc{A}}$ by the universal linear relations (U4) and the restrictive relation (R5) which expands over formal sums (and only formal sums!) along each edge, as in (U5).  We also institute a relation (R0) which takes any bordism $[M,T]$ in which an edge in $T$ is labeled by a zero object to the zero morphism.  (This final relation holds in $k\Bord_{\msc{S}\ot\msc{A}}^{\red}$ via internal composition and linear expansion.)
\par

Similarly, we let $k\Bord'_{\msc{S}\times\msc{A}}$ denote the quotient of $k\Bord_{\msc{S}\times \msc{A}}$ by the relations (U$'$4) and (R$'$0), where the latter relation takes any bi-labeled bordism $[M,T_{\msc{S}},T_{\msc{A}}]$ in which an edge in $T_{\msc{S}}$ or $T_{\msc{A}}$ is labeled by a zero object to the zero morphism.  We note that the concatenation functor reduces to a map $\opn{ccat}':k\Bord'_{\msc{S}\times \msc{A}}\to k\Bord'_{\msc{S}\otimes \msc{A}}$ which fits into a diagram
\[
\xymatrix{
k\Bord'_{\msc{S}\times \msc{A}}\ar[d]_{\opn{surj}}\ar[rr]^{\opn{ccat}'} & & k\Bord'_{\msc{S}\otimes \msc{A}}\ar[d]^{\opn{surj}}\\
k\Bord_{\msc{S}\times \msc{A}}^{\red}\ar[rr]_{\opn{ccat}^{\opn{red}}} & &  k\Bord_{\msc{S}\otimes \msc{A}}^{\red}.
}
\]
We claim that the map $\opn{ccat}'$ is fully faithful.
\par

By the same arguments employed in the reduced setting, using the relation (R5), one sees that morphisms between monomial surfaces in $k\Bord_{\msc{S}\ot\msc{A}}'$ are spanned by monomial bordisms.  Rather, we see that the functor $\opn{ccat}'$ is full, as in Lemma \ref{lem:837}.

Consider monomial surfaces $\Sigma_{\vec{z}}$ and $\Sigma_{\vec{w}}$, which we consider as objects in both $k\Bord_{\msc{S}\times\msc{A}}'$ and $k\Bord_{\msc{S}\otimes\msc{A}}'$ by an abuse of notation.  The linear space of bordisms
\[
\Hom_{k\Bord_{\msc{S}\otimes\msc{A}}'}(\Sigma_{\vec{z}},\Sigma_{\vec{w}})
\]
decomposes naturally
\[
\Hom_{k\Bord_{\msc{S}\otimes\msc{A}}'}(\Sigma_{\vec{z}},\Sigma_{\vec{w}})=\oplus_{d\in \mfk{D}}H_{\msc{S}\otimes \msc{A}}(\Sigma_{\vec{z}},\Sigma_{\vec{w}})_d
\]
over the collection $\mfk{D}$ of partially labeled monomial bordisms $(M,\bar{T})$, where $\bar{T}$ consists of the underlying geometry $\gamma:\Gamma\to M$ along with the edge labels $u:\opn{edge}(\Gamma)\to \msc{S}\times \msc{A}\subseteq \msc{S}\otimes\msc{A}$, which we require to be nonzero.  We similarly decompose morphisms in the $\msc{S}\times \msc{A}$-labeled category
\[
\Hom_{k\Bord_{\msc{S}\times\msc{A}}'}(\Sigma_{\vec{v}},\Sigma_{\vec{w}})=\oplus_{d\in \mfk{D}}H_{\msc{S}\times\msc{A}}(\Sigma_{\vec{v}},\Sigma_{\vec{w}})_d
\]
and note that the concatenation map
\begin{equation}\label{eq:886}
\opn{ccat}':\Hom_{k\Bord_{\msc{S}\times\msc{A}}'}(\Sigma_{\vec{v}},\Sigma_{\vec{w}})\to \Hom_{k\Bord_{\msc{S}\otimes\msc{A}}'}(\Sigma_{\vec{z}},\Sigma_{\vec{w}})
\end{equation}
preserves the given decompositions over $\mfk{D}$.
\par

Now, by fullness of the functor $\opn{ccat}'$, to establish full faithfulness it suffices to prove that on each component the map
\[
\opn{ccat}'_d:H_{\msc{S}\times\msc{A}}(\Sigma_{\vec{v}},\Sigma_{\vec{w}})_d\to H_{\msc{S}\otimes \msc{A}}(\Sigma_{\vec{z}},\Sigma_{\vec{w}})_d
\]
is injective.  Noting that each such component is now of finite-dimension, it then suffices to prove that the $d$-th components for $\msc{S}\otimes\msc{A}$ and $\msc{S}\times \msc{A}$ have the same dimension.
\par

Writing out $d=(M,\gamma,u:\opn{edge}(\Gamma)\to \msc{S}\times \msc{A})$, and taking $u_{\msc{S}}$ and $u_{\msc{A}}$ the projections onto the relevant components, we have
\[
\begin{array}{l}
\dim H_{\msc{S}\otimes\msc{A}}(\Sigma_{\vec{v}},\Sigma_{\vec{w}})_d=\sum_{v\in \opn{vert}(\Gamma)^{\opn{int}}}\dim\Hom_{\msc{S}\otimes \msc{A}}(u_{\opn{in}(v)},u_{\opn{out}(v)})\vspace{2mm}\\
=\sum_{v}\dim\Hom_{\msc{S}}(u_{\msc{S},\opn{in}(v)},u_{\msc{S},\opn{out}(v)})\cdot\dim\Hom_{\msc{A}}(u_{\msc{A},\opn{in}(v)},u_{\msc{A},\opn{out}(v)})\ \ \ (\text{Lemma \ref{lem:708} ii})\vspace{2mm}\\
=\sum_{v}\dim\opn{Bilin}_k\left(\Hom_{\msc{S}}(u_{\msc{S},\opn{in}(v)},u_{\msc{S},\opn{out}(v)})\times \Hom_{ \msc{A}}(u_{\msc{A},\opn{in}(v)},u_{\msc{A},\opn{out}(v)}),k\right)\vspace{2mm}\\
=\dim H_{\msc{S}\times \msc{A}}(\Sigma_{\vec{v}},\Sigma_{\vec{w}})_d
\end{array}
\]
via the linearity relations (U4) and (U$'$4).  From this identification of dimensions we see that each $\opn{ccat}'_d$ is a linear isomorphism, that each map \eqref{eq:886} is a linear isomorphism, and hence that the functor $\opn{ccat}'$ is fully faithful.
\par

We now have a fully faithful functor
\[
\opn{ccat}':k\Bord_{\msc{S}\times \msc{A}}'\to k\Bord_{\msc{S}\otimes\msc{A}}'
\]
which is an equivalence onto the full subcategory spanned by monomial surfaces in $k\Bord_{\msc{S}\otimes \msc{A}}'$.  In $k\Bord_{\msc{S}\otimes\msc{A}}'$ the internal composition relation can be instituted via the internal composition relation $[M,T]-[M,T']=0$ occurring exclusively at monomial bordisms $[M,T]$.  Via linear expansion at the vertices, we can furthermore restrict and apply this relation only to those bordism whose vertices are labeled by product morphisms $f_{\msc{S}}\otimes f_{\msc{A}}$.  It follows that the functor $\opn{ccat}'$ reduces to another fully faithful functor
\[
\opn{ccat}^{\opn{red}}:k\Bord_{\msc{S}\times\msc{A}}^{\opn{red}}\to k\Bord_{\msc{S}\otimes\msc{A}}^{\opn{red}}.
\]
\end{proof}

\bibliographystyle{abbrv}

\end{document}